\documentclass[12pt,leqno,a4paper]{memo-l}
\usepackage{amssymb,enumerate}
\usepackage{xypic}

\newcommand{\CC}{{\mathbb{C}}}
\newcommand{\FF}{{\mathbb{F}}}
\newcommand{\QQ}{{\mathbb{Q}}}

\newcommand{\ZZ}{{\mathbb{Z}}}

\newcommand{\bB}{{\mathbf{B}}}
\newcommand{\bG}{{\mathbf{G}}}
\newcommand{\bH}{{\mathbf{H}}}
\newcommand{\bL}{{\mathbf{L}}}
\newcommand{\bP}{{\mathbf{P}}}
\newcommand{\bS}{{\mathbf{S}}}
\newcommand{\bT}{{\mathbf{T}}}

\newcommand{\cB}{{\mathcal{B}}}
\newcommand{\cE}{{\mathcal{E}}}
\newcommand{\cF}{{\mathcal{F}}}
\renewcommand{\cH}{{\mathcal{H}}}
\newcommand{\cO}{{\mathcal{O}}}

\newcommand{\fS}{{\mathfrak{S}}}

\newcommand{\ad}{{\operatorname{ad}}}
\newcommand{\Arg}{{\operatorname{Arg}}}
\newcommand{\End}{{\operatorname{End}}}

\newcommand{\IBr}{{\operatorname{IBr}\,}}
\newcommand{\Irr}{{\operatorname{Irr}\,}}
\newcommand{\Hom}{{\operatorname{Hom}}}

\newcommand{\rad}{{\operatorname{rad}}}
\newcommand{\rnk}{{\operatorname{rank}}}
\newcommand{\GL}{{\operatorname{GL}}}
\newcommand{\PGL}{{\operatorname{PGL}}}
\newcommand{\SL}{{\operatorname{SL}}}
\newcommand{\GU}{{\operatorname{GU}}}
\newcommand{\PSU}{{\operatorname{U}}}
\newcommand{\Sp}{{\operatorname{Sp}}}
\newcommand{\GO}{{\operatorname{GO}}}
\newcommand{\SO}{{\operatorname{SO}}}
\newcommand{\St}{{\operatorname{St}}}
\newcommand{\hgt}{{\operatorname{ht}}}
\newcommand{\RLG}{{R_L^G}}
\newcommand{\sRLG}{{{}^*\!R_L^G}}
\newcommand{\RLM}{{R_L^M}}
\newcommand{\RTG}{{R_T^G}}

\newcommand{\tw}[1]{{}^#1\!}
\newcommand{\Ph}[1]{\Phi_{#1}}
\newcommand{\hlf}{\frac{1}{2}}
\newcommand{\thrd}{\frac{1}{3}}
\newcommand{\frth}{\frac{1}{4}}
\newcommand{\sxt}{\frac{1}{6}}
\newcommand{\egt}{\frac{1}{8}}
\newcommand{\Chevie}{{\sf Chevie}{}}
\newcommand\vr{{\vrule width14pt height3pt depth-2pt}\,}
\newcommand{\pl}{{\!+\!}}
\newcommand{\mn}{{\!-\!}}
\newcommand{\wtA}{{\widetilde A}}
\newcommand{\tPsi}{{\tilde\Psi}}
\newcommand\co{{\!\colon\!}}

\let\al=\alpha

\let\la=\lambda
\let\vhi=\varphi
\let\ze=\zeta
\let\wt=\widetilde
\let\sqtens=\boxtimes

\newtheorem{thm}{Theorem}[chapter]
\newtheorem{lem}[thm]{Lemma}
\newtheorem{prop}[thm]{Proposition}
\newtheorem{cor}[thm]{Corollary}
\newtheorem{ass}[thm]{Assumption}
\newtheorem{conj}[thm]{Conjecture}

\theoremstyle{remark}
\newtheorem{rem}[thm]{Remark}
\newtheorem{exmp}[thm]{Example}

\begin{document}
\frontmatter

%%%%%%%%%%%%%%%%%%%%%%%%%%%%%%%%%%%%%%%%%%%%%%%%%%%%%%%%%%%%%%%%%%%%%%%%%
\title[Decomposition matrices]{Decomposition matrices for groups of Lie type in non-defining characteristic}
%%%%%%%%%%%%%%%%%%%%%%%%%%%%%%%%%%%%%%%%%%%%%%%%%%%%%%%%%%%%%%%%%%%%%%%%%

\date{\today}

\author{Olivier Dudas}
\address{CNRS, Universit\'e Paris Diderot, Sorbonne Universit\'e, Institut de
  Math\'ematiques de Jussieu-Paris Rive Gauche, IMJ-PRG, F-75013, Paris, France.}
\email{olivier.dudas@imj-prg.fr}

\author{Gunter Malle}
\address{FB Mathematik, TU Kaiserslautern, Postfach 3049,
         67653 Kaisers\-lautern, Germany.}
\email{malle@mathematik.uni-kl.de}

\thanks{Both authors gratefully acknowledge financial support by SFB TRR 195.
  This first author also gratefully acknowledges financial support by
  ANR-16-CE40-0010-01.}

\keywords{decomposition numbers, classical groups, exceptional groups, unitary primes, unipotent blocks, Hecke algebras, Craven's conjecture}

\subjclass[2010]{20C20, 20C33}

\begin{abstract}
We determine approximations to the decomposition matrices for unipotent
$\ell$-blocks of several series of finite reductive groups of classical and
exceptional type over $\FF_q$ of low rank in non-defining good
characteristic~$\ell$.
%%{\bf[???]}
\end{abstract}

\maketitle

%\pagestyle{myheadings}
%%\markboth{}{}
%\markboth{for personal use only}{preliminary}

\tableofcontents

%%%%%%%%%%%%%%%%%%%%%%%%%%%%%%%%%%%%%%%%%%%%%%%%%%%%%%%%%%%%%%%%%%%%%%%%%
\section{Introduction} \label{sec:intro}

It is a fundamental problem in representation theory of finite groups to
construct all irreducible representations of the finite nearly simple groups,
or at least determine their dimensions. This question is open even for the
family of symmetric groups. While it was solved by Frobenius more than a
hundred years ago for representations over the complex numbers, it remains a
challenging problem when the characteristic of the base field is positive.
In this manuscript we consider certain families of finite quasi-simple groups
of Lie type. For these, the ordinary characters were classified by G.~Lusztig
in the 1980s (see \cite{Lu84}). For positive characteristic
representations it is open in general. It comes in two fundamentally
different flavours. In the defining characteristic case, the representations of
the finite group are closely related to the representations of the ambient
algebraic group, to which a highest weight theory applies, but still deep
problems remain. Here, we are concerned with the other, the cross
characteristic case. The main problem here can be phrased as finding the
$\ell$-modular decomposition matrices for the so-called unipotent blocks.

We obtain very close approximations to these decomposition matrices for all
finite groups of Lie type and Lie rank at most~6, and sometimes even beyond,
that is to say, we either determine the matrices completely or up to at most
a very small number of undetermined entries.

Our motivation for doing this is at least threefold. First, it is of interest
for many applications to know the dimensions of all irreducible modules of the
``small'' finite simple groups, or at least the few smallest such dimensions.
Secondly, we expect that the data obtained in this work may be useful
to derive and test conjectures for example on
\begin{itemize}
 \item the number and labels of cuspidal unipotent Brauer characters,
 \item the parameters of modular Hecke algebras,
 \item the subdivision of unipotent characters into $\ell$-modular
  Harish-Chandra series,
 \item and last not least on the shape of decomposition matrices and on
  individual decomposition numbers, for example Craven's conjecture
  \cite{Cra12}.
\end{itemize}
For example, the parametrisation of cuspidal Brauer characters and of the
distribution into Harish-Chandra series for the class of unitary groups was
predicted and proved based on similar data obtained in \cite{GHM}. Thirdly, we
expect further applications and connections based on the data obtained here.
In fact our previous results on unitary groups \cite{DM15} and on exceptional
groups \cite{DM16} have already been used by E.~Norton \cite{N19} to predict and
then prove the distribution of characters into modular Harish-Chandra series.
\par
Our results are mostly phrased under the assumption ($T_\ell$) that unipotent
$\ell$-blocks of finite reductive groups $G$ have uni-triangular decomposition
matrix. This is known to hold whenever $\ell$ and the defining characteristic of
$G$ are not too small, see \cite[Thm.~A]{BDT19}.
\medskip

The work is structured as follows. The first part contains definitions and
some general results. We recall some notions and results from the ordinary and
modular representation theory of finite reductive groups in
Chapter~\ref{chap:finred}. In Chapter~\ref{chap:hecke} we derive criteria for
when the parameters of modular Hecke algebras can be obtained from a
corresponding characteristic zero situation. The second part is devoted to
the determination of decomposition matrices in the various groups under
consideration. It is divided up according to the order $d=d_\ell(q)$ of the
size $q$ of the finite field over which our group is defined, modulo~$\ell$,
the characteristic of the representations we consider. We start off with
the most difficult case of $d=2$, for which the ranks of Sylow $\ell$-subgroups
are the largest. Here, in Section~2.\ref{sec:St} we prove a general result on
multiplicities in the $\ell$-modular reduction of the Steinberg character,
before we go on to treat the different families of finite quasi-simple reductive
groups. \par
In the subsequent Chapters~\ref{chap:d=3} and~\ref{chap:d=6} we consider
decomposition matrices for the cases $d=3$ and $d=6$ respectively. At the end
of each section we give an overview of the distribution into Harish-Chandra
series  observed in the blocks under consideration. Finally, in
Chapter~\ref{chap:d>=7} we discuss the $\ell$-modular Brauer trees when $d\ge7$.
In the last chapter we check that our results are in agreement with a
conjecture of Craven, and indicate in Proposition~\ref{prop:craven} which
further decomposition numbers would follow if this conjecture holds.

\medskip\noindent
{\bf Acknowledgement:} The computation of Harish-Chandra and Deligne--Lusztig
induction (and thus restriction) of unipotent characters was done in Jean
Michel's development version of the \Chevie{} system \cite{Mi15} using in
particular the command \textsc{LusztigInductionTable}. We thank him for having
provided this and various further very convenient functionalities.
Chapter~\ref{chap:craven} was written after an observation by Emily Norton. We
thank her for her valuable suggestion as well as for her help in the proof of
Proposition~\ref{prop:B6,d=6}.

\mainmatter
%%%%%%%%%%%%%%%%%%%%%%%%%%%%%%%%%%%%%%%%%%%%%%%%%%%%%%%%%%%%%%%%%%%%%%%%%
%%%%%%%%%%%%%%%%%%%%%%%%%%%%%%%%%%%%%%%%%%%%%%%%%%%%%%%%%%%%%%%%%%%%%%%%%
\part{Setting and general methods}
%%%%%%%%%%%%%%%%%%%%%%%%%%%%%%%%%%%%%%%%%%%%%%%%%%%%%%%%%%%%%%%%%%%%%%%%%
\vskip 2pc

%%%%%%%%%%%%%%%%%%%%%%%%%%%%%%%%%%%%%%%%%%%%%%%%%%%%%%%%%%%%%%%%%%%%%%%%%
\chapter{Finite groups of Lie type}   \label{chap:finred}

%%%%%%%%%%%%%%%%%%%%%%%%%%%%%%%%%%%%%%
\section{Basic sets and uni-triangularity}   \label{subsec:setting}
Let $\ell$ be a prime number and $(K,\cO,k)$ be an $\ell$-modular system
which we assume to be large enough for all the finite groups encountered.
More specifically, since we will be working with $\ell$-adic cohomology we
assume throughout this paper that $K$ is a finite extension of the field
$\QQ_\ell$ of $\ell$-adic numbers.
\par
Let $H$ be a finite group.
A representation of $H$ will always be assumed to be finite-dimensional.
The set of irreducible characters (resp.~irreducible Brauer characters)
will be denoted by $\Irr H$ (resp. $\IBr H$).
Given a simple $kH$-module $N$, we shall denote by $\vhi_N$ (resp.~$P_N$,
resp.~$\Psi_N$) its Brauer character (resp.~its projective cover, resp.~the
ordinary character of its projective cover). The restriction of an ordinary
character $\chi$ of $H$ to the set of $\ell'$-elements will be denoted by
$\chi^\circ$, and referred to as the \emph{$\ell$-modular reduction} of $\chi$
(or sometimes just the \emph{$\ell$-reduction} of $\chi$). It decomposes
uniquely on the family of irreducible Brauer characters of $H$ as
$$\chi^\circ = \sum_{\vhi\in\IBr\, H} d_{\chi,\vhi}\, \vhi.$$
The non-negative integral coefficients
$(d_{\chi,\vhi})_{\chi \in \Irr H, \vhi\in \IBr\, H}$ form the
\emph{$\ell$-modular decomposition matrix} of $H$. Equivalently, if
$\vhi=\vhi_N$ is the Brauer character of a simple $kH$-module $N$, then
the character of its projective cover satisfies
$\Psi_N = \sum_{\chi\in\Irr H} d_{\chi,\vhi}\,\chi$ by Brauer reciprocity.
\par
We denote by $\langle - ; - \rangle_H$ the usual inner product on the
space of $K$-valued class functions on $H$. If $N$ is a simple $kH$-module with
Brauer character $\vhi$ and $P$ a projective $kH$-module with ordinary
character $\Psi$ then
$\langle \Psi;\vhi \rangle = \dim \Hom_{kH}(P,N)$ gives the multiplicity of $P_N$
as a direct summand of $P$. 
\par
Recall that a \emph{basic set of characters} is a set $\cB$ of complex
irreducible characters of $H$ such that $\cB^\circ=\{\chi^\circ\mid\chi\in\cB\}$
is a $\ZZ$-basis of $\ZZ \IBr\, H$. This means that the restriction of the
decomposition matrix to $\cB$ is invertible over $\ZZ$.
\smallskip

According to Brauer, the ordinary and $\ell$-modular irreducible characters of
$H$ are subdivided into $\ell$-blocks such that the decomposition matrix is block
diagonal with respect to this partition. It then also makes sense to speak of
basic sets of individual $\ell$-blocks.
Now assume that we have a basic set $\cB$ for an $\ell$-block $b$ of $H$
ordered in such a way that the decomposition matrix of $b$ with respect to
$\cB$ is known \emph{a priori} to be lower uni-triangular. Then this
facilitates very much the explicit determination of this decomposition matrix
as explained in the following example.

\begin{exmp}   \label{exmp:uni}
Let $\cB=\{\rho_1,\ldots,\rho_r\}\subseteq\Irr b$ be a basic set of a block~$b$
of $H$ such that the decomposition matrix is known to be uni-triangular with
respect to this ordering, and let $\{\Psi_1,\ldots,\Psi_r\}$ be the
corresponding projective indecomposable modules (PIMs for short) in $b$. Then:
\begin{enumerate}[(a)]
\item If $\Psi$ is a projective character in $b$ that involves $\rho_i$ but no
 $\rho_j$ with $j<i$ then $\Psi_i$ is a direct summand of $\Psi$.
\item More generally, let $\Psi,\Psi'$ be two projective characters such that
 $\Psi-\Psi'$ involves $\rho_i$ with positive multiplicity, but no $\rho_j$
 with $j<i$. Then $\Psi_i$ is a direct summand of $\Psi$.
\end{enumerate}
Thus, under certain conditions known PIMs can be subtracted from given
projectives, and in this way a set of projective characters can be partially
echelonised.
\end{exmp}

%%%%%%%%%%%%%%%%%%%%%%%%%%%%%%%%%%%%%%
\section{Unipotent characters and unipotent blocks}   \label{subsec:finred}
Let $\bG$ be a connected reductive linear algebraic group over
$\overline{\FF}_p$, and $F:\bG\rightarrow\bG$ be a Frobenius endomorphism with
respect to an $\FF_q$-structure. (In particular, we do not consider Ree and
Suzuki groups here.)  We let $(\bG^*,F)$ be a Langlands dual group to $(\bG,F)$.
We let $G:=\bG^F$ and $G^*:=\bG^{*F}$ denote the associated finite reductive
groups. Throughout, whenever $\bH\le\bG$ is an $F$-stable subgroup we write
$H=\bH^F$.
\par
Recall from \cite[\S13]{DM91} that the irreducible characters of $G$ over $K$
fall into rational Lusztig series attached to conjugacy classes of semisimple
elements in $G^*$. Given $s$ a semisimple element of $G^*$ we denote by
$\cE(G,s)$ the corresponding Lusztig series. The \emph{unipotent characters} of
$G$ are by definition the elements of $\cE(G,1)$, and the \emph{unipotent
$\ell$-blocks} are the blocks that contain at least one unipotent character.
By Brou\'e--Michel (see \cite[Thm.~9.12]{CE}), the irreducible characters
lying in the unipotent blocks are exactly those lying in $\cup_t\cE(G,t)$, for
$t$ running over the $\ell$-elements of $G^*$.
\par
Unipotent characters were classified by Lusztig \cite{Lu84}. When working
with explicit characters later, we will use the notation in \cite{Ca} for
exceptional groups. For classical groups, we will denote by $[\lambda]$ a
unipotent character parametrised by a partition or a bipartition~$\lambda$.
\smallskip

\begin{ass}   \label{ass:ell}
 Throughout this work we shall make the following assumptions on~$\ell$:
 \begin{itemize}
  \item $\ell \neq p$ (non-defining characteristic),
  \item $\ell$ is \emph{very good} for $\bG$. (\emph{i.e.}, $\ell$ is good for
   $\bG$ and $\ell$ divides neither $|Z(\bG)/Z^\circ(\bG)|$ nor
   $|Z(\bG^*)/Z^\circ(\bG^*)|$).
 \end{itemize}
\end{ass}

Note that both conditions are inherited by all Levi subgroups of $\bG$, so
that inductive arguments can be applied.

In this situation, the unipotent characters lying in a given unipotent
$\ell$-block $b$ of $G$ form a basic set of this block (see \cite{GH91,Ge93}).
Consequently, the restriction of the decomposition matrix of $b$ to the
unipotent characters is invertible and the whole decomposition matrix of the
unipotent blocks can be recovered from these unipotent parts together with the
ordinary character table. This square matrix will be referred to as
the \emph{unipotent $\ell$-decomposition matrix} of $G$.
In particular every (virtual) unipotent character is a virtual projective
character, up to adding and removing some non-unipotent characters. In addition,
under our Assumption~\ref{ass:ell} on $\ell$, not only the parametrisation
of unipotent characters but also the decomposition matrix of the unipotent
$\ell$-blocks is independent of the isogeny type of $\bG$ (see \cite[Thm. 17.7]{CE}).
Therefore, we will not specify it in our statements and proofs.
\par
Furthermore, under these assumptions on $\ell$, the basic set is widely expected
to be uni-triangular. This means that one can order the set of unipotent
characters and unipotent Brauer characters in such a way that the unipotent
decomposition matrix has uni-triangular shape (see \cite[Conj.~3.4]{GH97}).
More precisely, in many of our proofs we will assume the following property
of a finite group of Lie type $G$ at a prime $\ell$:
\begin{itemize}
 \item[$(T_\ell)$] \emph{The decomposition matrix of the unipotent
  $\ell$-blocks of $G$ has a uni-triangular shape with respect to any total
  ordering of the unipotent characters compatible with the order on families.}
\end{itemize}
This has recently been shown when $p$ is a good prime for $\bG$ (see
\cite[Thm.~A]{BDT19}).

The unipotent $\ell$-blocks of the finite reductive groups are known by the
work of Fong--Srinivasan \cite{FS}, Brou\'e--Malle--Michel \cite{BMM} and
Cabanes--Enguehard \cite{CE94}. We will use these results throughout without
further reference.

%%%%%%%%%%%%%%%%%%%%%%%%%%%%%%%%%%%%%%
\section{Deligne--Lusztig theory}   \label{ssec:dltheory}

We fix an $F$-stable maximal torus $\bT$ of $\bG$ contained in an $F$-stable
Borel subgroup $\bB$ of $\bG$. Such a torus is said to be \emph{maximally split}
(or sometimes \emph{quasi-split} as in \cite{DM91}).
Given an $F$-stable Levi complement $\bL$ of a parabolic subgroup $\bP$ of $\bG$,
we denote by $\RLG$ and $\sRLG$ the Deligne--Lusztig induction and restriction
maps
$$\RLG : \ZZ \Irr_\Lambda L \longrightarrow \ZZ \Irr_\Lambda G \quad
\text{and}\quad
  \sRLG  : \ZZ \Irr_\Lambda G\longrightarrow \ZZ \Irr_\Lambda L$$
where $\Lambda$ is the field $K$ or $k$. We will only use them when 
in a situation where they do not depend on $\bP$, which justifies our
notation. We refer to \cite{BoMi11} for conditions ensuring the independence
from the parabolic subgroup.
When $\bP$ can be chosen to be $F$-stable --- in which case we will say that
$\bL$ is a \emph{$1$-split} Levi subgroup of $\bG$ --- these coincide with the
Harish-Chandra induction and restriction map induced by the Harish-Chandra
induction and restriction functors. These will be also denoted by $\RLG$ and
$\sRLG$.
\smallskip

Let $t$ be a semisimple element of $G^*$.
Assume that $\bL^*:=C_{\bG^*}(t)$ is a Levi subgroup of $\bG^*$, in duality
with an $F$-stable Levi subgroup $\bL$ of $\bG$. Since $t\in Z(\bL^*)^F$, the
element $t$ corresponds to a linear character of $L=\bL^F$, which we will
denote by $\hat t$. By \cite[Thm.~13.25]{DM91} there is a sign
$\varepsilon\in\{\pm1\}$ which depends only on $\bL$, $\bG$ and $F$ such that
the maps
$$\cE(L,1) \xrightarrow{- \otimes \hat t} \cE(L,t)
  \xrightarrow{\varepsilon \RLG} \cE(G,t)$$
are bijections. When $t$ is an $\ell$-element and $\rho$ a unipotent character
of $L$, the irreducible character $\varepsilon \RLG(\rho \otimes \hat t)$ lies
in a unipotent block of $G$. Consequently, its $\ell$-reduction can be written
uniquely as a linear combination of the $\ell$-reductions of unipotent
characters, which form a basic set.
On the other hand, the $\ell$-reduction of $\hat t$ is trivial, therefore
by \cite[Prop. 12.2]{DM91} the character $\varepsilon \RLG(\rho\otimes\hat t)$
has the same $\ell$-reduction as the virtual unipotent character
$\varepsilon \RLG(\rho)$. This will prove to be a powerful tool to derive
relations on the decomposition numbers of $G$, as shown in
Example~\ref{exmp:use (RED)}.

\begin{lem}   \label{lem:red}
 Let $t$ be a semisimple $\ell$-element of $G^*$ such that $C_{\bG^*}(t)= \bL^*$
 is a Levi subgroup of $\bG^*$ with dual $\bL\le\bG$. Then for every unipotent
 character $\rho$ of $L$, $\varepsilon \RLG(\rho)^\circ$ is a non-negative
 linear combination of irreducible Brauer characters.
\end{lem}

%\begin{proof}
%Since $s$ is an $\ell$-element, the character $\hat s$ is a linear character
%of $\ell$-power order, and as such its $\ell$-modular reduction is trivial.
%Now $\RLG$ commutes with the decomposition maps by \cite[??]{BR03}, so we
%deduce that the $\ell$-reduction of $\varepsilon \RLG(\rho)$ coincides with
%the $\ell$-reduction of $\varepsilon \RLG(\rho \otimes\hat s)$. Since
%$\varepsilon \RLG(\rho \otimes\hat s)$ is an irreducible character of $G$,
%its $\ell$-reduction is a non-negative linear combination of irreducible Brauer
%characters.
%\end{proof}

\begin{exmp}   \label{exmp:use (RED)}
Let $G = \PGL_2(q)$ and hence $G^* = \SL_2(q)$. The unipotent characters of $G$
are $\St_G$ and $1_G$. There are semisimple element $t \in G^*$ with
eigenvalues $\{\lambda,\lambda^{-1}\}$ for every $\lambda \in \FF_{q^2}$
satisfying $\lambda^q = \lambda^{-1}$. When $\lambda \neq \pm 1$, such an
element is regular. If $\ell \mid (q+1)$ is an odd prime number, then one can
take $\lambda$ to be a non-trivial $\ell$-element. The corresponding semisimple
element $t$ is a regular $\ell$-element of $G^*$, therefore
$\bL^* = C_{\bG^*}(t)$ is a maximal torus of $\bG^*$.
By Lemma~\ref{lem:red}, the $\ell$-modular reduction of the virtual character
$-R_L^G(1_L) = \St_G - 1_G$ is a
non-negative combination of irreducible Brauer characters.
This shows that when $\ell$ is an odd prime number dividing $q+1$, then the
$\ell$-reduction of $\St_G$ contains at least one copy of the trivial Brauer
character and the unipotent $\ell$-decomposition matrix has the following shape
$$\begin{array}{c|cc} 1_G & 1 & \cdot \\
	\St_G & \alpha & 1 \\ \end{array}$$
with $\alpha \geq 1$ (the coefficient $1$ in the bottom right corner comes
from the fact that this square matrix is invertible since the unipotent
characters form a basic set for the unipotent blocks). We have
$-R_L^G(1_L)^\circ = (\alpha-1)\varphi_{k} + \varphi_{\St}$.
\end{exmp}

Let $W:=N_\bG(\bT)/\bT$ denote the Weyl group of $\bG$. It is a finite Coxeter
group with distinguished set of generators $S$ and associated length
function~$l$. Recall from \cite[Prop.~3.3.3]{Ca} that the $G$-conjugacy classes
of $F$-stable maximal tori of $\bG$ are parame\-trised by the $F$-conjugacy
classes of $W$. Given $w \in W$, we will denote by $\bT_w$ a torus of type $w$
with respect to $\bT$ and we will write $R_w := R_{T_w}^G(1_{T_w})$. This is a
virtual character of $G$ all of whose constituents are unipotent characters.

As pointed out above, under our Assumption~\ref{ass:ell} the unipotent
characters form a basic set for the unipotent blocks. Therefore $R_w$ can be
seen as the unipotent part of a virtual projective module $\widetilde R_w$. The
decomposition of $\widetilde R_w$ on the basis of characters of PIMs can be
read off from the decomposition of $R_w$ on the family of unipotent parts of
characters of PIMs. Given a simple $kG$-module $N$, the \emph{coefficient of
$\Psi_N$ on $R_w$} will refer to the coefficient of $\Psi_N$ on
$\widetilde R_w$, which is given by $ \langle \widetilde R_w ; \varphi_N \rangle$.
The following proposition gives some control on the sign of this coefficient
(see \cite[Prop.~1.5]{Du13}).

\begin{prop}   \label{prop:dl}
 Let $w \in W$. If $\Psi_N$ does not occur in $R_v$ for any $v < w$ in the
 Bruhat order on $W$ then the coefficient of $\Psi_N$ in $(-1)^{l(w)}R_w$
 is non-negative.
\end{prop}

\begin{exmp}
As in Example~\ref{exmp:use (RED)}, we work with $G= \PGL_2(q)$ and an odd
prime number $\ell$ dividing $q+1$. There are two elements in the Weyl group
of $G$, namely the trivial element $e$ and the simple reflection $s$.
The Deligne--Lusztig characters are $R_e = 1_G+\St_G$ and $R_s = 1_G - \St_G$.
With the decomposition matrix given in Example \ref{exmp:use (RED)} they
decompose on the basis of PIMs as
$$\widetilde R_e = \Psi_k + (1-\alpha)\Psi_{\St} \qquad \text{and} \qquad
  \widetilde R_s = \Psi_k - (1+\alpha)\Psi_{\St}.$$
By Proposition~\ref{prop:dl} we have $1-\alpha \geq 0$, which forces
$\alpha =1$. Therefore $\Psi_{\St}$ does not occur in $R_e$; it has a cuspidal
head and $\Psi_{\St}$ occurs with multiplicity $1+\alpha = 2$ in $-R_s$, a
non-negative number, as predicted by Proposition~\ref{prop:dl}.
The unipotent $\ell$-decomposition matrix is hence
$$  \begin{array}{c|cc} 1_G & 1 & \cdot \\
	\St_G & 1 & 1 \\\hline 
	& ps & c\\ \end{array}$$
The projective character corresponding to the first column lies in the principal
series (indicated by the label ``ps''), while the second one is cuspidal
(and therefore labelled~``c'').
\end{exmp}

If $w \in W$ is such that $wF$ is a $d$-regular element (in the sense of
Springer \cite{Sp74}) for some $d\ge1$ then $\bT_w$ contains a (Sylow)
$\Phi_d$-torus $\bS$ of $\bG$ such that $C_\bG(\bS) = \bT_w$. For $d$ the order
of $q$ modulo~$\ell$, under suitable conditions on the $\ell$-part
$\Phi_d(q)_\ell$ of the $d$th cyclotomic polynomial $\Phi_d$ evaluated at $q$,
the dual torus $(\bS^*)^F$ contains an $\ell$-element $t$ such that
$C_{\bG^*}(t) = \bT_w^*$ (see Proposition~\ref{prop:q+1 reg}
for the case $d=2$ and the examples below). In that case by Lemma~\ref{lem:red},
$(-1)^{l(w)}(R_w)^\circ$ is a non-negative linear combination of irreducible
Brauer characters.

\begin{exmp}   \label{exmp:reg}
(a) Let $G$ be a group of type $C_n(q)$. Then a regular semisimple element
of $G$ is an element of $\bG$ with $2n$ distinct eigenvalues of the form
$\{\la_1^{\pm 1},\ldots,\la_n^{\pm 1}\}$ in the natural
$2n$-dimensional matrix representation that are permuted under the Frobenius
map $\la \longmapsto \la^q$. If the eigenvalues are $\ell$-elements
and $d$ is the order of $q$ modulo $\ell$, then each eigenvalue has an orbit
under $F$ of size $d$, therefore $d$ must divide $2n$. In that case regular
$\ell$-elements of $G$ exist whenever $\Phi_d(q)_\ell> 2n$.
\smallskip

(b) Let $G$ be a group of type $D_n(q)$. There are two types of regular
semisimple elements of odd order: elements of $\bG$ with $2n$ distinct
eigenvalues of the form $\{\la_1^{\pm 1},\ldots,\la_n^{\pm 1}\}$,
or elements with $2n-2$ distinct eigenvalues
$\{\la_1^{\pm 1},\ldots,\la_{n-1}^{\pm 1}\}$ and two eigenvalues both
equal to~$1$. If such an element is $F$-stable then $d$, the order of $q$
modulo~$\ell$, divides $2n$ or~$2n-2$ respectively. It can be chosen to be an
$\ell$-element provided that $\Phi_d(q)_\ell>2n$ or $\Phi_d(q)_\ell>2n-2$
respectively.
\smallskip

(c) Let $G$ be a group of type $\tw2D_n(q)$ with $n$ odd. Then $w_0 F$ acts
as $-q$ on the set of characters and cocharacters of $\bT$. Let $e$ be the
order of $-q$ modulo $\ell$. Using Ennola duality we deduce from (b) that when
$e$ divides $2n$ and $\Phi_e(q)_\ell > 2n$ (resp.~when $e$ divides $2n-2$
and $\Phi_e(q)_\ell > 2n-2$) then there exists a regular $\ell$-element
of $G$.
\smallskip

(d) Let $\bS$ be an $F$-stable torus of $\bG^*$ with $S=\bS^F$ of order
$\Phi_d(q)$ (a $\Phi_d$-torus of rank $1$). If $\ell$ is good and does not divide
the order of $(Z(\bG^*)/Z^\circ(\bG^*))^F$ then by \cite[Lemma~13.17]{CE},
the $d$-split Levi subgroup $\bL^* = C_{\bG^*}(\bS)$ of $\bG^*$ is also equal
to $C_{\bG^*}(S_\ell)$, the centraliser of a Sylow $\ell$-subgroup $S_\ell$
of~$S$. Consequently, if $t \in S_\ell$ is any generator of the cyclic group
$S_\ell$ then $\bL^* = C_{\bG^*}(t)$.
\end{exmp}

%%%%%%%%%%%%%%%%%%%%%%%%%%%%%%%%%%%%%%%%%%%%%%%%%%%%%%%%%%%%%%%%%%%%%%%%%
\chapter{Hecke algebras}   \label{chap:hecke}
Let $\Lambda$ be one of $K$, $\cO$ or $k$. Throughout this chapter, let $\bL$ be
a 1-split Levi subgroup of $\bG$, that is, an $F$-stable Levi complement of an
$F$-stable parabolic subgroup of $\bG$. In this case the maps $\RLG$ and $\sRLG$
introduced in \S 1.\ref{ssec:dltheory} are induced by the Harish-Chandra
induction and restriction functors, which we still denote by $\RLG$ and $\sRLG$.

Throughout this section we shall assume that $Z(\bG)$ is connected so 
that the results in \cite[\S8]{Lu84} apply.

%%%%%%%%%%%%%%%%%%%%%%%%%%%%%%%%%%%%%%
\section{Reduction stability}
A $\Lambda L$-module $X$ is \emph{cuspidal} if it is killed under every proper
Harish-Chandra restriction. In that case the algebra
$\cH_G(L,X) := \End_{\Lambda G}(\RLG(X))$ has a very specific structure. When
$\Lambda = K$ and $X$ is simple, Lusztig showed in
\cite[\S8.6]{Lu84} that the group $W_G(L,X):= N_G(\bL,X)/L$ admits a structure of
a Coxeter system and $\cH_G(L,X)$ is an Iwahori--Hecke algebra associated to
$W_G(L,X)$. More precisely, if $S_G(L,X)$ denotes the set of simple reflections
associated to the Coxeter structure then  $\cH_G(L,X)$ has a $K$-basis
$\{T_w\}_{w \in  W_G(L,X)}$ satisfying, for all $s \in S_G(L,X)$ and
$w \in W_G(L,X)$
$$T_s T_{w} = \begin{cases}  T_{sw}  & \text{if } sw > w, \\
  (q_s-1)T_{sw} + q_sT_w & \text{otherwise}. \end{cases}$$
In addition the parameters $\{q_s\}_{s \in S_G(L,X)}$ can be computed
explicitly. They are actually already
determined by the Hecke algebras of the cuspidal module $X$ inside the minimal
Levi overgroups of $L$ in $G$.  When $\Lambda = k$, the endomorphism algebra
was studied for example in \cite{GHM}; is can still be shown to be closely
related to an Iwahori--Hecke algebra, but much less is known about the
parameters. The best situation occurs  when the normaliser of the cuspidal
lattice is invariant under change of scalars, as studied in
\cite[\S2.6]{GeThesis}. More precisely, given an $\cO L$-lattice $X$ we will say
that  $\RLG(X)$ is \emph{reduction stable} if
\begin{itemize}
 \item[(1)] $KX$ is irreducible, and
 \item[(2)] $N_G(\bL,X) = N_G(\bL,KX)=N_G(\bL,kX)$.
\end{itemize}
In that case $\cH_G(L,kX)$ is an Iwahori--Hecke algebra whose
parameters are the $\ell$-reduction of the parameters of the Iwahori--Hecke
algebra $\cH_G(L,KX)$.

\begin{prop}[Geck]   \label{prop:redstable-geck}
 Let $X$ be a cuspidal $\cO L$-module. If $\RLG(X)$ is reduction stable then
 there is an $\cO$-basis $\{T_w\}_{w \in  W_G(L,X)}$ of $\cH_G(L,X)$ endowing it
 with a structure of an Iwahori--Hecke algebra.

 Furthermore, if $\Lambda$ is one of $K$ or $k$ then
 $\{1_\Lambda \otimes_\cO T_w\}_{w \in W_G(L,X)}$ is a $\Lambda$-basis of
 $\cH_G(L,\Lambda X)$.
\end{prop}

The standard setup for reduction stability is when both $KX$ and $kX$ are simple
modules. In that case it is enough to check that $N_G(\bL,KX) = N_G(\bL,kX)$. We
will often need to work with a slightly more general setup.

\begin{prop}   \label{prop:redstab}
 Let $X$ be a non-zero cuspidal $\cO L$-lattice. We assume that
 \begin{itemize}
 % \item[()] $KX$ is a simple module;
  \item[(1)] the head $Y$ of $kX$ is a simple module;
  \item[(2)] $N_G(\bL,KX)=N_G(\bL,Y)$; and
  \item[(3)] $\Hom_{kL}({}^w (kX), \rad(kX))=0$ for all $w \in N_G(\bL)$.
 \end{itemize}
 Then $R_L^G(X)$ is reduction stable. Furthermore, $\cH_G(L,kX)\simeq \cH_G(L,Y)$.
\end{prop}

\begin{proof}
First note that $N_G(\bL,X)\subset N_G(\bL,kX)\subset N_G(\bL,Y) = N_G(\bL,KX)$
where the second inclusion comes from the fact that the head of $kX$ is the
simple module $Y$. Therefore we only need to show that
$N_G(\bL,KX)\subset N_G(\bL,X)$ to prove the reduction stability.

Let $w \in N_G(\bL)$. Consider the short exact sequence of $kL$-modules
$$ 0\longrightarrow \rad(kX) \longrightarrow kX\longrightarrow Y
  \longrightarrow 0.$$
Since $\Hom_{kL}({}^w (kX), \rad(kX))=0$ by (4), the covariant functor
$\Hom_{kL}({}^w(k X),-)$ gives an injective map
$$ \psi_w:\,\Hom_{kL}({}^w(kX),kX) \hookrightarrow \Hom_{kL}({}^w(kX),Y).$$
On the other hand, with the head of ${}^w(kX)$ being ${}^w Y$ we have a natural
isomorphism 
$$ \phi_w :\, \Hom_{kL}({}^w Y,Y) \simeq \Hom_{kL}({}^w(kX),Y)$$
induced by the map ${}^w(kX) \twoheadrightarrow {}^w Y$. This shows that
$\Hom_{kL}({}^w(kX),kX)$ has dimension at most~$1$. Now let $w \in N_G(\bL,KX)$.
Since $K\Hom_{\cO L}({}^wX,X) \simeq \Hom_{KL}({}^w(KX),KX)$ we must have that
$\Hom_{\cO L}({}^wX,X)$ is non-zero. On the other hand, the natural map 
$$k\Hom_{\cO L}({}^wX,X) \hookrightarrow \Hom_{kL}({}^w(kX),kX)$$ 
is an embedding. By the previous observation on the dimension of
$\Hom_{kL}({}^w(kX),kX)$, we deduce that it must be an isomorphism and that
$\Hom_{kL}({}^w(kX),kX)\simeq k$. 
Consequently $\psi_w$ is also an isomorphism. In particular, any non zero
morphism from ${}^w(kX)$ to $kX$ will send the head of ${}^w kX$ to the head of
$kX$ and therefore must be an isomorphism. This shows that ${}^w(kX) \simeq kX$. 
In addition, since $k\Hom_{\cO L}({}^wX,X) \simeq \Hom_{kL}({}^w(kX),kX)$ then
we also  have ${}^w X \simeq X$ by Nakayama's lemma. This proves that $\RLG(X)$
is reduction stable. Note that the fact that $\End_{kL}(kX)$ has dimension $1$
forces $KX$ to be a simple $KL$-module.
\smallskip

Note that if $w \notin N_G(\bL,Y)$ then $\psi_w$ is the zero map, therefore it is
also an isomorphism in that case. Let $\pi : kX \twoheadrightarrow Y$.
By adjunction and the Mackey formula, the natural map
$\End_{kG}(\RLG(kX)) \longrightarrow \Hom_{kG}(\RLG(kX),\RLG(Y))$
induced by $\RLG(\pi)$ is an isomorphism since for all $w\in N_G(\bL)$
each map $\psi_w : \Hom_{kL}({}^w(kX),kX)  \longrightarrow \Hom_{kL}({}^w(kX),Y)$
is an isomorphism. The same holds for the natural map 
$\End_{kG}(\RLG(Y))\longrightarrow\Hom_{kG}(\RLG(kX),\RLG(Y))$ since 
$\phi_w$ is an isomorphism for all $w \in N_G(\bL)$.
The combination of the two gives the asserted isomorphism
$\End_{kG}(\RLG(kX)) \simeq \End_{kG}(\RLG(Y))$.
\end{proof}

%%%%%%%%%%%%%%%%%%%%%%%%%%%%%%%%%%%%%%
\section{Embedding of decomposition matrices}
Let $X$ be a cuspidal simple $kL$-module.
Then the simple representations of the algebra $\cH_G(L,X)$ encode the simple
$kG$-modules occurring in the head of $\RLG(X)$, that is the simple modules
lying in the Harish-Chandra series above the cuspidal pair $(L,X)$. In addition,
when $X$ comes from a reduction stable $\cO L$-lattice $\widetilde X$, one can 
compute the decomposition of $\RLG(P_X)$ into PIMs using the decomposition
matrix of $\cH_G(L,\widetilde X)$, as explained in~\cite[\S3]{DGHM}. This gives a
powerful method to compute the decomposition numbers of $G$ for PIMs whose heads
lie in the series above $(L,X)$. As in the previous section, we explain how to
generalise this method when $X$ is no longer simple. 

\begin{prop}\label{prop:dechecke}
 Let $X$ be a cuspidal $\cO L$-lattice. We assume that
 \begin{itemize}
 % \item[()] $KX$ is a simple module;
  \item[(1)] the head $Y$ of $kX$ is a simple module;
  \item[(2)] $N_G(\bL,KX)=N_G(\bL,Y)$; and
  \item[(3)] $\Hom_{kL}({}^w P_Y, \rad(kX))=0$ for all $w \in N_G(\bL)$.
 \end{itemize}
 Then $R_L^G(X)$ is reduction stable, $\cH_G(L,kX) \simeq \cH_G(L,Y)$ and
 the decomposition matrix of $\cH_G(L,X)$ embeds into that of $G$.
\end{prop}

\begin{proof} Let $w \in N_G(\bL)$. Condition~(3) implies that ${}^w Y$ is not a
composition factor of $\rad(kX)$ which shows in particular that
$\Hom_{kL}({}^w (kX), \rad(kX))=0$. Therefore the conditions of
Proposition~\ref{prop:redstab} are satisfied and we deduce that $R_L^G(X)$ is
reduction stable and $\cH_G(L,kX) \simeq \cH_G(L,Y)$.

To show the statement on the decomposition matrices we only need to check the
condition (D) given in \cite[4.1.13, 4.1.14]{GJ11}. Since the head of $kX$ is
simple, equal to $Y$, we have a surjective map $P_Y\twoheadrightarrow kX$, which
induces a surjective map $\RLG(P_Y) \twoheadrightarrow \RLG(kX)$.
Let $\widetilde P_Y$ be the projective $\cO L$-module which is the (unique) lift 
of $P_Y$, and let $P:= \RLG(\widetilde P_Y)$. Since $P$ is projective we have a
map $P \longrightarrow \RLG(X)$ which lifts the surjective map
$\RLG(P_Y) \twoheadrightarrow \RLG(kX)$. By Nakayama's lemma it should also be
surjective. Now condition (D) is equivalent to
$$ \langle KP;\RLG(KX)\rangle_G = \langle \RLG(KX);\RLG(KX)\rangle_G.$$
To verify it we use the condition (3), the Mackey formula and cuspidality. We
have
$$ \langle KP;\RLG(KX)\rangle_G = \langle \RLG(KP_Y);\RLG(KX)\rangle_G
   = \sum_{w \in N_G(\bL)/L} \langle KP_Y; {}^wKX\rangle_G.$$
Now $\langle KP_Y; {}^wKX\rangle_G$ equals the multiplicity of $Y^w$ as a
composition factor of $kX$,
which by (3) is also the multiplicity of $Y^w$ in $Y = kX/\rad(kX)$. It is
therefore $1$ if $w \in N_G(\bL,Y)$ or $0$ otherwise. By (2) we have
$N_G(\bL,Y) = N_G(\bL,KX)$ hence
$\langle KP_Y; {}^wKX\rangle_G = \langle KX; {}^wKX\rangle_G$ and we are done. 
\end{proof}

For convenience we state a version of Proposition~\ref{prop:dechecke} when
$N_G(\bL,kX)$ is as big as possible, in which case condition (3) becomes simpler.

\begin{cor}   \label{cor:dechecke}
 Let $X$ be a cuspidal $\cO L$-lattice in a block $b$ of $\cO L$ such that
 \begin{itemize}
  \item[(1)] the head $Y$ of $kX$ is a simple module;
  \item[(2)] $N_G(\bL,Y) = N_G(\bL,KX) = N_G(\bL,b)$; and
  \item[(3)] $Y$ occurs only once as a composition factor of $kX$.
 \end{itemize}
 Then $R_L^G(X)$ is reduction stable, $\cH_G(L,kX) \simeq \cH_G(L,Y)$ and
 the decomposition matrix of $\cH_G(L,X)$ embeds into that of $G$.
 \end{cor}

\begin{exmp}   \label{exmp:hecke}
Let $\chi$ be a cuspidal ordinary irreducible character of a 1-split Levi
subgroup $L$ of $G$.
\begin{enumerate}[(a)]
 \item Assume that $\chi$ lies in a block $b$ with cyclic defect groups.
 Let $Y$ be any $kL$-composition factor of the $\ell$-reduction of $\chi$.
 Then by Green \cite{Gr74} there exists an $\cO L$-lattice $X$ with character
 $\chi$ such that $kX$ is uniserial, has mutually distinct composition factors,
 and has $X$ as its head (the various composition factors are labelling the
 edges incident to the vertex labelled by $\chi$ in the Brauer tree of $b$).
 Then the assumptions of Corollary \ref{cor:dechecke} are for example
 satisfied whenever $N_G(\bL,b)$ acts trivially on the Brauer tree and on
 the character $\chi$.
\item Assume that the $\ell$-modular reduction of $\chi$ has only two
 composition factors, say $Y$ and $Z$, and that $Y$ is self-dual.
 Then there exists an $\cO L$-lattice $X$
 with character $\chi$ or $\chi^*$ such that $kX$ is uniserial with head $Y$.
 In that case it is enough to check assumption~(2) of
 Corollary~\ref{cor:dechecke}. Note that it can be checked equivalently 
 on $Y$ or $Z$.
\item More generally, assume that $Y$ is a simple $kL$-module occurring in the
 $\ell$-modular reduction of $\chi$ with multiplicity one. In other words we
 assume that $\langle \Psi_Y;\chi\rangle_L = 1$. Denote by $\widetilde{P}_Y$
 a projective $\cO L$-module lifting $P_Y$. Let
 $$e := \sum_{\rho \in \Irr_K b \smallsetminus \{\chi\}} e_\rho$$
 where $e_\rho$ is the central idempotent associated to the irreducible
 character $\rho$. Then $N:= e \widetilde{P}_Y \cap \widetilde{P}_Y$ is a pure
 $\cO L$-submodule of $\widetilde{P}_Y$ and $X:= \widetilde{P}_Y / N$ is an
 $\cO L$-lattice with character $\chi$ such that $kX$ has simple head $Y$. In
 particular, $X$ satisfies conditions~(1) and (3) of
 Corollary~\ref{cor:dechecke}.
\end{enumerate}
\end{exmp}

%%%%%%%%%%%%%%%%%%%%%%%%%%%%%%%%%%%%%%
\section{Verifying reduction stability}   \label{sec:veri}

Let us comment on how to guarantee the assumptions of
Corollary~\ref{cor:dechecke} in certain situations. In this work, we will solely
be concerned with unipotent blocks of finite reductive groups $G=\bG^F$. 
By \cite[Thm. 17.7]{CE} unipotent characters and unipotent Brauer characters are
insensitive to the centre of the group whenever $\ell$ is very good, therefore
we may assume
that $\bG$ has connected centre. Now let $\bL$ be an $F$-stable Levi subgroup
with a cuspidal $\cO L$-lattice $X$ such that the head of $kX$ is simple.
If $X$ lies in a unipotent block then it is trivial on $Z(L)$.
Now $N_G(\bL)=N_\bG(\bL)^F$ induces algebraic automorphisms of $\bL$, hence
inner, diagonal and graph automorphisms. If $\bG$ has connected centre, then so
has $\bL$, and all diagonal automorphisms of $L$ are inner; hence $N_\bG(\bL)^F$
induces graph automorphisms on $L$. Then there exists an $N_G(\bL)$-stable
$\cO L$-lattice $X'$ with $kX'=kX$ if one of the following holds:
\begin{itemize}
 \item $[\bL,\bL]$ does not have non-trivial graph automorphisms;
 \item $[\bL,\bL]$ is a product of simple factors regularly permuted by the
  graph automorphisms induced by $N_G(\bL)$ (since then we may choose a
  lift in one of the factors and then take the product over the
  $N_G(\bL)$-orbit).
\end{itemize}
This will deal with most of the cases we encounter. It thus remains to discuss
situations in which $N_\bG(\bL)^F$ induces non-trivial graph automorphisms on
a simple factor of $\bL^F$.

%%%%%%%%%%%%%%%%%%%%%%%%%%%%%%%%%%%%%%%%%%%%%%%%%%%%%%%%%%%%%%%%%%%%%%%%%
%%%%%%%%%%%%%%%%%%%%%%%%%%%%%%%%%%%%%%%%%%%%%%%%%%%%%%%%%%%%%%%%%%%%%%%%%
\part{Decomposition matrices}
%%%%%%%%%%%%%%%%%%%%%%%%%%%%%%%%%%%%%%%%%%%%%%%%%%%%%%%%%%%%%%%%%%%%%%%%%

%%%%%%%%%%%%%%%%%%%%%%%%%%%%%%%%%%%%%%%%%%%%%%%%%%%%%%%%%%%%%%%%%%%%%%%%%
\chapter{Description of the strategy}   \label{chap:strategy}

We keep the notation and setup from the previous chapter. In particular,
$\bG$ is a connected reductive linear algebraic group, $F:\bG\rightarrow\bG$
is a Frobenius endomorphism with respect to an $\FF_q$-rational structure
and $G=\bG^F$ is the corresponding finite group of Lie type.

In our proofs we shall use the following tools, most of which already served
well in our papers \cite{DM15,DM16}:

\begin{itemize}
\item[(HCi)] Harish-Chandra inducing PIMs from proper Levi subgroups and cutting
  by a block of $G$ gives projective characters, which are hence non-negative
  integral   linear combinations of PIMs of $G$. Similarly, projective
  characters can also be
  obtained by Harish-Chandra restricting projectives from a larger group
  containing $G$ as a Levi subgroup, or by a succession of such steps.
\item[(HCr)] If $\Psi$ is a projective character of $G$ such that no non-zero
  proper subcharacter has the property that its Harish-Chandra restriction to
  any Levi subgroup $L$ of $G$ decomposes non-negatively on the PIMs of $L$,
  then $\Psi$ is the character of a PIM.
\item[($\cH$)] The number of Brauer characters in a Harish-Chandra series
  equals the number of simple modules of the Hecke algebra $\cH$ of the
  corresponding cuspidal Brauer character. More precisely, the decomposition
  of induced PIMs in that series can be read off from the corresponding
  decomposition for $\cH$ (see Proposition~\ref{prop:dechecke}).
\item[(Csp)] There exist cuspidal unipotent Brauer characters for $G$ if and
  only if the centraliser of a Sylow $\ell$-subgroup of $G$ is not contained
  in any proper 1-split Levi subgroup (see \cite[Cor.~2.7]{GHM}).
\item[(St)] The \emph{$\ell$-modular Steinberg character of $G$}, i.e., the
  unique unipotent Brauer constituent in the $\ell$-modular reduction of an
  ordinary Gelfand--Graev character of $G$, is cuspidal if and only if a Sylow
  $\ell$-subgroup of $G$ is not contained in any proper 1-split Levi subgroup of
  $G$ (see \cite[Thm.~4.2]{GHM}).
\item[(DL)] Let $w \in W$ and $\Psi$ be the character of a PIM of $G$. If
  $\Psi$ does not occur in the Deligne--Lusztig character $R_v$ for any
  $v<w$ then the coefficient of $\Psi$ in $(-1)^{l(w)}R_w$ is non-negative
  (see Proposition~\ref{prop:dl}).
\item[(Red)] Let $\bL$ be an $F$-stable Levi subgroup of $\bG$ such that
  $\bL^*$ is the centraliser of a semisimple $\ell$-element of $G^*$. Then
  there is a sign $\varepsilon\in\{\pm1\}$ such that for any unipotent
  character $\rho$ of $L$, $\varepsilon \RLG(\rho)^\circ$ is a non-negative
  linear combination of irreducible Brauer characters (see Lemma~\ref{lem:red}).
\item[(Tri)] Assume that the (unipotent) decomposition matrix of $G$ is
  uni-triangular with respect to some total ordering of unipotent
  characters compatible with increasing $a$-values. Then, we can partially
  echelonise any set of projective characters of $G$ (as explained in
  Example~\ref{exmp:uni}).
\end{itemize}

\begin{rem}   \label{rem:dual}
When $\ell$ is very good for $G$ there is an $F$-equivariant bijection
between the conjugacy classes of $\ell$-elements of $G$ and of $G^*$
preserving the centralisers (see \cite[Prop.~4.2]{GH91}). Therefore in
that case the assumptions of (Red) hold for $G$ whenever they hold
for the dual group $G^*$.
\end{rem}

Throughout, for a prime $\ell$ and a finite reductive group $G$ defined
over $\FF_q$ we will denote by $d=d_\ell(q)$ the order of $q$ modulo~$\ell$.
Our results turn out to be uniform in $d$, not depending on the prime $\ell$
(once $\ell$ is large enough with respect to the root system of~$G$).

We will not consider the case $d=1$, since there by a result of Puig for all
large enough primes the unipotent decomposition matrix is always the identity
matrix. Indeed, in this case the $\ell$-blocks are unions of Harish-Chandra
series and all Hecke algebras are semisimple after reduction modulo~$\ell$.
Furthermore, we will not deal with
the case when $d=4$ since this was already considered in our predecessor
paper \cite{DM16}; we shall only indicate how some of the remaining unknowns
can be computed using $(T_\ell)$, see Remark~\ref{rem:craven4}. Finally, we will
not consider the general linear groups,
as their unipotent decomposition matrices were determined up to rank~10 by
James \cite{Ja90}, nor the general unitary groups whose unipotent decomposition
matrices up to rank~10 were computed in \cite{GHM} for linear primes and in
\cite{DM15} for unitary primes.

%%%%%%%%%%%%%%%%%%%%%%%%%%%%%%%%%%%%%%%%%%%%%%%%%%%%%%%%%%%%%%%%%%%%%%%%%
\chapter{Decomposition matrices at $d_\ell(q)=2$}   \label{chap:d=2}
In this section we determine decomposition matrices for unipotent $\ell$-blocks
of various classical and exceptional groups over $\FF_q$ for primes $\ell$
dividing $q+1$. This is by some measure the most complicated case since the
ranks of Sylow $\ell$-subgroups for such primes are generally larger than for
any prime $\ell$ with $d_\ell(q)\ge3$. Nevertheless, we obtain almost complete
results in the cases considered.

%%%%%%%%%%%%%%%%%%%%%%%%%%%%%%%%%%%%%%
\section{Centralizers of $\ell$-elements}
Recall that $\bG$ is connected reductive with Frobenius map $F$, $\bT$ is a
maximally split torus of $\bG$ and $\bB$ is an $F$-stable Borel subgroup
containing it. We denote
by $\Phi$ the set of roots of $\bG$ with respect to $\bT$, by $\Phi^+$ the
set of positive roots corresponding to $\bB$, and by $\Delta$ the set of
simple roots. Given $\alpha \in \Phi$ we write $\hgt(\alpha)$ for the height of
$\alpha$, that is the sum of the coefficients of $\alpha$ expressed in the basis
$\Delta$. The Weyl group of $\Phi$ can be identified with the Weyl group $W$ of
$\bG$. 

To any subset 
$I$ of $S$ there is a corresponding parabolic subgroup $W_I$ of $W$
and a standard Levi subgroup $\bL_I = \langle \bT, W_I\rangle$ of $\bG$.
When $I$ is $F$-stable then $\bL_I$ is $F$-stable and is a $1$-split
Levi subgroup of $\bG$. In this section we shall rather be interested in the
$2$-split Levi subgroups as defined, e.g., in \cite[p.~17]{BMM}. They are
obtained as centralisers of
$\Phi_2$-tori, which are $F$-stable tori of $\bG$ of order $(q+1)^r$ for some
$r \geq 0$. Let $w_0$ be the  longest element of~$W$ and by $S$ the set
of simple reflections in $W$ corresponding to $\Delta$. In the case
where $w_0F$ acts trivially on $S \smallsetminus I$ then the pair $(\bL_I,w_0 F)$
is conjugate to a pair $(\bL,F)$ where $\bL$ is a $2$-split Levi subgroup.
We give here some further conditions on $\ell$ to ensure the existence of an
$\ell$-element whose centralizer is $\bL$. This will be needed
in order to use the method~(Red) from Chapter~\ref{chap:strategy}.

\begin{prop}   \label{prop:q+1 reg}
 Let $I \subseteq S$ be a subset of the set of simple reflections of $W$ and
 $\bL_I$ be the corresponding standard Levi subgroup of $\bG$. We assume that
 \begin{itemize}
  \item[(1)] $w_0 F$ acts trivially on $S \smallsetminus I$;
  \item[(2)] $\ell$ is very good for $\bG$; and
  \item[(3)] $(q+1)_\ell>\hgt(\pi_I(\alpha))$ for all $\alpha\in\Phi^+$, where
   $\pi_I$ is the projection of $\ZZ\Phi^+$ to $\ZZ\Phi_{S \smallsetminus I}^+$.
 \end{itemize}
 Then there exists $t\in Z^\circ(\bL_I)^{w_0 F}$ such that $C_\bG(t)=\bL_I$.
\end{prop}

\begin{proof}
Let $\pi_\ad : \bG \twoheadrightarrow \bG_\ad := \bG/Z(\bG)$.
From \cite[Prop.~2.3]{DM91} it follows that
$\pi_{\ad}(C_\bG^\circ(t)) = C_{\bG_\ad}^\circ(\pi_\ad(t))$ for any semisimple
element $t \in \bG$. Since $\ell$ is very good, both $Z(\bG)/Z^\circ(\bG)$ and
$Z(\bG^*)/Z^\circ(\bG^*)$ are $\ell'$-groups.
The first one ensures that any $\ell$-element of $G_\ad$ lifts to an
$\ell$-element of $G$, whereas the second implies that the centralisers of
$\ell$-elements are connected by \cite[13.14(iii) and 13.15(i)]{DM91}. Finally,
by \cite[Prop.~13.12]{CE} we also have that $Z(\bL)/Z^\circ(\bL)$ is an
$\ell'$-group for any Levi subgroup of $\bG$. This shows that we can assume that
$\bG$ is semisimple of adjoint type without loss of generality.

Write $\Delta = \{\alpha_1,\ldots,\alpha_n\} \subset X(\bT)$ and denote by
$\{\varpi_1,\ldots,\varpi_n\} \subset Y(\bT)$ the dual basis for the pairing
between characters and cocharacters.
Here $X(\bT)$ is the lattice of characters of $\bT$ whereas $Y(\bT)$
is the lattice of cocharacters. We reorder the simple roots so that
$\{\alpha_1,\ldots,\alpha_m\}$ are the simple roots corresponding
to the simple reflections in $S\smallsetminus I$.
\smallskip

Given $\la \in \overline{\FF}_p^\times$ we consider the
semisimple element
$$t(\la):= \varpi_1(\la) \varpi_2(\la) \cdots
  \varpi_m(\la)$$
of $\bG$. If $\alpha = \sum_{i=1}^n x_i \alpha_i \in X(\bT)$ then
\begin{equation} \label{eq:semisimple}
 \alpha\big(t(\la)\big) =\la^{ \sum_{i=1}^m x_i} = \lambda^{\mathrm{ht}(\pi_I(\alpha))},
\end{equation}
where $\pi_I$ is the projection of $\ZZ\Phi^+$ to $\ZZ\Phi_{S \smallsetminus I}^+$.
In particular $t(\la)$ lies in the kernel of every root
$\alpha\in\Phi_I$, therefore it lies in $Z(\bL_I)$. In addition,
$w_0F(\varpi_i) = -q\varpi_i$ for every $i\leq m$. Indeed, by assumption~(1)
$w_0 F$ permutes the elements in $I$ but fixes the elements in
$S\smallsetminus I$. Therefore we have
$$\begin{aligned}
\langle w_0F(\varpi_i)+q\varpi_i,\alpha_j \rangle &\, = \langle w_0F(\varpi_i),
\alpha_j\rangle + q\langle\varpi_i,\alpha_j \rangle\\
& \, =\langle\varpi_i,w_0F(\alpha_j)\rangle+q\langle\varpi_i,\alpha_j\rangle\\
& \, = \langle \varpi_i,w_0F(\alpha_j)\rangle + q\delta_{i,j}.\\
\end{aligned}$$
If $j > m$ then $w_0F(\alpha_j) \in \ZZ\Phi_I$ therefore
$\langle\varpi_i,w_0F(\alpha_j)\rangle = 0$ whereas if $j\leq m$ we have
$w_0F(\alpha_j) = -q \alpha_j$ and hence
$\langle \varpi_i,w_0F(\alpha_j)\rangle = -q\delta_{i,j}$.
In each case $\langle w_0F(\varpi_i)+q\varpi_i,\alpha_j \rangle = 0$ which
proves that  $w_0F(\varpi_i)=-q\varpi_i$ for all $i \leq m$. We deduce that
$t(\la)$ is $w_0 F$-stable whenever $\la$ is an
$\ell$-element satisfying $\la^q= \la^{-1}$.
\smallskip

By \eqref{eq:semisimple} and assumption~(3), there exists an $\ell$-element
$\la$ of $\overline{\FF}_p^\times$ such that $\lambda^{q+1} = 1$ and
$\alpha\big(t(\la)\big) \neq 1$. By \cite[Prop. 2.3]{DM91} this implies that
$C_\bG^\circ(t(\la)) = \bL_I$.  But since $\ell$ is very good the centraliser
of every $\ell$-element is connected by \cite[13.14(iii)]{DM91}.
\end{proof}

%%%%%%%%%%%%%%%%%%%%%%%%%%%%%%%%%%%%%%
\section{Multiplicities in the Steinberg character}   \label{sec:St}

Before considering individual series of groups, we first prove a general
result. It demonstrates how Deligne--Lusztig characters can be used to obtain
relations between the entries of decomposition matrices of unipotent
$\Phi_2$-block, yielding new decomposition numbers in the ``bottom right
corner'' of the matrix. This is the part of the decomposition matrix about
which Harish-Chandra methods usually yield the least information.
\smallskip

We call
$$h:= 1+\max\{\hgt(\alpha)\mid\alpha\in\Phi^+\}$$
the \emph{Coxeter number} of $\bG$. When $\Phi$ is irreducible, it equals the
order of any Coxeter element of $W$ (see \cite[Prop.~VI.1.31]{Bki456}).

\begin{thm}   \label{thm:genpos}
 Let $\bG$ be connected reductive. We assume that
 \begin{itemize}
   \item[\rm(1)] $w_0 F$ acts trivially on $W$;
   \item[\rm(2)] $\ell$ is very good for $\bG$; and
   \item[\rm(3)] $(q+1)_\ell\ge h$, where $h$ is the Coxeter number of $\bG$.
 \end{itemize}
 Then there exists a linear character $\theta$ of $T_{w_0}$ in general
 position such that $\theta^\circ = 1$ and
 $$(-1)^{l(w_0)} R_{T_{w_0}}^G(\theta)^\circ = \vhi_\St.$$
\end{thm}

\begin{proof}
%If the claim holds for $[\bG,\bG]$ then so it does for $\bG$: under our
%assumption on $\ell$, any $\ell$-character $\theta$ obtained for
%$T_{w_0}\cap[\bG,\bG]$ extends to a character $\widetilde \theta$ of $T_{w_0}$ of
%$\ell$-power order and $G/[\bG,\bG]^F$ is an $\ell'$-group. Since
%$\mathrm{Res}_{[\bG,\bG]^F}^G R_{T_{w_0}}^G(\widetilde \theta) = R_{T_{w_0}\cap[\bG,\bG]}^G(\theta)$ by \cite[13.22]{DM91}, the conclusion holds for $G$ if it does for $[\bG,\bG]^F$.
Let $(\bG^*,\bT^*,F)$ be in duality with $(\bG,\bT,F)$.
By Proposition \ref{prop:q+1 reg} applied to $\bG^*$ and $I = \emptyset$, there
exists an $\ell$-element $t \in T_{w_0}^*$ such that $C_{\bG^*}(t) = \bT^*$.
Therefore $t$ is a regular $\ell$-element. Under the duality, this shows that
there exists an irreducible character $\theta$ of $T_{w_0}$ in general position
such that $\theta^\circ = 1$.
\par
By \cite[Cor.~2.10]{Lu78} the property that $\theta$ is regular implies that
$\chi_\theta = (-1)^{l(w_0)} R_{T_{w_0}}^G(\theta)$
is an irreducible character. Furthermore, since $w_0 F$ is central $w_0$ lies
in a cuspidal $F$-conjugacy class of $W$ which implies that $\chi_\theta$
is a cuspidal character by \cite[Cor.~2.19]{Lu78}. In particular, any
irreducible Brauer character $\vhi$ occurring in $\chi_\theta^\circ$ is also
cuspidal.
\par
Since $\theta$ is an $\ell$-character then $\chi_\theta^\circ =
(-1)^{l(w_0)} R_{T_{w_0}}^G(1)^\circ$, which gives the expression of
$\chi_\theta^\circ$ on the basic set of unipotent characters. We denote by
$P_\St$ the unique projective indecomposable summand of a Gelfand--Graev
character which contains the Steinberg character $\St$ of $G$, and by
$\Psi_\St$ its character. The Steinberg character occurs with multiplicity $1$
in $\Psi_\St$ and any other constituent is non-unipotent. Since
$\langle R_{T_{w_0}}^G(1);\St\rangle = \langle R_{T_{w_0}}^G(1);\Psi_\St\rangle=
 (-1)^{l(w_0)}$ (see for example \cite[Cor.~12.18(ii)]{DM91}) we deduce that
$\vhi_\St$ occurs with multiplicity one in $\chi_\theta^\circ$.   \par
We need to show that no other Brauer character can occur. Recall from
\S\ref{ssec:dltheory} that for $w \in W$, we denote by
$\wt R_w$ a virtual projective character obtained by adding and removing suitable
non-unipotent characters to $R_w = R_{T_w}^G(1)$. The orthogonality relations for
Deligne--Lusztig characters, together with the fact that
$R_{w_0}$ contains only unipotent constituents, yield the
following relation for $w \neq w_0$:
\begin{equation}\label{eq:ortho}
  0 = \langle R_{T_w}^G(1) ;
  (-1)^{l(w_0)} R_{T_{w_0}}^G(1)  \rangle =
  \langle \wt R_w ; \chi_\theta^\circ \rangle =
  \sum_{\vhi \in \IBr\, G} \langle \wt R_w ; \vhi \rangle
  \langle \Psi_\vhi; \chi_\theta^\circ \rangle.
\end{equation}
Note that since $\chi_\theta$ is cuspidal, the Brauer characters contributing
to this sum are all cuspidal. Let $w \neq w_0$ be of minimal length
in its conjugacy class. We prove by induction on its length $l(w)$ that for
every cuspidal Brauer character
$\vhi$, if $\langle \wt R_w ; \vhi \rangle \neq 0$ then
$\langle \Psi_\vhi; \chi_\theta^\circ \rangle = 0$. This already holds for
any element $w$ lying in a proper $F$-stable parabolic subgroup since in that
case there are no cuspidal Brauer characters $\vhi$ such that
$\langle \wt R_w ; \vhi \rangle \neq 0$. Assume now that the property holds for
any $v \in W$ such that $l(v) < l(w)$. If $\varphi$ is an irreducible Brauer
character such that $\langle \wt R_w;\vhi\rangle\langle \Psi_\vhi;
\chi_\theta^\circ\rangle\neq 0$, then by induction, $\Psi_\vhi$
cannot occur in any $\wt R_v$ for $l(v) < l(w)$. It follows from
Proposition~\ref{prop:dl} that $(-1)^{l(w)}\langle\wt R_w;\vhi\rangle>0$,
so that $(-1)^{l(w)}\langle\wt R_w;\vhi\rangle
\langle\Psi_\vhi;\chi_\theta^\circ \rangle > 0$
which contradicts \eqref{eq:ortho}.
\par
In other words, if $\Psi_\vhi$ occurs in some $\wt R_w$ for $w\neq w_0$ then
$\vhi$ is not a constituent of $\chi_\theta^\circ$. It remains to see that
all the PIMs but one will actually occur. Note that since
$\langle \chi_\theta^\circ; \vhi_\St \rangle = 1$ we already know that
$\Psi_\St$ occurs only in $\wt R_{w_0}$. Let $\vhi \neq \varphi_\St$ be an
irreducible Brauer character. Let us consider the virtual projective character
$$ \wt Q = \sum_{w \in W} (-1)^{l(w)}\wt R_w.$$
Its unipotent part equals $\St$, therefore we must have $\wt Q = \Psi_{\St}$.
In particular, $\Psi_\vhi$ does not occur in $\wt Q$. We deduce that if
$\Psi_\vhi$ does not occur in $\wt R_w$ for all $w \neq w_0$, then it
does not occur in any $\wt R_w$, which contradicts \cite[Thm.~A]{BR03}.
\end{proof}

We can use the previous theorem to compute non-trivial decomposition numbers
of the Steinberg character.

\begin{cor}   \label{cor:family1}
 Let $(\bG,F)$ be simple of classical type $\tw2A_{n-1}(q)$, $D_{2n}(q)$,
 $\tw2D_{2n+1}(q)$ with  $n\ge2$, or of exceptional type $\tw2E_6(q)$,
 $E_7(q)$, or $E_8(q)$. Assume that $p$ is good and $\ell$ is very good for
 $\bG$. Then:
 \begin{itemize}
  \item[\rm(a)] There is a unique unipotent character $\rho$ of $\bG^F$ with
   $a$-value $1$.
  \item[\rm(b)] Let $\rho^*$ be the Alvis--Curtis dual of $\rho$. If
   $(q+1)_\ell \ge h$, there exists a PIM of $\bG^F$ whose unipotent part is
   given by $\rho^* + (\rnk\, \bG)\, \St$.
 \end{itemize}
\end{cor}

\begin{proof}
When $\ell$ is very good unipotent characters and Brauer characters are
insensitive to the centre by \cite[Thm.~17.7]{CE}. Therefore we may and we will
assume that $\bG$ has trivial centre. Let $\cO$ be the
subregular unipotent class of $\bG$, that is the maximal unipotent class
outside the regular unipotent class. It is the unique class of codimension $2$
in the variety of unipotent elements in $\bG$ and thus $F$-stable. From the
classification of unipotent classes in good characteristic (see
e.g.~\cite{SpB}) one can check that it is special. We list below for each type
considered the class (with Jordan form in the natural matrix representation for
classical types), the special unipotent character $\rho$ of $\bG^F$ with
unipotent support $\cO$ and its Alvis--Curtis dual $\rho^*$. Recall from
\S\ref{subsec:finred} that a unipotent character corresponding to a
bipartition $\lambda$ is denoted by~$[\lambda]$.
$$\begin{array}{c|ccc}
\text{Type} & \cO & \rho & \rho^* \\\hline
\tw2A_{n-1} & (n-1,1)  &  [(n-1,1)] & [21^{n-2}] \\
%B_n & (2n-1,1,1) & \rho_{n-1.1} \\
%C_n & (2n-2,2) & \rho_{n-1.1} \\
D_{n} & (2n-3,3) & [n-1.1] & [1.1^{n-1}]\\
\tw2D_n & (2n-3,3) & [(n-2,1).] & [.21^{n-3}] \\
\tw2E_6 & E_6(a_1) & \phi_{2,4}' & \phi_{2,16}''\\
E_7 & E_7 (a_1) & \phi_{7,1} & \phi_{7,46} \\
E_8 & E_8(a_1) & \phi_{8,1} & \phi_{8,91} \\
%F_4 & F_4(a_1) & \phi_{4,1} \\
%G_2 & G_2(a_1) & \phi_{2,1} \\
\end{array}$$
In each of these cases, the $a$-value of $\rho$ equals $1$, and $\rho$ (as well
as $\rho^*$) is alone in its family. By maximality of $\cO$, every other
non-trivial unipotent character has $a$-value at least $2$, which proves (a).

\smallskip
Since the Springer correspondence sends the trivial local system on $\cO$ to
the reflection representation $\phi_\text{ref}$ of $W$, then $\rho$ is equal
to the almost character associated with $\phi_\text{ref}$ (see for example
\cite[\S13.3]{Ca}). More precisely, there exists an extension
$\wt\phi_\text{ref}$ of $\phi_\text{ref}$ to $W\rtimes \langle F \rangle$ such
that
\begin{equation} \label{eq:rchiref}
  \rho = R_{\wt \phi_\text{ref}} := \frac{1}{|W|} \sum_{w \in W}
    \wt \phi_{\text{ref}} (wF) R_w.
\end{equation}
When $(\bG,F)$ is split, $\rho$ is just the principal series character
corresponding to $\phi_\text{ref}$.
\smallskip

By \cite[Thm.~A]{BDT19} there exists a generalised Gelfand--Graev module
$\Gamma$ of $\cO G$ whose character has unipotent part $\rho^*+x\St$ for some
non-negative integer $x$ (depending on $q$). Let $P$ be the unique direct
summand of $\Gamma$ whose character $\Psi$ has $\rho^*$ as a constituent. Then
the unipotent part of $\Psi$ equals $\rho^*+y\St$ for some non-negative
integer~$y$. Let $\varphi$ be the irreducible Brauer character such that
$\Psi = \Psi_\vhi$.
By Theorem \ref{thm:genpos}, the multiplicity of $\varphi$ in $(R_{w_0})^\circ$
is zero, yielding the equation
\begin{equation}\label{eq:coeff}
  0 = \langle \Psi; R_{w_0} \rangle
    = \langle \rho^*; R_{w_0} \rangle + (-1)^{l(w_0)}y.
\end{equation}
Now, the Alvis--Curtis dual of $R_{w_0}$ is $(-1)^{l(w_0)} R_{w_0}$ (see for
example \cite[Thm.~12.8]{DM91}). Using equation~\eqref{eq:rchiref} we get
$$\langle \rho^*; R_{w_0} \rangle = (-1)^{l(w_0)} \langle \rho; R_{w_0} \rangle
  = (-1)^{l(w_0)} \wt \phi_{\text{ref}}(w_0 F) = - (-1)^{l(w_0)} \rnk\, \bG.$$
Then (b) follows from \eqref{eq:coeff}.
\end{proof}

For the groups not listed in Corollary \ref{cor:family1}, but for which $w_0F$
acts trivially on $W$, there are several unipotent characters with minimal
non-zero $a$-value and they form a non-trivial family. We can still give an
approximation of the previous decomposition number as we illustrate for groups
of type $B$ and $C$.

\begin{cor} \label{cor:family4}
 Let $(\bG,F)$ be simple of type $B_n(q)$ or $C_n(q)$, $n\ge2$. Assume that
 $p$ and $\ell$ are odd. If $(q+1)_\ell > 2n$, then there exist two PIMs whose
 unipotent parts are given by
    $$[1^n.] + (n-\delta)\, \St \quad \text{and} \quad
     [B_2\co.1^{n-2}]+ (n-1+\delta)\, \St,$$
 where $\delta = 1$ if $n$ is even, and $0$ otherwise.
\end{cor}

\begin{proof}
As above, we may and we will assume that the centre of $\bG$ is trivial. The
unique family of unipotent characters of $G$ with $a$-value $1$, which
corresponds to the subregular unipotent class, is
$\cF = \{ [n-1.1],[.n],[(n-1,1).], [B_2\co n-2.]\}$. In terms
of symbols, these unipotent characters are given in the same order by
$$\cF =
 \left\{\binom{0\ n}{1},\ \binom{0\ 1}{n},\ \binom{1\ n}{0},\ \binom{0\ 1\ n}{-}
 \right\}.
$$
Let us focus on the characters $[.n]$ and $[B_2\co n-2.]$. Their uniform parts
can be expressed in terms of almost characters (see \cite[Chap.~4]{Lu84}), from
which we can compute the scalar product with any Deligne--Lusztig character
$R_w$,
$w\in W$. This yields
$$\begin{aligned}
 \langle [.n];R_w \rangle = &\,  \frac{1}{2} \big( \phi^{n-1.1}(w)
   -\phi^{(n-1,1).}(w) + \phi^{.n}(w)\big), \\
 \langle[B_2\co n-2.];R_w \rangle = &\,  \frac{1}{2} \big( \phi^{n-1.1}(w)
   -\phi^{(n-1,1).}(w) - \phi^{.n}(w)\big), \\
\end{aligned}$$
where $\phi^\mu$ denotes the irreducible character of $W$ corresponding to the
bipartition $\mu$ of $n$. One can compute easily the values of these characters
at the central element $w_0$: the character $\phi^{n-1.1}$ is the reflection
character, therefore
$\phi^{n-1.1}(w_0) = -n$. The character $\phi^{.n}$ is linear, with value $-1$
on the first simple reflection and $1$ on the others, and thus
$\phi^{.n}(w_0) = (-1)^n$. Finally, for the value of $\phi^{(n-1,1).}$ we
use that the induction of $\phi^{n-1.}$ from $B_{n-1}$ to $B_n$ decomposes as
$\phi^{(n-1,1).}+\phi^{n.}+\phi^{n-1.1}$ which gives
$\phi^{(n-1,1).}(w_0) = n-1$. This yields
\begin{equation}\label{eq:multw0}
\begin{aligned}
  \langle[.n];R_{w_0} \rangle = &\, -n + \delta_{n\text{ even}}, \\
  \langle[B_2\co n-2.];R_{w_0} \rangle = &\, -n + \delta_{n\text{ odd}}. \\
\end{aligned}
\end{equation}
The Alvis--Curtis duals of $[.n]$ and $[B_2\co n-2.]$ are $[1^n.]$ and
$[B_2\co.1^{n-2}]$ respectively. By \cite[Thm.~B]{BDT19}, there exist Kawanaka
modules $K'$ and $K''$ which are projective modules whose characters have
unipotent parts $[1^n.] + x \St$ and $[B_2\co.1^{n-2}] + y \St$, for some
non-negative integers~$x$ and $y$. Therefore there is an indecomposable summand
$P'$ (resp. $P''$) of $K'$ (resp. $K''$), with $[1^n.]$ (resp.
$[B_2\co.1^{n-2}]$) occurring in the character of $P'$ (resp. $P''$). Now using
Theorem~\ref{thm:genpos}  and \eqref{eq:multw0} we can compute the multiplicity
of $\St$ in both of these projective characters as given in the statement.
\end{proof}

%%%%%%%%%%%%%%%%%%%%%%%%%%%%%%%%%%%%%%
\section{Unipotent decomposition matrix of $D_4(q)$}
We are now ready to compute decomposition matrices for specific series of
finite reductive groups. We start with the orthogonal groups $D_4(q)$.
As customary, we will label the unipotent characters in the principal series
by characters of the Weyl group of type $D_4$, that is, by unordered pairs of
partitions of~4, while the (unique) cuspidal unipotent character will be denoted
``$D_4$''. This is also the labelling used in the \Chevie{} system \cite{Mi15}.

In our tables of decomposition matrices the second column lists the degrees
of the unipotent characters as a product of cyclotomic polynomials $\Phi_e$
evaluated at $q$. In the last line, we give the $\ell$-modular Harish-Chandra
series of the Brauer characters (and, by abuse of notation, also of the
corresponding PIMs), by either writing ``ps'' for characters in the principal
series, or a type of Levi subgroup if that Levi subgroup has a unique cuspidal
unipotent Brauer character, or by the label of a cuspidal unipotent Brauer
character of that Levi subgroup.
The root system of type $D_4$ has three conjugacy classes of subsystems of
type $A_1^2$, cyclically permuted by the graph automorphism of order~3; we
denote them by $D_2$, $A_1^2$ and ${A_1^2}'$. The symbol ``c'' denotes cuspidal
PIMs. Also, in all of our tables for better readability we print ``.'' in
place of ``0''.

The groups of type $D_4$ have 14 unipotent characters. For primes $\ell>2$ with
$\ell|(q+1)$, thirteen of them lie in the principal $\ell$-block, while the one
with label $1.21$ is of $\ell$-defect zero. 

\begin{thm}   \label{thm:D4}
 The decomposition matrix for the unipotent $\ell$-blocks of $D_4(q)$,
 $q$ odd, $q\equiv-1\pmod\ell$, $\ell\ge11$, is as given in Table~\ref{tab:D4}.
\end{thm}

\begin{table}[htbp]
\caption{$D_4(q)$, $11\le \ell|(q+1)$}   \label{tab:D4}
$$\vbox{\offinterlineskip\halign{$#$\hfil\ \vrule height11pt depth4pt&
      \hfil\ $#$\ \hfil\vrule&& \hfil\ $#$\hfil\cr
    .4&                  1& 1\cr
   1.3&            q\Ph4^2& 2& 1\cr
    2+&        q^2\Ph3\Ph6& 1& 1& 1\cr
    2-&        q^2\Ph3\Ph6& 1& 1& .& 1\cr
   .31&        q^2\Ph3\Ph6& 1& 1& .& .& 1\cr
  .2^2& \hlf q^3\Ph4^2\Ph6& .& 1& 1& 1& 1& 1\cr
 1^2.2& \hlf q^3\Ph3\Ph4^2& 2& 2& 1& 1& 1& .& 1\cr
  1.21& \hlf q^3\Ph2^4\Ph6& .& .& .& .& .& .& .& 1\cr
   D_4& \hlf q^3\Ph1^4\Ph3& .& .& .& .& .& .& .& .&   1\cr
  1^2+&        q^6\Ph3\Ph6& 1& 1& 1& 1& 1& 1& 1& .& 2& 1\cr
  1^2-&        q^6\Ph3\Ph6& 1& 1& 1& 1& 1& 1& 1& .& 2& .& 1\cr
 .21^2&        q^6\Ph3\Ph6& 1& 1& 1& 1& 1& 1& 1& .& 2& .& .& 1\cr
 1.1^3&          q^7\Ph4^2& 2& 1& 1& 1& 1& 1& 2& .& 4& 1& 1& 1& 1\cr
  .1^4&             q^{12}& 1& .& .& .& .& 1& 1& .& 6& 1& 1& 1& 4& 1\cr
\noalign{\hrule}
  & & ps& ps& A_1^2& {A_1^2}'& D_2& D_2A_1\!& A_1& ps& c& A_1^2& {A_1^2}'& D_2& c& c\cr
  }}$$
\end{table}

\begin{proof}
By \cite[Thm.~5.1]{GH91} the unipotent characters form a basic set for the
unipotent $\ell$-blocks of $G=D_4(q)$ when $\ell\ne2$. It was shown in
\cite{GP92}, by constructing projective characters by Harish-Chandra induction
and from generalised Gelfand--Graev characters, that when $q$ is odd the
decomposition matrix of the unipotent blocks of $G$ is uni-triangular. This
provides a unique labelling for the irreducible unipotent Brauer characters;
we denote them by $\vhi_x$ if $x$ is the label of the corresponding ordinary
unipotent character.  \par
Let us denote by $\Psi_i$, $1\le i\le 14$, the linear combinations of unipotent
characters given by the columns in Table~\ref{tab:D4}. Harish-Chandra
induction from proper Levi subgroups now yields these projectives except for
$\Psi_9$, $\Psi_{13}$ and~$\Psi_{14}$. The Steinberg-PIM is cuspidal by (St)
since no proper Levi subgroup contains a Sylow $\ell$-subgroup of $G$.
\par
The first two columns correspond to PIMs as can be seen from the decomposition
matrix of the Hecke algebra for the principal series. By Harish-Chandra theory,
if any of the listed induced projective characters is decomposable, it can only
contain constituents from lower Harish-Chandra series. It is then easily seen
that all other projective characters given in Table~\ref{tab:D4} must also be
indecomposable.
\par
It remains to show that $\Psi_9$ and $\Psi_{13}$ are the unipotent parts of PIMs.
The graph automorphism of $G$ of order~3 fixes the unique cuspidal unipotent
character but permutes the characters with labels $1^2+,1^2-,.21^2$ cyclically.
It follows that the first three entries below the diagonal in the 9th column
of the decomposition matrix must be equal. We denote them by $a_1$. The other
two unknown entries in this column will be denoted by $a_2$ and $a_3$. The
last entry in the $13$th column will be denoted by $a_4$ so that
$$\begin{aligned}
  \Psi_{9}=& [D_4]+ a_1( [1^2+]+[1^2-]+[.21^2])+a_2 [1.1^3]+a_3[.1^4], \\
  \Psi_{13}=& [1.1^3]+a_4[.1^4].
\end{aligned}$$
By Theorem \ref{thm:genpos} we know that $R_{w_0}^\circ$
has $\vhi_{.1^4}$ as only Brauer constituent. The relation
$\langle  \Psi_{1.1^3} ; R_{w_0}^\circ \rangle = 0$ shows that
$a_4 = 4$, and the relation $\langle\Psi_{D_4};R_{w_0}^\circ\rangle=0$
yields $a_3=8-9a_1+4a_2$.
\par
Let $w\in W$ be a Coxeter element. The coefficient of $\Psi_{13}$ on $R_w$
equals $-1+3a_1-a_2$. By (DL) this forces $-1+3a_1-a_2 \geq 0$. On the other
hand by Proposition~\ref{prop:q+1 reg} for $\ell>4$ a 2-split Levi subgroup
$\tw2A_2(q).(q+1)^2$ of $G$ has a linear $\ell$-character in general position.
Then (Red) with $\rho$ the unique cuspidal unipotent character
of $L$ gives $3a_1-a_2\leq 3$. Therefore there exists $c\in\{-1,0,1\}$ such
that $a_2 = 3a_1-2+c$. With this notation we have $a_3 = 4c+3a_1$. We will see
in the proof of Theorem~\ref{thm:D5} that $a_1 = 2$ and in the proof of
Theorem~\ref{thm:D6} that $c = 0$.
\end{proof}

\begin{rem}
Assuming $(T_\ell)$, the result in Table~\ref{tab:D4} remains true for even
$q$. In fact, it was shown by Paolini \cite{Pa18} that $(T_\ell)$ holds for
$D_4(q)$ in characteristic~2 and thus the unipotent decomposition matrix of
$D_4(2^f)$ agrees with the one given in Table~\ref{tab:D4}, up to the knowledge
of $a_1$ and $c$ which will require $(T_\ell)$ for the types $D_5$ and $D_6$.
\end{rem}

%%%%%%%%%%%%%%%%%%%%%%%%%%%%%%%%%%%%%%
\section{Unipotent decomposition matrix of $D_5(q)$}

We now turn to the orthogonal groups $D_5(q)$. For the determination of the
decomposition
matrices for unipotent blocks of non-cyclic defect we will need to know the
structure and parameters of the Hecke algebras attached to cuspidal
$\ell$-modular Brauer characters of certain Levi subgroups. Here and later, by
convention, $\cH(X;q_1)$ for $X\in\{A_n,D_n\}$ denotes the Iwahori--Hecke
algebra over $k$ of type $X$ with parameters $(q_1,-1)$, and $\cH(B_n;q_1;q_2)$
denotes the Iwahori--Hecke algebra over $k$ of type $B_n$ with parameters
$(q_1,-1)$ at the type $B_1$-node, and parameters $(q_2,-1)$ at the remaining
nodes on the $A_{n-1}$-branch. In Table~\ref{tab:hecke irr} we have collected
the number $|\Irr\cH|$ for some small rank modular Iwahori--Hecke $\cH$
algebras occurring later. They can easily be computed using for example the
programme of Jacon \cite{Ja05}.

\begin{table}[ht]
\caption{$|\Irr\cH|$ for some modular Hecke algebras}   \label{tab:hecke irr}
$\begin{array}{l|cccccc}
 \qquad\qquad n=& 1& 2& 3& 4& 5& 6\cr
\hline
 \cH(B_n;1;1)&   2& 5& 10& 20& 36& 65\cr
 \cH(B_n;-1;1)&  1& 2& 3& 5& 7& ?\cr  %% \tw2A{2n-1},\tw2A_n
 \cH(B_n;1;-1)&  2& 2& 4& 6& 8& 12\cr  %% \tw2D_{n+1} MatrixDec([1,2],2,2,n);
 \cH(B_n;-1;-1)& 1& 2& 3& 4& 6& 9\cr  %% B_n; MatrixDec([1,1],2,2,n);
 \cH(D_n;1)&     1& 2& 3& 13& 18& 37\cr
 \cH(D_n;-1)&    1& 1& 2& 3& 4& 6\cr  %% D_n
\end{array}$
\end{table}

In this as well as in all later tables, $W(D_1)$ has to be interpreted as the
trivial group.

\begin{lem}   \label{lem:paramDn,d=2}
 Let $q$ be a prime power and $2<\ell|(q+1)$. The Hecke algebras of various
 $\ell$-modular cuspidal pairs $(L,\la)$ of Levi subgroups $L$ in $D_n(q)$ and
 their respective numbers of simple modules are as given in
 Table~\ref{tab:hecke Dn,d=2}.
\end{lem}

\begin{table}[ht]
\caption{Hecke algebras and $|\Irr\cH|$ in $D_n(q)$ for $d_\ell(q)=2$}   \label{tab:hecke Dn,d=2}
$\begin{array}{l|c|ccccc}
 (L,\la)& \qquad\qquad\cH\qquad\qquad\qquad & n=4& 5& 6\cr
\hline
 (A_1,\vhi_{1^2})& \cH(A_1;q)\otimes\cH(D_{n-2};q)& 1& 2& 2+1\cr
 (A_1^2,\vhi_{1^2}^{\sqtens2})& \cH(B_2;q^2;q)\otimes\cH(D_{n-4};q)& 2& 2& 2\cr
 (A_1^3,\vhi_{1^2}^{\sqtens3})& \cH(B_3;q;q^2)\otimes\cH(D_{n-6};q)& -& -& 3\cr
 (D_2,\vhi_{.1^2})& \cH(B_{n-2};q^2;q)& 2& 4& 4+2\cr
 (D_2A_1,\vhi_{.1^2}\sqtens\vhi_{1^2})&\cH(A_1;q)\otimes\cH(B_{n-4};q^2;q)& 1& 2& 2\cr
 (D_2A_1^2,\vhi_{.1^2}\sqtens\vhi_{1^2}^{\sqtens2})&\cH(B_2;q;q)\otimes\cH(B_{n-6};q^2;q)& -& -& 2\cr
 (D_4,D_4)& \cH(B_{n-4};q^4;q)& 1& 2& 2\cr
 (D_4,\vhi_{1.1^3})& \cH(B_{n-4};q^2;q)& 1& 2& 2\cr
 (D_4,\vhi_{.1^4})& \cH(B_{n-4};q^2;q)& 1& 2& 2\cr
 (D_4A_1,D_4\sqtens\vhi_{1^2})& \cH(A_1;q^9)\otimes\cH(B_{n-6};q^4;q)& -& -& 1\cr
\end{array}$
\end{table}

\begin{proof}
The relative Weyl group of a Levi subgroup $L$ of $D_n(q)$ of type $A_1$ has
type $D_{n-2}A_1$ (see either \cite[p.~72]{Ho80} or use \Chevie\ \cite{Chv}).
As the modular Steinberg character $\vhi_{1^2}$ of $A_1(q)$ is liftable to a
cuspidal character in characteristic~0, the
parameters of the Hecke algebra are the same as in characteristic~0 by
Proposition~\ref{prop:redstab}. They can hence be determined locally inside the
minimal Levi overgroups of $L$: in type $A_1^2$ the parameter is~$q$, and
similarly in a Levi subgroup of type $D_2A_1$.
\par
The relative Weyl group of a Levi subgroup of type $A_1^2$ has type
$B_2D_{n-4}$. The modular Steinberg character of $A_1(q)^2$ is liftable, so we
can determine the parameters locally inside $A_3(q)$ (here the parameter is
$q^2$), in $D_2(q)A_1(q)$ and in $A_1(q)^3$ (where the parameters clearly are
$q$). The relative Weyl group of a Levi subgroup of type $A_1^3$ has type
$B_3D_{n-6}$, and it is contained in minimal Levi overgroups of types $A_3A_1$
$D_2A_1^2$ and $A_1^4$.
\par
The relative Weyl group of a Levi subgroup of type $D_2$ (the two end nodes
interchanged by the graph automorphism) has type $B_{n-2}$ by
\cite[p.~71]{Ho80}. Again the modular Steinberg character $\vhi_{.1^2}$ of
$D_2(q)$ is liftable to a cuspidal character in characteristic~0, and the
parameters of the Hecke algebra can hence be determined locally: inside a Levi
subgroup of type $D_3$ the parameter is $q^2$, while inside a Levi subgroup of
type $D_2A_1$, it clearly is $q$. \par
The relative Weyl group of a Levi subgroup $D_2A_1$ is of type $A_1B_{n-4}$;
the parameters of the Hecke algebra for its cuspidal $\ell$-modular (liftable)
Steinberg character $\vhi_{.1^2}\sqtens\vhi_{1^2}$ are again determined
locally inside proper Levi subgroups $D_4$, $D_3A_1$ and $D_2A_1^2$ of $G$.
In all cases discussed so far, the assumptions of reduction stability in
Corollary~\ref{cor:dechecke} are satisfied as all graph automorphisms of $L$
just permute simple components. The same considerations apply to the cuspidal
Brauer character of a Levi subgroup of type $D_2A_1^2$.
\par
The ordinary cuspidal unipotent character of $G=D_4(q)$ remains irreducible
upon reduction modulo~$\ell$, by Table~\ref{tab:D4}, hence the corresponding
cuspidal Brauer character is reduction stable. The parameters of the Hecke
algebra in characteristic~0 are known, see \cite[p.~464]{Ca}. The parameters in
the remaining cases are determined analogously. The cuspidal modular Steinberg
character is the $\ell$-modular reduction of a Deligne--Lusztig character
$\RTG(\theta)$ for $\theta\in\Irr T$ in general position with $T$ a Sylow
$\Phi_2$-torus. The normaliser of a Levi subgroup of type $D_4$ in larger groups
of type $D_n$ induces the graph automorphism of order~2 on $G$. (This can be seen
from the inclusion $(\GO_8\GO_{2n-8})\cap\SO_{2n}\le\SO_{2n}$.) It is clear that
there exist $\ell$-characters $\theta\in\Irr T$ in general position invariant
under this graph automorphism. Finally, the cuspidal Brauer character
$\vhi_{1.1^3}$ occurs as one of two composition factors in the $\ell$-modular
reduction of a character $\RLG(\theta)$, with $\bL^F=\tw2A_2(q).(q+1)^2$
and $\theta=\theta_1\sqtens\theta_2$ with $\theta_1$ the cuspidal unipotent
character of $\tw2A_2(q)$, and $\theta_2\in\Irr Z$ an $\ell$-character in
general position of the $\Phi_2$-torus $Z=Z(\bL)^F$. Here, $\theta_1$ is
certainly stable under all automorphisms, and for the two dimension torus $Z$ it
is immediate that there is a stable $\ell$-character in general position. Thus,
Corollary~\ref{cor:dechecke} applies.
\end{proof}

The groups of type $D_5$ have 20 unipotent characters. All of them lie in the
principal $\ell$-block for primes $\ell$ dividing $q+1$. 

\begin{thm}   \label{thm:D5}
 Assume $(T_\ell)$. The decomposition matrix for the principal $\ell$-block of
 $D_5(q)$, $11\le\ell|(q+1)$, is as given in Table~\ref{tab:D5}, where
 $d\in\{0,1\}$.
\end{thm}

For reasons of space, in the 2nd column of Table~\ref{tab:D5} we just give the
leading coefficient and $q$-power of the degree polynomials of the
corresponding unipotent characters.
The Harish-Chandra series of the cuspidal PIMs of $D_4(q)$ are indicated by
the label of the unipotent character corresponding to it by triangularity
of the decomposition matrix given in Table~\ref{tab:D4}.

\begin{table}[htbp]
{\small\caption{$D_5(q)$, $11\le\ell|(q+1)$, principal block}   \label{tab:D5}
$$\vbox{\offinterlineskip\halign{$#$\hfil\ \vrule height11pt depth4pt&
      \hfil\ $#$\ \hfil\vrule&& \hfil\ $#$\hfil\cr
     .5&         1& 1\cr
    1.4&         q& 1& 1\cr
    .41&       q^2& .& 1& 1\cr
    2.3&       q^2& 1& 1& .& 1\cr
  1^2.3&  \hlf q^3& 1& 1& 1& 1& 1\cr
    .32&  \hlf q^3& .& .& .& 1& .& 1\cr
   1.31&  \hlf q^3& 1& 1& 1& .& .& .& 1\cr
 D_4\co2& \hlf q^3& .& .& .& .& .& .& .& 1\cr
   2.21&       q^4& 1& 1& 1& 1& 1& .& 1& .& 1\cr
  1.2^2&       q^5& .& .& 1& .& .& .& 1& .& 1& 1\cr
  .31^2&       q^6& 1& .& .& 1& 1& 1& .& 2& .& .& 1\cr
 1^2.21&       q^6& 1& 1& 1& 1& 1& 1& 1& 2\mn d& 1& 1& .& 1\cr
  1^3.2&  \hlf q^7& 1& 1& .& 1& 1& 1& .& 2\mn d& .& .& 1& 1& 1\cr
  .2^21&  \hlf q^7& .& .& .& 1& 1& 1& .& 2& 1& .& 1& .& .& 1\cr
 1.21^2&  \hlf q^7& 1& 1& 2& .& 1& .& 1& 2\mn d& 1& 1& .& 1& .& .& 1\cr
D_4\co1^2&\hlf q^7& .& .& .& .& .& .& .& .& .& .& .& .& .& .& .& 1\cr
1^2.1^3&    q^{10}& 1& 1& 1& 1& 2& 1& .& 4\mn d& 1& 1& 1& 1& 1& 1& .& 2& 1\cr
  .21^3&    q^{12}& .& 1& 1& .& .& .& .& 2\mn d& .& 1& .& 1& 3& .& 1& 2& .& 1\cr
  1.1^4&    q^{13}& 1& 1& 1& .& 1& .& .& 4\mn d& .& 1& .& 1& 3& .& 1& 2& 1& 1& 1\cr
   .1^5&    q^{20}& 1& .& .& .& 1& .& .& 4& .& .& 1& .& 1& 1& .& 2& 1& .& 3& 1\cr
\noalign{\hrule}
  & & ps& ps& D_2& ps& A_1& D_2& ps& D_4& A_1^2& D_2A_1\!& D_2& A_1& 1.1^3& D_2A_1\!& D_2& D_4& A_1^2& .1^4& 1.1^3& .1^4\cr
  }}$$}
\end{table}

\begin{proof}
Let $G=D_5(q)$. Again, the unipotent characters form a basic set for
the principal $\ell$-block. Note that the $\ell$-modular decomposition matrices
for unipotent blocks of all proper Levi subgroups are known, except for the
one PIM in the decomposition matrix of the principal $\ell$-block of $D_4(q)$
whose unipotent part is given by
$$[D_4]+ a_1( [1^2+]+[1^2-]+[.21^2])+a_2[1.1^3]+a_3[.1^4]$$
with $a_2 = c+3a_1-2$, $a_3 = 4c+3a_1$ and $c \in\{-1,0,1\}$
(see the proof of Theorem~\ref{thm:D4}). Let us again denote by $\Psi_i$,
$1\le i\le 20$, the linear combinations of unipotent characters given by the
columns in Table~\ref{tab:D5}. We shall show that these are the unipotent
parts of projective indecomposable characters of $G$.

The columns $\Psi_1,\Psi_2,\Psi_4$ and $\Psi_7$ are obtained from the
decomposition matrix of the principal series Hecke algebra $\cH(D_5;q)$,
which in turn is easily deduced from the one for $\cH(B_5;1;q)$; see the
programme of Jacon \cite{Ja05}. Harish-Chandra induction of PIMs from
proper Levi subgroups gives $\Psi_9$, $\Psi_{12}$, $\Psi_{17}$ and the
projectives with unipotent parts
$$\begin{aligned}
  &\tPsi_3=\Psi_3+2\Psi_6,\qquad  &&\tPsi_5=\Psi_5+2\Psi_{12},\\
 %%$\tPsi_6=\Psi_6+\Psi_{10}+\Psi_{14}+\Psi_{15}$,
  &\tPsi_8=\Psi_8'+a_1\Psi_{12}+(a_2-a_1)\Psi_{15}+\Psi_{16}',\\
\end{aligned}$$
$$\begin{aligned}
  &\tPsi_{10}=\Psi_{10}+\Psi_{14},\qquad  &&\tPsi_{11}=\Psi_{11}+2\Psi_{15},\\
  &\tPsi_{13}=\Psi_{13}+\Psi_{15}+\Psi_{19},\qquad  &&\tPsi_{18}=\Psi_{18}+\Psi_{20}.
\end{aligned}$$
The Hecke algebra for the cuspidal $\ell$-modular Steinberg character of a
Levi subgroup of $G$ of type $D_2$ (the two end nodes interchanged by the
graph automorphism) was determined in Lemma~\ref{lem:paramDn,d=2} to be
$\cH=\cH(B_3;q^2;q)$. It has four PIMs, thus there will be four PIMs of $G$
in the Harish-Chandra series $D_2$. Now the projectives $\tPsi_3$ and
$\tPsi_{11}$ contain summands in that series. The only proper subsums of
$\tPsi_3$ satisfying~(HCr) are $\Psi_3$, $\Psi_6$, $\Psi_3+\Psi_6$ and
$2\Psi_6$. On the other hand, from the decomposition matrix of $\cH$ we see
that this projective has to have two distinct summands in the
$D_2$-Harish-Chandra series. It follows that $\Psi_3$ and $\Psi_6$ are PIMs
for $G$. Exactly the same reasoning applies to the projective $\tPsi_{11}$:
it must decompose as $\Psi_{11}+2\Psi_{15}$, and both $\Psi_{11}$ and
$\Psi_{15}$ are PIMs by~(HCr). \par
The Hecke algebra $\cH(A_1;q)\otimes\cH(A_1;q^2)$ for the cuspidal modular
Steinberg character of a Levi subgroup $D_2A_1$ has two irreducible characters
(see Table~\ref{tab:hecke Dn,d=2}). It follows that the corresponding
Harish-Chandra series contains two PIMs. The projective $\tPsi_{10}$ contains
summands in that series, and the only splitting satisfying~(HCr) is
$\Psi_{10}+\Psi_{14}$; this yields that $\Psi_{10}$ and $\Psi_{14}$ are PIMs.
\par
For the Harish-Chandra series above the Steinberg PIM of a Levi subgroup of
type $A_1$, the Hecke algebra $\cH=\cH(A_1;q)\otimes\cH(D_3;q)$ determined in
Lemma~\ref{lem:paramDn,d=2} has two PIMs; by~(HCr) and the decomposition
matrix of $\cH$ the only admissible splitting of the projective $\tPsi_5$ in
this series is as $\Psi_5+2\Psi_{12}$, thus we find the PIM $\Psi_5$.
The Harish-Chandra induction $\tPsi_{18}$ of the Steinberg PIM from $D_4$
contains the Steinberg PIM of $G$ by (St); this gives $\Psi_{18}$ and
$\Psi_{20}$, both of which are indecomposable by~(HCr).
\par
The Hecke algebra for the ordinary unipotent cuspidal character of $D_4$ is
$\cH(A_1;q^4)$, so $\tPsi_8$ has at least two projective summands
$\tPsi_8',\tPsi_{16}$, the first containing $[D_4\co2]$, the second
$[D_4\co1^2]$. On the other hand by (Tri) there is a PIM involving
$[D_4\co1^2], [.2^21]$, and unipotent characters of larger
$a$-value. As $\tPsi_{16}$ must be a summand of this, it only contains
unipotent characters with $a$-value at least~7. The only possible summand of
$\tPsi_8$ satisfying~(HCr) and containing none of $[1^3.2]$ and
$[1.21^2]$ is $\Psi_{16}'$ as listed below, with a parameter $b\geq0$
satisfying $a_2\le b\le a_3$, and so
$\tPsi_8'=\Psi_8'+a_1\Psi_{12}+(a_2-a_1)\Psi_{15}$ is projective:
{\small
$$\vbox{\offinterlineskip\halign{$#$\hfil\ \vrule height10pt depth3pt&
      \hfil\ $#$\ \hfil&& \hfil\ $#$\hfil\cr
 D_4$:$2 & 1\cr
  .31^2&  a_1 \cr
 1^2.21& a_1\cr
  1^3.2& a_2\mn a_1&.\cr
  .2^21& a_1&.\cr
 1.21^2&a_1&. \cr
D_4$:$1^2& .& 1\cr
1^2.1^3&a_2& a_1\cr
  .21^3& b\mn a_2& a_1\pl a_3\mn b\cr
  1.1^4& b\pl a_1\mn a_2&  a_2\pl a_3\mn a_1\mn b\cr
   .1^5& a_2\pl a_3\mn b&b\mn a_2 \cr
\noalign{\hrule}
  & \Psi_8' & \Psi_{16}'\cr
  }}$$}
\par
As the centraliser of a Sylow $\ell$-subgroup of $G$ is contained inside a
proper parabolic subgroup of type $D_4$, $G$ cannot have cuspidal Brauer
characters by (Csp). Since we already accounted
for all Harish-Chandra series except for the one above $\vhi_{1.1^3}$ of $D_4$,
$\Psi_{19}$ must lie in that series, and must hence be a summand of the
Harish-Chandra induction $\tPsi_{13}$ of the cuspidal PIM $\Psi_{1.1^3}$ of
$D_4(q)$. The only such subsum of $\tPsi_{13}$ satisfying~(HCr) is $\Psi_{19}$,
so this gives the PIM $\Psi_{19}$. Harish-Chandra inducing $\Psi_{19}$ to
$D_6(q)$, restricting back again and decomposing shows that
$\tPsi_{13}-\Psi_{19}=\Psi_{13}+\Psi_{15}$ must be decomposable, and~(HCr)
then yields $\Psi_{13}$.
\par
Modulo the knowledge of $a_1,b$ and~$c$ we have now obtained all columns in
Table~\ref{tab:D5} except for $\Psi_8$. When Harish-Chandra inducing
$\Psi_{16}'$ to $D_6(q)$ and restricting back again, the
decomposition in terms of the projectives obtained so far has coefficient~1
on $\tPsi_8'$ and negative coefficients $-a_1$ on $\Psi_{12}$ and $a_1-a_2$
on $\Psi_{15}$. So $\tPsi_8'$ is not indecomposable, but contains $\Psi_{12}$
at least $a_1$ times and $\Psi_{15}$ at least $a_2-a_1$ times. This shows
that $\Psi_{8}'$ is projective. In addition, (HCr) shows the inequalities
$2(a_2-a_1)\le b\le a_3$. Note however that at this stage we cannot show
that $\Psi_{8}'$ is indecomposable since it could still contain some copies of
$\Psi_{11}$, $\Psi_{12}$ and $\Psi_{15}$.
\par
Let us consider the Deligne--Lusztig character $R_w$ associated to the Coxeter
element $w\in W$. The PIMs $\Psi_{18}$ and $\Psi_{20}$ do not occur in any
$R_v$ for $v < w$. The multiplicity of $R_w$ on $\Psi_{20}$ is equal to
$4-8a_1+2b-8c$ whereas the multiplicity on $\Psi_{18}$ is the opposite. We
deduce from (DL) that both quantities are zero and therefore $b=-2+4a_1+4c$.
Now since $b \leq a_3$ and $a_3 = 4c+3a_1$ we obtain $a_1 \leq 2$. On the other
hand, if $\ell > 6$ then by Proposition~\ref{prop:q+1 reg} there exists a
linear $\ell$-character of a Levi subgroup $L$ of $G$ of type
$D_2(q).(q-1)(q+1)^2$. Using (Red) with the trivial character of $L$ we get
$a_1 \geq 2$, therefore $a_1=2$ and $b=a_3=4c+6$ and thus $a_2 = c+4$.
In addition, the relation $a_2 \leq b$ forces $c\geq 0$ hence $c\in\{0,1\}$.
\par
Finally we use (Red) with the following pairs $(L,\rho)$ to show that
$\Psi_8'$ is almost indecomposable:
\begin{itemize}
 \item $(A_3(q).(q+1)^2,[4])$ shows that $\Psi_{11}$ cannot be a direct
  summand of $\Psi_{8}'$;
 \item  $(\tw2D_4(q).(q+1),[1.2])$ shows that $\Psi_{15}$ cannot be a
  direct summand of $\Psi_{8}'$;
 \item $(A_1(q) A_3(q).(q+1),[2]\sqtens[31])$ shows that $\Psi_{12}$
  can be a direct summand of $\Psi_{8}'$ with multiplicity at most $1$.
\end{itemize}
Therefore $\Psi_8:=\Psi_{8}'-d\Psi_{12}$ is indecomposable for some
$d\in\{0,1\}$.
\end{proof}

%%{\bf[Question: does our conjecture say something about $d$?]
%%[O. Nothing...]}

%%%%%%%%%%%%%%%%%%%%%%%%%%%%%%%%%%%%%%
\section{Unipotent decomposition matrix of $D_6(q)$}

The groups of type $D_6$ have 42 unipotent characters. For primes $\ell>2$ with
$\ell|(q+1)$, 37 of them lie in the principal $\ell$-block, the other five lie
in a block of defect $(q+1)^2_\ell$. 

\begin{thm}   \label{thm:D6}
 Assume $(T_\ell)$. The decomposition matrices for the unipotent $\ell$-blocks
 of $D_6(q)$, for $11\le\ell|(q+1)$, are as given in
 Tables~\ref{tab:D6}--\ref{tab:D6B}, where $d\in\{0,1\}$ is as in
 Theorem~\ref{thm:D5} and the unknown entries satisfy moreover
 $$c_{20} = 24-15\,c_{17}-5\,c_{18}+6\,c_{19},$$
 $5\le 4\,c_{17}+c_{18}-c_{19} \le 7$, $c_4,c_{12},c_{17}\ge2$, and
 $c_{18}\ge4$.
\end{thm}

\begin{table}[ht]
\caption{$D_6(q)$, $11\le\ell|(q+1)$, principal block}   \label{tab:D6}
{\small
$$\vbox{\offinterlineskip\halign{$#$\hfil\ \vrule height10pt depth3pt&
      \hfil\ $#$\ \hfil\vrule&& \hfil\ $#$\hfil\cr
      .6& 1& 1\cr
     1.5& q& 2& 1\cr
     .51& q^2& 1& 1& 1\cr
     2.4& q^2& 2& 2& .& 1\cr
      3+& q^3& 1& 1& .& 1& 1\cr
      3-& q^3& 1& 1& .& 1& .& 1\cr
     .42& \hlf q^3& .& 1& 1& 1& .& .& 1\cr
   1^2.4& \hlf q^3& 2& 2& 1& 1& .& .& .& 1\cr
D_4\co2.& \hlf q^3& .& .& .& .& .& .& .& .& 1\cr
    .3^2& \hlf q^4& .& .& .& 1& 1& 1& 1& .& .& 1\cr
    2.31& \hlf q^4& 2& 3& 1& 2& 1& 1& .& 1& .& .& 1\cr
D_4\co1^2.& \hlf q^4& .& .& .& .& .& .& .& .& 1& .& .& 1\cr
    1.32& q^5& .& 1& 1& 1& 1& 1& 1& .& .& 1& 1& .& 1\cr
   .41^2& q^6& 1& 1& 1& 1& .& .& 1& 1& 2& .& .& .& .& 1\cr
   2.2^2& q^6& .& 1& 1& 1& 1& 1& .& 1& .& .& 1& .& 1& .& 1\cr
  1^2.31& q^6& 2& 3& 2& 2& 1& 1& 1& 2& 2\mn d& .& 1& .& 1& .& .& 1\cr
   1^3.3& \hlf q^7& 2& 2& 1& 2& 1& 1& 1& 2& 2\mn d& .& .& .& .& 1& .& 1& 1\cr
  1.31^2& \hlf q^7& 2& 2& 2& 1& 1& 1& 1& 2& 4\mn d& 1& 1& 2& 1& 1& .& 1& .& 1\cr
D_4\co1.1& \hlf q^7& .& .& .& .& .& .& .& .& 1& .& .& 1& .& .& .& .& .& .& 1\cr
     21+& q^7& 1& 2& 1& 2& 2& 2& 1& 2& 2\mn d& 1& 1& 2& 1& .& 1& 1& .& .& .& 1\cr
     21-& q^7& 1& 2& 1& 2& 2& 2& 1& 2& 2\mn d& 1& 1& 2& 1& .& 1& 1& .& .& .& .& 1\cr
 1^2.2^2& q^8& .& 1& 1& 1& 2& 2& 1& 1& 2\mn d& 2& 1& 4& 2& .& 1& 1& .& .& .& 1& 1\cr
  2.21^2& q^8& 2& 3& 2& 2& 2& 2& 1& 3& 4\mn 2d& 1& 1& 4& 1& 1& 1& 2& 1& 1& .& 1& 1\cr
  1.2^21& q^9& .& 1& 2& 1& 2& 2& 1& 2& 4\mn d& 2& 1& 6& 2& 1& 2& 1& .& 1& .& 1& 1\cr
    .2^3& \hlf q^{10}& .& .& .& 1& 1& 1& 1& 1& 2& 1& .& 2& .& 1& 1& .& .& .& .& .& .\cr
1^2.21^2& \hlf q^{10}& 2& 3& 3& 2& 3& 3& 2& 4& 8\mn 3d& 3& 1& 10& 3& 1& 2& 3& 1& 1& 2& 2& 2\cr
D_4\co.2& \hlf q^{10}& .& .& .& .& .& .& .& .& .& .& .& .& .& .& .& .& .& .& 1& .& .\cr
   .31^3& q^{12}& 1& 1& 1& 1& 1& 1& 1& 1& 4\mn d& 1& .& 2& 1& 1& .& 1& 3& 1& 2& .& .\cr
 .2^21^2& \hlf q^{13}& .& 1& 1& 1& 1& 1& 1& 1& 4\mn d& 2& .& 8& 1& 1& 1& 1& 3& 1& 2& 1& 1\cr
   1^4.2& \hlf q^{13}& 2& 2& 1& 1& 1& 1& 1& 2& 6\mn 2d& 1& .& 4& 1& 1& .& 2& 4& 1& 2& 1& 1\cr
D_4\co.1^2& \hlf q^{13}& .& .& .& .& .& .& .& .& .& .& .& 1& .& .& .& .& .& .& 1& .& .\cr
    1^3+& q^{15}& 1& 1& 1& 1& 1& 1& 1& 2& 4\mn d& 1& .& 6& 1& 1& 1& 1& 1& .& 2& 1& 1\cr
    1^3\-& q^{15}& 1& 1& 1& 1& 1& 1& 1& 2& 4\mn d& 1& .& 6& 1& 1& 1& 1& 1& .& 2& 1& 1\cr
 1^2.1^4& q^{16}& 2& 2& 2& 1& 1& 1& 1& 3& 8\mn 2d& 2& .& 14& 2& 1& 1& 2& 4& 1& 4& 2& 2\cr
   .21^4& q^{20}& 1& 1& 1& .& .& .& .& 1& 6\mn d& 1& .& 8& 1& 1& .& 1& 4& 1& 4& 1& 1\cr
   1.1^5& q^{21}& 2& 1& 1& .& .& .& .& 2& 8\mn d& 1& .& 10& 1& 1& .& 1& 4& 1& 4& 1& 1\cr
    .1^6& q^{30}& 1& .& .& .& .& .& .& 1& 4& .& .& 6& .& 1& .& .& 1& .& 2& .& .\cr
\noalign{\hrule}
 & & ps& ps& D_2& ps& A_1^3& {A_1^3}'& D_2& A_1& D_4& A_1^4& ps& D_4A_1& D_2A_1& D_2& A_1^2& A_1& 1.1^3& D_2& D_4& A_1^3& {A_1^3}'\cr
 \text{dec}& & & & & & & & & & *& & & * \cr
  }}$$}
Here, we write $A_1^4$ for the HC-series of type $D_2A_1^2$,
\end{table}

\begin{table}[ht]
\caption{$D_6(q)$, $11\le\ell|(q+1)$, cntd.}   \label{tab:D6cntd}
{\small $$\vbox{\offinterlineskip\halign{$#$\hfil\ \vrule height10pt depth3pt&
      \hfil\ $#$\ \hfil\vrule&& \hfil\ $#$\hfil\cr
 1^2.2^2& q^8&  1\cr
  2.21^2& q^8& .& 1\cr
  1.2^21& q^9& 1& 1& 1\cr
    .2^3& \hlf q^{10}& .& .& 1& 1\cr
1^2.21^2& \hlf q^{10}& 1& 1& 1& .& 1\cr
D_4\co.2& \hlf q^{10}& .& .& .& .& .& 1\cr
   .31^3&      q^{12}& .& .& .& .& .& .& 1\cr
 .2^21^2& \hlf q^{13}& 1& 4& 1& 1& .& 2& 1& 1\cr
   1^4.2& \hlf q^{13}& .& 1& .& .& 1& .& 1& .& 1\cr
D_4\co.1^2& \hlf q^{13}& .& .& .& 2& .& 1& .& .& .& 1\cr
    1^3+& q^{15}& 1& 1& 1& c_4& 1& c_{12}& .& .& .& .& 1\cr
    1^3-& q^{15}& 1& 1& 1& c_4& 1& c_{12}& .& .& .& .& .& 1\cr
 1^2.1^4& q^{16}& 2& 5& 1& 2c_4& 2& 2c_{12}& 1& 1& 1& c_{17}& 1& 1& 1\cr
   .21^4& q^{20}& 1& 4& 1& 3& 1& 2& 1& 1& 3& c_{18}& .& .& .& 1\cr
   1.1^5& q^{21}& 1& 5& 1& 2c_4& 2& 2c_{12}& 1& 1& 4& c_{19}& 1& 1& 4& 1& 1\cr
    .1^6& q^{30}& 1& 4& 1& 3\pl2c_4& 1& 4\pl2c_{12}& .& 1& 3& c_{20}& 1& 1& 9& 1& 6& 1\cr
\noalign{\hrule}
 & & A_1^4& A_11.1^3& D_2A_1& c& A_1^2& c& .1^4& A_1.1^4& 1.1^3& c& A_1^3& {A_1^3}'& c& .1^4& c& c\cr
  }}$$}
\end{table}

\begin{table}[htbp]
\caption{$D_6(q)$, $11\le\ell|(q+1)$, block of defect $\Phi_2^2$}   \label{tab:D6B}
$$\vbox{\offinterlineskip\halign{$#$\hfil\ \vrule height11pt depth4pt&&
      \hfil\ $#$\hfil\cr
   1.41&    \hlf q^3\Ph2^4\Ph3\Ph6^2\Ph{10}& 1\cr
   21.3&    \hlf q^4\Ph2^4\Ph5\Ph6^2\Ph{10}& 1& 1\cr
 .321&    \hlf q^7\Ph2^4\Ph6^2\Ph8\Ph{10}& .& 1& 1\cr
 1^3.21& \hlf q^{10}\Ph2^4\Ph5\Ph6^2\Ph{10}& 1& 1& 1& 1\cr
 1.21^3& \hlf q^{13}\Ph2^4\Ph3\Ph6^2\Ph{10}& 1& .& .& 1& 1\cr
\noalign{\hrule}
 \omit& & ps& ps& D_2& A_1& D_2\cr
   }}$$
\end{table}

\begin{proof}
For the non-principal unipotent block all columns as given in
Table~\ref{tab:D6B} are obtained directly from (HCi) and are indecomposable by
(HCr). So we are left to consider the principal block.
\par
Denote by $\Psi_i$, $1\le i\le 37$, the linear combinations corresponding to
the columns of Tables~\ref{tab:D6} and~\ref{tab:D6cntd}. All $\Psi_i$ apart
from the ones with
$$i\in\{3,9,10,25,27,31,34,36,37\}$$
are obtained by (HCi) from Levi subgroups of types $D_5$, $A_5$, $D_4A_1$
and $D_2A_3$, up to the knowledge of $c\in\{0,1\}$ 
which we will determine later in the proof. Moreover we obtain $\Psi_{10}+\Psi_{11}$ and
$\Psi_{10}+\Psi_{22}$ from which (Tri) yields $\Psi_{10}$. We can also
determine the Harish-Chandra series of all the $\Psi_i$ obtained so far, and
comparing with Table~\ref{tab:hecke Dn,d=2} it ensures that all non-cuspidal
series have been accounted for except for one PIM in series $D_2$, which
must be a summand of $\Psi_3+\Psi_{11}$, and one PIM in series $D_4$ which must
be a summand of~$\Psi_9$.
\par
% So left: 3,9,25,27,31,34,36,37
Thus the principal block must contain six cuspidal Brauer characters
corresponding to the columns with indices $25,27,31,34,36,37$. By our
assumption $(T_\ell)$ there will be projectives for these columns with non-zero
entries only on and below the diagonal, for characters lying in smaller families.
We denote these unknown entries by $c_1,\ldots,c_{24}$ according to the following table, where we have used that the
non-trivial graph automorphism of $D_6(q)$ of order~2 interchanges the
unipotent characters $[1^3+]$ and $[1^3-]$ but fixes the six cuspidal PIMs,
which means that their multiplicities in these PIMs must agree.
{\small $$\vbox{\offinterlineskip\halign{$#$\hfil\ \vrule height10pt depth3pt&
      \hfil\ $#$\ \hfil&& \hfil\ $#$\hfil\cr
%%$$\begin{array}{c|ccccccc}
    .2^3& 1\cr
1^2.21^2&    .\cr
D_4\co.2&    .&  1\cr
   .31^3&    .&  .\cr
 .2^21^2&   c_1&  c_9\cr
   1^4.2&   c_2&  c_{10}& \cr
D_4\co.1^2& c_3&  c_{11}& 1\cr
    1^3+,1^3-&   c_4&  c_{12}& . \cr
 1^2.1^4&   c_5&  c_{13}& c_{17}& 1\cr
   .21^4&   c_6&  c_{14}& c_{18}& c_{21}\cr
   1.1^5&   c_7&  c_{15}& c_{19}& c_{22}& 1\cr
    .1^6&   c_8&  c_{16}& c_{20}& c_{23}& c_{24}\cr
 \noalign{\hrule}
  & \Psi_{25}& \Psi_{27}& \Psi_{31}& \Psi_{34}& \Psi_{36}\cr
  }}$$}
%%\end{array}$$}

Note that the missing cuspidal columns can not occur as summands of any of
the projectives obtained so far. The value $c_{24}=6$ follows from
Corollary~\ref{cor:family1}(b). To derive conditions on the other unknowns
we start by looking at the coefficients of the various PIMs on the
Deligne--Lusztig character $R_w$ attached to a Coxeter element $w\in W$.
The PIMs corresponding to the cuspidal simple modules, as well as the PIMs
$\Psi_{35}$ and $\Psi_{29}$ do not occur in $R_v$ for $v < w$, therefore their
coefficients must be non-negative by (DL).

The coefficients on $\Psi_{29}$ and $\Psi_{31}$ are $3-c_1-c_9$ and
$3-c_3-c_{11}$ respectively. On the other hand, if $\ell > 4$ we can invoke
(Red) with Proposition~\ref{prop:q+1 reg} for $\rho$ the unipotent character
labelled by $.2^2$ (resp.~the cuspidal unipotent character) of $D_4(q).(q+1)^2$
to get $c_9 \geq 2$ and $c_1 \geq 1$ (resp. $c_3 \geq 2$ and $c_{11}\geq 1$).
This shows that $c_1 = c_{11} = 1$ and $c_3 = c_9 = 2$.

At this point, Harish-Chandra inducing the cuspidal PIM $\Psi_{.2^3}$ to
$D_7(q)$ and restricting the result to $D_5(q)A_1(q)$ only decomposes
non-negatively when the parameter $c$ left open in the decomposition matrix
for $D_4(q)$ satisfies $c=0$. This furnishes the final step in the proofs of
Theorems~\ref{thm:D4} and~\ref{thm:D5}.

The coefficient on $\Psi_{34}$ is $c_{10}+2c_{12}-c_{13}+c_2+2c_4-c_5$.
If $\ell > 4$, (Red) applied to the cuspidal unipotent character
of a 2-split Levi subgroup $\tw2A_2(q)^2.(q+1)^2$ gives
$$c_{10}+2c_{12}-c_{13} \leq 0\quad\text{and}\quad
  c_2+2c_4-c_5 \leq 0.$$
Therefore they must both be zero.

The coefficient on $\Psi_{35}$ is $5+3c_{10}-c_{14}+3c_2-c_6$. If $\ell > 8$
then by (Red) for the trivial character of $A_1(q)^2.(q+1)^4$ we get
$$2+3c_{10}-c_{14} \leq 0\quad\text{and}\quad 3+3c_2-c_6\leq 0.$$
Hence $c_{14} = 2+3c_{10}$ and $c_6 = 3 + 3c_2$.

Finally, the coefficient on $\Psi_{36}$ is $4c_{10}+2c_{12}-c_{15}+4c_2+2c_4
-c_7$, and hence is non-negative. On the other hand, the trivial character
of $A_1(q).(q+1)^5$ when $\ell > 10$ gives, by (Red), the relations
$$4c_{10}+2c_{12}-c_{15} \leq 0\quad\text{and}\quad 4c_2+2c_4-c_7\leq 0,$$
and therefore these expressions must both vanish. Using this first set of
relations we obtain by Theorem~\ref{thm:genpos} the relations
$$c_8 = 3+3c_2+2c_4\quad\text{and}\quad c_{16} = 4+3c_{10}+2c_{12}$$
for the multiplicities in the Steinberg character.

We continue by analysing the coefficients of the Deligne--Lusztig character
$R_w$ for $w = s_1s_3s_4s_3s_1s_2s_3s_4s_5s_6$. The coefficients on $\Psi_{35}$
and $\Psi_{36}$ are $-4c_{21}$ and $16+4c_{21}-4c_{22}$ respectively.
Hence $c_{21} = 0$ by (DL). On the other hand, when $\ell > 10$, (Red) with
the trivial character of $A_1(q).(q+1)^5$ gives $4-c_{22} \leq 0$,
and we deduce that $c_{22} = 4$. Theorem~\ref{thm:genpos} gives $c_{23} = 9$.

For the last relation we consider the coefficient of $\Psi_{36}$ in the
Deligne--Lusztig character $R_w$ for
$w = s_1s_2s_3s_1s_2s_3s_4s_3s_1s_2s_3s_4s_5s_6$. It gives
$$X:=-120+96c_{18}+24c_{19}-24c_{20} \geq 0.$$
On the other hand (Red) applied to the cuspidal unipotent character of
$\tw2A_2(q).(q+1)^4$ gives $X/24 \leq 2$. Theorem \ref{thm:genpos} finally
yields $c_{20} = 24-15c_{17}-5c_{18}+6c_{19}$.

The Harish-Chandra series above $[.2^3]$ in $D_7(q)$ has two summands;
decomposing the Harish-Chandra induction shows that we must have $c_2=0$.
Similarly, decomposing the Harish-Chandra series above $[D_4\co.2]$ in $D_7(q)$
shows that we must have $c_{10}=0$.

Now (HCr) proves that all $\Psi_i$ are indecomposable, apart possibly from
$\Psi_9$ and $\Psi_{12}$: The projective $\Psi_9$ can only contain
$\Psi_{14},\Psi_{16},\Psi_{18}$. When $\ell > 4$ one can invoke (Red) for the
trivial character of a Levi subgroup $D_4(q).(q+1)^2$ to exclude the
possibility that $\Psi_{14}$ occurs.  Moreover, using (Red) with the unipotent
character $[2] \sqtens[2-]$ of $A_1(q)D_4(q).(q+1)$ and the character
$[2.2]$ of $\tw2D_5(q).(q+1)$ we can check that $\Psi_{16}$ and
$\Psi_{18}$ each can occur at most once.

The projective module of character $\Psi_{12}$ can only contain
$\Psi_{20},\Psi_{21},\Psi_{22},\Psi_{26},\Psi_{27},\Psi_{32}$, $\Psi_{33}$.
When $\ell > 6$ one can use (Red) with the trivial character of a Levi
subgroup $A_1(q)^3.(q+1)^3$ to check that in fact only $\Psi_{22}$
and $\Psi_{26}$ might occur. In addition, they can occur only once, which
follows from (Red) with the unipotent characters $[1^2.]$ of
$\tw2A_4(q).(q+1)^2$ and $[.1^2]\sqtens [2]\sqtens [2]$ of
$\tw2A_3(q)A_1(q)^2.(q+1)$.

The lower bounds on the $c_i$'s are obtained using (Red) in the following cases:
\begin{itemize}
\item $L^*=\tw2A_3(q).A_1(q).(q+1)^2$ with $\rho =[.2]\sqtens[2]$ gives
  $c_4\geq 2$;
\item $L^*=A_1(q)^3.(q+1)^3$ with the trivial representation gives
  $c_{12}\geq 2$; and
\item $L^*=A_1(q)^2.(q+1)^4$ (two non-conjugate) with the trivial representation
  gives $c_{17} \geq 2$ and $c_{18} \geq 4$.
\end{itemize}
Note that the previous relations imply $c_{19} \geq 5$.
\end{proof}

\begin{rem}\label{rem:QwD6}
In \cite{DM14} we introduced, for any $w\in W$ a virtual character $Q_w$
representing the Alvis--Curtis dual of  the intersection cohomology of the
Deligne--Lusztig variety corresponding to $w$.
Let $w = s_1s_3s_1s_2s_3s_4s_5s_6$ and $w' = (s_1s_2s_3)^2 s_4s_3s_1s_2s_3s_4s_5s_6$.
Then 
$$ \begin{aligned}
\langle Q_w ; \varphi_{1^3+} \rangle & = 13-5c_{12}-c_4, \\
\langle Q_{w'} ; \varphi_{1^2.1^4} \rangle & = 48-24c_{17},\\
\langle Q_{w'} ; \varphi_{.21^4} \rangle & = 96-24c_{18}.\\
\end{aligned}$$
If the conjecture in \cite[Conj.~1.2]{DM14} holds then these multiplicities 
must be non-negative. With the lower bounds on the $c_i$'s given
in Theorem~\ref{thm:D6} this would imply $c_{12} = c_{17}= 2$, $c_4 \in \{2,3\}$,
$c_{18}= 4$ and $c_{19} \in \{5,\ldots,7\}$.
\end{rem}

%%%%%%%%%%%%%%%%%%%%%%%%%%%%%%%%%%%%%%
\section{Unipotent decomposition matrix of $E_6(q)$}

For the groups $E_6(q)$ we again first determine some Hecke algebras:

\begin{lem}   \label{lem:paramE6,d=2}
 Let $q$ be a prime power and $\ell$ odd such that $(q+1)_\ell>5$. The Hecke
 algebras of various $\ell$-modular cuspidal pairs $(L,\la)$ of Levi subgroups
 $L$ in $E_6(q)$ and their respective numbers of simple modules are as given in
 Table~\ref{tab:hecke E6,d=2}.
\end{lem}

\begin{table}[ht]
\caption{Hecke algebras in $E_6(q)$ for $d_\ell(q)=2$}   \label{tab:hecke E6,d=2}
$\begin{array}{l|c|ccccc}
 (L,\la)& \qquad\qquad\cH\qquad\qquad\qquad & |\Irr\cH|\cr
\hline
 (A_1,\vhi_{1^2})& \cH(A_5;q)& 3+1\cr
 (A_1^2,\vhi_{1^2}^{\sqtens2})& \cH(B_3;q^2;q)& 4\cr
 (A_1^3,\vhi_{1^2}^{\sqtens3})& \cH(A_1;q)\otimes\cH(A_2;q^2)& 3\cr
 (D_4,D_4)& \cH(A_2;q^4)& 3\cr
 (D_4,\vhi_{1.1^3})& \cH(A_2;q^2)& 3\cr
 (D_4,\vhi_{.1^4})& \cH(A_2;q^2)& 3\cr
\end{array}$
\end{table}

\begin{proof}
The arguments are very similar to the ones used in the proof of
Lemma~\ref{lem:paramDn,d=2}. In fact, most of the parameters of the relevant
Hecke algebras were already determined there. As an example let us consider
the $\ell$-modular Steinberg character of a Levi subgroup of type $A_1$. By
\cite[p.~75]{Ho80} its relative Weyl group in $E_6(q)$
is of type $A_5$, and as the character is liftable, its parameters are
determined locally, inside a Levi subgroup of type $A_1^2$, to be equal to $q$.
For the (liftable) modular Steinberg character $.1^4$ of a Levi subgroup of
type $D_4$ the relative Weyl group has type $A_2$, and the parameter was
again already in Lemma~\ref{lem:paramDn,d=2} shown to be equal to~$q^2$.
\par
Reduction stability for the first four cases has already been argued in the
proof of that lemma. The normaliser of a Levi subgroup of type $D_4$ inside
$E_6$ induces the full group $\fS_3$ of graph automorphisms. Thus, by our
description of $\vhi_{.1^4}$ in Lemma~\ref{lem:paramDn,d=2} we need to see
that there exist $\ell$-characters in general position for a Sylow $\Phi_2$-torus
which are stable under the graph automorphisms. For this we use that the
extension of $D_4$ by $\fS_3$ is realised inside the groups of type $F_4$
(see e.g. \cite[Exmp.~13.9]{MT}): the long root subgroups generate a subgroup
of type $D_4$, and it is normalised by the Weyl group of type $A_2$ generated
by two reflections at short roots. By \cite[Tab.~T.A.133]{Koe06} there exist
semisimple $\ell$-elements with centralisers a short root $A_2$ in $F_4(q)$
whenever $q>5$ and thus there is an $\ell$-character $\theta$ of a Sylow
$\Phi_2$-torus $T$ of $G=D_4(q)$ such that $\RTG(\theta)$ lifts the modular
Steinberg character~$\vhi_{.1^4}$. The argument for $\vhi_{1.1^3}$ is similar.
\end{proof}

The groups of type $E_6$ have 30 unipotent characters. For primes $\ell>3$ with
$\ell|(q+1)$, 25 of them lie in the principal $\ell$-block and three are of
defect zero. The remaining two lie in a unipotent block of cyclic defect,
with Brauer tree
$$\vbox{\offinterlineskip\halign{$#$
        \vrule height10pt depth 2pt width 0pt&& \hfil$#$\hfil\cr
& \phi_{64,4}\ & \vr& \ \phi_{64,13}\ & \vr& \ \bigcirc \cr
 & & ps& & A_1\cr
  }}$$

\begin{thm}   \label{thm:E6,d=2}
 Assume $(T_\ell)$. The decomposition matrix for the principal $\ell$-block of
 $E_6(q)$, $11\le\ell|(q+1)$, is as given in Table~\ref{tab:E6,d=2} where
 $d\in\{0,1\}$ is as in Theorem~\ref{thm:D5}. Here, the projectives in columns~7
 and~11 might be decomposable.
\end{thm}

\begin{table}[htbp]
\caption{$E_6(q)$,  $11\le\ell|(q+1)$, principal block}   \label{tab:E6,d=2}
{\small
$$\vbox{\offinterlineskip\halign{$#$\hfil\ \vrule height10pt depth3pt&
      \hfil\ $#$\ \hfil\vrule&& \hfil\ $#$\hfil\cr
 \phi_{1,0}& 1&1\cr
 \phi_{6,1}& q&. & 1\cr
 \phi_{20,2}& q^2&. & 1& 1\cr
 \phi_{30,3}& \hlf q^3& . & 1& 1& 1\cr
 \phi_{15,5}& \hlf q^3&1 & .& 1& .& 1\cr
 \phi_{15,4}& \hlf q^3&1 & .& .& .& .& 1\cr
     D_4\co3& \hlf q^3&. & .& .& .& .& .& 1\cr
 \phi_{60,5}& q^5 &. & .& 1& .& .& .& .& 1\cr
 \phi_{24,6}& q^6 &. & .& 1& 1& 1& .& .& .& 1\cr
 \phi_{81,6}& q^6 &1 & 1& 1& .& 1& 1&  4\mn d&  1& .& 1\cr
\phi_{20,10}& \sxt q^7 &. & 1& .& .& .& 1&  2\mn d&  .& .& 1& 1\cr
 \phi_{90,8}& \thrd q^7 &. & 1& 2& 1& 1& .& 6\mn 3d& 1& 1& 1& .& 1\cr
 \phi_{10,9}& \thrd q^7 &. & .& .& 1& .& .&     .& .& 1& .& .& .& 1\cr
 \phi_{60,8}& \hlf q^7 &. & .& .& .& .& 1&   4\mn d& 1& .& 1& .& .& .& 1\cr
    D_4\co21& \hlf q^7 &. & .& .& .& .& .&     .& .& .& .& .& .& .& .& 1\cr
\phi_{81,10}& q^{10} &1 & 1& 1& .& 1& 1& 10\mn 4d& 1& .& 2& 1& 1& .& 1& 2& 1\cr
\phi_{60,11}& q^{11} &. & .& 1& .& 1& .& 8\mn 3d& 1& 1& 1& .& 1& .& 1& 2& 1& 1\cr
\phi_{24,12}& q^{12} &. & .& 1& 1& .& .& 2\mn 2d& .& 1& .& 2& 1& 1& .& 2& .& .& 1\cr
\phi_{30,15}& \hlf q^{15} &. & 1& 1& 1& 1& .& 6\mn 3d& .& 2& 1& 3& 1& 1& .& 2& .& .& 1& 1\cr
\phi_{15,17}& \hlf q^{15} &1 & .& 1& .& 1& .& 4\mn 2d& .& .& .& 2& 1& .& .& 2& 1& .& 1& .& 1\cr
\phi_{15,16}& \hlf q^{15} &1 & .& .& .& 1& 1&   4\mn d& .& .& 1& 1& .& .& 1& 2& 1& .& .& .& .& 1\cr
   D_4\co1^3& \hlf q^{15} &. & .& .& .& .& .&     .& .& .& .& .& .& .& .& .& .& .& .& .& .& .& 1\cr
\phi_{20,20}& q^{20} &. & 1& 1& .& 1& .& 8\mn 3d& .& 1& 1& 4& 1& .& .& 4& 1& 1& 1& 1& 3& .& 2& 1\cr
 \phi_{6,25}& q^{25} &. & 1& .& .& .& .&   4\mn d& .& .& 1& 2& .& .& .& 2& .& .& .& 1& 2& .& .& 1& 1\cr
 \phi_{1,36}& q^{36} &1 & .& .& .& 1& .&     2& .& .& .& .& .& .& .& 2& 1& .& .& .& 1& 1& .& .& 2& 1\cr
\noalign{\hrule}
 & & ps\!& ps\!& ps\!& ps\!& A_1\!& ps\!& D_4& ps& A_1^2\!& A_1\!& 1.1^3\!& A_1\!& A_1^3\!& A_1^2\!& D_4\!& A_1^2\!& A_1^3\!& .1^4\!& A_1^2\!& 1.1^3\!& A_1^3\!& D_4\!& .1^4\!& 1.1^3\!& .1^4\!\cr
 \text{dec?}& & & & & & & & *& & & & * \cr
  }}$$}
\end{table}

\begin{proof}
Let $G=E_6(q)$. We denote by $\Psi_i$, $1\le i\le 25$, the linear
combinations of unipotent characters given by the columns in
Table~\ref{tab:E6,d=2}. We shall show that these are the unipotent parts of
projective indecomposable characters of $G$. \par
The columns $\Psi_i$ with $i\in\{1,2,3,4,6,8\}$ are obtained from the
decomposition matrix of the Hecke algebra of type $E_6$ which has been
determined by Geck \cite[Table~D]{Ge93b}.  \par
By Table~\ref{tab:hecke E6,d=2} the $A_1$-Harish-Chandra series contains three
Brauer characters in the principal block. Harish-Chandra induction of the two
PIMs of $D_5(q)$ in that series yields
$\tPsi_5=\Psi_5+\Psi_8+2\Psi_{10}+\Psi_{12}$ and
$\tPsi_{10}=\Psi_{10}+2\Psi_{12}$. From the decomposition matrix of the Hecke
algebra $\cH(A_5;q)$ it follows that $\tPsi_{10}$ must contain one PIM in that
series plus two copies of another one. But the only splitting of $\tPsi_{10}$
compatible with~(HCr) is into $\Psi_{10}$ plus two times $\Psi_{12}$. So these
two lie in the $A_1$-series. Then, $\tPsi_5$ has to contain two copies
of $\Psi_{10}$ and one copy of $\Psi_{12}$, so that $\tPsi_5'=\Psi_5+\Psi_8$
is projective, and it contains one indecomposable summand from the $A_1$-series.
Inducing the first $A_1$-PIM from $A_5(q)$ we see that the same holds for
$\tPsi_5''=\Psi_4+\Psi_5$. Hence neither of $\tPsi_5',\tPsi_5''$ is
indecomposable, and all of their summands apart from the one in the
$A_1$-series lie in the principal series. This yields $\Psi_5$.
\par
The $A_1^2$-series contains four Brauer characters by
Table~\ref{tab:hecke E6,d=2}. Harish-Chandra induction of the two PIMs of
$A_5(q)$ in that series  yields $\tPsi_9=\Psi_9+2\Psi_{14}$ and
$\tPsi_{16}=\Psi_{16}+2\Psi_{19}$.  Thus, both $\tPsi_9$ and $\tPsi_{16}$
contain two summands from that series, one with multiplicity~2. The only
splitting consistent with~(HCr) is just given by $\Psi_9$ plus $2\Psi_{14}$,
respectively $\Psi_{16}$ plus $2\Psi_{19}$.
\par
The $A_1^3$-series contains three Brauer characters. From (HCi) and using
(Tri) we obtain $\Psi_{13}+\Psi_{14}$, $\Psi_{13}+\Psi_{14}+\Psi_{17}$
(which together yield $\Psi_{17}$), $\Psi_{14}+\Psi_{17}+\Psi_{19}$ and
$\Psi_{21}$. By (Tri) we have that $\Psi_{13}+\Psi_{14}$ cannot be
indecomposable, and the only admissible splitting with (HCr) is into
$\Psi_{13}$ plus $\Psi_{14}$. \par
The Harish-Chandra series above the three cuspidal unipotent Brauer characters
of $D_4(q)$ all contain three Brauer characters by Table~\ref{tab:hecke E6,d=2}.
For $1.1^3$, (HCi) gives $\Psi_{11}+\Psi_{12}+\Psi_{20}$ and
$\Psi_{20}+\Psi_{24}$, which both must contain two PIMs from this series.
Now, by (Tri), $\Psi_{11}+\Psi_{12}$ cannot be indecomposable, and the only
possible splitting with (HCr) leads to $\Psi_{11},\Psi_{20}$ and $\Psi_{24}$.
The situation for $.1^4$ is entirely similar, giving $\Psi_{18},\Psi_{23}$
and $\Psi_{25}$, all three of which are indecomposable by (HCr). We also find
$\Psi_7+\Psi_{15}+2\Psi_{19}$ and $\Psi_{15}+2\Psi_{19}+\Psi_{22}$ which split
up at least into $\Psi_7,\Psi_{15}+2\Psi_{19}$ and $\Psi_{22}$.
\par
Harish-Chandra inducing $\Psi_{22}$ to $E_7(q)$, cutting by the principal
block, and restricting back shows that $\Psi_{19}$ is twice contained in
$\Psi_{15}+2\Psi_{19}$.

We have thus accounted for all columns in the table. An application of (HCr)
shows that all of them are indecomposable, except possibly for $\Psi_7$ and
$\Psi_{11}$:
$\Psi_{11}$ might contain $\Psi_{19}$ once, and $\Psi_7$ might
contains copies of $\Psi_{10}$ and of $\Psi_{12}$.
\end{proof}

%%{\bf[G: can we say something about those ??] [O. Nothing it seems...
%%so maybe we can actually get rid of the $d$ since it includes removing copies
%%of $\Psi_{10}$ and $\Psi_{12}$.]}

%%%%%%%%%%%%%%%%%%%%%%%%%%%%%%%%%%%%%%
\section{Unipotent decomposition matrix of $\tw2D_4(q)$}
We now turn to the groups of twisted type, where again we start with the
orthogonal groups.
Note that $\tw2D_3(q)\cong\tw2A_3(q)$ is a unitary group, whose unipotent
decomposition matrix we already determined in \cite{DM15}.

The groups of type $\tw2D_4$ have 10 unipotent characters, all of which lie in
the principal $\ell$-block for primes $\ell$ dividing $q+1$.
For the following result, we need no restriction on $q$:

\begin{thm}   \label{thm:2D4}
 The decomposition matrix for the principal $\ell$-block of $\tw2D_4(q)$,
 with $2<\ell|(q+1)$, is as given in Table~\ref{tab:2D4}.
\end{thm}

\begin{table}[htbp]
\caption{$\tw2D_4(q)$, $2<\ell|(q+1)$}   \label{tab:2D4}
$$\vbox{\offinterlineskip\halign{$#$\hfil\ \vrule height11pt depth4pt&
      \hfil\ $#$\ \vrule&& \hfil\ $#$\hfil\cr
    3.&               1& 1\cr
   21.&           q\Ph8& .& 1\cr
   2.1&     q^2\Ph3\Ph6& .& 1& 1\cr
1^3.&\hlf q^3\Ph6\Ph8& 1& .& .& 1\cr
  .3&\hlf q^3\Ph6\Ph8& .& .& 1& .& 1\cr
 1^2.1&\hlf q^3\Ph3\Ph8& .& 1& 1& .& .& 1\cr
   1.2&\hlf q^3\Ph3\Ph8& 1& .& .& .& .& .& 1\cr
 1.1^2&     q^6\Ph3\Ph6& 1& .& .& 1& \alpha& .& 1& 1\cr
   .21&         q^7\Ph8& .& .& .& .& \alpha& .& 1& 1& 1\cr
  .1^3&          q^{12}& .& .& 1& .& 1& 1& .& .& \alpha& 1\cr
\noalign{\hrule}
  & & ps& ps& ps& A_1& .2& A_1& ps& .1^2& .2& .1^2\cr
  }}$$
\end{table}

Here $\alpha+1$ is the multiplicity of the cuspidal Brauer character
$\vhi_{.2}$ in the $\ell$-modular reduction of the Steinberg character
$[.1^2]$ of $\tw2D_3(q)$; in particular $\alpha=2$ if $(q+1)_\ell>3$ by
\cite[Tab.~1]{DM15}, and $\alpha\le2$ always.

\begin{proof}
We employ similar arguments as in the untwisted case. All unipotent characters
of $G=\tw2D_4(q)$ lie in the principal $\ell$-block, and they form a basic set.
Let $\Psi_1,\ldots,\Psi_{10}$ denote the linear combinations of unipotent
characters corresponding to the ten columns of Table~\ref{tab:2D4}. We
construct projective characters as follows: the decomposition matrix of the
Hecke algebra $\cH(B_3;q^2;q)$ for the principal series gives the four
PIMs $\Psi_1,\Psi_2,\Psi_3,\Psi_7$ labelled by ``ps" in Table~\ref{tab:2D4}.
Next Harish-Chandra induction of the $A_1$-series PIM of
a Levi subgroup of type $A_2$ gives a projective character with unipotent part
$\tPsi_4=\Psi_4+\Psi_6$, but the induction of neither of the $A_1$-series PIMs
from a Levi subgroup of type $\tw2D_2A_1$ contains this character, so
$\tPsi_4$ must be decomposable. The only subsums compatible with (HCr) are
$\Psi_4$ and $\Psi_6$, so these are the two PIMs of $G$ in the $A_1$-series.
\par
Note that the centraliser of a Sylow $\Phi_2$-torus of $G$ lies inside a Levi
subgroup of type $\tw2D_3$, so by (Csp), all Brauer characters of $G$ have a
Harish-Chandra vertex contained in that
Levi subgroup. Furthermore, by \cite[Thm.~4.2]{GHM}, the Harish-Chandra vertex
of the modular Steinberg character of $G$ is the (cuspidal) modular Steinberg
character $\vhi_{.1^2}$ of $\tw2D_3(q)$. But the Harish-Chandra induction of
this PIM has unipotent part $\tPsi_8=\Psi_8+\Psi_{10}$, so $\tPsi_8$ is
decomposable and yields the two PIMs $\Psi_8$ and $\Psi_{10}$. We have now
accounted for all Harish-Chandra series except for that of the cuspidal
character $\vhi_{.2}$ of $\tw2D_3(q)$. Hence the two remaining PIMs must lie
in that series.
\par
The column $\Psi_9$ (with a yet undetermined entry $a\ge0$ in the last row)
comes from the tensor product of the unipotent character $[21.]$ with an
irreducible Deligne--Lusztig character for a Coxeter torus
(which is projective). This can be computed using the table of unipotent
characters for $\tw2D_4(q)$ available in \Chevie~\cite{Chv}. Harish-Chandra
induction of $\Psi_{.2}$ from $\tw2D_3(q)$ gives a projective character with
unipotent part $\tPsi_5=\Psi_5+\Psi_7+\Psi_9$. As the remaining two PIMs must
lie in the Harish-Chandra series above $\Psi_{.2}$, $\tPsi_5$ has at least two
summands, and it has three if $\Psi_7$ is a summand of $\tPsi_5$. We thus
obtain the following lower right-hand corner of the decomposition matrix, for
some $b\in\{0,1\}$ (with $b=0$ if $\tPsi_5$ has two summands, and $b=1$ if it
has three).
$$\vbox{\offinterlineskip\halign{$#$\hfil\ \vrule height11pt depth4pt&&
      \hfil\ $#$\hfil\cr
    .3& 1\cr
 1^2.1& .& 1\cr
   1.2& 1-b& .& 1\cr
 1.1^2& u-b& .& 1& 1\cr
   .21& u-b& .& 1& 1& 1\cr
  .1^3& u-v& 1& .& .& v& 1\cr
\noalign{\hrule}
  & .2& A_1& ps& .1^2& .2& .1^2\cr
  }}$$
Note that for $(q+1)_\ell=3$, the result in \cite[Tab.~1]{DM15}
only gives an upper bound~$3$ for the multiplicity of the Brauer character
$\vhi_{.2}$ in the $3$-modular reduction of the Steinberg character. We write
$u$ for this multiplicity (which equals~3 unless $(q+1)_\ell=3$), and we
have $0\le v\le u$.

The Harish-Chandra restriction to $\tw2D_3(q)$ of the two projectives in the
$\vhi_{.2}$-series decomposes as
$(1-2b+u)\Psi_{1.1}+\Psi_{1^2.}+(u-v-1)\Psi_{.1^2}$,
respectively $\Psi_{1^2.}+(1+v-u)\Psi_{.1^2}$ (see \cite[Table~1]{DM15}).
This shows that $v=u-1$. Finally, Harish-Chandra induction of $\Psi_5$ to
$\tw2D_5(q)$ and restriction back to $\tw2D_4(q)$ should decompose
non-negatively into PIMs.
This forces $b=1$. Our claim follows by setting $\alpha:=v$.
\end{proof}

\begin{rem}
From the known 3-modular decomposition matrices it can be seen that $\alpha=1$
for $\SO_6^-(2)\cong\PSU_4(2)$ and $\SO_8^-(2)$, so the case $(q+1)_\ell=3$
does indeed behave differently from the generic one where $\al=2$.
\end{rem}

%%%%%%%%%%%%%%%%%%%%%%%%%%%%%%%%%%%%%%
\section{Unipotent decomposition matrix of $\tw2D_5(q)$}

We now turn to the groups $\tw2D_5(q)$, where we first need to determine the
structure of certain Hecke algebras.

\begin{lem}   \label{lem:param2Dn,d=2}
 Let $q$ be a prime power and $\ell\ne2$ with $(q+1)_\ell\ge7$. The Hecke
 algebras of various $\ell$-modular cuspidal pairs $(L,\la)$ of Levi subgroups
 $L$ in $\tw2D_n(q)$ and their respective numbers of irreducible characters are
 as given in Table~\ref{tab:hecke 2Dn,d=2}.
\end{lem}

\begin{table}[ht]
\caption{Hecke algebras and $|\Irr\cH|$ in $\tw2D_n(q)$ for $d_\ell(q)=2$}   \label{tab:hecke 2Dn,d=2}
$\begin{array}{l|c|ccccc}
 (L,\la)& \cH& n=4& 5& 6& 7\cr
\hline
 (A_1,\vhi_{1^2})& \cH(A_1;q)\otimes\cH(B_{n-3};q^2;q)& 2& 2& 4& 6\cr
 (A_1^2,\vhi_{1^2}^{\sqtens2})& \cH(B_2;q^2;q)\otimes\cH(B_{n-5};q^2;q)& -& 2& 4& 4\cr
 (\tw2D_3,\vhi_{.2})& \cH(B_{n-3};q^2;q)& 2& 2& 4& 6\cr
 (\tw2D_3,\vhi_{.1^2})& \cH(B_{n-3};q^2;q)& 2& 2& 4& 6\cr
 (\tw2D_3A_1,\vhi_{.2}\sqtens\vhi_{1^2})& \cH(A_1;q^3)\otimes\cH(B_{n-5};q^2;q)& -& 1& 2& 2\cr
 (\tw2D_3A_1,\vhi_{.1^2}\sqtens\vhi_{1^2})& \cH(A_1;q)\otimes\cH(B_{n-5};q^2;q)& -& 1& 2& 2\cr
\end{array}$
\end{table}

\begin{proof}
Recall that $\tw2D_n(q)$ has Weyl group of type $B_{n-1}$. The relative Weyl
group of a Levi subgroup of type $A_1$ inside $\tw2D_n(q)$ is of type
$A_1B_{n-3}$, see \cite[p.~70]{Ho80}. Since the modular Steinberg character
$\vhi_{1^2}$ of $A_1(q)$ is liftable,
we may determine the parameters locally, inside minimal Levi overgroups of
types $\tw2D_3$, $\tw2D_2A_1$ and $A_1^2$. The relative Weyl group for a Levi
of type $A_1^2$ is of type $\tw2D_3B_{n-5}$, and the minimal Levi overgroups
have types $A_3$, $\tw2D_3A_1$ $\tw2D_2A_1^2$ and $A_1^3$.
For the modular Steinberg character of a Levi subgroup $\tw2D_3(q)$ the
relative Weyl group has type $B_{n-3}$ and the minimal Levi overgroups are of
types $\tw2D_4$ and $\tw2D_3A_1$, with corresponding parameters $q^2$ and $q$.
Finally, the minimal Levi overgroups for a Levi subgroup of type $\tw2D_3A_1$
have types $\tw2D_5$, $\tw2D_4A_1$ and $\tw2D_3A_1^2$, with parameters $q^2$
and $q$ in the latter two. The parameter for the containment in $\tw2D_5(q)$
will be determined in the proof of Theorem~\ref{thm:2D5}.
\par
The cuspidal Brauer character $\vhi_{.2}$ of $G=\tw2D_3(q)$ lift to an ordinary
character $\RLG(\theta)$ for $\theta\in\Irr L$ by \cite[Prop.~5.4]{DM15}, where
$L=\tw2A_2(q).(q+1)$ and $\theta$ is the product of the cuspidal unipotent
character $\theta_1$ of $\tw2A_2(q)$ with an $\ell$-character $\theta_2$ of
$Z(L)$ in general position. Now
$\RLG(\theta_1\sqtens\theta_2)=\RLG(\theta_1\sqtens\theta_2^{-1})$, so
$\RLG(\theta)$ is invariant under the graph automorphism of $\tw2D_3(q)$ and
hence reduction stable. Similarly, the cuspidal modular Steinberg character
$\vhi_{.1^2}$ lifts to $\RTG(\theta)$ for $\theta$ an $\ell$-character in general
position of a Sylow $\Phi_2$-torus $T$ of $G$. Here the normaliser of $G$ in a
larger type $D$-group induces the Weyl group $W$ of type $B_3$ on $T$, and direct
calculation then shows that there is an $\ell$-character of $T$ in general
position stabilised by a short root reflection of $W$ when $(q+1)_\ell\ge7$.
\end{proof}

The groups of type $\tw2D_5$ have 20 unipotent characters. For primes $\ell>2$
with $\ell|(q+1)$, 18 of them lie in the principal $\ell$-block, and there is
a further unipotent $\ell$-block of cyclic defect, with Brauer tree
$$\vbox{\offinterlineskip\halign{$#$
        \vrule height10pt depth 2pt width 0pt&& \hfil$#$\hfil\cr
& 21.1\ & \vr& \ \bigcirc\ & \vr& \ 1.21\cr
 & & ps& & ps\cr
  }}$$

\begin{thm}   \label{thm:2D5}
 Assume $(T_\ell)$. The decomposition matrix for the principal $\ell$-block of
 $\tw2D_5(q)$, $11\le\ell|(q+1)$, is as given in Table~\ref{tab:2D5}, where
 $b \geq 2$ and $d\in\{0,1\}$.
\end{thm}

\begin{table}[htbp]
\caption{$\tw2D_5(q)$, $11\le\ell|(q+1)$, principal block}   \label{tab:2D5}
$$\vbox{\offinterlineskip\halign{$#$\hfil\ \vrule height11pt depth4pt&
      \hfil\ $#$\ \hfil\vrule&& \hfil\ $#$\hfil\cr
     4.&                       1& 1\cr
    31.&            q\Ph3\Ph{10}& 1& 1\cr
    3.1&             q^2\Ph4\Ph8& 1& 1& 1\cr
   2^2.&          q^2\Ph8\Ph{10}& .& 1& .& 1\cr
     .4& \hlf q^3\Ph6\Ph8\Ph{10}& .& .& 1& .& 1\cr
  21^2.& \hlf q^3\Ph3\Ph8\Ph{10}& 1& 1& .& 1& .& 1\cr
  2.2& \hlf q^3\Ph3\Ph4^2\Ph{10}& 1& 1& 1& 1& .& .& 1\cr
    1.3&      q^4\Ph4\Ph8\Ph{10}& 1& .& 1& .& 1& .& 1& 1\cr
  1^2.2&      q^5\Ph3\Ph8\Ph{10}& 1& 1& 1& 1& .& .& 1& 1& 1\cr
  2.1^2&         q^6\Ph3\Ph6\Ph8& 1& 1& 1& 1& 2& 1& 1& .& .& 1\cr
  1^3.1&      q^6\Ph4\Ph8\Ph{10}& 1& 1& 1& 1& .& 1& .& .& 1& .& 1\cr
   1^4.& \hlf q^7\Ph6\Ph8\Ph{10}& 1& .& .& .& .& 1& .& .& .& .& 1& 1\cr
    .31& \hlf q^7\Ph3\Ph8\Ph{10}& .& .& 1& .& 3& .& 1& 1& .& 1& .& .& 1\cr
1^2.1^2&\hlf q^7\Ph3\Ph4^2\Ph{10}& 1& 1& 1& 2& 2& 1& 1& 3& 1& 1& 1& .& .& 1\cr
   .2^2&       q^{10}\Ph8\Ph{10}& .& .& .& 1& 2& .& 1& 3& .& 1& .& b& 1& 1& 1\cr
  1.1^3&          q^{12}\Ph4\Ph8& 1& .& 1& 1& 3& 1& 1& 3& 1& 1& 1& 3& 2& 1& .& 1\cr
  .21^2&       q^{13}\Ph3\Ph{10}& .& .& 1& 1& 3& .& 1& 4& 1& 1& .& 3b\mn d& 3& 1& 3& 1& 1\cr
   .1^4&                  q^{20}& .& .& 1& .& 1& .& .& 3& 1& .& 1& 3\pl5b\mn5d& 2& 1& 5& 1& 5& 1\cr
\hline
  & & ps& ps& ps& A_1^2& .2& A_1& ps& .2\!\cdot\!A_1& A_1& .1^2& A_1^2& c& .2& .1^2\!\cdot\!A_1& c& .1^2& c& c\cr
  }}$$
\end{table}

\begin{proof}
The group $G=\tw2D_5(q)$ has 18 unipotent characters in its principal
$\ell$-block. Since $\ell>2$ these unipotent characters form a basic set. As
in the previous proofs let $\Psi_1,\ldots,\Psi_{18}$ denote projective
characters with unipotent part as in the columns of the matrix in
Table~\ref{tab:2D5}. All projectives $\Psi_i$ except for $i\in\{12,15,17,18\}$
are found by (HCi). \par
Since the Sylow $\ell$-subgroups of $G$ are not contained in any proper Levi
subgroup of $G$, the PIM $\Psi_{18}$ is cuspidal by (St). Application of (HCr)
now shows that all projectives obtained so far are in fact indecomposable,
except possibly for $\Psi_5$. (For $\Psi_4$ and $\Psi_8$, there are two
possible splittings consistent with (HCr), but in both cases, neither summand
occurs among the other PIMs, so that a splitting would lead to a
non-independent set of PIMs.) This implies in particular that the series above
the two cuspidal unipotent Brauer characters of $\tw2D_3(q)A_1(q)$ both just
contain one Brauer character and hence that the parameter of the Hecke algebra
has to be~$-1\in k$, as claimed in Lemma~\ref{lem:param2Dn,d=2}.
\par
As $\ell > 10$, by Theorem~\ref{thm:genpos} there exists a non-unipotent
cuspidal character $\rho$ with $\rho^\circ = R_{w_0}^\circ = \vhi_{1^4.}$.
Assume $\vhi \in \{\vhi_{1^4.},\vhi_{.2^2}, \vhi_{.21^2}\}$ is not
cuspidal. Then the corresponding projective cover can be obtained from the
other columns of the decomposition matrix. The condition
$\langle P_\vhi;\rho\rangle = 0$ would then force $\vhi = \vhi_{.2^2}$ and
the unipotent part of $\Psi_{15}$ to be either
$\rho_{.2^2} + \rho_{1.1^3} + 3\rho_{.21^2} + \rho_{1^4.}$ or
$\rho_{.2^2} + 2\rho_{1.1^3} + 4\rho_{.21^2} + 2\rho_{1^4.}$.
This contradicts the fact that the entry at $[1.1^3]$ in $\Psi_{15}$ vanishes
by (Tri).
\par
Now, denote by $b_1,\ldots,b_7$ the yet unknown decomposition numbers below the
diagonal in the columns corresponding to $\vhi_{1^4.},\vhi_{.2^2}$
and $\vhi_{.21^2}$ (recall that the first entry below the diagonal in
$\Psi_{12}$ and in $\Psi_{15}$ vanishes). Thus, we have
\begin{align*}
&&\Psi_{12} & = [1^4.]+b_1[.2^2]+b_2[1.1^3]+b_3[.21^2]+b_4[.1^4],&& \\
&&\Psi_{15} & = [.2^2]+b_5[.21^2]+b_6[.1^4],&& \\
\text{and}&&\Psi_{17} & = [.21^2]+b_7[.1^4].&&
\end{align*}
From Theorem~\ref{thm:genpos} we obtain the relations
$$-15+10b_1+4b_2-5b_3+b_4 = 0, \quad 10-5b_5+b_6 = 0, \quad b_7 = 5.$$
To obtain further relations we decompose suitable Deligne--Lusztig characters
$R_w$ in terms of projective characters and then apply (DL).
We start with a Coxeter element $w = s_2 s_3 s_4 s_5$, whose coefficient on
$\Psi_{17}$ is $3-b_5$, forcing $b_5\geq 3$. On the other hand, if $\ell>8$
then by Proposition~\ref{prop:q+1 reg} there exists a linear $\ell$-character
in general position of a 2-split Levi subgroup $L$ of type $A_1(q).(q+1)^4$.
Then (Red) with $\rho$ being the Steinberg character of $L$ shows that
$b_5 \geq 3$, which yields $b_5= 3$ and $b_6 = 5$ using the previous equations.
Note that as a consequence neither $\Psi_{12}$, $\Psi_{17}$ nor
$\Psi_{18}$ occurs in $R_w$.
\par
With $w = s_1 s_2 s_3 s_1 s_2 s_3 s_4 s_5$, the coefficient of  $\Psi_{17}$
on $R_w$ equals $-3+3b_1+b_2-b_3$ and must be non-negative by (DL). On the
other hand, if $\ell >6$ one can use (Red) with a Levi subgroup of type
$\tw2A_2(q).(q+1)^3$ and its cuspidal unipotent character to find that
$5-3b_1-b_2+b_3 \geq 0$. This shows that there exists $d\in\{0,1,2\}$ such that
$b_3 = -3+3b_1+b_2 -d$.
\par
At this stage we can only find lower bounds on the remaining unknowns
$b_1,b_2,d$. We have already seen that $d\in\{0,1,2\}$. Using (Red) for the
two non-conjugate pairs of the form
$(L,\rho)=(A_1(q)^2.(q+1)^3,[2]^{\sqtens2})$ gives $b_1 \geq 2$ and
$b_2 \geq 3$ whenever $\ell > 6$. The other relations will be obtained in
the proof of Theorem~\ref{thm:2D6}. In the table we have set $b:=b_1$.
\end{proof}

\begin{rem}
As in Remark~\ref{rem:QwD6} consider the virtual character $Q_w$ introduced in
\cite{DM14}, for $w = (s_1s_2s_3)^2s_4s_5s_4s_3s_1s_2s_3$.
%, $w' = (s_1s_2s_3)^2s_4s_3s_1s_2s_3s_5s_4s_3s_1$. 
Then in the principal
block $B_0$ of $\tw2D_5(q)$ we have
$$\begin{aligned}
  B_0Q_w =&\, 48([1^4.]+[.31]+3[.2^2]+5[1.1^3]+9[.21^2]+15[.1^4]).\\
%  B_0Q_{w'} =&\ 24([1^4.]+[1^2.1^2]+3[.2^2]+4[1.1^3]+13[.21^2]+44[.1^4]).\\
\end{aligned}$$
If the conjecture in \cite[Conj.~1.2]{DM14} holds, this should have
$\Psi_{1^4.}$ as a direct summand, which would prove that $\Psi_{1^4.}$ is a
direct summand of $[1^4.]+2[.2^2]+3[1.1^3]+6[.21^2]+13[.1^4]$. This would force
in particular $b \leq 2$ hence $b=2$. We deduce that assuming the conjecture,
there are the following two possibilities left for the character of
$\Psi_{1^4.}$:
$$\begin{array}{c|cc}
  1^4.&  1&   1\\
  .2^2&  2&   2 \\
 1.1^3&  3&   3 \\
 .21^2&  6&   5 \\
  .1^4& 13&   8 \\
\end{array}$$
\end{rem}

%%%%%%%%%%%%%%%%%%%%%%%%%%%%%%%%%%%%%%
\section{Unipotent decomposition matrix of $\tw2D_6(q)$}

The groups of type $\tw2D_6$ have 36 unipotent characters, all of which lie in
the principal $\ell$-block for primes $\ell$ dividing $q+1$.

\begin{thm}   \label{thm:2D6}
 Assume $(T_\ell)$. The decomposition matrix for the principal $\ell$-block
 of $\tw2D_6(q)$, $11\le\ell|(q+1)$, is as given in Table~\ref{tab:2D6}, where
 $b \geq 2$ and $d \in \{0,1\}$ are as in Theorem~\ref{thm:2D5}. All columns but
 possibly $\Psi_5$ are indecomposable.
\end{thm}

\begin{table}[htbp]
\caption{$\tw2D_6(q)$, $11\le\ell|(q+1)$}   \label{tab:2D6}
{\small
$$\vbox{\offinterlineskip\halign{$#$\hfil\ \vrule height9pt depth4pt&
      \hfil\ $#$\ \hfil\vrule&& \hfil\ $#$\hfil\cr
    5.& 1& 1\cr
   41.& q& .& 1\cr
   4.1& q^2& .& 1& 1\cr
   32.& q^2& .& .& .& 1\cr
    .5& \hlf q^3& .& .& 1& .& 1\cr
   3.2& \hlf q^3& 1& .& .& 1& .& 1\cr
  31.1& \hlf q^3& .& 1& 1& .& .& .& 1\cr
 31^2.& \hlf q^3& 1& .& .& 1& .& .& .& 1\cr
   1.4& \hlf q^4& 1& .& .& .& .& 1& .& .& 1\cr
   2.3& \hlf q^4& .& 1& 1& .& 1& .& .& .& .& 1\cr
  2^2.1& \hlf q^4& .& .& .& .& .& .& 1& .& .& .& 1\cr
  2^21.& \hlf q^4& .& .& .& 1& .& .& .& 1& .& .& .& 1\cr
  21.2& q^5& .& 1& 1& .& .& .& 1& .& .& 1& 1& .& 1\cr
  3.1^2& q^6& 1& .& .& 1& 2& 1& .& 1& .& .& .& .& .& 1\cr
  1^2.3& q^6& .& 1& 1& .& 1& .& .& .& .& 1& .& .& 1& .& 1\cr
 21^2.1& q^6& .& 1& 1& .& .& .& 1& .& .& .& 1& .& 1& .& .& 1\cr
   .41& \hlf q^7& .& .& .& .& 2& 1& .& .& 1& .& .& .& .& 1& .& .& 1\cr
  2.21& \hlf q^7& 1& .& .& 1& 2& 1& .& 1& .& .& .& 1& .& 1& .& .& .& 1\cr
 21.1^2& \hlf q^7& .& 1& 1& .& .& .& 1& .& 2& 1& 1& .& 1& .& .& 1& .& .& 1\cr
  21^3.& \hlf q^7& .& 1& .& .& .& .& .& .& .& .& .& .& .& .& .& 1& .& .& .& 1\cr
  1.31& q^8& 1& .& .& .& 2& 1& .& .& 1& .& .& .& .& 1& .& .& 1& 1& .& .\cr
 1^3.2& q^8& 1& .& .& 1& .& 1& .& 1& .& .& .& .& .& .& .& .& .& .& .& .\cr
 1^2.21& q^9& 1& .& .& 1& 2& 1& .& 1& .& .& .& 1& .& 1& 2& .& .& 1& .& .\cr
  .32& \hlf q^{10}& .& .& .& .& 1& .& .& .& 2& 1& .& .& .& .& .& .& .& .& 1& .\cr
  1.2^2& \hlf q^{10}& .& .& .& .& 2& .& .& .& .& .& .& 1& .& 1& 2& .& 1& 1& .& b\cr
1^3.1^2& \hlf q^{10}& 1& .& .& 1& 2& 1& .& 2& .& .& .& 1& .& 1& 2& .& .& .& .& .\cr
1^4.1& \hlf q^{10}& .& 1& 1& .& .& .& .& .& .& .& .& .& 1& .& .& 1& .& .& .& 1\cr
 2.1^3& q^{12}& .& 1& 1& .& 1& .& .& .& 2& 1& .& .& 1& .& .& 1& 2& .& 1& 3\cr
 .31^2& \hlf q^{13}& .& .& 1& .& 2& .& .& .& 2& 1& .& .& 1& .& 1& .& 2& .& 1& 2b\mn d\cr
 1.21^2& \hlf q^{13}& 1& .& .& .& 4& 1& .& 1& 1& .& .& 1& .& 2& 2& .& 1& 1& .& b\cr
1^2.1^3& \hlf q^{13}& .& 1& 1& .& 1& .& .& .& 2& 1& 1& .& 2& .& 1& 1& 2& .& 1& 3\cr
  1^5.& \hlf q^{13}& 1& .& .& .& .& .& .& 1& .& .& .& .& .& .& .& .& .& .& .& .\cr
 .2^21& q^{16}& .& .& .& .& 1& .& .& .& 2& 1& 1& .& 1& .& 1& .& 2& .& 1& 2b\mn d\cr
 1.1^4& q^{20}& 1& .& .& .& 2& 1& .& 1& 1& .& .& .& .& 1& 2& .& .& .& .& 3b\mn4d\cr
 .21^3& q^{21}& .& .& .& .& 2& 1& .& .& 1& .& .& .& .& 1& 2& .& 1& .& .& 4b\mn4d\cr
  .1^5& q^{30}& .& .& 1& .& 1& .& .& .& .& .& .& .& 1& .& 1& .& 2& .& .& 3\pl2b\mn d\cr
\hline
  & & ps& ps& ps& ps& .2& ps& ps& A_1& .2& ps& A_1^2& A_1^2& A_1& .1^2& .2\!\cdot\!A_1& A_1& .2& ps& .1^2& 1^4.\cr
   \text{dec?}& & & & & & *\cr
  }}$$}
\end{table}

\begin{table}[htbp]
\caption{$\tw2D_6(q)$, $11\le\ell|(q+1)$, cntd.}   \label{tab:2D6,cntd}
{\small
$$\vbox{\offinterlineskip\halign{$#$\hfil\ \vrule height9pt depth4pt&
      \hfil\ $#$\ \hfil\vrule&& \hfil\ $#$\hfil\cr
   1.31& q^8& 1\cr
  1^3.2& q^8& .& 1\cr
 1^2.21& q^9& 1& 1& 1\cr
    .32& \hlf q^{10}& .& .& .& 1\cr
  1.2^2& \hlf q^{10}& 1& .& 1& .& 1\cr
1^3.1^2& \hlf q^{10}& .& 1& 1& .& .& 1\cr
  1^4.1& \hlf q^{10}& .& .& .& .& .& .& 1\cr
  2.1^3& q^{12}& .& .& .& .& .& .& .& 1\cr
  .31^2& \hlf q^{13}& .& .& .& 1& 2& .& .& 1& 1\cr
 1.21^2& \hlf q^{13}& 1& 1& 1& 2& 1& .& .& .& .& 1\cr
1^2.1^3& \hlf q^{13}& 2& .& .& .& .& .& 1& 1& .& .& 1\cr
   1^5.& \hlf q^{13}& .& .& .& .& .& 1& .& .& .& .& .& 1\cr
  .2^21& q^{16}& 2& .& .& 1& 2& .& .& 1& 1& .& 1& b& 1\cr
  1.1^4& q^{20}& .& 1& 1& 2& 3& 1& .& .& 4& 1& .& 3& .& 1\cr
  .21^3& q^{21}& 1& 1& 1& 2& 4& .& .& .& 4& 1& .& 2b\mn d& 2& 1& 1\cr
   .1^5& q^{30}& 2& .& .& .& 2& .& 1& 1& 1& .& 1& 3b\mn 4d& 3& .& 4& 1\cr
\hline
 & & \!.2\!\cdot\!A_1& A_1& .1^2\!\cdot\!A_1& .2& .2^2& A_1^2& A_1^2& .1^2& .21^2& .1^2& .1^2\!\cdot\!A_1\!\!& 1^4.& .2^2& .1^4& .21^2& .1^4\cr
  }}$$}
\end{table}

\begin{proof}
Since $\ell>2$ the unipotent characters form a
basic set for the unipotent blocks. As before, we denote by $\Psi_i$,
$1\le i\le 36$, (virtual) projective characters of $G$ whose unipotent parts
decompose as given in the respective columns of Table~\ref{tab:2D6}, and we
propose to show that these are the unipotent PIMs of $G$.
\par
Note that the decomposition matrices of the unipotent blocks of all proper
Levi subgroups are known, up to the undetermined entries in the PIM
$\Psi_{1^4.}$ of $\tw2D_5(q)$. The $\Psi_i$, $i\in\{1,2,3,4,6,7,10,18\}$ are
the PIMs in the principal series, so are obtained from the decomposition matrix
of the Hecke algebra $\cH(B_5;q^2;q)$ (which can be computed with \cite{Ja05}).
Now consider the Harish-Chandra series above the cuspidal $\ell$-modular
Steinberg character $\vhi_{1^2}$ of a
Levi subgroup of type $A_1$. Harish-Chandra inducing the two PIMs in that
series from a Levi subgroup of type $\tw2D_5$ gives $\Psi_8+2\Psi_{16}$ and
$\Psi_{13}+2\Psi_{22}$. The decomposition matrix of the corresponding Hecke
algebra $\cH(A_1;q)\otimes\cH(B_3;q^2;q)$ (see Lemma~\ref{lem:param2Dn,d=2})
shows that each of these must have two summands in the $A_1$-series, one with
multiplicity~2. The only possible splitting compatible with (HCr) gives the
four PIMs $\Psi_8$, $\Psi_{13}$, $\Psi_{16}$ and $\Psi_{22}$.
\par
Next consider the Harish-Chandra series of the cuspidal $\ell$-modular
Steinberg character $\vhi_{1^2}^{\sqtens2}$ of a
Levi subgroup of type $A_1^2$. HC-inducing the two PIMs in that series from
a Levi subgroup of type $A_4$ gives $\Psi_{11}+\Psi_{12}$ and
$\Psi_{26}+\Psi_{27}$. Comparison with the corresponding Hecke algebra
$\cH(B_2;q^2;q)\otimes\cH(B_1;q^2)$ shows that both have to have two PIMs as
summands in this series. Again splitting via (HCr) yields
$\Psi_{11},\Psi_{12},\Psi_{26}$ and $\Psi_{27}$.
\par
Now consider the Harish-Chandra series of the cuspidal $\ell$-modular Steinberg
character $\vhi_{.1^2}$ of a Levi subgroup of type $\tw2D_3$. Here,
Harish-Chandra inducing the two PIMs in that series from a Levi subgroup of
type $\tw2D_5$ gives $\Psi_{14}+2\Psi_{19}$ and $\Psi_{28}+2\Psi_{30}$.
Comparison with the decomposition matrix of the Hecke algebra $\cH(B_3;q^2;q)$
and application of (HCr) then shows that $\Psi_{14},\Psi_{19},\Psi_{28}$ and
$\Psi_{30}$ are projective.
\par
The situation for the Harish-Chandra series of the cuspidal unipotent Brauer
character $\vhi_{.2}$ of a Levi subgroup of type $\tw2D_3$ is precisely as in
the previous case, yielding the columns $\Psi_5,\Psi_9,\Psi_{17}$ and
$\Psi_{24}$.
Now consider the Harish-Chandra series above the cuspidal unipotent character
$\vhi_{.2}\boxtimes\vhi_{1^2}$ of a Levi subgroup of type $\tw2D_3A_1$. By
Table~\ref{tab:hecke 2Dn,d=2} this series will contain two PIMs, whence the
Harish-Chandra induction of the unique PIM of a Levi subgroup of type
$\tw2D_3A_2$ in that series, viz. $\Psi_{15}+\Psi_{21}+\Psi_{22}$, must have
at least two summands in that series. The only minimal summands compatible
with (HCr) are $\Psi_{15}$, $\Psi_{21}$, $\Psi_{22}$, of which the later lies
in a lower Harish-Chandra series. But $\Psi_{22}$ cannot be a summand in
the PIMs induced from $\tw2D_4A_1$, so this yields $\Psi_{15}$ and $\Psi_{21}$.
Next consider the Harish-Chandra series of the modular Steinberg character
$\vhi_{.1^2}\boxtimes\vhi_{1^2}$ of a Levi subgroup of type $\tw2D_3.A_1$.
The Harish-Chandra induction of the corresponding PIM from $\tw2D_3.A_2$ gives
$\tPsi_{23}=\Psi_{23}+\Psi_{31}$. As a Levi subgroup of type $\tw2D_4.A_1$
has two PIMs in this series, $\tPsi_{23}$ must have at least two summands in
this series, and the only possibility according to (HCr) is $\Psi_{23}$ plus
$\Psi_{31}$.
\par
Since a Sylow $\ell$-subgroup of $G$ is contained in a Levi subgroup of $G$
of type $\tw2D_5$, $G$ does not possess cuspidal unipotent Brauer characters by
(Csp) and moreover, by (St) the Harish-Chandra induction of the Steinberg PIM
from $\tw2D_5(q)$ must contain the Steinberg PIM of $G$, which gives
$\Psi_{34}$ and $\Psi_{36}$.
\par
We have now accounted for all Harish-Chandra series except for those above
the three cuspidal Brauer characters $\vhi_{1^4.}$, $\vhi_{.2^2}$ and
$\vhi_{.21^2}$ of $\tw2D_5(q)$. As we are still missing six projectives, each
of these three series will contribute two PIMs, and they must occur as summands
of the Harish-Chandra induction
$\tPsi_{20},\tPsi_{25}$ and $\tPsi_{29}$ of the respective cuspidal PIMs from
$\tw2D_5(q)$. The only way that $\tPsi_{25}$ can split according to (HCr) is as
$\tPsi_{25}=\Psi_{25}+\Psi_{33}$. By (Tri) there is a PIM of $G$ containing
$[.21^3]$ once plus some copies of the Steinberg character. This cannot
be a summand of $\tPsi_{20}$, so it must occur in $\tPsi_{29}$. Then (HCr)
shows that $\Psi_{35}$ is a summand of $\tPsi_{29}$. Harish-Chandra restriction
of a PIM from $\tw2D_7(q)$ which contains the unipotent character $[.31^3]$
once and other unipotent characters of at least the same $a$-value decomposes
as $\tPsi_{29}-\Psi_{30}$ plus $\Psi_i$'s with higher $a$-value, so $\Psi_{30}$
is a summand of $\tPsi_{29}$ as well. This yields $\Psi_{29}$.
\par
Finally, the PIM involving $[1^2.1^3]$ and constituents with higher $a$-value
must be a summand of $\tPsi_{20}$. Again, (HCr) then yields two projective
modules $\Psi_{20}'$ and $\Psi_{32}$ as given below:
$$\vbox{\offinterlineskip\halign{$#$\hfil\ \vrule height9pt depth4pt&
      \hfil\ $#$\ \hfil&& \hfil\ $#$\hfil\cr
  21^3.& 1\cr
    .32& b_1\cr
  1.2^2& b_1\cr
1^3.1^2& .\cr
  1^4.1& 1\cr
  2.1^3& b_2\cr
  .31^2& 3b_1\pl b_2\mn d\mn 3\cr
 1.21^2& 3b_1\pl 2b_2\mn d\mn 3\cr
1^2.1^3& b_2\cr
   1^5.& . & 1\cr
  .2^21& 3b_1\pl b_2\mn d\mn 3 & b_1\cr
  1.1^4& 5b_1\pl b_2\mn 5d\mn c& b_2\pl c\cr
  .21^3& 6b_1\pl b_2\mn 5d\mn c& 2b_1\pl b_2\mn d\pl c\mn 3\cr
   .1^5& 2b_1\pl 2b_2\mn d\pl c\mn 3& 3\pl 3b_1\mn b_2\mn 4d\mn c\cr
\hline
  & \Psi_{20}'& \Psi_{32}&\cr
  }}$$
where $c$ is some non-negative integer.
\par
Let $w=s_2s_1s_3s_1s_2s_3s_4s_5s_6$. The coefficients of $\Psi_{34}$ and
$\Psi_{36}$ on $R_w$ since these PIMs cannot occur in  $ \Psi_{20}'$ or  $\Psi_{32}$.
We find that the coeeficients are opposed therefore they must be zero by (DL).
This gives $-24+8b_2+8c = 0$, hence $c=3-b_2$. But we showed that
$b_2 \geq 3$ at the end of the proof of Theorem~\ref{thm:2D5}, therefore we
must have $b_2 = 3$ and $c=0$.  In addition, if $\ell > 4$ then (Red) with
$(\tw2A_2(q)A_1(q^2).(q+1)^2, \tw2A_2\sqtens[2])$ gives $d \leq 1$.
\par
All columns are indecomposable, except possibly for $\Psi_5$ which might
contain once $\Psi_{18}$, and for $\Psi_{20}'$ which could contain the PIMs
$\Psi_{24},\Psi_{28}$ and $\Psi_{30}$.
Now Harish-Chandra inducing $\Psi_{32}$ to $\tw2D_7(q)$, restricting it back
and decomposing it in the $\Psi_i$ shows that $\Psi_{20}'$ must indeed contain
$\Psi_{24}$ at least $b_1$ times, and $\Psi_{30}$ at least $3-d$ times. The
new character could only contain $\Psi_{28}$.
However, when $\ell > 6$, one can use (Red) with the trivial character of
a Levi subgroup of type $A_3(q).(q+1)^3$ to see that $\Psi_{28}$ cannot be a
summand of $\Psi_{20}$.
\end{proof}

%%{\bf[G: anything about the unknowns from our conjecture?]
%%[O. Only that $b=2$ as before.]}

%%%%%%%%%%%%%%%%%%%%%%%%%%%%%%%%%%%%%%
\section{Unipotent decomposition matrix of $\tw2E_6(q)$}

\begin{lem}   \label{lem:param2E6,d=2}
 Let $q$ be a prime power and $2<\ell|(q+1)$. The Hecke algebras of various
 $\ell$-modular cuspidal pairs $(L,\la)$ of Levi subgroups $L$ in $\tw2E_6(q)$
 and their respective numbers of irreducible characters are as given in
 Table~\ref{tab:hecke 2E6,d=2}.
\end{lem}

\begin{table}[ht]
\caption{Hecke algebras in $\tw2E_6(q)$ for $d_\ell(q)=2$}   \label{tab:hecke 2E6,d=2}
$\begin{array}{l|c|ccccc}
 (L,\la)& \cH& |\Irr\cH|\cr
\hline
 (A_1,\vhi_{1^2})& \cH(B_3;q;q^2)& 3\cr
 (\tw2D_3,\vhi_{.2})& \cH(B_2;q^2;q)& 2\cr
 (\tw2D_3,\vhi_{.1^2})& \cH(B_2;q^2;q)& 2\cr
 (\tw2A_5,\vhi_{321})& \cH(A_1;q^9)& 1\cr
 (\tw2A_5,\vhi_{2^21^2})& \cH(A_1;q)& 1\cr
 (\tw2A_5,\vhi_{21^4})& \cH(A_1;q)& 1\cr
 (\tw2A_5,\vhi_{1^6})& \cH(A_1;-1)& 1\cr
\end{array}$
\end{table}

\begin{proof}
For type $A_1$ the minimal Levi overgroups in $\tw2E_6(q)$ are of types
$\tw2D_3$ and $A_1\tw2D_2$; for $\tw2D_3$ the minimal overgroups are of types
$\tw2D_4$ and $\tw2A_5$. For the ordinary cuspidal unipotent character $[321]$
of $\tw2A_5(q)$ the Hecke algebra inside $\tw2E_6(q)$ is known to have
parameter $q^9$ (see \cite[p.~464]{Ca}). The induction of the Steinberg PIM
of $\tw2A_5(q)$ to $\tw2E_6(q)$ is indecomposable by (HCr), so the parameter
here is~$-1\in k$. By \cite[Prop.~5.4]{DM15}, the cuspidal Brauer character
$\vhi_{21^4}$ of $\tw2A_5(q)$ lifts to an ordinary cuspidal character in the
Lusztig series of an $\ell$-element with centraliser $\tw2A_2(q)(q+1)^4$, whose
centraliser in $\tw2E_6(q)$ is of type $\tw2A_2(q)A_1(q)(q+1)^3$. Similarly,
$\vhi_{2^21^2}$ lifts to an ordinary cuspidal character in the Lusztig series of
an $\ell$-element with centraliser $\tw2A_2(q)^2(q+1)^2$ (resp.
$\tw2A_2(q)^2 A_1(q) (q+1)$) in $\tw2A_5(q)$ (resp.~$\tw2E_6(q)$). Thus the
parameter of the Hecke algebra is~$q$ in either case.
\par
The reduction stability of the cuspidal unipotent Brauer characters of
$\tw2D_3(q)$ was already discussed in Lemma~\ref{lem:param2Dn,d=2}. From the
extended Dynkin diagram it can be seen that the relative Weyl group of an
$A_5$-Levi subgroup inside $E_6(q)$ acts trivially on the $A_5$-factor, so
all cuspidal Brauer characters of $\tw2A_5(q)$ are reduction stable.
\end{proof}

The groups $\tw2E_6(q)$ have 30 unipotent characters. For primes $\ell>3$
dividing $q+1$, 25 of these lie in the principal $\ell$-block and three more,
namely those labelled $\phi_{8,3}',\phi_{8,9}'',\phi_{16,5}$, lie in a block
of defect $(q+1)^2_\ell$. The unipotent part of the decomposition matrix
of this latter $\ell$-block is easily shown to be the identity matrix. The
last two unipotent characters, both of which are cuspidal, are of defect~0.
We obtain the following approximation to the decomposition matrix of the
principal block:

\begin{thm}   \label{thm:2E6,d=2}
 The decomposition matrix for the principal $\ell$-block of $\tw2E_6(q)$,
 $\ell|(q+1)$, $\ell >3$ and $(q+1)_\ell > 11$, is approximated as given in
 Table~\ref{tab:2E6,d=2}.
 The unknown entries satisfy $f \in \{0,1,2\}$ and
 $$\begin{aligned}
 x_5 &= 6x_1-x_2+x_3+2x_4-7\\
 x_7 &= 3x_1-3x_2+x_3+2x_4-5\\
 x_8 &= 6x_1-3x_2+x_3+2x_4+2x_6-12\\
 x_9 &= 6x_1-5x_2+4x_4+5x_6-15,\\
 x_{10} &= 15x_1-15x_2+x_3+10x_4+5x_6-25,\\
 v_2 &= 5v_1+f-9,\\
 v_3 &= 10v_1+6f-24.
 \end{aligned}$$
 All projectives listed are indecomposable except possibly for $\Psi_6$ and
 $\Psi_7$.
\end{thm}

\begin{table}[htbp]
\caption{$\tw2E_6(q)$, $\ell|(q+1)$, $\ell>3$ and $(q+1)_\ell>11$}   \label{tab:2E6,d=2}
{\small
$$\vbox{\offinterlineskip\halign{$#$\hfil\ \vrule height10pt depth4pt&
      \hfil\ $#$\ \vrule&& \hfil\ $#$\hfil\cr
   \phi_{1,0}& 1& 1\cr
  \phi_{2,4}'& q& .& 1\cr
   \phi_{4,1}& q^2& .& 1& 1\cr
   \phi_{9,2}& \hlf q^3& 1& 1& 1& 1\cr
 \phi_{1,12}'& \hlf q^3& 1& .& .& .& 1\cr
 \phi_{2,4}''& \hlf q^3& .& .& 1& .& .& 1\cr
    \tw2A_5\co2& q^4& .& .& .& .& .& .& 1\cr
  \phi_{4,7}'& q^5& .& 1& 1& .& .& .& .& 1\cr
 \phi_{8,3}''& q^6& 1& .& 1& 1& .& 1& 2& .& 1\cr
  \phi_{9,6}'& q^6& 1& 1& 1& 1& 1& 2& 2&  1& .& 1\cr
   \tw2E_6[1]& \sxt q^7& .& .& .& .& .& .& .&  .& .& .& 1\cr
  \phi_{12,4}& \sxt q^7& .& 1& 2& 1& .& 2& 4&  1& .& 1& .&   1\cr
  \phi_{6,6}'& \thrd q^7& 1& .& .& 1& 1& 2& 2& .& 1& 1& .&  .& 1\cr
 \phi_{6,6}''& \thrd q^7& 1& .& .& 1& .& .& 2& .& 1& .& .&  .& .& 1\cr
   \phi_{4,8}& \hlf q^7& .& .& .& 1& .& 2& 2&  .& .& 1& .&  .& .& .& 1\cr
 \phi_{9,6}''& q^{10}& .& .& 1& 1& .& 3& 6&  .& 1& 1& x_1&  1& 3& .& 1& 1\cr
 \phi_{4,7}''& q^{11}& .& .& 1& .& .& 1& 4&  .& .& .& x_2&  1& 2& .& .& 1&  1\cr
  \phi_{8,9}'& q^{12}& 1& .& 1& 1& 1& 3& 4&  1& 1& 1& x_3&  .& 1& 1& 2& .& .&
1\cr
  \tw2A_5\co1^2& q^{13}& .& .& .& .& .& .& 1&  .& .& .& x_4&  .& .& .& .& .& 2&
.&  1\cr
  \phi_{9,10}& \hlf q^{15}& .& .& 1& 1& .& 5& 12& 1& 1& 2& x_5&  1& 6& 1& 3& 5&
3& 1& 2& 1\cr
\phi_{1,12}''& \hlf q^{15}& .& .& .& .& .& 1& 2&  .& 1& .& x_6&  .& 3& .& .& 1&
.& .& z& .&  1\cr
 \phi_{2,16}'& \hlf q^{15}& .& .& 1& .& .& 1& 2&  1& .& .& x_7&  .& .& .& 2& .&
1& 1& 2& .& .&  1\cr
  \phi_{4,13}& q^{20}& .& .& 1& .& .& 3& 8&  1& .& 1& x_8&  1& 6& .& 2& 5& 1&
1& 2z\pl2& 1& 2& v_1& 1\cr
\phi_{2,16}''& q^{25}& .& .& .& .& .& 2& 8&  .& .& 1& x_9&  1& 8& .& 1& 6& 3&
.& 5z\pl4& 1& 5& v_2&  5& 1\cr
  \phi_{1,24}& q^{36}& .& .& .& .& .& 1& 6&  .& 1& .& x_{10}& .& 6& 1& 2& 5& 5
& 1& 5z\pl10& 1& 5& v_3& 10& 6& 1\cr
\noalign{\hrule}
 & & ps& ps& ps& ps& A_1& .2& 321\!& A_1& ps& .1^2& c& ps& 2^21^2\!& A_1& .2& 21^4& c& .1^2& c& 1^6& c& c& c& c& c\cr
 \text{dec?}& & & & & & & *&*\cr
  }}$$}
\end{table}

Here the symbols ``$.2$'', ``$.1^2$'', ``$321$'',``$2^21^2$'', ``$21^4$''
and ``$1^6$'' stand for Harish-Chandra series of cuspidal Brauer characters
of $\tw2D_3(q)$, $\tw2A_5(q)$ respectively, see \cite[Table~1]{DM15}.

\begin{proof}
As usual, we write $\Psi_i$ for linear combinations of unipotent characters
corresponding to the column~$i$ of Table~\ref{tab:2E6,d=2}. For
$i\in\{6,7,10,12,15,16,18,20\}$ they are obtained by (HCi).
We also find $\Psi_2+\Psi_4$ and $\Psi_2+\Psi_5$ which by (Tri) and (HCr)
yield $\Psi_2,\Psi_4$ and $\Psi_5$. Then $\Psi_1+2\Psi_2$ and $\Psi_1+\Psi_3$
give $\Psi_1$ and $\Psi_3$. From this, $\Psi_3+\Psi_9$ allows to isolate
$\Psi_9$. Also $\Psi_4+\Psi_8$ again by (Tri) yields $\Psi_8$. Then
$\Psi_8+\Psi_{14}$ leads to $\Psi_{14}$ and $\Psi_{12}+\Psi_{13}$ gives
$\Psi_{13}$. \par
Counting the columns obtained so far and comparing with the number of
characters in non-cuspidal Harish-Chandra series with
Lemma~\ref{lem:param2E6,d=2} shows that the eight columns with indices
$i\in\{11,17,19,21,22,23,24,25\}$ must be cuspidal.

By Corollary~\ref{cor:family1} the second to last entry in the last row equals
$6 = \rnk\,\bG$. We then use (DL) to find some of the entries in columns
corresponding to the other 
cuspidal Brauer characters. Note that by Proposition~\ref{prop:q+1 reg} we can
use (Reg) for any $d$-split Levi subgroup whenever $(q+1)_\ell > 12$.
We start with $\Psi_{17}$: the decomposition numbers in that column will be
denoted by $y_1,\ldots,y_8$ so that the unipotent part of $\Psi_{17}$ equals
$$\Psi_{17} = \phi_{4,7}''+ y_1\phi_{8,9}' + \cdots + y_7 \phi_{2,16}''
              + y_8 \phi_{1,24}.$$
Let us consider the Deligne--Lusztig character $R_w$ associated to a Coxeter
element $w$. One first checks that $\Psi_{18}$ does not occur in any $R_v$ for
$v < w$. Its coefficient at $\Psi_{18}$ is $-y_1$, therefore by (DL) the entry
$y_1$ must be zero. The coefficient of $\Psi_{19}$ equals $2-y_2$ which forces
$y_2 \leq 2$. One can use (Red) with the character
$\tw2A_2\sqtens [2]$ of a Levi subgroup $\tw2A_2(q)^2.(q+1)^2$ to see that
we also have $y_2 \geq 2$, which proves that $y_2 =2$. The coefficients of
$\Psi_{20}$, $\Psi_{21}$ and $\Psi_{22}$ are $3-y_3$, $-y_4$ and $1-y_5$
respectively. Using (Red) for a Levi subgroup $D_4(q).(q+1)^2$ and the
character $[.2^2]$ (resp.~the cuspidal unipotent character) we have
$y_3 \geq 3$ (resp. $y_5 \geq 1$). We deduce that
$$y_3 =3,\quad y_4 = 0\quad\text{and}\quad y_5 =1.$$
The coefficient of $\Psi_{23}$ is $1-y_6$, and (Red) for the trivial character
of a Levi subgroup $A_1(q)^2.(q+1)^4$ gives $y_6 \geq 1$,
whence $y_6 =1$. Finally, the coefficient of $\Psi_{24}$ is $3-y_7$;
one can invoke (Red) with the trivial character of a Levi
subgroup $A_1(q).(q+1)^5$ to ensure that $y_7 \geq 3$, hence $y_7 =3$. We
conclude that $y_8 = 5$ using Theorem~\ref{thm:genpos}.

We now turn to $\Psi_{21}$ using the Deligne--Lusztig character $R_w$ with
$w = s_1s_2s_3s_1s_4s_3$. We denote by $u_i$ the unknown decomposition numbers
in the $21$st column: under the assumption (Tri) there are $u_i\ge0$ such that
$$\Psi_{21} =  \phi_{1,12}''+u_1\phi_{4,13}+u_2\phi_{2,16}''+u_3 \phi_{1,24}.$$
The coefficient of $\Psi_{23}$ on $R_w$ is $2-u_1$; on the other hand, (Red)
with the trivial character of a Levi subgroup $A_1(q)^2.(q+1)^4$ gives
$u_2 \geq 2$. The coefficient of $\Psi_{24}$ is
$5-u_2$ and (Red) with the trivial character of $A_1(q).(q+1)^5$ gives
$u_5 \geq 5$. This shows that $u_1=2$ and $u_2=5$. Theorem \ref{thm:genpos}
gives $u_3 = 5$.

We continue with the PIM $\Psi_{19}$ and the Deligne--Lusztig character $R_w$
with $w = s_4s_5s_4s_2s_3s_1s_4s_5$. The unknown entries will be denoted by
$z_1,\ldots,z_6$. The coefficients of the PIMs $\Psi_{20}$ and $\Psi_{22}$
on $R_w$ are $4-2z_1$ and $4-2z_3$ respectively. The unipotent characters
$[.2^2]$ and $[D_4]$ of $D_4(q).(q+1)^2$ give, by (Red), the relations
$z_1,z_3 \geq 2$. We deduce that $z_1 =z_3 =2$. The PIM $\Psi_{23}$ has
coefficient $4+4z_2-2z_4$; using (Red) with the cuspidal unipotent
character of $\tw2A_2(q)^2.(q+1)^2$  we get $-2-2z_2+z_4\geq 0$,
which proves that $z_4 = 2+2z_2$. Next, $\Psi_{24}$ has coefficient
$8+10z_2-2z_5$ and by (Red) with the trivial character of $A_1(q).(q+1)^5$ we
must also have $-4-5z_2+z_5 \geq 0$, whence $z_5 = 4+5z_2$. Finally,
Theorem \ref{thm:genpos}~shows that $z_6 = 10+5z_2$. We set $z:=z_2$.

Let us now consider the PIM $\Psi_{11}$ for which many entries are unknown.
Using (Tri), only $10$ entries below the diagonal can be non-zero, starting
with the row corresponding to $\phi_{9,6}''$. We denote these by $x_i$. Let
$w = s_1s_2s_4s_3s_1s_5s_4s_3s_6s_5s_4s_3$. The coefficients of $\Psi_{20}$
and $\Psi_{22}$ in $R_w$ are
$$X:=6x_1-x_2+x_3+2x_4-x_5-7\quad\text{and}\quad Y:=3x_1-3x_2+x_3+2x_4-x_7-5$$
respectively. By (Red) with the trivial character of $\tw2A_2(q).(q+1)^4$,
the sum $X+Y$ must be non-positive. Therefore (DL) forces
$X=Y=0$ which gives the values of $x_5$ and $x_7$ in terms of $x_1,\ldots,x_4$.
The coefficient of $\Psi_{23}$ is
$-12+6x_1-3x_2+x_3+2x_4+2x_6-x_8$. By (Red) applied to the cuspidal unipotent
character of $\tw2A_2(q)^2.(q+1)^2$ it is also non-positive, therefore it must
be zero, so that
$$x_8  = 6x_1-3x_2+x_3+2x_4+2x_6-12.$$
Finally, the coefficient of $\Psi_{24}$ is $-15+6x_1-5x_2+4x_4+5x_6-x_9$; it
is non-negative by (DL) and non-positive by (Red) applied to the trivial
character of $A_1(q).(q+1)^5$, forcing
$$x_9 = 6x_1-5x_2+4x_4+5x_6-15.$$
Theorem~\ref{thm:genpos} then gives $x_{10} = 15x_1-15x_2+x_3+10x_4+5x_6-25$.

Let us denote by $w_1$ and $w_2$ the two unknown entries in the $23$rd column.
Let $w =  s_2s_3s_4s_2s_3s_4s_6s_5s_4s_2s_3s_4s_5s_6$. The coefficient of
$\Psi_{23}$ in $R_w$ is $60-12w_1$ therefore $w_1 \leq 5$ by (DL). On the other
hand, (Red) applied to the trivial character of $A_1(q).(q+1)^5$ forces
$w_1 \geq 5$, hence $w_1 =5$. We get $w_2 = 10$ using Theorem~\ref{thm:genpos}.

The last Deligne--Lusztig character $R_w$ we look at is the one associated
to the element $w = s_1s_2s_3s_1s_4s_3s_1s_5s_4s_3s_1s_6s_5s_4s_3s_1$.
Let $v_1,v_2,v_3$ be the three unknown entries in the 22nd column.
The coefficient of $\Psi_{24}$ on $R_w$ equals
$-252+180v_1-36v_2 = 36(-7+5v_1-v_2)$. By (DL) this implies that
$v_2 \leq 5v_1-7$. On the other hand, (Red) for the cuspidal unipotent
character of $\tw2A_2(q).(q+1)^4$ yields the relation $9-5v_1+v_2\geq 0$.
We deduce that $v_2 = 5v_1-9+f$ with some $f \in \{0,1,2\}$.
Theorem~\ref{thm:genpos} then gives $v_3 = 30-20v_1+6v_2 = 10v_1+6f-24$.
\par
Now all projectives in the table are indecomposable except possibly for
$\Psi_6$ which might contain $\Psi_{12}$ once, and for $\Psi_7$ which might
contain $\Psi_9$ (twice), $\Psi_{12}$ (four times), $\Psi_{14}$ and
$\Psi_{15}$ (once each).
\end{proof}

%%%%%%%%%%%%%%%%%%%%%%%%%%%%%%%%%%%%%%
\section{Unipotent decomposition matrices of $B_4(q)$ and $C_4(q)$}   \label{sec:Bn}

Here, we find the decomposition matrices for the unipotent blocks of
odd-dimensional orthogonal groups $B_4(q)$ and symplectic groups $C_4(q)$,
assuming $(T_\ell)$.

The decomposition matrix for the principal $\ell$-block of $\SO_7(q)$,
$2<\ell|(q+1)$, was determined by Himstedt--Noeske \cite{HN14}. Again we first
record the parameters of certain Hecke algebras.

\begin{lem}   \label{lem:paramB4,d=2}
 Let $q$ be a prime power and $2<\ell|(q+1)$. The Hecke algebras of various
 $\ell$-modular cuspidal pairs $(L,\la)$ of Levi subgroups $L$ in $B_4(q)$ and
 their respective numbers of irreducible characters are as given in
 Table~\ref{tab:hecke B4,d=2}.
\end{lem}

\begin{table}[ht]
\caption{Hecke algebras in $B_4(q)$ for $d_\ell(q)=2$}   \label{tab:hecke B4,d=2}
$\begin{array}{l|c|cc}
 (L,\la)& \cH& |\Irr\cH|\cr
\hline
 (A_1,\vhi_{1^2})& \cH(A_1;q)\otimes\cH(B_2;q;q)& 2\cr
 (B_1,\vhi_{.1})& \cH(B_3;q;q)& 3\cr
 (A_1^2,\vhi_{1^2}^{\otimes2})& \cH(B_2;q^2;q)& 2\cr
\end{array}$
\end{table}

\begin{proof}
The relative Weyl groups of the relevant Levi subgroups can be found in
\cite[p.~70]{Ho80}. All three cuspidal characters are the $\ell$-modular
Steinberg characters of the respective groups and thus liftable. The minimal
Levi overgroups for type $A_1$ are of types $B_2$, $B_1A_1$ and $A_1^2$ and
thus lead to parameters $q$ in all three cases. For a Levi subgroup of type
$B_1$ the minimal overgroups have types $B_2$ and $B_1A_1$, with parameters
$q$ as just seen. Finally for a Levi subgroup of type $A_1^2$ the minimal
overgroups are of types $A_3$ and $B_2A_1$, with parameters $q^2$, $q$
respectively.
\par
As for reduction stability, by the considerations in Section~2.\ref{sec:veri}
we only need to worry about the last case. Here, the embedding
$\SO_5\SO_4\le\SO_9$ shows that the relative Weyl group $B_2$ centralises the
$A_1^2$ Levi subgroup and so any lift of the cuspidal Brauer character will
be reduction stable.
\end{proof}

The groups of type $B_4$ have 25 unipotent characters. For primes $\ell>2$
dividing $q+1$, 20 of these lie in the principal $\ell$-block, the other five
lie in a block of defect $(\Phi_2^2)_\ell$.

\begin{thm}   \label{thm:B4}
 Assume $(T_\ell)$. The $\ell$-modular decomposition matrices of the unipotent
 $\ell$-blocks of $G=B_4(q)$, $7<\ell|(q+1)$, are as given in
 Table~\ref{tab:B4} and~\ref{tab:B4bl2}. 
\end{thm}

Here $B_3^a$ denotes the Harish-Chandra series of the cuspidal unipotent
Brauer character $B_2\co1^2$ of $B_3(q)$, $B_2^b$, $B_2^c$ the ones of the
cuspidal unipotent Brauer characters $B_2\sqtens \vhi_{1^2}$ and
$\vhi_{.1^2}\sqtens\vhi_{1^2}$ of $B_2(q)A_1(q)$ and $A_1^{2*}$ the one of
the $\ell$-modular Steinberg character $\vhi_{1^2}\sqtens \vhi_{1^2}$ of
$B_1(q)A_1(q)$.

\begin{table}[ht]
\caption{$B_4(q)$, $7<\ell|(q+1)$, principal block}   \label{tab:B4}
$$\vbox{\offinterlineskip\halign{$#$\hfil\ \vrule height10pt depth4pt&
      \hfil\ $#$\ \hfil\vrule&& \hfil\ $#$\hfil\cr
      4.&        1&  1\cr
B_2$:$2.&   \hlf q&  .& 1\cr
      .4&   \hlf q&  1& .& 1\cr
     31.&   \hlf q&  1& .& .& 1\cr
B_2$:$1^2.&\hlf q^2& .& 1& .& .& 1\cr
    2^2.& \hlf q^2&  .& .& .& 1& .& 1\cr
     2.2& \hlf q^2&  1& .& 1& .& .& .& 1\cr
    21.1&      q^3&  1& .& .& 1& .& 1& 1& 1\cr
     .31& \hlf q^4&  1& .& 1& 1& .& .& .& .& 1\cr
   21^2.& \hlf q^4&  1& .& .& 1& .& 1& .& 1& .& 1\cr
   2.1^2& \hlf q^4&  1& 2& 1& .& .& .& 1& 1& .& .& 1\cr
   1^2.2&      q^4&  1& .& 1& .& .& .& 1& 1& .& .& .& 1\cr
    1.21&      q^5&  1& 2& 2& 1& 2& 1& 1& 1& 1& .& 1& 1& 1\cr
 B_2$:$.2&\hlf q^6&  .& 1& .& .& .& .& .& .& .& .& .& .& .& 1\cr
    .2^2& \hlf q^6&  .& .& .& 1& 2& 1& .& .& 1& .& .& .& 1& .& 1\cr
 1^2.1^2& \hlf q^6&  1& 2& 1& .& 2& 1& 1& 2& .& .& 1& 1& 1& .& .& 1\cr
B_2$:$.1^2&\hlf q^9& .& 1& .& .& 1& .& .& .& .& .& .& .& .& 1& 2& .& 1\cr
    1^4.& \hlf q^9&  1& .& .& .& .& .& .& 1& .& 1& .& .& .& .& .& 1& .& 1\cr
   .21^2& \hlf q^9&  1& 2& 1& 1& 2& 1& .& 1& 1&3& 1& 1& 1& 2& 1& .& .& .& 1\cr
    .1^4&   q^{16}&  1& 2& 1& .& 2& .& .& 1& .& 3& 1& 1& 1& 2& 3& 1& 4& 3& 1& 1\cr
\noalign{\hrule}
  & & ps& B_2& B_1& ps& B_2^b& A_1^2& ps& A_1& B_1& 1^3.& .1^2& A_1^{2*}& B_2^c& B_3^a& c& A_1^2& c& c& .1^3& c\cr
  }}$$
\end{table}

\begin{table}[ht]
\caption{$B_4(q)$, $5\le\ell|(q+1)$, block of defect $\Phi_2^2$}  \label{tab:B4bl2}
$$\vbox{\offinterlineskip\halign{$#$\hfil\ \vrule height11pt depth4pt&
      \hfil\ $#$\ \hfil\vrule&& \hfil\ $#$\hfil\cr
    3.1&        \hlf q\Ph2^2\Ph4\Ph6& 1\cr
    1.3&      \hlf q^2\Ph2^2\Ph6\Ph8& 1& 1\cr
B_2\co1.1&\hlf q^4\Ph1^2\Ph2^2\Ph3\Ph6& .& .& 1\cr
  1^3.1&      \hlf q^6\Ph2^2\Ph6\Ph8& 1& .& .& 1\cr
  1.1^3&      \hlf q^9\Ph2^2\Ph4\Ph6& 1& 1& 2& 1& 1\cr
\noalign{\hrule}
  & & ps& B_1& B_2& A_1& .1^2\cr
  }}$$
\end{table}

\begin{proof}
The five projectives in the non-principal unipotent $\ell$-block are obtained
by (HCi) and are easily seen to be indecomposable. This proves
Table~\ref{tab:B4bl2}. So now consider the principal block.
The three columns labelled ``ps'' come from the decomposition matrix of the
Hecke algebra $\cH(B_4;q;q)$. Harish-Chandra inducing the unipotent PIMs from
proper Levi subgroups $L$ of $G$ and cutting by the principal block we
obtain projectives which are non-negative integral linear combinations of the
sixteen columns $\Psi_1,\ldots,\Psi_{14},\Psi_{16},\Psi_{19}$ in our table.
(For Levi subgroups of type $B_3$ the decomposition matrix, depending on
$(q+1)_\ell$, was obtained in \cite[Table~5]{HN14}.) Furthermore, among
these induced projectives we actually find all of the $\Psi_i$ not labelled
``ps", except that instead of $\Psi_6$ we obtain $\Psi_6+\Psi_7$ and
$\Psi_6+\Psi_{19}$. Since the space spanned by these projectives together with
$\Psi_7$ and $\Psi_{19}$ is only 3-dimensional, we conclude that $\Psi_6$ is
also a projective character.
\par
We next claim that all of the $\Psi_i$ are indecomposable. For
$i\notin\{3,6,8\}$ this follows by application of (HCr). For the remaining
three columns there is only one possible non-trivial decomposition each, into
$$\Psi_3=\Psi_3'+\Psi_7',\quad \Psi_6=\Psi_6'+\Psi_{8}',\quad
  \Psi_{8}=\Psi_{8}'+\Psi_{10}'$$
(with $\Psi_7'+\Psi_{8}'=\Psi_7$). The three projective characters are induced
from Steinberg PIMs in the series $B_1$, $A_1^2$, $A_1$. The corresponding
Hecke algebras were determined in Lemma~\ref{lem:paramB4,d=2}, by
Table~\ref{tab:hecke 2Dn,d=2} they have 3,2,2 modular irreducibles,
respectively. We already saw that the other unipotent block of $G$ has
one PIM in each of the series $B_1$ and $A_1$. If $\Psi_3$ decomposes, its
summands must lie in the $B_1$-series, which would produce four PIMs in that
series, but we just argued that there should be exactly three. The same
argument applies to the other two series. Thus all columns not labelled ``c"
in Table~\ref{tab:B4} correspond to PIMs. Since $(q+1)_\ell > 8$,
Corollary~\ref{cor:family4} gives the columns $\Psi_{17}$ and $\Psi_{18}$.
\par
Now by uni-triangularity there are $x_i \geq 0$ such that
$$\Psi_{15} = [.2^2]+x_1[B_2\co.1^2]+x_2[1^4.]+x_3[.21^2]+x_4[.1^4].$$
For determining the $x_i$'s we proceed as usual. We first use
Theorem \ref{thm:genpos} to get $x_4 = 4x_1+3x_2+x_3-6$. Then we decompose
the Deligne--Lusztig character $R_w$ for $w$ a Coxeter element on the
basis of PIMs. We have, up to adding and removing non-unipotent characters
$$\begin{aligned}
  R_w= &\, \Psi_{1}+ \Psi_{2} -\Psi_{3} -\Psi_{4} -\Psi_{5}+ \Psi_{6} - \Psi_{8}
  +\Psi_{9} +  \Psi_{12}-\Psi_{14}+\Psi_{15} -\Psi_{16} \\
& \,+  (2-x_1)\Psi_{17} -x_2\Psi_{18}+ (1-x_3)\Psi_{19}.
\end{aligned}$$
By decomposing $R_v$ for $v<w$, one checks that none of $\Psi_{17}$,
$\Psi_{18}$ $\Psi_{19}$ occurs (for $\Psi_{17}$ and $\Psi_{18}$ one could also
invoke the fact that they correspond to cuspidal modules). Then (DL) yields
$x_1 \leq 2$, $x_2=0$ and $x_3 \leq 1$. On the other hand, for a 2-split Levi
subgroup $B_2(q).(q+1)^2$ and $\ell> 4$ one can use (Red) with the unipotent
characters $[B_2]$ and $[.2]$ to get respectively
$x_1 \geq 2$ and $x_3 \geq 1$, showing $x_1=2$ and $x_3=1$.
\end{proof}

As Lusztig has shown, the unipotent characters of groups of types $B_n$ and
$C_n$ are parametrised by the same combinatorial objects. With this we may
state:

\begin{thm}   \label{thm:C4}
 Assume $(T_\ell)$. Then the $\ell$-modular decomposition matrices for the
 unipotent blocks of $G=C_4(q)$, for $8<\ell|(q+1)$, are as given in
 Tables~\ref{tab:B4} and~\ref{tab:B4bl2} for $B_4(q)$ above.
\end{thm}

\begin{proof}
All of the arguments in the proof of Theorem~\ref{thm:B4} go through for $C_4(q)$ as well.
\end{proof}

%%%%%%%%%%%%%%%%%%%%%%%%%%%%%%%%%%%%%%
\section{Unipotent decomposition matrix of $F_4(q)$}

The groups of type $F_4$ have 37 unipotent characters. For $d_\ell(q)=2$ 
and $\ell>3$, 25 of these lie in the principal $\ell$-block, five more lie in
a block of defect $(q+1)^2_\ell$, and the last seven are of defect~0.
For $p$ good, the decomposition matrices of the unipotent $\ell$-blocks of
$F_4(q)$ were partially computed by K\"ohler in \cite{Koe06}. He completely
determined the matrix for the non-principal block of positive defect. We
obtain here most of the entries that were left undetermined for the principal
block.

\begin{thm}   \label{thm:F4}
 The decomposition matrix for the principal $\ell$-block of $F_4(q)$,
 $(q,6)=1$, with $3<\ell|(q+1)$ and $(q+1)_\ell>11$ is approximated as
 given in Table~\ref{tab:F4}.
  \par
 Here, the unknown parameters satisfy
 $$x_6 = 4-2x_1-2x_2+2x_3\quad\text{and}\quad  x_7 = 4-x_3+2x_4+2x_5.$$
 Furthermore, $x_1,x_2,x_3,z \geq 2$.
\end{thm}

\begin{table}[ht]
\caption{$F_4(q)$, $\ell >3$ and $(q+1)_\ell > 11$}\label{tab:F4}
\begin{small}
$$\vbox{\offinterlineskip\halign{$#$\hfil\ \vrule height9pt depth4pt&
      \hfil\ $#$\ \hfil\vrule&& \hfil\ $#$\hfil\cr
    \phi_{1,0}&               1&  1\cr
  \phi_{2,4}''&          \hlf q&  .& 1\cr
   \phi_{2,4}'&          \hlf q&  .& .& 1\cr
       B_2$:$2.&         \hlf q&  .& .& .& 1\cr
    \phi_{9,2}&             q^2&  .& 1& 1& .& 1\cr
   \phi_{8,3}'&             q^3&  1& .& 1& .& 1& 1\cr
  \phi_{8,3}''&             q^3&  1& 1& .& .& 1& .& 1\cr
   F_4^{II}[1]& \frac{1}{24}q^4&  .& .& .& .& .& .& .& 1\cr
  \phi_{6,6}''& \frac{1}{12}q^4&  1& .& .& 2& 1& 1& 1& .& 1\cr
   \phi_{9,6}'&        \egt q^4&  .& 1& 1& .& 1& 1& .& .& .& 1\cr
  \phi_{1,12}'&        \egt q^4&  1& .& .& .& .& 1& .& .& .& .& 1\cr
  \phi_{9,6}''&        \egt q^4&  .& 1& 1& .& 1& .& 1& .& .& .& .& 1\cr
      F_4^I[1]&        \egt q^4&  .& .& .& .& .& .& .& .& .& .& .& .& 1\cr
 \phi_{1,12}''&        \egt q^4&  1& .& .& .& .& .& 1& .& .& .& .& .& .& 1\cr
       F_4[-1]&       \frth q^4&  .& .& .& .& .& .& .& .& .& .& .& .& .& .& 1\cr
    B_2$:$1^2.&       \frth q^4&  .& .& .& 1& .& .& .& .& .& .& .& .& .& .& .& 1\cr
      B_2$:$.2&       \frth q^4&  .& .& .& 1& .& .& .& .& .& .& .& .& .& .& .& .& 1\cr
   \phi_{6,6}'&       \thrd q^4&  1& .& .& .& 1& 1& 1& .& .& .& .& .& .& .& .& .& .& 1\cr
   \phi_{8,9}'&             q^9&  1& 1& .& 2& 1& 2& 1& x_1& 1& 1& 3& .& .& .& 2& 2& .& 1& 1\cr
  \phi_{8,9}''&             q^9&  1& .& 1& 2& 1& 1& 2& x_2& 1& .& .& 1& .& 3& 2& .& 2& 1& .& 1\cr
   \phi_{9,10}&          q^{10}&  .& 1& 1& 2& 1& 1& 1& x_3& 1& 1& 2& 1& y& 2& z& 2& 2& 1& 1& 1& 1\cr
 \phi_{2,16}''&     \hlf q^{13}&  .& .& 1& .& .& .& .& x_4& .& .& .& 1& y\pl2& 3& z& .& 2& .& .& 1& 1& 1\cr
  \phi_{2,16}'&     \hlf q^{13}&  .& 1& .& .& .& .& .& x_5& .& 1& 3& .& y\pl2& .& z& 2& .& .& 1& .& 1& .& 1\cr
     B_2$:$.1^2&    \hlf q^{13}&  .& .& .& 1& .& .& .& x_6& .& .& .& .& 2y& .& 2z\mn4& 1& 1& .& .& .& 2& .& .& 1\cr
   \phi_{1,24}&          q^{24}&  1& .& .& 2& .& 1& 1& x_7& 1& .& 3& .& 3y\pl 4& 3& 3z& 2& 2& 1& 1& 1& 3& 2& 2& 4& 1\cr
\noalign{\hrule}
  & & ps& ps& ps& C_2& ps& A_1& C_1& c& .1^2& A_1& 1^3.& C_1& c& 1^3.& c& B_3^a& B_3^a& B_2^b& .1^3& .1^3& c& c& c& c & c\cr
  }}$$
\end{small}
\end{table}

Here $B_3^a$ denotes the Harish-Chandra series of the cuspidal unipotent
Brauer character $B_2\co\vhi_{1^2}$ of $B_3(q)$, and $A_1^{2*}$ the one of the
$\ell$-modular Steinberg character $\vhi_{1^2}\sqtens\vhi_{1^2}$ of
$\tilde A_1(q)A_1(q)$.

\begin{proof}
The values $a,b,c,d$ left undetermined in \cite[T.A.157]{Koe06} are obtained
by Harish-Chandra induction from the decomposition matrices of the Levi factors
$B_3(q)$ and $C_3(q)$: using the values
given in \cite[Table~5 and Thm.~4.3]{HN14}, we get $b,d \leq \gamma = 2$ and
$a,c \leq \beta-1 = 2$ (recall that $(q+1)_\ell>5$). This forces $a=b=c=d=2$.
Furthermore \cite[T.A.157]{Koe06} gives the following approximation to the
lower right hand corner of the decomposition matrix:
$$\begin{array}{c|ccccc}
   \phi_{9,10}&    1\cr
 \phi_{2,16}''&    1& 1\cr
  \phi_{2,16}'&    1& .& 1\cr
     B_2$:$.1^2&   b_1& .& .& 1\cr
   \phi_{1,24}&    b_2& b_3& b_4& b_5& 1\cr
\end{array}$$
Now, Theorem~\ref{thm:genpos} yields the values of $b_3,b_4$ and $b_5$,
together with the relation $b_2 = 4b_1-5$.

Let $w$ be a Coxeter element of the Weyl group of $G=F_4(q)$. The corresponding
Deligne--Lusztig character $R_w$ has the following decomposition in terms of
projective characters in the principal block $B_0$:
$$B_0R_w = \Psi_1+\Psi_4-\Psi_6-\Psi_7+\Psi_8-\Psi_{13}+\Psi_{14}+\Psi_{17}
        -\Psi_{18}+\Psi_{21}+(2-b_1)\Psi_{24}.$$
By \cite[Prop.~1.5]{Du13}, we deduce that $b_1 \leq2$, and therefore $b_1=2$
and $b_2=3$.

We now focus on the columns corresponding to the ordinary cuspidal unipotent
characters, whose entries will be denoted by $x_1,\ldots,x_{21}$ as follows:
\begin{small}
$$\vbox{\offinterlineskip\halign{$#$\hfil\ \vrule height10pt depth4pt&&
      \hfil\ $#$\hfil\cr
   F_4^{II}[1]&	   1\cr
      F_4^I[1]&   .&        1\cr
       F_4[-1]&   .&       .&        1\cr
   \phi_{8,9}'& x_1&     x_8&   x_{15}\cr
  \phi_{8,9}''& x_2&     x_9&   x_{16}\cr
   \phi_{9,10}& x_3&  x_{10}&   x_{17}\cr
 \phi_{2,16}''& x_4&  x_{11}&   x_{18}\cr
  \phi_{2,16}'& x_5&  x_{12}&   x_{19}\cr
    B_2$:$.1^2& x_6&  x_{13}&   x_{20}\cr
   \phi_{1,24}& x_7&  x_{14}&   x_{21}\cr }}$$
\end{small}
Let $w =s_1s_2s_3s_4s_2s_3$. The coefficients of the PIMs $\Psi_{19}$,
$\Psi_{22}$ and $\Psi_{24}$ are  respectively $2-x_{15}$,
$2-x_{15}+x_{17}-x_{18}$ and $4-2x_{15}-2x_{16}+2x_{17}-x_{20}$. They are all
non-negative by (DL) and they add up to
$8-4x_{15}-2x_{16}+3x_{17}-x_{18}-x_{20}$. But  when $(q+1)_\ell > 7$, this sum
should also be non-positive by (Red) applied to the trivial character of
$\widetilde A_1(q).(q+1)^3$ (see Proposition \ref{prop:q+1 reg}).
Therefore we must have
$$x_{15} = 2,\quad x_{18} = x_{17}\quad\text{and}\quad
  x_{20} = 2x_{17}-2x_{16}.$$
Similarly, the coefficients of $\Psi_{20}$, $\Psi_{23}$ and $\Psi_{24}$ are
respectively $2-x_{16}$, $2-x_{16}+x_{17}-x_{19}$ and
$4-2x_{15}-2x_{16}+2x_{17}-x_{20}$. They add up to a number which is both
non-negative by (DL) and non-positive by (Red) applied to the trivial
character of $A_1(q).(q+1)^3$. Therefore
$$x_{16}=2,\quad x_{19} = x_{17}\quad\text{and}\quad x_{20} = 2x_{17}-4.$$
Then by Theorem~\ref{thm:genpos} we get $x_{21} = 3x_{17}$. We set $z:=x_{17}$.

Let $w=(s_1s_2s_3s_4)^2$ and $R_w$ be the corresponding Deligne--Lusztig
character. We proceed as in the previous paragraph. The sum of the coefficients
of $\Psi_{19}$, $\Psi_{22}$ and $\Psi_{24}$ is
$$(-x_8) + (2-x_8+x_{10}-x_{11}) + (-2x_8-2x_9+2x_{10}-x_{13})$$
and the sum of the coefficients of $\Psi_{20}$, $\Psi_{24}$ and $\Psi_{24}$ is
$$(-x_9) + (2-x_9+x_{10}-x_{12}) + (-2x_8-2x_9+2x_{10}-x_{13}).$$
Both are sums of non-negative integers by (DL) and are
non-positive by (Red) (for the same unipotent characters as before).
We deduce that
$$x_8 = x_9 = 0,\quad x_{11} = x_{12} = x_{10}+2,\quad\text{and}\quad
  x_{13} = 2x_{10}.$$
With Theorem~\ref{thm:genpos} we get $x_{14} = 3x_{10}+4$. We set $y:=x_{10}$.

We now consider the Deligne--Lusztig character associated with
$w=(s_1s_2s_3s_4)^3$. Only the PIMs $\Psi_{24}$ and $\Psi_{25}$ do not occur
in $R_v$ for $v <w$. The coefficient of $\Psi_{24}$ on $R_w$ equals
$4-2x_1-2x_2+2x_3-x_6$ and therefore it must be non-negative by (DL). On the
other hand, if $(q+1)_\ell > 4$ on can invoke (Red) for the cuspidal unipotent
character of the Levi subgroup $B_2(q).(q+1)^2$ to ensure that it is also
non-positive. We deduce that $x_6 =  4-2x_1-2x_2+2x_3$. The relation
$x_7 =4-x_3+2x_4+2x_5$ now follows from Theorem~\ref{thm:genpos}.

Finally, we use (Red) to obtain lower bounds on some of the missing entries.
With the trivial character of the $2$-split Levi subgroups of type
$\tw2 A_2(q).(q+1)^2$ we find $x_1 \geq 2$ and $x_2 \geq2$. Using the character 
$[21]\boxtimes [2]$ of a $2$-split Levi subgroup of type
$\tw2A_2(q) A_1(q).(q+1)$ we find  $x_3 \geq x_1+x_2-2$ hence $x_3 \geq 2$. 
\end{proof}

\begin{rem}
As before we can obtain upper bounds on the missing decomposition numbers if
\cite[Conj.~1.2]{DM14} holds.
Let $w_1 = s_2s_3s_2s_1s_3s_2s_3s_4s_3s_2s_1s_3s_4$. We have
$$\langle Q_{w_1} ; \varphi_{19} \rangle  = 4-2x_1 \quad \text{and} \quad
  \langle Q_{w_1} ; \varphi_{20} \rangle  = 4-2x_2,
$$
which would show that $x_1=x_2=2$. Now with 
$w_2 = s_1s_2s_1s_3s_2s_1s_4s_3s_2s_1s_3s_2s_4s_3s_2s_1$ and
$w_3=s_3s_4s_3s_2s_1s_3s_2s_3s_4s_3s_2s_1s_3s_2s_3s_4$ we have
$$ \begin{aligned}
  \langle Q_{w_2} ; \varphi_{22} \rangle & = 72+12x_3-12x_4,\\
  \langle Q_{w_3} ; \varphi_{21} \rangle & = 156-12z-12x_3, \\
  \langle Q_{w_3} ; \varphi_{23} \rangle & = 72+12x_3-12x_5.\\
\end{aligned}$$
This gives $x_4,x_5 \leq x_3+6$ and $x_3+z \leq 13$. Finally with
$w_4 = s_3s_2s_3s_4s_3s_2s_1s_3s_4$ and 
$w_5 = s_2s_1s_4s_3s_2s_1s_3s_2s_3$ we find
$$\langle Q_{w_4} ; \varphi_{21} \rangle  = 18-6z \quad \text{and} \quad
  \langle Q_{w_4} ; \varphi_{21} \rangle  = 14-2y.
$$
We conclude that $y \leq 7$ and $z \in \{2,3\}$. Using the previous inequalities
we obtain $x_3\leq 11$ and $x_4,x_5 \leq 17$.
\end{rem}

%%%%%%%%%%%%%%%%%%%%%%%%%%%%%%%%%%%%%%%%%%%%%%%%%%%%%%%%%%%%%%%%%%%%%%%%%
\chapter{Decomposition matrices at $d_\ell(q)=3$}   \label{chap:d=3}

Here, we consider decomposition matrices of unipotent blocks of groups of Lie
type $G=G(q)$ for primes $\ell$ with $d_\ell(q)=3$, so $\ell|(q^2+q+1)$ and in
particular $\ell\ge7$. If $G$ is of classical type, then such primes $\ell$
with are \emph{linear} for $G$, and so the decomposition numbers are known
by work of Gruber and Hiss \cite{GrHi} to be given by suitable $q$-Schur
algebras. (This implies, for example, that the block distribution of unipotent
characters refines the subdivision into ordinary Harish-Chandra series, and
that $(T_\ell)$ is always satisfied.) Nevertheless, to our knowledge they have
never been written out explicitly, so we derive them here, also as an induction
base for blocks of groups of exceptional type for which the theory of
linear primes from \cite{GrHi} does not apply.

%%%%%%%%%%%%%%%%%%%%%%%%%%%%
\section{Even-dimensional split orthogonal groups}
We begin with groups of type $D_n$ for $n\le7$. The Brauer trees of unipotent
blocks with cyclic defect were first determined by Fong and Srinivasan
\cite{FS90} and in our situation can also easily be obtained by Harish-Chandra
induction:

\begin{prop}   \label{prop:trees Dn d=3}
 Let $q$ be a prime power and $\ell$ a prime with $d_\ell(q)=3$.
 The Brauer trees of the unipotent $\ell$-blocks of $D_n(q)$, $3\le n\le 7$,
 with cyclic defect are as given in Table~\ref{tab:Dn,d=3,def1}.
\end{prop}

\begin{table}[ht]
\caption{Brauer trees for $D_n(q)$ ($3\le n\le 7$), $7\le\ell| (q^2+q+1)$ } \label{tab:Dn,d=3,def1}
$$\vbox{\offinterlineskip\halign{$#$
        \vrule height10pt depth 2pt width 0pt&& \hfil$#$\hfil\cr
D_3(q):& .3& \vr& .21& \vr& .1^3& \vr& \bigcirc\cr
\cr
D_5(q):& 1.4& \vr& 1.2^2& \vr& 1.1^4& \vr& \bigcirc\cr
\cr
D_7(q):& 2.5& \vr& 2.2^21& \vr& 2.21^3& \vr& \bigcirc\cr
\cr
 & 1^2.41& \vr& 1^2.32& \vr& 1^2.1^5& \vr& \bigcirc\cr
 & & ps& & ps& & A_2\cr
\cr
D_4(q):& 1.3& \vr& 1.21& \vr& 1.1^3& \vr& \bigcirc& \vr& .1^4 & \vr& .2^2& \vr& .4\cr
\cr
D_5(q):& .5& \vr& .2^21& \vr& .21^3& \vr& \bigcirc& \vr& 2.1^3& \vr& 2.21& \vr& 2.3\cr
\cr
 & .41& \vr& .32& \vr& .1^5& \vr& \bigcirc& \vr& 1^2.1^3& \vr& 1^2.21& \vr& 1^2.3\cr
\cr
D_6(q):& 1.5& \vr& 1.2^21& \vr& 1.21^3& \vr& \bigcirc& \vr& 2.1^4& \vr& 2.2^2& \vr& 2.4\cr
 \cr
 & 1^2.4& \vr& 1^2.2^2& \vr& 1^2.1^4& \vr& \bigcirc& \vr&  1.1^5& \vr& 1.32& \vr& 1.41\cr
\cr
D_7(q):& .61& \vr& .32^2& \vr& .31^4& \vr& \bigcirc& \vr& 1^3.31& \vr& 21.31& \vr& 3.31\cr
\cr
 & .51^2& \vr& .3^21& \vr& .21^5& \vr& \bigcirc& \vr& 1^3.21^2& \vr& 21.21^2& \vr& 3.21^2\cr
\cr
 & 1^2.5& \vr& 1^2.2^21& \vr& 1^2.21^3& \vr& \bigcirc& \vr& 2.1^5& \vr& 2.32& \vr& 2.41\cr
 & & ps& & ps& & A_2& & A_2& & ps& & ps\cr
\cr
D_7(q): & D_4\co3.& \vr& D_4\co21.& \vr& D_4\co1^3.& \vr& \bigcirc& \vr& D_4\co.1^3& \vr& D_4\co.21& \vr& D_4\co.3\cr
 & & D_4& & D_4& & D_4A_2& & D_4A_2& & D_4& & D_4\cr
  }}$$
\end{table}

Here and later on, we label the ordinary unipotent characters by their
Harish-Chandra series; for the principal series this means by the irreducible
characters of a Weyl group of type $D_n$, hence by unordered pairs of
partitions of $n$, and for characters in the Harish-Chandra series of the
cuspidal unipotent character of a Levi subgroup of type $D_4$ by the symbol
``$D_4$'' and a character of the relative Weyl group, which is of type
$B_{n-4}$, hence by a bipartition of $n-4$.

Under the edges of the Brauer trees, which represent the irreducible Brauer
characters (or equivalently the PIMs) of the block, we have indicated their
corresponding $\ell$-modular Harish-Chandra series; here ``ps'' stands for the
principal series, while ``$A_2$'' stands for the series of the cuspidal
$\ell$-modular Steinberg character of a Levi subgroup of type $A_2$.

Again, we first need to determine the parameters of the Hecke algebras
attached to cuspidal $\ell$-modular Brauer characters of certain Levi
subgroups:

\begin{lem}   \label{lem:paramDn,d=3}
 Let $q$ be a prime power and $\ell|(q^2+q+1)$.
 \begin{enumerate}[\rm(a)]
  \item The Hecke algebra for the cuspidal $\ell$-modular Steinberg character
   $\vhi_{1^3}$ of a Levi subgroup of type $A_2$ inside $D_n(q)$, $n\ge4$, is
   $\cH(B_{n-3};1;q)$.
  \item The Hecke algebra for the cuspidal $\ell$-modular Steinberg character
   $\vhi_{1^3}^{\sqtens2}$ of a Levi subgroup of type $A_2^2$ inside $D_n(q)$
   is $\cH(A_1;q^3)\otimes \cH(A_1;q^3)$ if $n=6$ and
   $\cH(B_2;q^3;1)\otimes \cH(D_{n-6};q)$ when $n\ge7$ (where $D_1$ has to be
   interpreted as the trivial group).
 \end{enumerate}
\end{lem}

\begin{proof}
First, by \cite[p.~72]{Ho80} the relative Weyl group of $A_2$ inside $D_n$ is
of type $B_{n-3}$. The parameters of the corresponding Hecke algebra are
determined by the parameters inside the minimal Levi subgroups above $A_2$,
viz.\ those of types $D_4$ and $A_2A_1$. Clearly the parameter inside the
product $A_2(q)A_1(q)$ is equal to $q$, by \cite[Lemma~3.19]{GHM2}. Now the
$\ell$-modular cuspidal Steinberg character of $L=A_2(q)$ is liftable to an
ordinary cuspidal character $\la$ by \cite[Thm.~7.8]{GHM}, lying in the
Lusztig series of a regular semisimple $\ell$-element of $L^*$. Then
the parameters of the relative Hecke algebra inside $M=D_4(q)$ can be
determined as the $\ell$-modular reduction of the quotient of the degrees of
the two ordinary constituents of $\RLM(\la)$, see \cite[Lemma~3.17]{GHM2}.
This equals~1 as claimed in~(a). Reduction stability holds by
Example~\ref{exmp:hecke}(a).
\par
In~(b), the cuspidal modular Steinberg character of a Levi subgroup of type
$A_2^2$ is the exterior tensor product of those of the two factors, so is
again liftable to
characteristic zero by \cite[Thm.~7.8]{GHM}. Again by \cite[p.~72]{Ho80} the
relative Weyl group has type as stated. The minimal Levi overgroups here are
of types $A_5$ (twice) inside $D_6$, and of type $D_4A_2$ and $A_2^2A_1$ in
$D_7$, $D_8$ respectively, and the parameters in either case can again be
determined using \cite[Lemma~3.17]{GHM2}. Here, reduction stability holds if we
choose the same lift to characteristic~0 in both factors.
\end{proof}

The only unipotent $\ell$-blocks of non-cyclic defect of $D_6(q)$ and $D_7(q)$
are their principal blocks.

\begin{prop}   \label{prop:D6,d=3}
 The decomposition matrix for the principal $\ell$-block of $D_6(q)$ for
 primes $7\le \ell| (q^2+q+1)$ is as given in Table~\ref{tab:D6,d=3}.
\end{prop}

Here, in the second column we print the degrees of the corresponding unipotent
characters as products of cyclotomic polynomials evaluated at $q$.

\begin{table}[ht]
\caption{$D_6(q)$, $7\le\ell| (q^2+q+1)$}   \label{tab:D6,d=3}
$$\vbox{\offinterlineskip\halign{$#$\hfil\ \vrule height11pt depth4pt&
      \hfil\ $#$\ \hfil\vrule&& \hfil\ $#$\hfil\cr
    .6&                                 1& 1\cr
   .51&                    q^2\Ph5\Ph{10}& 1& 1\cr
    3+&                q^3\Ph5\Ph8\Ph{10}& .& .& 1\cr
    3-&                q^3\Ph5\Ph8\Ph{10}& .& .& .& 1\cr
  .3^2&     \hlf q^4\Ph5\Ph6^2\Ph8\Ph{10}& .& 1& .& .& 1\cr
  21.3&   \hlf q^4\Ph2^4\Ph5\Ph6^2\Ph{10}& .& .& 1& 1& .& 1\cr
 .41^2&                q^6\Ph5\Ph8\Ph{10}& .& 1& .& .& .& .& 1\cr
 1^3.3&      \hlf q^7\Ph4^2\Ph5\Ph6^2\Ph8& .& .& .& .& .& 1& .& 1\cr
  .321&   \hlf q^7\Ph2^4\Ph6^2\Ph8\Ph{10}& 1& 1& .& .& 1& .& 1& .& 1\cr
   21+&          q^7\Ph4^2\Ph5\Ph8\Ph{10}& .& .& 1& .& .& 1& .& .& .& 1\cr
   21-&          q^7\Ph4^2\Ph5\Ph8\Ph{10}& .& .& .& 1& .& 1& .& .& .& .& 1\cr
  .2^3&  \hlf q^{10}\Ph5\Ph6^2\Ph8\Ph{10}& 1& .& .& .& .& .& .& .& 1& .& .& 1\cr
1^3.21&\hlf q^{10}\Ph2^4\Ph5\Ph6^2\Ph{10}& .& .& .& .& .& 1& .& 1& .& 1& 1& .& 1\cr
 .31^3&             q^{12}\Ph5\Ph8\Ph{10}& .& .& .& .& .& .& 1& .& 1& .& .& .& .& 1\cr
  1^3+&             q^{15}\Ph5\Ph8\Ph{10}& .& .& .& .& .& .& .& .& .& 1& .& .& 1& .& 1\cr
  1^3-&             q^{15}\Ph5\Ph8\Ph{10}& .& .& .& .& .& .& .& .& .& .& 1& .& 1& .& .& 1\cr
 .21^4&                 q^{20}\Ph5\Ph{10}& .& .& .& .& 1& .& .& .& 1& .& .& 1& .& 1& .& .& 1\cr
  .1^6&                            q^{30}& .& .& .& .& 1& .& .& .& .& .& .& .& .& .& .& .& 1& 1\cr
\noalign{\hrule}
 \omit& & ps& ps& ps& ps& ps& ps& ps& A_2& ps& ps& ps& A_2^2& A_2& A_2& A_2^2& A_2^2& A_2& A_2^2\cr
   }}$$
\end{table}

\begin{proof}
By Harish-Chandra induction we obtain the projectives in the principal series
as well as those in the $A_2$-series, which are easily seen to be
indecomposable by (HCr). Since the centraliser of a Sylow $\ell$-subgroup of
$D_6(q)$ is contained in a Levi subgroup of type $A_5$, there are no cuspidal
Brauer characters by (Csp) so the remaining four PIMs must belong to Brauer
characters lying in the Harish-Chandra series above the cuspidal Steinberg PIM
of a Levi subgroup of type $A_2^2$. The corresponding Hecke algebra
$\cH(A_1;q^3)^{\sqtens2}$, determined in Lemma~\ref{lem:paramDn,d=3}(b), remains
semisimple modulo $\ell$. Now Harish-Chandra induction also yields the
projectives $\Psi_{10}+\Psi_{12}+\Psi_{15}$, $\Psi_{16}+\Psi_{18}$,
$\Psi_{11}+\Psi_{12}+\Psi_{16}$, $\Psi_{15}+\Psi_{18}$, and
$\Psi_{12}+\Psi_{15}+\Psi_{16}+\Psi_{18}$, where $\Psi_i$ denotes the linear
combination of unipotent characters with coefficients as given in the $i$th
column of our table. The only way that these can split into four characters
all satisfying (HCr) is the one given.
\end{proof}

\begin{prop}   \label{prop:D7,d=3}
 The decomposition matrix for the principal $\ell$-block of $D_7(q)$ for
 primes $7\le\ell| (q^2+q+1)$, is as given in Table~\ref{tab:D7,d=3}.
\end{prop}

\begin{table}[ht]
{\small\caption{$D_7(q)$, $7\le\ell| (q^2+q+1)$}   \label{tab:D7,d=3}
$$\vbox{\offinterlineskip\halign{$#$\hfil\ \vrule height9pt depth4pt&
      \hfil\ $#$\ \hfil\vrule&& \hfil\ $#$\hfil\cr
     .7&           1& 1\cr
    1.6&           q& .& 1\cr
   1.51&    \hlf q^3& .& 1& 1\cr
    .52&    \hlf q^3& 1& .& .& 1\cr
    3.4&         q^3& .& .& .& .& 1\cr
    .43&    \hlf q^4& .& .& .& 1& .& 1\cr
   21.4&    \hlf q^4& .& .& .& .& 1& .& 1\cr
  2^2.3&    \hlf q^6& .& .& .& .& 1& .& .& 1\cr
  1.3^2&    \hlf q^6& .& .& 1& .& .& .& .& .& 1\cr
   .421&    \hlf q^7& 1& .& .& 1& .& 1& .& .& .& 1\cr
  1^3.4&    \hlf q^7& .& .& .& .& .& .& 1& .& .& .& 1\cr
 1.41^2&    \hlf q^7& .& .& 1& .& .& .& .& .& .& .& .& 1\cr
 21.2^2&         q^9& .& .& .& .& 1& .& 1& 1& .& .& .& .& 1\cr
  1.321&         q^9& .& 1& 1& .& .& .& .& .& 1& .& .& 1& .& 1\cr
  1.2^3& \hlf q^{12}& .& 1& .& .& .& .& .& .& .& .& .& .& .& 1& 1\cr
1^3.2^2& \hlf q^{12}& .& .& .& .& .& .& 1& .& .& .& 1& .& 1& .& .& 1\cr
  .41^3&      q^{12}& .& .& .& .& .& .& .& .& .& 1& .& .& .& .& .& .& 1\cr
 1.31^3& \hlf q^{13}& .& .& .& .& .& .& .& .& .& .& .& 1& .& 1& .& .& .& 1\cr
 .321^2& \hlf q^{13}& 1& .& .& .& .& 1& .& .& .& 1& .& .& .& .& .& .& 1& .& 1\cr
  1^4.3& \hlf q^{13}& .& .& .& .& .& .& .& 1& .& .& 1& .& .& .& .& .& .& .& .& 1\cr
  .2^31& \hlf q^{16}& 1& .& .& .& .& .& .& .& .& .& .& .& .& .& .& .& .& .& 1& .& 1\cr
 1^4.21& \hlf q^{16}& .& .& .& .& .& .& .& 1& .& .& 1& .& 1& .& .& .& .& .& .& 1& .& 1\cr
 1.21^4& \hlf q^{21}& .& .& .& .& .& .& .& .& 1& .& .& .& .& 1& 1& .& .& 1& .& .& .& .& 1\cr
.2^21^3& \hlf q^{21}& .& .& .& .& .& 1& .& .& .& .& .& .& .& .& .& .& 1& .& 1& .& 1& .& .& 1\cr
1^3.1^4&      q^{21}& .& .& .& .& .& .& .& .& .& .& .& .& 1& .& .& 1& .& .& .& .& .& 1& .& .& 1\cr
  1.1^6&      q^{31}& .& .& .& .& .& .& .& .& 1& .& .& .& .& .& .& .& .& .& .& .& .& .& 1& .& .& 1\cr
   .1^7&      q^{42}& .& .& .& .& .& 1& .& .& .& .& .& .& .& .& .& .& .& .& .& .& .& .& .& 1& .& .& 1\cr
\noalign{\hrule}
% \omit& & p& p& p& p& p& p& p& p& p& p& A_2& p& p& p& A_2^2& A_2& A_2& A_2& p& A_2& A_2^2& A_2& A_2& A_2& A_2^2& A_2^2& A_2^2\cr
 \omit& & ps& ps& ps& ps& ps& ps& ps& ps& ps& ps& A_2\!& ps& ps& ps& A_2^2\!& A_2\!& A_2\!& A_2\!& ps& A_2\!& A_2^2\!& A_2\!& A_2\!& A_2\!& A_2^2\!& A_2^2\!& A_2^2\!\cr
   }}$$}
\end{table}

\begin{proof}
Again, since the centraliser of a Sylow $\ell$-subgroup of $D_7(q)$ is
contained in a Levi subgroup of type $A_5$, there are no cuspidal Brauer
characters by (Csp) so all Brauer characters must lie in the Harish-Chandra
series above cuspidal characters of proper Levi subgroups. We investigate these
in turn.   \par
The Hecke algebra $\cH(D_7;q)$ for the principal series can be considered as a
subalgebra of index~2 of an Iwahori--Hecke algebra $\cH(B_7;1;q)$. The latter
is Morita equivalent
to a sum of tensor products of Hecke algebras of type $A_n$, $n\le7$, by
\cite[4.7]{DJ92}, and the decomposition matrices of the latter are known by
the work of James \cite[p.~259]{Ja90}. This gives the 14 columns corresponding
to the principal series, for all $\ell\ge7$. Next, the columns $\Psi_i$, for
$i\in\{18,20,22,23,25\}$ are obtained by Harish-Chandra induction. Further,
the projectives
$\Psi_3+\Psi_7+\Psi_{17}$, $\Psi_{11}+\Psi_{20}$ and $\Psi_{17}+\Psi_{18}$ yield
$\Psi_{11}$ and $\Psi_{17}$; while $\Psi_9+\Psi_{16}+\Psi_{19}+\Psi_{24}$,
$\Psi_{14}+\Psi_{16}+\Psi_{20}+2\Psi_{22}+\Psi_{23}$ and $\Psi_{23}+\Psi_{24}$
yield $\Psi_{16}$ and $\Psi_{24}$. Harish-Chandra induction also gives the
projectives $\Psi_{15}+\Psi_{21}$, $\Psi_{15}+\Psi_{26}$, $\Psi_{21}+\Psi_{27}$
and $\Psi_{26}+\Psi_{27}$ (modulo addition of known projectives). The Hecke
algebra $H(B_2;q^3;1)$ for the cuspidal $\ell$-modular Steinberg character
$\vhi_{1^3}^{\sqtens2}$ of a Levi subgroup of type
$A_2^2$ remains semisimple modulo~$\ell$ by Lemma~\ref{lem:paramDn,d=3}(b).
This yields the remaining four columns of the asserted decomposition matrix.
Then (HCr) shows that all columns are indecomposable.
\end{proof}

%%%%%%%%%%%%%%%%%%%%%%%%%%%%
\section{Unipotent decomposition matrix of $E_6(q)$}
The triangular shape of the decomposition matrix for the unipotent blocks
of $E_6(q)$, when $q$ is a power of a good prime for $E_6$, has been shown by
Geck and Hiss \cite[7.5]{GH97} in the case when $d_\ell(q)=3,6$, by using
generalised Gelfand--Graev characters. This shows that $(T_\ell)$ is satisfied
under these assumptions on $q$. In fact, for $d_\ell(q)=3$, Geck and Hiss
\cite[Table 3]{GH97} give an approximation to the decomposition matrix
involving four unknown entries, plus the unknown decomposition of the two
ordinary cuspidal characters.

\begin{lem}   \label{lem:paramE6,d=3}
 Let $q$ be a prime power and $\ell|(q^2+q+1)$. The Hecke algebra for the
 cuspidal $\ell$-modular Steinberg character $\vhi_{1^3}^{\sqtens2}$ of a Levi
 subgroup of type $A_2^2$ inside $E_6(q)$ is $\cH(G_2;q^3;q)$, with four
 irreducible characters.
\end{lem}

This is easily seen as in the previous cases; reduction stability holds since the
normaliser just interchanges the two $A_2$-factors and we can choose the same
Steinberg module in both factors.

\begin{thm}   \label{thm:E6,d=3}
 Let $(q,6)=1$. Then the decomposition matrix for the principal $\ell$-block
 of $E_6(q)$ for primes $\ell>3$ with $(q^2+q+1)_\ell>7$ is as given in
 Table~\ref{tab:E6,d=3}.
 \par
 Here, the unknown parameters satisfy
 $a_8\leq -1-a_3-a_4+a_6+a_7$ and the relations
 $$a_5 = 1-a_1-a_2+a_3,\quad\text{and}\quad a_9 = 3+2a_3+3a_4-3a_6-2a_7+3a_8.$$
\end{thm}

There is a further unipotent $\ell$-block of cyclic defect with Brauer tree
$$\vbox{\offinterlineskip\halign{$#$
        \vrule height10pt depth 2pt width 0pt&& \hfil$#$\hfil\cr
D_4:3& \vr& D_4:21& \vr& D_4:1^3& \vr& \bigcirc\cr
   &  D_4& & D_4& &  c\cr
  }}$$

\begin{table}[ht]
\caption{$E_6(q)$, $7 < (q^2+q+1)_\ell$, $(q,6)=1$}   \label{tab:E6,d=3}
{\small
$$\vbox{\offinterlineskip\halign{$#$\hfil\ \vrule height10pt depth4pt&
      \hfil\ $#$\ \hfil\vrule&& \hfil\ $#$\hfil\cr
  \phi_{1,0}&         1& 1\cr
  \phi_{6,1}&         q& 1& 1\cr
 \phi_{20,2}&       q^2& 1& 1& 1\cr
 \phi_{30,3}&  \hlf q^3& .& 1& .& 1\cr
 \phi_{15,5}&  \hlf q^3& .& 1& .& .& 1\cr
 \phi_{15,4}&  \hlf q^3& .& .& 1& .& .& 1\cr
 \phi_{64,4}&       q^4& .& 1& 1& 1& 1& .& 1\cr
 \phi_{60,5}&       q^5& .& 1& 1& 1& .& 1& 1& 1\cr
 \phi_{24,6}&       q^6& .& .& 1& .& .& .& 1& .& 1\cr
\phi_{20,10}&  \sxt q^7& .& .& .& .& 1& .& 1& .& .& 1\cr
 \phi_{80,7}&  \sxt q^7& .& 1& .& 1& 1& .& 1& 1& .& .& 1\cr
 \phi_{10,9}& \thrd q^7& .& 1& .& .& .& .& .& 1& .& .& .& 1\cr
 \phi_{60,8}&  \hlf q^7& 1& 1& 1& .& 1& 1& 1& 1& 1& .& .& .& 1\cr
  E_6[\ze_3]&\thrd q^7& .& .& .& .& .& .& .& .& .& .& .& .& .& 1\cr
E_6[\ze_3^2]&\thrd q^7& .& .& .& .& .& .& .& .& .& .& .& .& .& .& 1\cr
\phi_{60,11}&   q^{11}& 1& 1& .& .& 1& .& .& 1& .& .& 1& 1& 1&a_1&a_1& 1\cr
\phi_{24,12}&   q^{12}& .& .& .& .& 1& .& .& .& .& .& .& .& 1&a_2&a_2& .& 1\cr
\phi_{64,13}&   q^{13}& .& .& .& .& 1& .& 1& 1& 1& 1& 1& .& 1&a_3&a_3& 1& 1& 1\cr
\phi_{15,16}& \hlf q^{15}& 1& .& .& .& .& .& .& .& 1& .& .& .& 1&a_4&a_4& 1& .& .& 1\cr
\phi_{15,17}& \hlf q^{15}& .& .& .& .& .& .& 1& 1& 1& 1& .& .& .&a_5&a_5& .& .& 1& .& 1\cr
\phi_{30,15}& \hlf q^{15}& .& .& .& .& .& .& .& 1& .& .& 1& 1& .&a_6&a_6& 1& .& 1& .& .& 1\cr
\phi_{20,20}&      q^{20}& .& .& .& .& .& 1& .& 1& 1& .& .& 1& 1&a_7&a_7& 1& 1& 1& 1& .& .& 1\cr
 \phi_{6,25}&      q^{25}& .& .& .& .& .& 1& .& 1& .& .& .& 1& .&a_8&a_8& .& .& 1& .& b& 1& 1& 1\cr
 \phi_{1,36}&      q^{36}& .& .& .& .& .& 1& .& .& .& .& .& .& .&a_9&a_9& .& .& .& 1& 3b\mn6& .& 1&3& 1\cr
\noalign{\hrule}
 \omit& & ps& ps& ps& ps& ps& ps& ps& ps& A_2^2& A_2& ps& A_2^2& ps& c& c& A_2& A_2& A_2& A_2^2& c& A_2^2& A_2& c& c\cr
   }}$$}
\end{table}

\begin{proof}
We explain how to find projective characters $\Psi_i$, $1\le i\le24$, with
unipotent parts as given in the columns of Table~\ref{tab:E6,d=3}.
First, (HCi) gives all columns except for those with index
$i=9,12,14,15,19,20,21,23,24$. (Alternatively, the ten PIMs in the principal
series can also be read off from the decomposition matrix of the Hecke algebra
$\cH(E_6;q)$, given in \cite[Tab.~7.13]{GJ11}.) Furthermore
(HCi) yields $\Psi_9+\Psi_{11}+\Psi_{12}$ and $\Psi_9+\Psi_{11}+\Psi_{21}$.
From these and (Tri) we get $\Psi_9+\Psi_{11}$, $\Psi_{12}$ and $\Psi_{21}$.
Harish-Chandra inducing the 16th PIM from a Levi subgroup of type $D_6$ to
$E_7(q)$ and
restricting it back to $E_6(q)$ yields $\Psi_9+\Psi_{12}+\Psi_{16}$, which
shows that $\Psi_9$ is a projective character. Finally, (HCi) from a Levi
subgroup of type $A_5$ also yields a projective character
$\Psi_{19}'=\Psi_{19}+\Psi_{21}$. Now the decomposition matrix of the Hecke
algebra for the $A_2^2$-series, determined in Lemma~\ref{lem:paramE6,d=3}
shows that $\Psi_{19}'$ must have two summands in that series, and then (HCr)
leads to the only admissible splitting $\Psi_{19}\oplus\Psi_{21}$.
\par
At this point we have accounted for projectives in all non-cuspidal
Harish-Chandra series, so the remaining five PIMs must belong to cuspidal
Brauer characters. (A priori, by (St), the modular Steinberg character is
known to be cuspidal.) Let us denote the unknown entries below the diagonal
of those PIMs as follows:
{\small
$$\vbox{\offinterlineskip\halign{$#$\hfil\ \vrule height10pt depth4pt&&
      \hfil\ $#$\hfil\cr
  E_6[\ze_3] & 1\cr
E_6[\ze_3^2]& .& 1\cr
\phi_{60,11}  &a_1&a_1 \cr
\phi_{24,12} &a_2&a_2 \cr
\phi_{64,13} &a_3&a_3\cr
\phi_{15,16} &a_4&a_4\cr
\phi_{15,17} &a_5&a_5 & 1\cr
\phi_{30,15} &a_6&a_6 & .\cr
\phi_{20,20} &a_7&a_7& b_1\cr
 \phi_{6,25}  &a_8&a_8& b_2& 1\cr
 \phi_{1,36}  &a_9&a_9& b_3&b_4\cr
   }}$$}
Note here that since the ordinary cuspidal unipotent characters $E_6[\ze_3]$
and $E_6[\ze_3^2]$ are Galois conjugate, while all other unipotent characters
are rational valued, the unknown decomposition numbers in the two corresponding
columns must agree. Now, by $(T_\ell)$, the PIM of $E_7(q)$ starting at
$\phi_{21,33}$ has coefficient 0 on $\phi_{56,30}$ and $\phi_{21,36}$.
Harish-Chandra restricting this to $E_6(q)$ gives a projective starting at
$\phi_{15,17}$, but not involving $\phi_{20,30}$. Thus again by (Tri) we
conclude that $b_1=0$.
\par
We now use the combination of (DL) and (Red) to determine some of the unknown
entries. We start with the Deligne--Lusztig character $R_w$ associated to a
Coxeter element $w$. The coefficient on $\Psi_{21}$ is $2-2a_1-2a_2+2a_3-2a_5$
and hence is non-negative by (DL). On the other hand, if $\ell > 3$ then (Red)
for the cuspidal unipotent character of the 3-split Levi subgroup
$\tw3D_4(q).(q^2+q+1)$ gives the relation $-1+a_1+a_2-a_3+a_5 \geq 0$
(see Example~\ref{exmp:reg}(d)).
We deduce that $a_5 = 1-a_1-a_2+a_3$. The coefficients of $\Psi_{23}$ and
$\Psi_{24}$ are
\begin{align*}
&& X & = -1-2a_3-2a_4+2a_6+2a_7-2a_8 \geq 0 &&\\
\text{and} && Y & = 3-2a_3+2a_7-2a_9 - b_4X \geq 0. &&
\end{align*}
Note that $X$ cannot be equal to zero by parity. Using explicit computations
in \Chevie~\cite{Chv}, one can see that there exists a regular
$\ell$-element in $G^*$ whenever $(q^2+q+1)_\ell > 7$. Then
(Red) applied to the case of a maximal torus gives $b_4 \geq 3$ and
$-3-2a_3-3a_4+3a_6+2a_7-3a_8+a_9 \geq 0$. Then
$$\begin{aligned} 0 \leq Y & = 3-2a_3+2a_7-2a_9 - b_4X \\
& \leq 3-2a_3+2a_7-2a_9 - 3X \\
& = 2(3+2a_3+3a_4-3a_6-2a_7+3a_8-a_9) \leq 0 \end{aligned}$$
forces $b_4 = 3$ and $3+2a_3+3a_4-3a_6-2a_7+3a_8-a_9 = 0$.
\par
Finally, the coefficient of $\Psi_{24}$ on the Deligne--Lusztig character
$R_w$ associated with $w=s_1s_2s_3s_1s_5s_4s_6s_5s_4s_2s_3s_4$ is equal to
$-18+9b_2-3b_3$. With (Red) applied to a maximal torus again we have
also $6-3b_2+b_3 \leq 0$ hence (DL) forces $b_3=-6+3b_2$. In the table, we write
$b$ for $b_2$.
\par
It now follows with (HCr) that all projectives in the table are indecomposable.
\end{proof}

\begin{rem} 
As before, we can use the virtual characters $Q_w$ for $w\in W$ to get
conjectural upper bounds for the unknown entries. With $w$ being the Coxeter
element, we find the following inequalities:
 $$\begin{aligned} a_1 & \leq 13,& 
 a_2 & \leq 1, & a_3 &\leq 4+a_1+a_2,\\
  a_4 & \leq 9+a_1, &
 a_6 & \leq 2-a_2+a_3, &
  a_7 &\leq 21-a_1+a_3+a_4, \\
 a_8 & \leq -1-a_3-a_4\,+&&\!\!\!\!\!\!\!\!\!a_6+a_7.\end{aligned}$$
On the other hand, the trivial character of a $2$-split Levi subgroup of type
$A_2(q)A_2(q).(q^2+q+1)$ (resp.~$A_2(q).(q^2+q+1)^2$) gives $a_2 \geq 1$
(resp.~$a_6 \geq 2-a_2+a_3$) by (Red). Therefore $a_2 = 1$ and $a_6 = 1+a_3$.
From this we deduce that we should have $a_1 \leq 13,a_3 \leq 18, a_4\leq 22,
a_5 \leq 5, a_6 \leq 14,a_7\leq 48$ and $a_8 \leq 26$.  

To obtain a bound on $b$ we use $Q_{w'}$ with 
$w' =  s_1s_2s_3s_1s_4s_3s_1s_5s_4s_3s_1s_6s_5s_4s_3s_1$ from which we get
$b \leq 8$ if  \cite[Conj.~1.2]{DM14} holds.
Most of these upper bounds are probably not sharp.
\end{rem}

%%%%%%%%%%%%%%%%%%%%%%%%%%%%
\section{Even-dimensional non-split orthogonal groups}
Next, we consider the uni\-potent blocks of twisted orthogonal groups $\tw2D_n$,
$4\le n\le7$. Again, the Brauer trees were described in \cite{FS90} (and can
readily be determined):

\begin{prop}   \label{prop:trees 2Dn d=3}
 Let $q$ be a prime power and $\ell$ a prime with $d_\ell(q)=3$.
 The Brauer trees of the unipotent $\ell$-blocks of $\tw2D_n(q)$, $4\le n\le 7$,
 with cyclic defect are as given in Table~\ref{tab:2Dn,d=3,def1}.
\end{prop}

\begin{table}[ht]
\caption{Brauer trees for $\tw2D_n(q)$ ($4\le n\le7$), $7\le\ell|(q^2+q+1)$} \label{tab:2Dn,d=3,def1}
$$\vbox{\offinterlineskip\halign{$#$
        \vrule height10pt depth 2pt width 0pt&& \hfil$#$\hfil\cr
\tw2D_4(q):& 3.& \vr& 21.& \vr& 1^3.& \vr& \bigcirc& \vr& .1^3& \vr& .21& \vr& .3\cr
\cr
\tw2D_5(q):& 4.& \vr& 2^2.& \vr& 1^4.& \vr& \bigcirc& \vr& 1.1^3& \vr& 1.21& \vr& 1.3\cr
\cr
 & .4& \vr& .2^2& \vr& .1^4& \vr& \bigcirc& \vr& 1^3.1& \vr& 21.1& \vr& 3.1\cr
\cr
\tw2D_6(q):& 5.& \vr& 2^21.& \vr& 21^3.& \vr& \bigcirc& \vr& 2.1^3& \vr& 2.21& \vr& 2.3\cr
\cr
 & .5& \vr& .2^21& \vr& .21^3& \vr& \bigcirc& \vr& 1^3.2& \vr& 21.2& \vr& 3.2\cr
\cr
 & 41.& \vr& 32.& \vr& 1^5.& \vr& \bigcirc& \vr& 1^2.1^3& \vr& 1^2.21& \vr& 1^2.3\cr
\cr
 & .41& \vr& .32& \vr& .1^5& \vr& \bigcirc& \vr& 1^3.1^2& \vr& 21.1^2& \vr& 3.1^2\cr
\cr
 & 4.1& \vr& 2^2.1& \vr& 1^4.1& \vr& \bigcirc& \vr& 1.1^4& \vr& 1.2^2& \vr& 1.4\cr
\cr
\tw2D_7(q):& 5.1& \vr& 2^21.1& \vr& 21^3.1& \vr& \bigcirc& \vr& 2.1^4& \vr& 2.2^2& \vr& 2.4\cr
\cr
 & 1.5& \vr& 1.2^21& \vr& 1.21^3& \vr& \bigcirc& \vr& 1^4.2& \vr& 2^2.2& \vr& 4.2\cr
\cr
 & 41.1& \vr& 32.1& \vr& 1^5.1& \vr& \bigcirc& \vr& 1^2.1^4& \vr& 1^2.2^2& \vr& 1^2.4\cr
\cr
 & 1.41& \vr& 1.32& \vr& 1.1^5& \vr& \bigcirc& \vr& 1^4.1^2& \vr& 2^2.1^2& \vr& 4.1^2\cr
 &  & ps& & ps& & A_2& & A_2& & ps& & ps\cr
  }}$$
\end{table}

Here and later, the unipotent characters of $\tw2D_n(q)$ in the principal
series are denoted by the corresponding character of the Weyl group, which in
this case is of type $B_{n-1}$, hence by bipartitions of $n-1$. Note that the
order of $\tw2D_n(q)$ with $n\le3$ is not divisible by primes $\ell$ with
$d_\ell(q)=3$.

Again, we first collect some information on Hecke algebras of cuspidal
characters.

\begin{lem}   \label{lem:param2Dn,d=3}
 Let $q$ be a prime power and $\ell|(q^2+q+1)$.
 \begin{enumerate}[\rm(a)]
  \item The Hecke algebra for the cuspidal $\ell$-modular Steinberg character
   $\vhi_{1^3}$ of a Levi subgroup of type $A_2$ inside $\tw2D_n(q)$, $n\ge4$,
   is $\cH(A_1;1)\otimes \cH(B_{n-4};q^2;q)$.
  \item The Hecke algebra for the cuspidal $\ell$-modular Steinberg character
   $\vhi_{1^3}^{\sqtens2}$ of a Levi subgroup of type $A_2^2$ inside
   $\tw2D_n(q)$, $n\ge7$, is $\cH(B_2;q^3;1)\otimes \cH(B_{n-7};q^2;q)$.
 \end{enumerate}
\end{lem}

\begin{proof}
Note that $\tw2D_n(q)$ has Weyl group of type $B_{n-1}$. By \cite[p.~70]{Ho80}
the relative Weyl group of $A_2$ inside $B_{n-1}$ is of type
$A_1 B_{n-4}$, and that of $A_2^2$ is of type $B_2 B_{n-7}$.
The cuspidal characters in question are the same as those in
Lemma~\ref{lem:paramDn,d=3} and thus lift to characteristic zero. The relevant
minimal Levi overgroups in~(a) are of types $\tw2D_4$ when $n=4$, $\tw2D_2A_2$
when $n=5$ and $A_2A_1$ when $n=6$, leading to the parameters $1,q^2,q$
respectively. In~(b) they are of types $A_5$, $\tw2D_4A_2$ when $n=7$,
$\tw2D_2A_2^2$ when $n=8$ and $A_2^2A_1$ when $n=9$, leading to the
parameters $q^3,1,q^2$ and~$q$. The argument is now completely analogous to
the one employed in the proof of Lemma~\ref{lem:paramDn,d=3}.
\par
The Steinberg character of a Levi subgroup of type $A_2$ is reduction stable
by Example~\ref{exmp:hecke}(a), and for a Levi subgroup $A_2^2$ we take the same
Steinberg module in both factors.
\end{proof}

The Sylow $\ell$-subgroups of $\tw2D_6(q)$ are cyclic when $d_\ell(q)=3$, so
the smallest rank not covered by Proposition~\ref{prop:trees 2Dn d=3} is $n=7$:

\begin{prop}   \label{prop:2D7,d=3}
 The decomposition matrix for the principal $\ell$-block of $\tw2D_7(q)$ for
 primes $7\le\ell| (q^2+q+1)$ is as given in Table~\ref{tab:2D7,d=3}.
\end{prop}

\begin{table}[ht]
{\small\caption{$\tw2D_7(q)$, $7\le\ell| (q^2+q+1)$}   \label{tab:2D7,d=3}
$$\vbox{\offinterlineskip\halign{$#$\hfil\ \vrule height9pt depth4pt&
      \hfil\ $#$\ \hfil\vrule&& \hfil\ $#$\hfil\cr
     6.&           1& 1\cr
    51.&           q& 1& 1\cr
     .6&    \hlf q^3& .& .& 1\cr
  41^2.&    \hlf q^3& .& 1& .& 1\cr
   3^2.&         q^3& .& 1& .& .& 1\cr
    3.3&    \hlf q^4& .& .& .& .& .& 1\cr
   321.&    \hlf q^4& 1& 1& .& 1& 1& .& 1\cr
   2^3.&    \hlf q^6& 1& .& .& .& .& .& 1& 1\cr
   21.3&    \hlf q^6& .& .& .& .& .& 1& .& .& 1\cr
    .51&    \hlf q^7& .& .& 1& .& .& .& .& .& .& 1\cr
  31^3.&    \hlf q^7& .& .& .& 1& .& .& 1& .& .& .& 1\cr
   3.21&    \hlf q^7& .& .& .& .& .& 1& .& .& .& .& .& 1\cr
  1^3.3&         q^9& .& .& .& .& .& .& .& .& 1& .& .& .& 1\cr
  21.21&         q^9& .& .& .& .& .& 1& .& .& 1& .& .& 1& .& 1\cr
   .3^2& \hlf q^{12}& .& .& .& .& .& .& .& .& .& 1& .& .& .& .& 1\cr
 1^3.21& \hlf q^{12}& .& .& .& .& .& .& .& .& 1& .& .& .& 1& 1& .& 1\cr
  3.1^3&              q^{12}& .& .& .& .& .& .& .& .& .& .& .& 1& .& .& .& .& 1\cr
  .41^2& \hlf q^{13}& .& .& .& .& .& .& .& .& .& 1& .& .& .& .& .& .& .& 1\cr
  21^4.& \hlf q^{13}& .& .& .& .& 1& .& 1& 1& .& .& 1& .& .& .& .& .& .& .& 1\cr
 21.1^3& \hlf q^{13}& .& .& .& .& .& .& .& .& .& .& .& 1& .& 1& .& .& 1& .& .& 1\cr
1^3.1^3& \hlf q^{16}& .& .& .& .& .& .& .& .& .& .& .& .& .& 1& .& 1& .& .& .& 1& 1\cr
   .321& \hlf q^{16}& .& .& 1& .& .& .& .& .& .& 1& .& .& .& .& 1& .& .& 1& .& .& .& 1\cr
   1^6.& \hlf q^{21}& .& .& .& .& 1& .& .& .& .& .& .& .& .& .& .& .& .& .& 1& .& .& .& 1\cr
  .31^3& \hlf q^{21}& .& .& .& .& .& .& .& .& .& .& .& .& .& .& .& .& .& 1& .& .& .& 1& .& 1\cr
   .2^3&      q^{21}& .& .& 1& .& .& .& .& .& .& .& .& .& .& .& .& .& .& .& .& .& .& 1& .& .& 1\cr
  .21^4&      q^{31}& .& .& .& .& .& .& .& .& .& .& .& .& .& .& 1& .& .& .& .& .& .& 1& .& 1& 1& 1\cr
   .1^6&      q^{42}& .& .& .& .& .& .& .& .& .& .& .& .& .& .& 1& .& .& .& .& .& .& .& .& .& .& 1& 1\cr
\noalign{\hrule}
 \omit& & ps& ps& ps& ps& ps& ps& ps& A_2^2\!& ps& ps& A_2\!& ps& A_2\!& ps& ps& A_2\!& A_2\!& ps& A_2\!& A_2\!& A_2^2\!& ps& A_2^2\!& A_2\!& A_2^2\!& A_2\!& A_2^2\!\cr
   }}$$}
\end{table}

\begin{proof}
First, observe that the centraliser of a Sylow $\ell$-subgroup of $\tw2D_7(q)$
is contained in a Levi subgroup of type $A_5$, so there do not exist cuspidal
unipotent $\ell$-modular Brauer characters by (Csp). The Hecke algebra for the
principal series is $\cH(B_6;q^2;q)$ (see e.g.~\cite[p.~464]{Ca}).
By \cite[4.7]{DJ92} it is Morita equivalent to a sum of products of Hecke
algebras of type $A_n$, $n\le6$, whose decomposition matrices are
known by \cite[p.~259]{Ja90} for all $\ell\ge7$. This yields the columns of
the principal series PIMs. All columns labelled $A_2$ are obtained by (HCi).
Then (HCr) shows that these are indeed indecomposable. Finally, by
Lemma~\ref{lem:param2Dn,d=3}(b) the Hecke algebra for the cuspidal
$\ell$-modular Steinberg character of a Levi subgroup of type $A_2^2$
is $\cH(B_2;q^3;1)$ and hence semisimple modulo~$\ell$. Then splitting
up the Harish-Chandra induction of the corresponding PIMs using (HCr) yields
the last five missing columns.
\end{proof}

%%%%%%%%%%%%%%%%%%%%%%%%%%%%
\section{Symplectic and odd-dimensional orthogonal groups}

Next, we consider the symplectic groups $C_n(q)$ and the odd-dimensional
orthogonal groups $B_n(q)$ with $n\le6$. Note that according to Lusztig's
classification, the unipotent characters of both series of groups are
parametrised in the same way, see e.g.~\cite[\S13.8]{Ca}.

\begin{prop}   \label{prop:trees Bn d=3}
 Let $q$ be a prime power and $\ell$ a prime with $d_\ell(q)=3$.
 The Brauer trees of the unipotent $\ell$-blocks of $B_n(q)$ and $C_n(q)$,
 $3\le n\le 6$, with cyclic defect are the same as for the unipotent
 $\ell$-blocks of $\tw2D_{n+1}(q)$ in Table~\ref{tab:2Dn,d=3,def1}, plus the
 three additional trees given in Table~\ref{tab:Bn,d=3,def1}.
\end{prop}

\begin{table}[ht]
\caption{Brauer trees for $B_n(q)$ and $C_n(q)$ ($5\le n\le6$), $7\le\ell|(q^2+q+1)$} \label{tab:Bn,d=3,def1}
$$\vbox{\offinterlineskip\halign{$#$
        \vrule height10pt depth 2pt width 0pt&& \hfil$#$\hfil\cr
B_5(q):& B_2\co3.& \vr& B_2\co21.& \vr& B_2\co1^3.& \vr& \bigcirc& \vr& B_2\co.1^3& \vr& B_2\co.21& \vr& B_2\co.3\cr
\cr
B_6(q):& B_2\co4.& \vr& B_2\co2^2.& \vr& B_2\co1^4.& \vr& \bigcirc& \vr& B_2\co1.1^3& \vr& B_2\co1.21& \vr& B_2\co1.3
\cr
\cr
 & B_2\co3.1& \vr& B_2\co21.1& \vr& B_2\co1^3.1& \vr& \bigcirc& \vr& B_2\co.1^4& \vr& B_2\co.2^2& \vr& B_2\co.4
\cr
 &  & B_2& & B_2& & B_2A_2& & B_2A_2& & B_2& & B_2\cr
  }}$$
\end{table}

Here, $B_2A_2$ denotes the Harish-Chandra series of the cuspidal unipotent
$\ell$-modular Brauer character $B_2\sqtens\vhi_{1^3}$ of a Levi subgroup of
type $B_2A_2$.

\begin{prop}   \label{prop:B6,d=3}
 The decomposition matrix for the principal $\ell$-block of $B_6(q)$ and of
 $C_6(q)$ for primes $7\le\ell| (q^2+q+1)$ is as given in
 Table~\ref{tab:B6,d=3}.
\end{prop}

\begin{table}[ht]
{\small\caption{$B_6(q)$ and $C_6(q)$, $7\le\ell| (q^2+q+1)$}   \label{tab:B6,d=3}
$$\vbox{\offinterlineskip\halign{$#$\hfil\ \vrule height9pt depth4pt&&
      \hfil\ $#$\hfil\cr
     6.& 1\cr
     .6& .& 1\cr
    51.& 1& .& 1\cr
   3^2.& .& .& 1& 1\cr
    3.3& .& .& .& .& 1\cr
    .51& .& 1& .& .& .& 1\cr
  41^2.& .& .& 1& .& .& .& 1\cr
   321.& 1& .& 1& 1& .& .& 1& 1\cr
   3.21& .& .& .& .& 1& .& .& .& 1\cr
   21.3& .& .& .& .& 1& .& .& .& .& 1\cr
  21.21& .& .& .& .& 1& .& .& .& 1& 1& 1\cr
   2^3.& 1& .& .& .& .& .& .& 1& .& .& .& 1\cr
  1^3.3& .& .& .& .& .& .& .& .& .& 1& .& .& 1\cr
   .3^2& .& .& .& .& .& 1& .& .& .& .& .& .& .& 1\cr
  31^3.& .& .& .& .& .& .& 1& 1& .& .& .& .& .& .& 1\cr
  3.1^3& .& .& .& .& .& .& .& .& 1& .& .& .& .& .& .& 1\cr
  .41^2& .& .& .& .& .& 1& .& .& .& .& .& .& .& .& .& .& 1\cr
   .321& .& 1& .& .& .& 1& .& .& .& .& .& .& .& 1& .& .& 1& 1\cr
 21.1^3& .& .& .& .& .& .& .& .& 1& .& 1& .& .& .& .& 1& .& .& 1\cr
 1^3.21& .& .& .& .& .& .& .& .& .& 1& 1& .& 1& .& .& .& .& .& .& 1\cr
1^3.1^3& .& .& .& .& .& .& .& .& .& .& 1& .& .& .& .& .& .& .& 1& 1& 1\cr
   .2^3& .& 1& .& .& .& .& .& .& .& .& .& .& .& .& .& .& .& 1& .& .& .& 1\cr
  21^4.& .& .& .& 1& .& .& .& 1& .& .& .& 1& .& .& 1& .& .& .& .& .& .& .& 1\cr
  .31^3& .& .& .& .& .& .& .& .& .& .& .& .& .& .& .& .& 1& 1& .& .& .& .& .& 1\cr
   1^6.& .& .& .& 1& .& .& .& .& .& .& .& .& .& .& .& .& .& .& .& .& .& .& 1& .& 1\cr
  .21^4& .& .& .& .& .& .& .& .& .& .& .& .& .& 1& .& .& .& 1& .& .& .& 1& .& 1& .& 1\cr
   .1^6& .& .& .& .& .& .& .& .& .& .& .& .& .& 1& .& .& .& .& .& .& .& .& .& .& .& 1& 1\cr
\noalign{\hrule}
 \omit& ps& ps& ps& ps& ps& ps& ps& ps& ps& ps& ps& A_2^2\!& A_2\!& ps& A_2\!& A_2\!& ps& ps& A_2\!& A_2\!& A_2^2\!& A_2^2\!& A_2& A_2& A_2^2\!& A_2& A_2^2\!\cr
   }}$$}
\end{table}

\begin{proof}
The arguments are exactly the same for $G=B_6(q)$ and $G=C_6(q)$.
All columns and their respective Harish-Chandra series except those labelled
$A_2^2$ are directly obtained by (HCi). Since the centraliser of a Sylow
$\ell$-subgroup of $G$ is contained inside a Levi subgroup of type $A_5$,
there are no cuspidal PIMs by (Csp). Thus the five missing PIMs belong to
Brauer characters in the series of the cuspidal Steinberg character of a Levi
subgroup of type $A_2^2$. With (HCi) we find projectives with unipotent parts
$$\begin{aligned}
  &\Psi_{12}+\Psi_{21}+\Psi_{25},\quad &\Psi_{21}+\Psi_{22}+\Psi_{27},\\
  &\Psi_{11}+\Psi_{12}+\Psi_{21}+\Psi_{22}\quad\text{and }
   &\Psi_{21}+\Psi_{25}+\Psi_{27}.
\end{aligned}$$
The only way these can decompose into sums of unipotent characters consistent
with (HCr) is the one given in the table. The indecomposability of all listed
projectives also readily follows with (HCr).
\end{proof}

%%%%%%%%%%%%%%%%%%%%%%%%%%%%
\section{Unipotent decomposition matrix of $F_4(q)$}

\begin{thm}   \label{thm:F4,d=3}
 Let $(q,6)=1$. Then the decomposition matrix for the principal $\ell$-block
 of $F_4(q)$ for primes $\ell\ge7$ with $(q^2+q+1)_\ell>7$ is approximated as
 given in Table~\ref{tab:F4,d=3}. \par
 Here the unknown parameters satisfy $x_1\geq2$, $y_1,y_2 \geq 1$,
 $y_6=3-2y_1-2y_2+y_3+y_4+2y_5$ and $y_5\leq y_1+y_2-1$.
\end{thm}

\begin{table}[ht]
\caption{$F_4(q)$, $(q^2+q+1)_\ell > 7$, $(q,6)=1$}   \label{tab:F4,d=3}
$$\vbox{\offinterlineskip\halign{$#$\hfil\ \vrule height11pt depth4pt&
      \hfil\ $#$\ \hfil\vrule&& \hfil\ $#$\hfil\cr
   \phi_{1,0}&              1& 1\cr
  \phi_{2,4}'&         \hlf q& 1& 1\cr
 \phi_{2,4}''&         \hlf q& 1& .& 1\cr
   \phi_{4,1}&         \hlf q& .& .& .& 1\cr
  \phi_{8,3}'&            q^3& .& .& .& 1& 1\cr
 \phi_{8,3}''&            q^3& .& .& .& 1& .& 1\cr
  F_4^{II}[1]&\frac{1}{24}q^4& .& .& .& .& .& .& 1\cr
 \phi_{1,12}'& \frac{1}{8}q^4& .& 1& .& .& .& .& .& 1\cr
\phi_{1,12}''& \frac{1}{8}q^4& .& .& 1& .& .& .& .& .& 1\cr
   \phi_{4,8}& \frac{1}{8}q^4& 1& 1& 1& .& .& .& .& .& .& 1\cr
  \phi_{4,7}'&      \frth q^4& .& .& .& .& 1& .& .& .& .& .& 1\cr
 \phi_{4,7}''&      \frth q^4& .& .& .& .& .& 1& .& .& .& .& .& 1\cr
  \phi_{16,5}&      \frth q^4& .& .& .& 1& 1& 1& .& .& .& .& .& .& 1\cr
   F_4[\ze_3]&      \thrd q^4& .& .& .& .& .& .& .& .& .& .& .& .& .& 1\cr
  F_4[\ze_3^2]&     \thrd q^4& .& .& .& .& .& .& .& .& .& .& .& .& .& .& 1\cr
  \phi_{8,9}'&            q^9& .& .& .& .& 1& .& .& .& .& .& 1& .& 1& y_1& y_1& 1\cr
 \phi_{8,9}''&            q^9& .& .& .& .& .& 1& .& .& .& .& .& 1& 1& y_2& y_2& .& 1\cr
 \phi_{2,16}'&    \hlf q^{13}& .& 1& .& .& .& .& .& 1& .& 1& .& .& .& y_3& y_3& .& .& 1\cr
\phi_{2,16}''&    \hlf q^{13}& .& .& 1& .& .& .& .& .& 1& 1& .& .& .& y_4& y_4& .& .& .& 1\cr
  \phi_{4,13}&    \hlf q^{13}& .& .& .& .& .& .& x_1& .& .& .& .& .& 1& y_5& y_5& 1& 1& .& .& 1\cr
  \phi_{1,24}&         q^{24}& .& .& .& .& .& .& 2x_1\mn3& .& .& 1& .& .& .& y_6& y_6& .& .& 1& 1& 2& 1\cr
\noalign{\hrule}
 \omit& & ps& ps& ps& ps& ps& ps& c& A_2& \wtA_2& ps& A_2& \wtA_2& ps& c& c& A_2& \wtA_2& A_2& \wtA_2& c& c\cr
   }}$$
 Here $\wtA_2$ denotes the Harish-Chandra series above the Steinberg PIM of
 $A_2<C_3$ in $F_4$.
\end{table}

\begin{proof}
The triangular shape of the decomposition matrix when $(q,6)=1$ was shown by
K\"ohler \cite[Tab.~A.160]{Koe06}. He also proves that the ordinary cuspidal
character $F_4^{II}[1]$ can only occur in the modular reduction of
$\phi_{4,13}$ or $\phi_{1,24}$. We will denote their multiplicities as
$x_1$ and $x_2$ respectively. When $(q^2+q+1)_\ell > 7$ they satisfy
$x_1 \geq 1$ and $3+x_2 -2x_1 \geq 0$. Similarly, he proved that the Brauer
character for $\phi_{4,13}$ can only possibly occur in the modular reduction
of the Steinberg character $\phi_{1,24}$, with a multiplicity $z$ that satisfies
$z\geq 2$ whenever $(q^2+q+1)_\ell > 7$. The multiplicities of unipotent
characters in the PIMs corresponding to the two Galois conjugate cuspidal
characters $F_4[\ze_3]$ and $F_4[\ze_3^2]$ will be denoted by $y_1,\ldots,y_6$,
so that we have
\begin{align*}
&&\Psi_{14} & = F_4[\ze_3] + y_1  \phi_{8,9}'+y_2\phi_{8,9}''+y_3\phi_{2,16}'
    +y_4\phi_{2,16}''+y_5\phi_{4,13}+y_6 \phi_{1,24} &&\\
\text{and} && \Psi_{15} & = F_4[\ze_3^2] + y_1  \phi_{8,9}'+y_2\phi_{8,9}''
    +y_3\phi_{2,16}'+y_4\phi_{2,16}''+y_5\phi_{4,13}+y_6 \phi_{1,24}. &&
\end{align*}
Under the assumption on $\ell$, we can use (Red) in the case of a maximal torus
to get the relation $-3+2y_1+2y_2-y_3-y_4-2y_5+y_6 \geq 0$ which will be useful
in the sequel.
\par
Let $w$ be a Coxeter element and $R_w$ be the corresponding Deligne--Lusztig
character. The multiplicities of $\Psi_{20}$ and $\Psi_{21}$ in $R_w$ are
respectively
\begin{align*}
&& 2X & = -2+2y_1+2y_2-2y_5 \geq 0 &&\\
  \text{and} && Y & = 2+2y_3+2y_4-2y_6-zX \geq 0. &&
\end{align*}
Since $z \geq 2$ we have
$$Y \geq 2+2y_3+2y_4-2y_6-2X = 6-4y_1-4y_2+2y_3+2y_4+4y_5-2y_6$$
which is non-positive by (Red). Therefore it must be zero, and we must also
have $X=0$ or $z=2$. In any case, the coefficient of $\Psi_{21}$ is zero.
\par
The coefficient of $\Psi_{21}$ in $R_w$ with $w=s_1s_2s_3s_4s_2s_3$ is $2-z$,
hence $z=2$ by (DL). The coefficient on $R_w$ with
$w = s_1s_2s_3s_4s_1s_2s_3s_4s_1s_2s_3s_4$ is $-3+2x_1-x_2$ and therefore
$x_2 = 2x_1-3$ by (DL) and the previous inequality on the $x_i$'s.
\par
Finally, using (Red) with a the trivial character of the $2$-split Levi
of type $A_2.(q^2+q+1)$ we get $y_1 \geq1$ and $y_2 \geq1$.
\end{proof}

\begin{rem} 
As before the virtual characters $Q_w$ for $w \in W$
 provide conjectural upper bounds for the remaining unknown entries.
We start with $w = s_1s_2s_1s_4s_3s_2s_1$, and we compute
$$ \begin{aligned}
\langle Q_w ; \varphi_{16} \rangle & = 12-12y_1, \\
\langle Q_{w} ; \varphi_{17} \rangle & = 12-12y_2,,\\
\langle Q_{w} ; \varphi_{18} \rangle & = 12-12y_3,\\
\langle Q_{w} ; \varphi_{19} \rangle & = 18-12y_4.\\
\end{aligned}$$
If \cite[Conj.~1.2]{DM14} holds then this forces $y_1,y_2,y_3,y_4 \leq 1$.
Using the relations in Theorem~\ref{thm:F4,d=3} we deduce that $y_1 = y_2 =1$, 
$y_5 \leq 1$ and $y_6 \leq 3$.  Furthermore,
with $w'= s_2s_3s_2s_1s_3s_2s_3s_4s_3s_2s_1s_3s_2s_4s_3s_2$ 
we have $\langle Q_{w'} ; \varphi_{20} \rangle = 176-8x_1$ from which
we get $x_1 \leq 22$.  
\end{rem}

We close this section by collecting in Table~\ref{tab:3-blocks} some data on
the $\Phi_3$-modular Harish-Chandra series in the principal $\ell$-blocks $B_0$
considered above. Here $W_G(B_0)$ denotes the relative Weyl group of a Sylow
$\Phi_3$-torus of the ambient group $G$ (which contains a Sylow $\ell$-subgroup
of $G$ and hence a defect group of $B_0$). This is known to be a complex
reflection group; more concretely, here it is one of two imprimitive complex
reflection groups of rank~2, denoted $G(6,1,2)$ and $G(6,2,2)$, respectively
the primitive reflection groups $G_5$, $G_{25}$ and $G_{26}$.

\begin{table}[ht]
\caption{Modular Harish-Chandra-series in $\Phi_3$-blocks}  \label{tab:3-blocks}
$$\begin{array}{r|cc|ccccccc}
 G&  W_G(B_0)& |\IBr B_0|& ps& A_2& A_2^2& E_6& c\\
\hline
    D_6&    G(6,2,2)& 18& 10&  4& 4&   \\
    B_6,C_6,D_7,\tw2D_7&    G(6,1,2)& 27& 14&  8& 5&   \\
    E_6&      G_{25}& 24& 10&  5& 4&  5\\
    E_7&      G_{26}& 48& 20& 10& 8& 10\\
%%    E_8&         G_5& 21&  8&  4&  & 4& 5\\
    F_4,\tw2E_6&         G_5& 21&  8&  4& 4&  & 5\\
\end{array}$$
\end{table}

It can be observed from the decomposition matrices
that the principal $\ell$-blocks of $D_7(q)$ and $\tw2D_7(q)$ are not
Morita equivalent, even though their modular Harish-Chandra series distribution
agrees. On the other hand, the decomposition matrices for the principal
$\ell$-blocks of $\tw2D_7(q)$, $B_6(q)$ and $C_6(q)$ agree after a suitable
simultaneous permutation of rows and columns of the one of $\tw2D_7(q)$, so
these blocks might be Morita equivalent.
%%{\bf[G: this should be decidable by Gruber-Hiss... Do we care?]}

%%%%%%%%%%%%%%%%%%%%%%%%%%%%%%%%%%%%%%%%%%%%%%%%%%%%%%%%%%%%%%%%%%%%%%%%%
\chapter{Decomposition matrices at $d_\ell(q)=6$}   \label{chap:d=6}
Now, we consider unipotent decomposition matrices for groups $G=G(q)$ at primes
$\ell$ with $d_\ell(q)=6$, so $\ell|(q^2-q+1)$ and $\ell\ge7$. For groups of
classical type, such primes are what is called \emph{unitary} for $G$, so the
theory of $q$-Schur algebras does not apply and the decomposition matrices
are not even understood theoretically.

While our assumption already implies that $\ell\geq 7$, we will often need
to make a further assumption on $(q^2-q+1)_\ell$ in order to use (Red).

%%%%%%%%%%%%%%%%%%%%%%%%%%%%
\section{Even-dimensional split orthogonal groups}
As in the previous sections we begin with groups of type $D_n$, $4\le n\le7$.
The Brauer trees for the unipotent $\ell$-blocks of cyclic defect can easily
be obtained by Harish-Chandra induction and are well-known \cite{FS90}. For
completeness and better reference we collect them here:

\begin{prop}   \label{prop:trees Dn d=6}
 Let $q$ be a prime power and $\ell$ a prime with $d_\ell(q)=6$.
 The Brauer trees of the unipotent $\ell$-blocks of $D_n(q)$, $4\le n\le 7$,
 with cyclic defect are as given in Table~\ref{tab:Dn,d=6,def1}.
\end{prop}

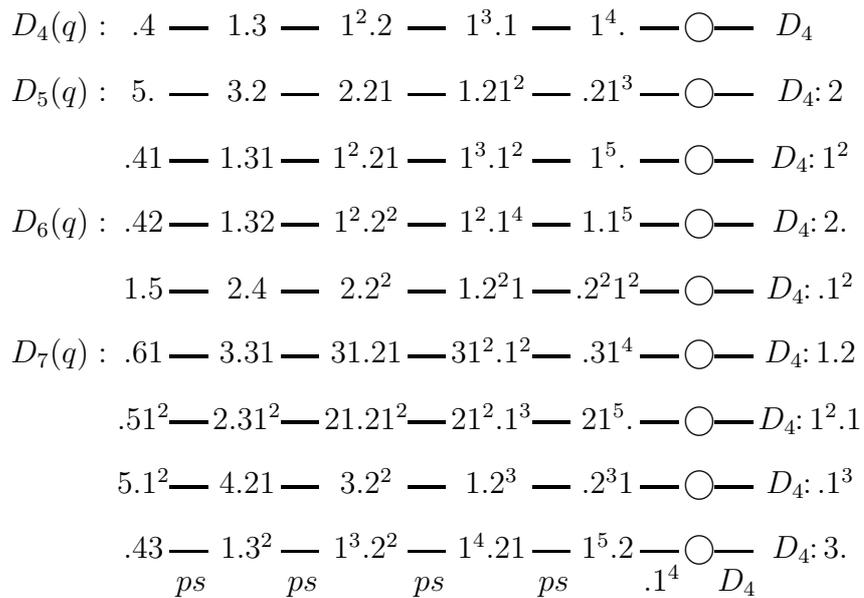
\begin{table}[ht]
\caption{Brauer trees for $D_n(q)$ ($4\le n\le7$), $7\le\ell| (q^2-q+1)$ } \label{tab:Dn,d=6,def1}
$$\vbox{\offinterlineskip\halign{$#$
        \vrule height10pt depth 2pt width 0pt&& \hfil$#$\hfil\cr
 D_4(q):& .4& \vr& 1.3 & \vr&  1^2.2& \vr& 1^3.1 & \vr&  1^4. & \vr& \bigcirc& \vr& D_4\quad\cr
\cr
 D_5(q):& 5.& \vr& 3.2& \vr& 2.21& \vr& 1.21^2& \vr& .21^3& \vr& \bigcirc& \vr& D_4\co2\cr
\cr
 & .41& \vr& 1.31& \vr& 1^2.21& \vr& 1^3.1^2& \vr& 1^5.& \vr& \bigcirc& \vr& D_4\co1^2\cr
\cr
 D_6(q):& .42& \vr& 1.32& \vr& 1^2.2^2& \vr& 1^2.1^4& \vr& 1.1^5& \vr& \bigcirc& \vr& D_4\co2.\cr
\cr
 & 1.5& \vr& 2.4& \vr& 2.2^2& \vr& 1.2^21& \vr& .2^21^2& \vr& \bigcirc& \vr& D_4\co.1^2\cr
\cr
 D_7(q):& .61& \vr& 3.31& \vr& 31.21& \vr& 31^2.1^2& \vr& .31^4& \vr& \bigcirc& \vr& D_4\co1.2\cr
\cr
 & .51^2& \vr& 2.31^2& \vr& 21.21^2& \vr& 21^2.1^3& \vr& 21^5.& \vr& \bigcirc& \vr& D_4\co1^2.1\cr
\cr
 & 5.1^2& \vr& 4.21& \vr& 3.2^2& \vr& 1.2^3& \vr& .2^31& \vr& \bigcirc& \vr& D_4\co.1^3\cr
\cr
 & .43& \vr& 1.3^2& \vr& 1^3.2^2& \vr& 1^4.21& \vr& 1^5.2& \vr& \bigcirc& \vr& D_4\co3.\cr
 & & ps& & ps& & ps& & ps& & .1^4& & D_4\cr
  }}$$
\end{table}

Here and later on, ``$.1^4$'' stands for the cuspidal unipotent $\ell$-modular
Brauer character of a Levi subgroup of type $D_4$ labelled by the unordered
bipartition $(-;1^4)$, and ``$D_4$'' for the $\ell$-modular reduction of the
cuspidal unipotent character of a Levi subgroup of type $D_4$.

To treat the blocks with non-cyclic defect, again we now first determine the
parameters of certain relative Hecke algebras:

\begin{lem}   \label{lem:paramDn,d=6}
 Let $q$ be a prime power and $\ell|(q^2-q+1)$. The Hecke algebras of various
 $\ell$-modular cuspidal pairs $(L,\la)$ of Levi subgroups $L$ in $D_n(q)$ and
 their respective numbers of irreducible characters are as given in
 Table~\ref{tab:hecke Dn,d=6}.
\end{lem}

\begin{table}[ht]
\caption{Hecke algebras and $|\Irr\cH|$ in $D_n(q)$ for $d_\ell(q)=6$}   \label{tab:hecke Dn,d=6}
$\begin{array}{l|c|ccccc}
 (L,\la)& \qquad\qquad\cH\qquad\qquad\qquad & n=6& 7& 8\cr
\hline
 (D_4,D_4)& \cH(B_{n-4};q^4;q)& 2& 4& 10\cr
 (D_4,\vhi_{.1^4})& \cH(B_{n-4};q^2;q)& 2& 4& 10\cr
 (A_5,\vhi_{1^6})& \cH(A_1;q^3)\otimes\cH(D_{n-6};q)& 1& 1& 4\cr
 (D_6,\vhi_{.1^6})& \cH(B_{n-6};q;q)& 1& 2& 5\cr
\end{array}$
\end{table}

\begin{proof}
In the first three cases, the cuspidal Brauer character lies in a block with
cyclic defect (see the Brauer tree in Table~\ref{tab:Dn,d=6,def1}), and hence
reduction stability follows from Example~\ref{exmp:hecke}(a). \par
A Levi subgroup $L$ of type $D_4$ has relative Weyl group of type $B_{n-4}$
inside $D_n$. The cuspidal Brauer character $\vhi_{.1^4}$ is a constituent of
the $\ell$-modular reduction of an ordinary cuspidal character $\la$ of $L$
lying in the Lusztig series of an $\ell$-element $s$ with centraliser
$(q^3+1)(q+1)$. By Corollary~\ref{cor:dechecke}, the Hecke algebra for
$\vhi_{.1^4}$ is the same as for $\la$. The minimal Levi overgroups are of types
$D_5$ when $n\ge5$, where $s$ has centraliser $\tw2D_2(q)(q^3+1)$, and $D_4A_1$
when $n\ge6$, which gives the parameters $q^2$ and $q$.
\par
The relative Weyl group of a Levi subgroup of type $A_5$ in $D_n$ is of type
$A_1D_{n-6}$, by \cite[p.~72]{Ho80}. As can be seen from the Brauer tree,
the $\ell$-modular Steinberg character $\vhi_{1^6}$ of $A_5$ is liftable, so
again the parameter $q^3$ for the Hecke algebra is determined inside the minimal
Levi overgroups of types $D_6$ and $A_5A_1$.
\par
Finally, the relative Weyl group of $D_6$ inside $D_7$ has type $A_1$. The
modular Steinberg character $\vhi_{.1^6}$ lifts to a cuspidal Deligne--Lusztig
character, so
it is reduction stable by Example~\ref{exmp:hecke}(c), and its Hecke algebra is
the $\ell$-modular reduction of an Iwahori--Hecke algebra in characteristic~0.
The parameters are determined already inside Levi subgroups of types $D_7$,
$D_6A_1$.
\end{proof}

\begin{prop}   \label{prop:D6,d=6}
 Assume $(T_\ell)$. The decomposition matrix for the principal $\ell$-block
 of $D_6(q)$ for primes $\ell$ with $(q^2-q+1)_\ell>7$ is as given in
 Table~\ref{tab:D6,d=6}.
\end{prop}

\begin{table}[htbp]
\caption{$D_6(q)$, $(q^2-q+1)_\ell>7$}   \label{tab:D6,d=6}
$$\vbox{\offinterlineskip\halign{$#$\hfil\ \vrule height11pt depth4pt&
      \hfil\ $#$\ \hfil\vrule&& \hfil\ $#$\hfil\cr
     .6&                               1& 1\cr
    .51&                  q^2\Ph5\Ph{10}& 1& 1\cr
     3+&              q^3\Ph5\Ph8\Ph{10}& 1& .& 1\cr
     3-&              q^3\Ph5\Ph8\Ph{10}& 1& .& .& 1\cr
   2.31&   \hlf q^4\Ph3^2\Ph5\Ph8\Ph{10}& 1& 1& 1& 1& 1\cr
D_4\co1^2.&\hlf q^4\Ph1^4\Ph3^2\Ph5\Ph{10}& .& .& .& .& .& 1\cr
  .41^2&              q^6\Ph5\Ph8\Ph{10}& .& 1& .& .& .& .& 1\cr
 1.31^2& \hlf q^7\Ph3^2\Ph4^2\Ph8\Ph{10}& .& 1& .& .& 1& .& 1& 1\cr
 D_4\co1.1&   \hlf q^7\Ph1^4\Ph3^2\Ph5\Ph8& .& .& .& .& .& 1& .& .& 1\cr
    21+&        q^7\Ph4^2\Ph5\Ph8\Ph{10}& .& .& 1& .& 1& .& .& .& .& 1\cr
    21-&        q^7\Ph4^2\Ph5\Ph8\Ph{10}& .& .& .& 1& 1& .& .& .& .& .& 1\cr
1^2.21^2&\hlf q^{10}\Ph3^2\Ph5\Ph8\Ph{10}&.& .& .& .& 1& .& .& 1& .& 1& 1& 1\cr
D_4\co.2&\hlf q^{10}\Ph1^4\Ph3^2\Ph5\Ph{10}&.& .& .& .& .& .& .& .& 1& .& .& .& 1\cr
  .31^3&           q^{12}\Ph5\Ph8\Ph{10}& .& .& .& .& .& .& 1& 1& .& .& .& .& .& 1\cr
   1^3+&           q^{15}\Ph5\Ph8\Ph{10}& .& .& .& .& .& .& .& .& .& 1& .& 1& .& .& 1\cr
   1^3-&           q^{15}\Ph5\Ph8\Ph{10}& .& .& .& .& .& .& .& .& .& .& 1& 1& .& .& .& 1\cr
  .21^4&               q^{20}\Ph5\Ph{10}& .& .& .& .& .& .& .& 1& .& .& .& 1& .& 1& .& .& 1\cr
   .1^6&                          q^{30}& .& .& .& .& .& .& .& .& .& .& .& 1& 2& .& 1& 1& 1& 1\cr
\noalign{\hrule}
  & & ps& ps& ps& ps& ps& D_4& ps& ps& D_4& ps& ps& ps& c& .1^4& A_5& A_5'& .1^4& c\cr
  }}$$
\end{table}

\begin{proof}
Let us write $\Psi_i$, $1\le i\le 18$, for the linear combinations of unipotent
characters given by the columns of Table~\ref{tab:D6,d=6}. We need to show that
these are restrictions to the principal block of $\ell$-modular PIMs of
$G=D_6(q)$. Note that the unipotent decomposition matrices of all proper
Levi subgroups are known, either by Proposition~\ref{prop:trees Dn d=6} or by
\cite{Ja90}.
\par
Projectives $\Psi_i$ with $i\ne 13,18$ are obtained by Harish-Chandra
induction of PIMs from proper Levi subgroups. This accounts, in particular,
for all PIMs in the principal series, which can be seen to be indecomposable
from the decomposition matrix of the Hecke algebra, computed with the
programme of N.~Jacon \cite{Ja05}. As a Sylow $\ell$-subgroup
of $G$ is not contained in any proper Levi subgroup, by (St) the $\ell$-modular
reduction of the Steinberg character contains a cuspidal Brauer character.
The relative Hecke algebra $\cH(B_2;q^4;q)$ for the cuspidal Brauer character
$D_4$ of a Levi subgroup of type $D_4$ (see Lemma~\ref{lem:paramDn,d=6}) has
two simple modules in characteristic~$\ell$, so the corresponding modular
Harish-Chandra series only contains two PIMs. Furthermore, the two projectives
obtained by Harish-Chandra induction of $\Psi_{.1^4}$ from $D_5(q)$ are
indecomposable by (HCr),
as are those in the $A_5$-series, so there are no further PIMs in those series.
Hence the remaining PIM must be cuspidal. From uni-triangularity of the
decomposition matrix, we deduce that it will have the form
$$[D_4\co.2]+a_1[1^3+]+a_1[1^3-]+a_2[.21^4]+a_3[.1^6].$$
Here we use that the graph automorphism of $G$ interchanges the two unipotent
characters labelled $1^3+$, $1^3-$ but fixes $[D_4\co.2]$ and hence the
corresponding two entries in that column must agree. Note also that the family
consisting of the unipotent character $[.31^3]$ is not comparable to the family
containing $[D_4:.2]$ which explains why it does not appear in the previous
projective character. The coefficient of
$\Psi_{18}$ in the Deligne--Lusztig character $R_w$ for $w$ a Coxeter
element is $2+2a_1+a_2-a_3$, and therefore by (DL) it must be non-negative.
\par
On the other hand we have $2+2a_1+a_2-a_3 \leq 0$ and therefore $a_3=a_2+2a_1+2$
provided that we can use (Red), that is when there exists a semisimple
regular $\ell$-element of $G^*$. By Example~\ref{exmp:reg}(b) such an element
exists whenever $(q^2-q+1)_\ell > 12$ (in particular whenever $(q^2-q+1)_\ell>7$
since we also assumed $\ell \geq 7$).
\par
It will be a consequence of the determination of the decomposition matrices
for $D_8(q)$ in Proposition~\ref{prop:D8,d=6} that in fact $a_1=a_2=0$ and thus
$a_3=2$, completing the proof.
\end{proof}

\begin{prop}   \label{prop:D7,d=6}
 Assume $(T_\ell)$. The decomposition matrix for the principal $\ell$-block
 of $D_7(q)$ for primes $\ell$ with $(q^2-q+1)_\ell>7$ is as given in
 Table~\ref{tab:D7,d=6}.
\end{prop}

Here $D_6^s$ denotes the Harish-Chandra series of the cuspidal $\ell$-modular
constituent of the $\ell$-modular reduction of the unipotent character
$[D_4\co.2]$ of $D_6(q)$.

\begin{table}[htbp]
\small{\caption{$D_7(q)$, $(q^2-q+1)_\ell>7$}   \label{tab:D7,d=6}
$$\vbox{\offinterlineskip\halign{$#$\hfil\ \vrule height10pt depth4pt&
      \hfil\ $#$\ \hfil\vrule&& \hfil\ $#$\hfil\cr
      .7&           1& 1\cr
     1.6&           q& .& 1\cr
    1.51&    \hlf q^3& .& 1& 1\cr
     .52&    \hlf q^3& 1& .& .& 1\cr
     3.4&         q^3& 1& 1& .& .& 1\cr
    2.41&    \hlf q^4& .& 1& 1& .& 1& 1\cr
D_4\co21.&   \hlf q^4& .& .& .& .& .& .& 1\cr
    2.32&    \hlf q^6& 1& .& .& 1& 1& .& .& 1\cr
D_4\co1^3.&  \hlf q^6& .& .& .& .& .& .& .& .& 1\cr
    .421&    \hlf q^7& .& .& .& 1& .& .& .& .& .& 1\cr
  1.41^2&    \hlf q^7& .& .& 1& .& .& 1& .& .& .& .& 1\cr
D_4\co2.1&   \hlf q^7& .& .& .& .& .& .& 1& .& .& .& .& 1\cr
  21.2^2&         q^9& .& .& .& .& 1& 1& .& 1& .& .& .& .& 1\cr
   1.321&         q^9& .& .& .& 1& .& .& .& 1& .& 1& .& .& .& 1\cr
1^2.2^21& \hlf q^{12}& .& .& .& .& .& .& .& 1& .& .& .& .& 1& 1& 1\cr
D_4\co.3& \hlf q^{12}& .& .& .& .& .& .& .& .& .& .& .& 1& .& .& .& 1\cr
   .41^3&      q^{12}& .& .& .& .& .& .& .& .& .& .& 1& .& .& .& .& .& 1\cr
  1.31^3& \hlf q^{13}& .& .& .& .& .& 1& .& .& .& .& 1& .& .& .& .& .& 1& 1\cr
  .321^2& \hlf q^{13}& .& .& .& .& .& .& .& .& .& 1& .& .& .& 1& .& .& .& .& 1\cr
\!D_4\co1.1^2\!\!&\hlf q^{13}& .& .& .& .& .& .& .& .& 1& .& .& .& .& .& .& .& .& .& .& 1\cr
1^2.21^3& \hlf q^{16}& .& .& .& .& .& 1& .& .& .& .& .& .& 1& .& .& .& .& 1& .& .& 1\cr
 D_4\co.21& \hlf q^{16}& .& .& .& .& .& .& .& .& .& .& .& .& .& .& .& .& .& .& .& 1& .& 1\cr
  1.21^4& \hlf q^{21}& .& .& .& .& .& .& .& .& .& .& .& .& .& .& .& .& 1& 1& .& .& 1& .& 1\cr
 .2^21^3& \hlf q^{21}& .& .& .& .& .& .& .& .& .& .& .& .& .& 1& 1& .& .& .& 1& .& .& .& .& 1\cr
 1^3.1^4&      q^{21}& .& .& .& .& .& .& .& .& .& .& .& .& 1& .& 1& .& .& .& .& .& 1& .& .& .& 1\cr
   1.1^6&      q^{31}& .& .& .& .& .& .& .& .& .& .& .& .& .& .& .& 2& .& .& .& .& 1& .& 1& .& 1& 1\cr
    .1^7&      q^{42}& .& .& .& .& .& .& .& .& .& .& .& .& .& .& 1& .& .& .& .& .& .& 2& .& 1& 1& .& 1\cr
\noalign{\hrule}
  & & ps& ps& ps& ps& ps& ps& D_4\!& ps& D_4\!& ps& ps& D_4\!& ps& ps& ps& \!D_6^s\!& \!.1^4\!& ps& \!.1^4\!& D_4\!& ps& D_6^s\!& \!.1^4\!& \!.1^4\!& A_5\!& \!.1^6\!& \!.1^6\!\!\cr
  }}$$}
\end{table}

\begin{proof}
As before, let's denote by $\Psi_i$, $1\le i\le 27$, the linear combinations
of unipotent characters corresponding to the columns of Table~\ref{tab:D7,d=6}.
Those $\Psi_i$ with $i$ not equal to
$$1,\ 15,\ 16,\ 19,\ 22,\ 24,\ 26\text{ and }27$$
are obtained by (HCi).
Furthermore, by uni-triangularity the projectives $\Psi_{17}+\Psi_{19}$,
$\Psi_{18}+\Psi_{19}$ yield $\Psi_{19}$, and $\Psi_{21}+\Psi_{24}$,
$\Psi_{23}+\Psi_{24}$ yield $\Psi_{24}$. Since a Levi subgroup $L$ of type
$D_6$ contains the centraliser of a Sylow $\ell$-subgroup of $D_7(q)$, there
are no cuspidal Brauer characters by (Csp) and in particular the Harish-Chandra
induction of the Steinberg PIM of $L$ splits off the Steinberg PIM of $G$;
this yields $\Psi_{26}$ and $\Psi_{27}$.
\par
The projective cover of the trivial character lies in the principal series and
is determined by the decomposition matrix of the Hecke algebra $\cH(D_7;q)$;
this yields $\Psi_1$ (HC-induction only gives $\Psi_1+\Psi_2$), and similarly
we obtain $\Psi_{15}$ (HC-induction only yields $\Psi_{15}+\Psi_{21}$).
Now the Harish-Chandra induction to $G$ of the PIM of the $\ell$-modular
cuspidal unipotent character $D_6^s$ of a Levi subgroup of type $D_6$ equals
$$\Psi=[D_4\co.3]+[D_4\co.21]+ a_2[1.21^4]+ a_2[.2^21^3]+ 2a_1[1^3.1^4]
  + a_3[1.1^6]+ a_3[.1^7].$$
As we have accounted for all other Harish-Chandra series, and there are no
cuspidal unipotent Brauer characters, this projective character has to have two
summands in that series, which by uni-triangularity must start at $[D_4\co.3]$
and $[D_4\co.21]$ respectively. The only possibility for splitting $\Psi$ into
two summands compatible with (HCr) is
$$\begin{aligned}
  \tPsi_{16}&=[D_4\co.3]+(a_2-z_2)[1.21^4]+ z_2[.2^21^3]+ a_1[1^3.1^4]
              + (a_3-z_3)[1.1^6]+ z_3[.1^7],\\
  \tPsi_{22}&=[D_4\co.21]+ z_2[1.21^4]+ (a_2-z_2)[.2^21^3]+ a_1[1^3.1^4]
              + z_3[1.1^6]+ (a_3-z_3)[.1^7],
\end{aligned}$$
with suitable non-negative integers $z_2\le a_2$ and $z_3\le a_3$. It will be
a consequence of the determination of the decomposition matrices for the
unipotent blocks of $D_8(q)$
in Proposition~\ref{prop:D8,d=6} that in fact $a_1=a_2=z_2=0$ and thus $a_3=2$.
This accounts for the last two missing columns $\Psi_{16}$ and $\Psi_{22}$.
\par
Finally, the coefficient of $\Psi_{26}$ on $R_w$ when $w$ is a Coxeter element
is $2z_4$. Since $\Psi_{26}$ does not occur in any $R_v$ for $v<w$, we deduce
from (DL) that $z_4\leq 0$ since $l(w)= 7$ is odd. Therefore $z_4=0$.
Then (HCr) shows that all of the $\Psi_i$ are in fact indecomposable.
\end{proof}

\begin{prop}   \label{prop:D8,d=6}
 Assume $(T_\ell)$. The decomposition matrices for the three unipotent
 $\ell$-blocks of $D_8(q)$ of non-cyclic defect for primes
 $\ell$ with $(q^2-q+1)_\ell>7$ are as given in Tables~\ref{tab:D8,d=6,bl1}
 and~\ref{tab:D8,d=6,bl2}.
\end{prop}

Here the blocks are labelled by the $6$-cuspidal unipotent characters of
the centraliser $D_2(q).\Phi_6^2$ of a Sylow $\Phi_6$-torus of $D_8(q)$
(according to the parametrisation of $\ell$-blocks in \cite{BMM}).

\begin{table}[htbp]
{\small\caption{$D_8(q)$, $(q^2-q+1)_\ell>7$, blocks~$\binom{2}{0}$ and~$\binom{1\ 2}{0\ 1}$}   \label{tab:D8,d=6,bl1}
$$\vbox{\offinterlineskip\halign{$#$\hfil\ \vrule height9pt depth0pt&&
      \hfil\ $#$\hfil\cr
      .8& 1\cr
     2.6& .& 1\cr
     3.5& 1& 1& 1\cr
    2.51& .& 1& 1& 1\cr
     .53& 1& .& .& .& 1\cr
D_4\co31.& .& .& .& .& .& 1\cr
  1.51^2& .& .& .& 1& .& .& 1\cr
D_4\co3.1& .& .& .& .& .& 1& .& 1\cr
  2.41^2& .& .& 1& 1& .& .& 1& .& 1\cr
    .431& .& .& .& .& 1& .& .& .& .& 1\cr
    2.33& 1& .& 1& .& 1& .& .& .& .& .& 1\cr
   1.331& .& .& .& .& 1& .& .& .& .& 1& 1& 1\cr
D_4\co1^3.1\!\!& .& .& .& .& .& .& .& .& .& .& .& .& 1\cr
   .51^3& .& .& .& .& .& .& 1& .& .& .& .& .& .& 1\cr
21^2.2^2& .& .& 1& .& .& .& .& .& 1& .& 1& .& .& .& 1\cr
D_4\co.4& .& .& .& .& .& .& .& 1& .& .& .& .& .& .& .& 1\cr
  2.31^3& .& .& .& .& .& .& 1& .& 1& .& .& .& .& 1& .& .& 1\cr
  .331^2& .& .& .& .& .& .& .& .& .& 1& .& 1& .& .& .& .& .& 1\cr
\!\!D_4\co1^2.1^2\!\!& .& .& .& .& .& .& .& .& .& .& .& .& 1& .& .& .& .& .& 1\cr
1^3.2^21& .& .& .& .& .& .& .& .& .& .& 1& 1& .& .& 1& .& .& .& .& 1\cr
 21.21^3& .& .& .& .& .& .& .& .& 1& .& .& .& .& .& 1& .& 1& .& .& .& 1\cr
1^4.21^2& .& .& .& .& .& .& .& .& .& .& .& .& .& .& 1& .& .& .& .& 1& 1& 1\cr
D_4\co.2^2& .& .& .& .& .& .& .& .& .& .& .& .& .& .& .& .& .& .& 1& .& .& .& 1\cr
  2.21^4& .& .& .& .& .& .& .& .& .& .& .& .& .& 1& .& .& 1& .& .& .& 1& .& .& 1\cr
 .2^21^4& .& .& .& .& .& .& .& .& .& .& .& 1& .& .& .& .& .& 1& .& 1& .& .& .& .& 1\cr
   1^6.2& .& .& .& .& .& .& .& .& .& .& .& .& .& .& .& 2& .& .& .& .& 1& 1& .& 1& .& 1\cr
   .21^6& .& .& .& .& .& .& .& .& .& .& .& .& .& .& .& .& .& .& .& 1& .& 1& 2& .& 1& .& 1\cr
\noalign{\hrule}
  & ps& ps& ps& ps& ps& D_4\!& ps& D_4\!& ps& ps& ps& ps& D_4\!& \!.1^4\!& ps& D_6^s\!& ps& \!.1^4\!& D_4\!& ps& ps& A_5\!& D_6^s\!& \!.1^4\!& \!.1^4\!& \!.1^6\!& \!.1^6\!\cr
\omit& \vphantom{A}\cr
\omit& \vphantom{A}\cr
     .71& 1\cr
   1^2.6& .& 1\cr
     .62& 1& .& 1\cr
    31.4& 1& 1& .& 1\cr
D_4\co2^2.& .& .& .& .& 1\cr
  1^2.51& .& 1& .& .& .& 1\cr
   21.41& .& 1& .& 1& .& 1& 1\cr
    3.32& 1& .& 1& 1& .& .& .& 1\cr
  2^2.31& .& .& .& 1& .& .& 1& 1& 1\cr
D_4\co1^4.& .& .& .& .& .& .& .& .& .& 1\cr
1^2.41^2& .& .& .& .& .& 1& 1& .& .& .& 1\cr
   .42^2& .& .& 1& .& .& .& .& .& .& .& .& 1\cr
D_4\co2.2& .& .& .& .& 1& .& .& .& .& .& .& .& 1\cr
  1.32^2& .& .& 1& .& .& .& .& 1& .& .& .& 1& .& 1\cr
D_4\co1.3& .& .& .& .& .& .& .& .& .& .& .& .& 1& .& 1\cr
1^2.31^3& .& .& .& .& .& .& 1& .& 1& .& 1& .& .& .& .& 1\cr
  .32^21& .& .& .& .& .& .& .& .& .& .& .& 1& .& 1& .& .& 1\cr
 1^2.2^3& .& .& .& .& .& .& .& 1& 1& .& .& .& .& 1& .& .& .& 1\cr
   .41^4& .& .& .& .& .& .& .& .& .& .& 1& .& .& .& .& .& .& .& 1\cr
  1.31^4& .& .& .& .& .& .& .& .& .& .& 1& .& .& .& .& 1& .& .& 1& 1\cr
\!D_4\co1.1^3\!& .& .& .& .& .& .& .& .& .& 1& .& .& .& .& .& .& .& .& .& .& 1\cr
1^2.21^4& .& .& .& .& .& .& .& .& 1& .& .& .& .& .& .& 1& .& .& .& 1& .& 1\cr
 .2^31^2& .& .& .& .& .& .& .& .& .& .& .& .& .& 1& .& .& 1& 1& .& .& .& .& 1\cr
\!D_4\co.21^2\!& .& .& .& .& .& .& .& .& .& .& .& .& .& .& .& .& .& .& .& .& 1& .& .& 1\cr
 1^3.1^5& .& .& .& .& .& .& .& .& 1& .& .& .& .& .& .& .& .& 1& .& .& .& 1& .& .& 1\cr
 1^2.1^6& .& .& .& .& .& .& .& .& .& .& .& .& .& .& 2& .& .& .& .& 1& .& 1& .& .& 1& 1\cr
    .1^8& .& .& .& .& .& .& .& .& .& .& .& .& .& .& .& .& .& 1& .& .& .& .& 1& 2& 1& .& 1\cr
\noalign{\hrule}
  & ps& ps& ps& ps& D_4\!& ps& ps& ps& ps& D_4\!& ps& ps& D_4\!& ps& D_6^s\!& ps& \!.1^4\!& ps& \!.1^4\!& \!.1^4\!& D_4\!& ps& \!.1^4\!& D_6^s\!& A_5\!& \!.1^6\!& \!.1^6\!\cr
  }}$$}
\end{table}

\begin{table}[htbp]
\caption{$D_8(q)$, $(q^2-q+1)_\ell>7$, block $\binom{1}{1}$}   \label{tab:D8,d=6,bl2}
$$\vbox{\offinterlineskip\halign{$#$\hfil\ \vrule height10pt depth4pt&&
      \hfil\ $#$\hfil\cr
       1.7& 1\cr
        4+& 1& 1\cr
        4-& 1& .& 1\cr
      1.52& 1& .& .& 1\cr
      2.42& 1& 1& 1& 1& 1\cr
 D_4\co21^2.& .& .& .& .& .& 1\cr
     1.421& .& .& .& 1& 1& .& 1\cr
      2^2+& .& 1& .& .& 1& .& .& 1\cr
      2^2-& .& .& 1& .& 1& .& .& .& 1\cr
    .421^2& .& .& .& .& .& .& 1& .& .& 1\cr
 D_4\co2.1^2& .& .& .& .& .& 1& .& .& .& .& 1\cr
   1.321^2& .& .& .& .& 1& .& 1& .& .& 1& .& 1\cr
1^2.2^21^2& .& .& .& .& 1& .& .& 1& 1& .& .& 1& 1\cr
   D_4\co.31& .& .& .& .& .& .& .& .& .& .& 1& .& .& 1\cr
  1.2^21^3& .& .& .& .& .& .& .& .& .& 1& .& 1& 1& .& 1\cr
      1^4+& .& .& .& .& .& .& .& 1& .& .& .& .& 1& .& .& 1\cr
      1^4-& .& .& .& .& .& .& .& .& 1& .& .& .& 1& .& .& .& 1\cr
     1.1^7& .& .& .& .& .& .& .& .& .& .& .& .& 1& 2& 1& 1& 1& 1\cr
\noalign{\hrule}
  & ps& ps& ps& ps& ps& D_4\!& ps& ps& ps& .1^4\!& D_4\!& ps& ps& D_6^s\!& .1^4& A_5\!& A_5\!& .1^6\!\cr
  }}$$
\end{table}

\begin{proof}
In the principal block, labelled by the trivial character $\binom{2}{0}$ of
$D_2(q)$, the columns $\Psi_i$, $i\ne 16,23$, are obtained by Harish-Chandra
inducing PIMs from Levi subgroups of type $D_7$ and $A_7$. (HCi) also yields
projectives
$$\begin{aligned}
  \tPsi_{16}=& [D_4\co.1^4]+a_1[1^4.21^2]+(a_2-z_2)[2.21^4]+z_2[.2^21^4]
              +(a_3-z_3)[1^6.2]+z_3[.21^6],\\
  \tPsi_{23}=& a_1[1^4.21^2]+[D_4\co.2^2]+z_2[2.21^4]+(a_2-z_2)[.2^21^4]
              +z_3[1^6.2]+(a_3-z_3)[.21^6],\\
\end{aligned}$$
with $a_i,z_i$ as in the proof of Proposition~\ref{prop:D7,d=6}.
By $(T_\ell)$, we must have that
$$\tPsi_{23}-a_1\Psi_{22}=[D_4\co.2^2]+z_2[2.21^4]+(a_2-z_2)[.2^21^4]
              +(z_3-a_1)[1^6.2]+(a_3-z_3-a_1)[.21^6]$$
is a projective character. Now Harish-Chandra restriction of this to a Levi
subgroup of type $D_6A_1$ yields negative multiplicities in PIMs unless
$a_1=0$.
\par
For the block labelled $\binom{1}{1}$ again all columns except the 15th and
24th are obtained by Harish-Chandra inducing suitable PIMs from proper Levi
subgroups. We also obtain
$$\begin{aligned}
  \tPsi_{15}=& [D_4\co1.3]+(a_2\!-\!z_2)[1.31^4]+(a_2-z_2)[1^2.21^4]+z_2[.2^31^2]
              +(a_3-z_3)[1^2.1^6]+z_33[.1^8],\\
  \tPsi_{24}=& z_2[1.31^4]+z_2[1^2.21^4]+(a_2-z_2)[.2^31^2]+[D_4\co.21^2]
              +z_3[1^2.1^6]+(a_3-z_3)[.1^8].
\end{aligned}$$
Triangularity shows that
$$\tPsi_{24}-z_2\Psi_{20}-(a_2-z_2)\Psi_{23}
  = [D_4\co.21^2]+(z_3-z_2)[1^2.1^6]+(a_3-z_3-a_2+z_2)[.1^8]$$
must be a projective character. Harish-Chandra restriction of this to a Levi
subgroup of type $D_6A_1$ yields negative multiplicities in PIMs unless
$z_2=a_2=0$. Using that $a_3=a_2+2a_1+2=2$ this completes the determination of
the decomposition matrices for blocks~1 and~3, as well as for the principal
blocks of $D_6(q)$ and $D_7(q)$.
\par
For the second block of $D_8(q)$, labelled $\binom{1\ 2}{0\ 1}$, all columns
are obtained directly by (HCi). Finally, (HCr) shows that all projectives
constructed in the three blocks are in fact indecomposable.
\end{proof}

%%%%%%%%%%%%%%%%%%%%%%%%%%%%
\section{Unipotent decomposition matrix of $E_6(q)$}
For $E_6(q)$ and $d_\ell(q)=6$, the triangular shape of the decomposition
matrix and thus property $(T_\ell)$ for the unipotent blocks of $G$ and primes
$\ell>3$ has been shown by Geck--Hiss \cite[Thm.~7.4]{GH97} under the
additional assumption that $q$ is a power of a good prime for $E_6$.

\begin{lem}   \label{lem:paramE6,d=6}
 Let $q$ be a prime power and $\ell|(q^2-q+1)$. The Hecke algebras of various
 $\ell$-modular cuspidal pairs $(L,\la)$ of Levi subgroups $L$ in $E_6(q)$ and
 their respective numbers of irreducible characters are as given in
 Table~\ref{tab:hecke E6,d=6}.
\end{lem}

\begin{table}[ht]
\caption{Hecke algebras in $E_6(q)$ for $d_\ell(q)=6$}   \label{tab:hecke E6,d=6}
$\begin{array}{l|c|ccc}
 (L,\la)& \qquad\qquad\cH\qquad\qquad\qquad & |\Irr\cH|\cr
\hline
  (D_4,D_4)& \cH(A_2;q^4)& 2\cr
 (D_4,\vhi_{.1^4})& \cH(A_2;q^2)& 2\cr
  (A_5,\vhi_{1^6})& \cH(A_1;q^3)& 1\cr
\end{array}$
\end{table}

\begin{proof}
Reduction stability holds in all cases as Sylow $\ell$-subgroups of $L$ are
cyclic. The proof is now as in the previous cases. For example, the
$\ell$-modular Steinberg character $\vhi_{1^6}$ of $A_5(q)$ lifts to an ordinary
cuspidal character in the Lusztig series of a regular $\ell$-element with
centraliser a maximal torus of order $\Phi_2\Phi_3\Phi_6$. In $E_6(q)$ such an
element has centraliser $\tw2A_2(q).\Phi_3\Phi_6$, whence we find the
parameter~$q^3$.
\end{proof}

The only unipotent block of positive $\ell$-defect of $E_6(q)$ is the principal
block.

\begin{thm}   \label{thm:E6,d=6}
 Assume $(T_\ell)$. The decomposition matrix for the principal $\ell$-block of
 $E_6(q)$ for primes $\ell>3$ with $(q^2-q+1)_\ell>13$ is approximated
 by Table~\ref{tab:E6,d=6}. Here the unknown entries satisfy $a_5\le1$,
 $b_4\le2$,
 $$a_7 = -1-a_1-a_2+a_6,\quad\text{and}\quad
   a_8 = -2-a_1-a_2-a_3+a_4+2a_5+a_6.$$
\end{thm}

\begin{table}[htbp]
\caption{$E_6(q)$, $\ell>3$, $(q^2-q+1)_\ell > 13$}   \label{tab:E6,d=6}
$$\vbox{\offinterlineskip\halign{$#$\hfil\ \vrule height11pt depth4pt&
      \hfil\ $#$\ \hfil\vrule&& \hfil\ $#$\hfil\cr
  \phi_{1,0}&         1& 1\cr
  \phi_{6,1}&         q& .& 1\cr
 \phi_{20,2}&       q^2& 1& 1& 1\cr
 \phi_{30,3}&  \hlf q^3& .& 1& 1& 1\cr
 \phi_{15,4}&  \hlf q^3& 1& .& .& .& 1\cr
     D_4\co3&  \hlf q^3& .& .& .& .& .& 1\cr
 \phi_{60,5}&       q^5& 1& .& 1& .& 1& .& 1\cr
 \phi_{24,6}&       q^6& .& .& 1& .& .& .& .& 1\cr
 \phi_{80,7}&  \sxt q^7& .& .& 1& 1& .& .& 1& 1& 1\cr
 \phi_{60,8}&  \hlf q^7& .& .& .& .& 1& .& 1& .& .& 1\cr
    D_4\co21&  \hlf q^7& .& .& .& .& .& 1& .& .& .& .& 1\cr
  E_6[\ze_3]& \thrd q^7& .& .& .& .& .& .& .& .& .& .& .& 1\cr
E_6[\ze_3^2]& \thrd q^7& .& .& .& .& .& .& .& .& .& .& .& .& 1\cr
\phi_{60,11}&    q^{11}& .& .& .& .& .& .& 1& .& 1& 1& .& a_1& a_1& 1\cr
\phi_{24,12}&    q^{12}& .& .& .& 1& .& .& .& .& 1& .& .& a_2& a_2& .& 1\cr
\phi_{30,15}& \hlf q^{15}& .& .& .& .& .& .& .& 1& 1& .& .& a_3& a_3& .& .& 1\cr
\phi_{15,16}& \hlf q^{15}& .& .& .& .& .& .& .& .& .& 1& .& a_4& a_4& 1& .& .& 1\cr
   D_4\co1^3& \hlf q^{15}& .& .& .& .& .& .& .& .& .& .& 1& a_5& a_5& .& .& .& .& 1\cr
\phi_{20,20}&    q^{20}& .& .& .& .& .& .& .& .& 1& .& .& a_6& a_6& 1& 1& 1& .& b_1& 1\cr
 \phi_{6,25}&    q^{25}& .& .& .& .& .& .& .& .& .& .& .& a_7& a_7& .& .& 1& .& b_1& 1& 1\cr
 \phi_{1,36}&    q^{36}& .& .& .& .& .& .& .& .& .& .& .& a_8& a_8& 1& .& .& 1& b_1\pl2& 1& b_4& 1\cr
\noalign{\hrule}
  & & ps& ps& ps& ps& ps& D_4& ps& ps& ps& ps& D_4& c& c& ps& .1^4& ps& A_5& c& .1^4& c& c\cr
  }}$$
\end{table}

\begin{proof}
Let $\Psi_i$ denote the linear combinations corresponding to the columns of
Table~\ref{tab:E6,d=6}. (HCi) yields all $\Psi_i$ except for those with
number
$$i\in\{1,4,8,12,13,14,17,18,20,21\}.$$
Further, $\Psi_4+\Psi_5$ and $\Psi_4+\Psi_7$ yield $\Psi_4$, $\Psi_4+\Psi_8$
and $\Psi_7+\Psi_8$ yield $\Psi_8$. The two principal series PIMs $\Psi_1$ and
$\Psi_{14}$ are obtained via the theorem of Dipper from the decomposition
matrix of the Hecke algebra $\cH=\cH(E_6;q)$ which has been determined by
Geck \cite[Table D]{Ge93b}. By a result of Geck--M\"uller
\cite[Thm.~3.10]{GM09} the decomposition matrix of $\cH$
does not depend on $\ell$ for all $\ell\ge7$.
\par
To determine some of the remaining entries we use the combination of (DL) and
(Red). We denote the three unknown entries below the diagonal in the 18th
column by $b_1,b_2,b_3$, and the eight unknown entries below the diagonal in
the 12th and 13th column by $a_1,\ldots,a_8$ (they agree in the two columns
since the cuspidal characters $E_6[\zeta_3]$ and $E_6[\zeta_3^2]$ are
Galois conjugate).
Let us first note that explicit computations in \Chevie~\cite{Chv} show
that there exists a regular $\ell$-element in $G^*$ whenever
$(q^2-q+1)_\ell>13$. In addition, the condition $\ell>3$ forces $\ell$ to
be good and therefore by Example~\ref{exmp:reg}(d) one can also use (Red) for
centralisers of $\Phi_d$-tori of rank $1$. Let $w$ be a Coxeter element and
$R_w$ be the corresponding Deligne--Lusztig character. We consider the
generalised $1$-eigenspace of $F$ on $R_w$. The coefficient of the PIM
$\Psi_{20}$ is $b_1-b_2$, hence $b_1\geq b_2$ by (DL). On the other hand,
the cuspidal unipotent character of the 6-split Levi subgroup
$\tw2A_2(q).\Phi_3\Phi_6$ gives $b_2-b_1 \geq 0$ by (Red), hence
$b_2 =b_1$. The coefficient of the PIM $\Psi_{21}$ is $2+b_1-b_3$.
This number is non-positive by (Red) applied to a maximal torus.
Therefore $b_3 = 2+b_1$.
\par
We now turn to the generalised $q$-eigenspace of $F$ on $R_w$. The coefficient
of the PIM $\Psi_{18}$ is $1-a_5$, therefore $a_5 \leq 1$ by (DL). The
coefficient of $\Psi_{20}$ is $-1-a_1-a_2+a_6-a_7$.
It is non-negative by (DL) but also non-positive by (Red) applied to the
cuspidal unipotent character of $\tw2A_2(q).\Phi_3\Phi_6$. Therefore
$a_7 = -1-a_1-a_2+a_6$. Finally, the coefficient of $\Psi_{21}$ is
$-2-a_1-a_2-a_3+a_4+2a_5+a_6-a_8$. Again, it must be zero by the combination
of (DL) and (Red) for a maximal torus. Consequently
$a_8 = -2-a_1-a_2-a_3+a_4+2a_5+a_6$.
\par
We finish with the Deligne--Lusztig character $R_w$ for
$w = s_1s_2s_3s_1s_5s_4s_6s_5s_4s_2s_3s_4$. The coefficient
of the PIM $\Psi_{21}$ is $48 - 24b_4$, which forces $b_4 \leq 2$.
\end{proof}

%%%%%%%%%%%%%%%%%%%%%%%%%%%%
\section{A non-principal block of $E_8(q)$}

\begin{thm}   \label{thm:E8,d=6}
 Assume $(T_\ell)$. The decomposition matrix for the unipotent $\ell$-block of
 $E_8(q)$ for primes $\ell>3$ with $(q^2-q+1)_\ell>13$, of defect
 $(\Phi_6)_\ell^2$, is approximated by Table~\ref{tab:E8,d=6,bl2}. Here the
 unknown entries $a_1,\ldots,a_8,b_1,b_4$ are as in Theorem~\ref{thm:E6,d=6},
 and starred entries in the column $E_6^b$ are understood to be the same as in
 the column for $E_6^a$.
\end{thm}

\begin{table}[htbp]
\caption{$E_8(q)$, $\ell>3$, $(q^2-q+1)_\ell>13$, block of defect~$\Phi_6^2$}   \label{tab:E8,d=6,bl2}
$$\vbox{\offinterlineskip\halign{$#$\hfil\ \vrule height11pt depth4pt&&
      \hfil\ $#$\hfil\cr
    \phi_{112,3}& 1\cr
    \phi_{160,7}& .& 1\cr
    \phi_{400,7}& 1& .& 1\cr
   \phi_{1344,8}& 1& 1& .& 1\cr
 D_4\co\phi_{8,3}'& .& .& .& .& 1\cr
  \phi_{2240,10}& 1& .& 1& 1& .& 1\cr
  \phi_{3360,13}& .& 1& .& 1& .& .& 1\cr
  \phi_{3200,16}& .& .& .& 1& .& .& .& 1\cr
  \phi_{7168,17}& .& .& .& 1& .& 1& 1& 1& 1\cr
E_6[\zeta_3]\co\phi_{2,2}& .& .& .& .& .& .& .& .& .& 1\cr
E_6[\zeta_3^2]\co\phi_{2,2}& .& .& .& .& .& .& .& .& .& .& 1\cr
  \phi_{1344,19}& .& .& 1& .& .& 1& .& .& .& .& .& 1\cr
 D_4:\phi_{16,5}& .& .& .& .& 1& .& .& .& .& .& .& .& 1\cr
  \phi_{3200,22}& .& .& .& .& .& .& 1& .& 1& 2a_2\mn a_1\mn a_3\pl a_6& *& .& .& 1\cr
  \phi_{3360,25}& .& .& .& .& .& .& .& 1& 1& 2a_3\mn a_2\pl a_4\pl a_6& *& .& .& .& 1\cr
  \phi_{2240,28}& .& .& .& .& .& 1& .& .& 1& a_1\mn a_2\pl a_3\pl a_4\pl a_6& *& 1& .& .& .& 1\cr
D_4\co\phi_{8,9}''& .& .& .& .& .& .& .& .& .& a_5& *& .& 1& .& .& .& 1\cr
  \phi_{1344,38}& .& .& .& .& .& .& .& .& 1& a_4\mn a_1\pl4a_6\pl a_7\pl a_8& *& .& .& 1& 1& 1& 3b_1\pl2& 1\cr
   \phi_{400,43}& .& .& .& .& .& .& .& .& .& a_3\mn a_2\pl2a_4\pl a_6& *& 1& .& .& .& 1& .& .& 1\cr
   \phi_{160,55}& .& .& .& .& .& .& .& .& .& a_6\pl2a_7\pl a_8& *& .& .& .& 1& .& 3b_1\pl2& 1& .& 1\cr
   \phi_{112,63}& .& .& .& .& .& .& .& .& .& a_6\pl a_7\pl2a_8& *& .& .& .& .& 1& 3b_1\pl4& 1& 1& b_4& 1\cr
\noalign{\hrule}
  & p& p& p& p& D_4& p& p& p& p& E_6^a& E_6^b& p& D_4& .1^4& p& p& E_6^c& .1^4& A_5& E_6^d& E_6^e\cr
  }}$$
\end{table}

\begin{proof}
The principal series PIMs can be found in \cite[Table~7.15]{GJ11}. Then (HCi)
yields the other listed projectives.
%{\bf[G: anything from (DL) and/or (Red) ??][O. Nothing...]}
\end{proof}

%%%%%%%%%%%%%%%%%%%%%%%%%%%%
\section{Even-dimensional non-split orthogonal groups}
We continue with the uni\-potent blocks of twisted orthogonal groups
$\tw2D_n(q)$, $n\le7$, for primes $\ell$ with $d_\ell(q)=6$. Again the Brauer
trees in the case of cyclic defect were first described by Fong and
Srinivasan \cite{FS90} and are easily obtained:

\begin{prop}   \label{prop:trees 2Dn d=6}
 Let $q$ be a prime power and $\ell$ a prime with $d_\ell(q)=6$.
 The Brauer trees of the unipotent $\ell$-blocks of $\tw2D_n(q)$, $3\le n\le 7$,
 with cyclic defect are as given in Table~\ref{tab:2Dn,d=6,def1}.
\end{prop}

\begin{table}[ht]
\caption{Brauer trees for $\tw2D_n(q)$ ($3\le n\le7$), $7\le\ell| (q^2-q+1)$} \label{tab:2Dn,d=6,def1}
$$\vbox{\offinterlineskip\halign{$#$
        \vrule height10pt depth 2pt width 0pt&& \hfil$#$\hfil\cr
\tw2D_3(q):& 2.& \vr& 1.1& \vr& .1^2& \vr& \bigcirc\cr
\cr
\tw2D_5(q):& 31.& \vr& 1^2.2& \vr& .21^2& \vr& \bigcirc\cr
\cr
\tw2D_7(q):& 42.& \vr& 1^2.31& \vr& 1.31^2& \vr& \bigcirc\cr
\cr
 & 31^2.1& \vr& 21^2.2& \vr& .2^21^2& \vr& \bigcirc\cr
 & & ps& & ps& & .1^2 \cr
\cr
\tw2D_4(q):& 3.& \vr& 1.2& \vr& .21& \vr& \bigcirc& \vr& .1^3& \vr& 1^2.1& \vr& 21.\cr
\cr
\tw2D_5(q):& 4.& \vr& 1.3& \vr& .31& \vr& \bigcirc& \vr& 1.1^3& \vr& 1^2.1^2& \vr& 2^2.\cr
\cr
 & 3.1& \vr& 2.2& \vr& .2^2& \vr& \bigcirc& \vr& .1^4& \vr& 1^3.1& \vr& 21^2.\cr
\cr
\tw2D_6(q):& 5.& \vr& 1.4& \vr& .41& \vr& \bigcirc& \vr& 2.1^3& \vr& 21.1^2& \vr& 2^2.1\cr
\cr
 & 41.& \vr& 1^2.3& \vr& .31^2& \vr& \bigcirc& \vr& 1.21^2& \vr& 1^2.21& \vr& 32.\cr
\cr
 & 4.1& \vr& 2.3& \vr& .32& \vr& \bigcirc& \vr& 1.1^4& \vr& 1^3.1^2& \vr& 2^21.\cr
\cr
 & 31^2.& \vr& 1^3.2& \vr& .21^3& \vr& \bigcirc& \vr& .2^21& \vr& 21.2& \vr& 31.1\cr
\cr
 & 3.1^2& \vr& 2.21& \vr& 1.2^2& \vr& \bigcirc& \vr& .1^5& \vr& 1^4.1& \vr& 21^3.\cr
\cr
\tw2D_7(q):& 5.1& \vr& 2.4& \vr& .42& \vr& \bigcirc& \vr& 2.1^4& \vr& 21^2.1^2& \vr& 2^21.1\cr
\cr
& 41.1& \vr& 21.3& \vr& .321& \vr& \bigcirc& \vr& 1.21^3& \vr& 1^3.21& \vr& 321.\cr
\cr
 & 4.2& \vr& 3.3& \vr& .3^2& \vr& \bigcirc& \vr& 1^2.1^4& \vr& 1^3.1^3& \vr& 2^3.\cr
\cr
 & 4.1^2& \vr& 2.31& \vr& 1.32& \vr& \bigcirc& \vr& 1.1^5& \vr& 1^4.1^2& \vr& 2^21^2.\cr
 & & ps& & ps& & .1^2& & .1^2& & ps& & ps& \cr
  }}$$
\end{table}

Here, $.1^2$ denotes the Harish-Chandra series of the cuspidal $\ell$-modular
Steinberg character of $\tw2D_3(q)$.

In Table~\ref{tab:hecke irr,d=6} we have collected the number $|\Irr\cH|$
for some small rank modular Iwahori--Hecke $\cH$ algebras occurring later.
Again, they can be computed using the programme of Jacon \cite{Ja05}.

\begin{table}[ht]
\caption{$|\Irr\cH|$ for some modular Hecke algebras}   \label{tab:hecke irr,d=6}
$\begin{array}{l|cccccc}
 \qquad\qquad n=& 1& 2& 3& 4& 5\cr
\hline
   \cH(B_n;1;q)& 2& 5& 10& 18& 32\cr  % MatrixDecomposition([3,6],6,2,n);
   \cH(B_n;q;q)& 2& 5&  9& 18& 30\cr  % MatrixDecomposition([1,3],6,2,n);
 \cH(B_n;q^2;q)& 2& 4&  8& 15& 26\cr  % MatrixDecomposition([2,3],6,2,n);
 \cH(B_n;q^3;q)& 1& 3&  5& 10& 16\cr  % MatrixDecomposition([3,3],6,2,n);
 \cH(B_n;q^4;q)& 2& 4&  8& 15& 26\cr  % MatrixDecomposition([3,4],6,2,n);
%%     \cH(D_n;q)& 1& 4\cr
\end{array}$
\end{table}

\begin{lem}   \label{lem:param2Dn,d=6}
 Let $q$ be a prime power and $\ell|(q^2-q+1)$. The Hecke algebras of various
 $\ell$-modular cuspidal pairs $(L,\la)$ of Levi subgroups $L$ in $\tw2D_n(q)$,
 $n\ge4$, and their respective numbers of irreducible characters are as given
 in Table~\ref{tab:hecke 2Dn,d=6}.
\end{lem}

\begin{table}[ht]
\caption{Hecke algebras and $|\Irr\cH|$ in $\tw2D_n(q)$ for $d_\ell(q)=6$}   \label{tab:hecke 2Dn,d=6}
$\begin{array}{l|c|ccccc}
 (L,\la)& \qquad\qquad\cH\qquad\qquad\qquad & 6& 7& 8& 9\cr
\hline
 (\tw2D_3,\vhi_{.1^2})& \cH(B_{n-3};1;q)& 10& 18& 32& 54\cr
 (A_5,\vhi_{1^6})& \cH(A_1;q^3)\otimes\cH(B_{n-7};q^2;q)& -& 1& 2& 4\cr
 (\tw2D_7,\vhi_{1^6.})& \cH(B_{n-7};q^4;q)& -& 1& 2& 4\cr
% (\tw2D_7,\vhi_{.2^3})& \cH(B_{n-7};q^2;q)& -& 1& 2& 4\cr
 (\tw2D_7,\vhi_{.1^6})& \cH(B_{n-7};q^4;q)& -& 1& 2& 4\cr
\end{array}$
\end{table}

\begin{proof}
Note that $\tw2D_n(q)$ has Weyl group of type $B_{n-1}$. Now the relative Weyl
group of a Levi subgroup $B_2(q)$ inside $B_{n-1}(q)$ has type $B_{n-3}$, and
the one of a Levi subgroup $A_5(q)$ has type $B_{n-6}$ by \cite[p.~70]{Ho80}.
The modular Steinberg character $\vhi_{.1^2}$ of $\tw2D_3(q)$ is liftable by
the $\ell$-modular Brauer tree in Table~\ref{tab:2Dn,d=6,def1}, and the one of
$A_5(q)$ by the Brauer tree for $\GL_6(q)$, so we can argue as usual, with
reduction stability following by Example~\ref{exmp:hecke}(a). \par
Next, the $\ell$-modular Steinberg character $\vhi_{.1^6}$ of $\tw2D_7(q)$
occurs with multiplicity~1 in the reduction of an ordinary Deligne--Lusztig
character $\RTG(\theta)$ from a maximal torus $T$ centralising a Sylow
$\Phi_6$-torus as can be seen from the decomposition matrix in
Table~\ref{tab:2D7,d=6}. Thus we obtain reduction stability with
Example~\ref{exmp:hecke}(c).

Finally, again using the decomposition matrix one checks that $\vhi_{1^6.}$
lifts to a non-unipotent character whose Jordan correspondent is the cuspidal
unipotent character of $D_4(q)\Phi_2\Phi_6$. The latter is invariant under all
automorphisms, and thus $\vhi_{1^6.}$ is also reduction stable.
\end{proof}

Observe that Sylow $\ell$-subgroups of $\tw2D_n(q)$, $n\le6$, are cyclic for
primes $\ell$ with $d_\ell(q)=6$, so we only need to deal with the case $n\ge7$.

\begin{prop}   \label{prop:2D7,d=6}
 Assume $(T_\ell)$. The decomposition matrix for the principal $\ell$-block
 of $\tw2D_7(q)$ for primes $\ell$ with $(q^2-q+1)_\ell>7$ is as given in
 Table~\ref{tab:2D7,d=6}. \par
 Here, the unknown entries satisfy $y_2 \leq 5$ and $y_3\le2$.
\end{prop}

\begin{table}[ht]
{\small\caption{$\tw2D_7(q)$, $(q^2-q+1)_\ell>7$}   \label{tab:2D7,d=6}
$$\vbox{\offinterlineskip\halign{$#$\hfil\ \vrule height10pt depth4pt&
      \hfil\ $#$\ \hfil\vrule&& \hfil\ $#$\hfil\cr
     6.&           1& 1\cr
    51.&           q& 1& 1\cr
     .6&    \hlf q^3& .& .& 1\cr
  41^2.&    \hlf q^3& .& 1& .& 1\cr
   3^2.&         q^3& .& .& .& .& 1\cr
    1.5&    \hlf q^4& 1& .& 1& .& .& 1\cr
   32.1&    \hlf q^4& .& .& .& .& 1& .& 1\cr
  1^2.4&    \hlf q^6& 1& 1& .& .& .& 1& .& 1\cr
  2^2.2&    \hlf q^6& .& .& .& .& .& .& 1& .& 1\cr
    .51&    \hlf q^7& .& .& 1& .& .& 1& .& .& .& 1\cr
 31.1^2&    \hlf q^7& .& .& .& .& .& .& 1& .& .& .& 1\cr
  31^3.&    \hlf q^7& .& .& .& 1& .& .& .& .& .& .& .& 1\cr
  1^3.3&         q^9& .& 1& .& 1& .& .& .& 1& .& .& .& .& 1\cr
  21.21\!&         q^9& .& .& .& .& 1& .& 1& .& 1& .& 1& .& .& 1\cr
1^2.2^2& \hlf q^{12}& .& .& .& .& 1& .& .& .& .& .& .& .& .& 1& 1\cr
  1^4.2& \hlf q^{12}& .& .& .& 1& .& .& .& .& .& .& .& 1& 1& .& .& 1\cr
  3.1^3&      q^{12}& .& .& .& .& .& .& .& .& .& .& 1& .& .& .& .& .& 1\cr
  .41^2& \hlf q^{13}& .& .& .& .& .& 1& .& 1& .& 1& .& .& .& .& .& .& .& 1\cr
 2.21^2& \hlf q^{13}& .& .& .& .& .& .& .& .& .& .& 1& .& .& 1& .& .& 1& .& 1\cr
  21^4.& \hlf q^{13}& .& .& .& .& .& .& .& .& .& .& .& 1& .& .& .& .& .& .& .& 1\cr
 1.2^21& \hlf q^{16}& .& .& .& .& .& .& .& .& 1& .& .& .& .& 1& 1& .& .& .& 1& .& 1\cr
  1^5.1& \hlf q^{16}& .& .& .& .& .& .& .& .& .& .& .& 1& .& .& .& 1& .& .& .& 1& .& 1\cr
   1^6.& \hlf q^{21}& .& .& .& .& .& .& .& .& .& .& .& .& .& .& .& .& .& .& .& 1& .& 1& 1\cr
  .31^3& \hlf q^{21}& .& .& .& .& .& .& .& 1& .& .& .& .& 1& .& .& .& .& 1& .& .& .& .& .& 1\cr
   .2^3&      q^{21}& .& .& .& .& .& .& .& .& 1& .& .& .& .& .& .& .& .& .& .& .& 1& .& .& .& 1\cr
  .21^4&      q^{31}& .& .& .& .& .& .& .& .& .& .& .& .& 1& .& .& 1& .& .& .& .& .& .& .& 1& y_3& 1\cr
   .1^6&      q^{42}& .& .& .& .& .& .& .& .& .& .& .& .& .& .& .& 1& .& .& .& .& .& 1& y_2& .& y_3\pl2& 1& 1\cr
\noalign{\hrule}
  & & ps& ps& ps& ps& ps& ps& ps& ps& ps& .1^2\!& ps& ps& ps& ps& .1^2\!& ps& .1^2\!& .1^2\!& .1^2\!& ps& .1^2\!& A_5& c& .1^2\!& c& .1^2\!& c\cr
  }}$$}
\end{table}

\begin{proof}
All columns in Table~\ref{tab:2D7,d=6} except for the 23rd, the 25th and the
27th are obtained by Harish-Chandra induction of PIMs from Levi subgroups of
types $\tw2D_6$ and $A_5$. With Lemma~\ref{lem:param2Dn,d=6} this accounts for
all Harish-Chandra series from proper Levi subgroups, so the remaining three
PIMs must be cuspidal. The ordinary Gelfand-Graev character gives the Steinberg
PIM in column~27. The uni-triangular shape of the decomposition matrix shows
that there must exist PIMs with unipotent part of the form
$$\Psi_{23}=[1^6.]+y_1[.21^4]+y_2[.1^6],\quad
  \Psi_{25}=[.2^3]+y_3[.21^4]+y_4[.1^6],$$
with suitable coefficients $y_i\ge0$. Not that $[1^6.]$ and $[.2^3]$ lie in families
which are not comparable, which explains why $[.2^3]$ does not occur
in $\Psi_{23}$. It will turn
out in the determination of the decomposition matrices for $\tw2D_8(q)$ 
 in Proposition~\ref{prop:2D82D9,d=6} that $y_1=0$.
It remains to compute $y_2$ and $y_4$. To this end we consider the
decomposition of Deligne--Lusztig characters $R_w$. For $w$ being a
Coxeter element the coefficient of $\Psi_{27}$ in $R_w$ is equal to
$2+y_3-y_4$ which by (DL) forces $y_4 \leq 2+y_3$. On the other hand,
by Example \ref{exmp:reg}(c) there exists a regular $\ell$-element of $G^*$
whenever $(q^2-q+1)_\ell > 12$ (in particular whenever $(q^2-q+1)_\ell > 7$
since we also assumed $\ell \geq 7$). Therefore one can use (Red) to
show that $y_4 \geq 2+y_3$, which proves that $y_4=2+y_3$. Now
let $w= s_2s_1s_3s_4s_5s_4s_3s_1s_2s_3s_4s_5s_6s_7$.
By explicit computations and using that $y_4 = 2+y_3$, the PIM $\Psi_{27}$
does not occur in $R_v$ for $v <w$. In addition, the coefficient of $\Psi_{27}$
in $R_w$ is $60-12y_2$, which by (DL) gives $y_2 \leq 5$. (HCr) and the
decomposition matrix of the Hecke algebra for the principal series show that
all columns correspond to PIMs. Moreover this proves that $y_3\le2$.
\end{proof}

It seems that in order to complete the proof of Proposition~\ref{prop:2D7,d=6}
we need to also study some decomposition matrices for $\tw2D_8(q)$ and
$\tw2D_9(q)$:

\begin{prop}   \label{prop:2D82D9,d=6}
 Assume $(T_\ell)$. The decomposition matrices for the unipotent
 $\ell$-blocks of $\tw2D_8(q)$ and $\tw2D_9(q)$ of defect $(q^2-q+1)_\ell^2$
 for primes $\ell$ with $(q^2-q+1)_\ell>7$ are as given in
 Table~\ref{tab:2D8,d=6,bl1}. \par
 Here, the unknown entries $y_2\leq 5$ and $y_3\le2$ are as in
 Proposition~\ref{prop:2D7,d=6}.
\end{prop}

\begin{table}[ht]
{\small\caption{$\tw2D_8(q)$, $(q^2-q+1)_\ell>7$, blocks $\binom{0\ 2}{}$
   and $\binom{0\ 1\ 2}{1}$}   \label{tab:2D8,d=6,bl1}
$$\vbox{\offinterlineskip\halign{$#$\hfil\ \vrule height9pt depth0pt&&
      \hfil\ $#$\hfil\cr
7.       & 1\cr
52.      & 1& 1\cr
43.      & .& .& 1\cr
421.     & .& 1& .& 1\cr
42.1     & .& .& 1& .& 1\cr
1.6      & 1& .& .& .& .& 1\cr
41.1^2   & .& .& .& .& 1& .& 1\cr
.61      & .& .& .& .& .& 1& .& 1\cr
321^2.   & .& .& .& 1& .& .& .& .& 1\cr
1.51     & 1& .& .& .& .& 1& .& 1& .& 1\cr
2^2.3    & .& .& .& .& 1& .& .& .& .& .& 1\cr
21.31    & .& .& 1& .& 1& .& 1& .& .& .& 1& 1\cr
1^2.41   & 1& 1& .& .& .& .& .& .& .& 1& .& .& 1\cr
4.1^3    & .& .& .& .& .& .& 1& .& .& .& .& .& .& 1\cr
1^3.31   & .& 1& .& 1& .& .& .& .& .& .& .& .& 1& .& 1\cr
1^2.32   & .& .& 1& .& .& .& .& .& .& .& .& 1& .& .& .& 1\cr
2.31^2   & .& .& .& .& .& .& 1& .& .& .& .& 1& .& 1& .& .& 1\cr
1.41^2   & .& .& .& .& .& .& .& 1& .& 1& .& .& 1& .& .& .& .& 1\cr
2^21^3.  & .& .& .& .& .& .& .& .& 1& .& .& .& .& .& .& .& .& .& 1\cr
1.321    & .& .& .& .& .& .& .& .& .& .& 1& 1& .& .& .& 1& 1& .& .& 1\cr
1^4.21   & .& .& .& 1& .& .& .& .& 1& .& .& .& .& .& 1& .& .& .& .& .& 1\cr
1.31^3   & .& .& .& .& .& .& .& .& .& .& .& .& 1& .& 1& .& .& 1& .& .& .& 1\cr
.32^2    & .& .& .& .& .& .& .& .& .& .& 1& .& .& .& .& .& .& .& .& 1& .& .& 1\cr
1^5.1^2 & .& .& .& .& .& .& .& .& 1& .& .& .& .& .& .& .& .& .& 1& .& 1& .& .& 1\cr
1^7. & .& .& .& .& .& .& .& .& .& .& .& .& .& .& .& .& .& .& 1& .& .& .& .& 1& 1\cr
1.21^4  & .& .& .& .& .& .& .& .& .& .& .& .& .& .& 1& .& .& .& .& .& 1& 1& y_3& .& .& 1\cr
1.1^6 & .& .& .& .& .& .& .& .& .& .& .& .& .& .& .& .& .& .& .& .& 1& .& y_3\pl2& 1& y_2& 1& 1\cr
\noalign{\hrule}
 & ps& ps& ps& ps& ps& ps& ps& .1^2\!& ps& ps& ps& ps& ps& .1^2\!& ps& .1^2\!& .1^2\!& .1^2\!& ps& .1^2\!& ps& .1^2\!& .2^3\!& A_5& 1^6\!.& .1^2\!& .1^6\cr
\omit& \vphantom{A}\cr
\omit& \vphantom{A}\cr
6.1     & 1\cr
.7      & .& 1\cr
51.1    & 1& .& 1\cr
3^21.   & .& .& .& 1\cr
2.5     & 1& 1& .& .& 1\cr
41^2.1  & .& .& 1& .& .& 1\cr
21.4    & 1& .& 1& .& 1& .& 1\cr
321.1   & .& .& .& 1& .& .& .& 1\cr
2^21.2  & .& .& .& .& .& .& .& 1& 1\cr
21^2.3  & .& .& 1& .& .& 1& 1& .& .& 1\cr
31^2.1^2& .& .& .& .& .& .& .& 1& .& .& 1\cr
.52     & .& 1& .& .& 1& .& .& .& .& .& .& 1\cr
31^3.1  & .& .& .& .& .& 1& .& .& .& .& .& .& 1\cr
21^3.2  & .& .& .& .& .& 1& .& .& .& 1& .& .& 1& 1\cr
21^2.21 & .& .& .& 1& .& .& .& 1& 1& .& 1& .& .& .& 1\cr
21^4.1  & .& .& .& .& .& .& .& .& .& .& .& .& 1& 1& .& 1\cr
.421    & .& .& .& .& 1& .& 1& .& .& .& .& 1& .& .& .& .& 1\cr
1^3.2^2 & .& .& .& 1& .& .& .& .& .& .& .& .& .& .& 1& .& .& 1\cr
3.1^4   & .& .& .& .& .& .& .& .& .& .& 1& .& .& .& .& .& .& .& 1\cr
21^5.   & .& .& .& .& .& .& .& .& .& .& .& .& .& .& .& 1& .& .& .& 1\cr
2.21^3  & .& .& .& .& .& .& .& .& .& .& 1& .& .& .& 1& .& .& .& 1& .& 1\cr
1^6.1   & .& .& .& .& .& .& .& .& .& .& .& .& .& 1& .& 1& .& .& .& 1& .& 1\cr
1.2^21^2& .& .& .& .& .& .& .& .& 1& .& .& .& .& .& 1& .& .& 1& .& .& 1& .& 1\cr
.321^2  & .& .& .& .& .& .& 1& .& .& 1& .& .& .& .& .& .& 1& .& .& .& .& .& .& 1\cr
.2^31   & .& .& .& .& .& .& .& .& 1& .& .& .& .& .& .& .& .& .& .& .& .& .& 1& .& 1\cr
.2^21^3 & .& .& .& .& .& .& .& .& .& 1& .& .& .& 1& .& .& .& .& .& .& .& .& .& 1& y_3& 1\cr
.1^7    & .& .& .& .& .& .& .& .& .& .& .& .& .& 1& .& .& .& .& .& y_2& .& 1& .& .& y_3\pl2& 1& 1\cr
\noalign{\hrule}
 & ps& ps& ps& ps& ps& ps& ps& ps& ps& ps& ps& .1^2\!& ps& ps& ps& ps& \!.1^2\!& \!.1^2\!& .1^2\!& 1^6\!.& .1^2\!& A_5& .1^2\!& .1^2\!& .2^3\!& .1^2\!& .1^6\cr
   }}$$}
\smallskip
\end{table}

\begin{table}[ht]
{\small\caption{$\tw2D_9(q)$, $(q^2-q+1)_\ell>7$, blocks $\binom{1\ 2}{}$
   and $\binom{0\ 1\ 2}{2}$}   \label{tab:2D9,d=6,bl1}
$$\vbox{\offinterlineskip\halign{$#$\hfil\ \vrule height9pt depth0pt&&
      \hfil\ $#$\hfil\cr
71.       & 1\cr
62.       & 1& 1\cr
4^2.      & .& .& 1\cr
42^2.     & .& 1& .& 1\cr
42.2      & .& .& 1& .& 1\cr
1^2.6     & 1& .& .& .& .& 1\cr
32.3      & .& .& .& .& 1& .& 1\cr
41.21     & .& .& .& .& 1& .& .& 1\cr
32^21.    & .& .& .& 1& .& .& .& .& 1\cr
31.31     & .& .& 1& .& 1& .& 1& 1& .& 1\cr
1^2.51    & 1& 1& .& .& .& 1& .& .& .& .& 1\cr
4.21^2    & .& .& .& .& .& .& .& 1& .& .& .& 1\cr
.61^2     & .& .& .& .& .& 1& .& .& .& .& .& .& 1\cr
3.31^2    & .& .& .& .& .& .& .& 1& .& 1& .& 1& .& 1\cr
1.51^2    & .& .& .& .& .& 1& .& .& .& .& 1& .& 1& .& 1\cr
2^31^2.   & .& .& .& .& .& .& .& .& 1& .& .& .& .& .& .& 1\cr
1^2.41^2  & .& 1& .& .& .& .& .& .& .& .& 1& .& .& .& 1& .& 1\cr
1^2.3^2   & .& .& 1& .& .& .& .& .& .& 1& .& .& .& .& .& .& .& 1\cr
1^3.31^2  & .& 1& .& 1& .& .& .& .& .& .& .& .& .& .& .& .& 1& .& 1\cr
1.3^21    & .& .& .& .& .& .& 1& .& .& 1& .& .& .& 1& .& .& .& 1& .& 1\cr
1^4.21^2  & .& .& .& 1& .& .& .& .& 1& .& .& .& .& .& .& .& .& .& 1& .& 1\cr
.3^22     & .& .& .& .& .& .& 1& .& .& .& .& .& .& .& .& .& .& .& .& 1& .& 1\cr
1^2.31^3  & .& .& .& .& .& .& .& .& .& .& .& .& .& .& 1& .& 1& .& 1& .& .& .& 1\cr
1^5.1^3   & .& .& .& .& .& .& .& .& 1& .& .& .& .& .& .& 1& .& .& .& .& 1& .& .& 1\cr
1^2.21^4  & .& .& .& .& .& .& .& .& .& .& .& .& .& .& .& .& .& .& 1& .& 1& y_3& 1& .& 1\cr
1^8.      & .& .& .& .& .& .& .& .& .& .& .& .& .& .& .& 1& .& .& .& .& .& .& .& 1& .& 1\cr
1^2.1^6   & .& .& .& .& .& .& .& .& .& .& .& .& .& .& .& .& .& .& .& .& 1& y_3\pl2& .& 1& 1& y_2& 1\cr
\noalign{\hrule}
 & ps& ps& ps& ps& ps& ps& ps& ps& ps& ps& ps& \!.1^2\!& \!.1^2\!& \!.1^2\!& \!.1^2\!& ps& ps& .1^2\!& ps& .1^2\!& ps& .2^3\!& .1^2\!& A_5& .1^2\!& 1^6\!.& .1^6\cr
\omit& \vphantom{A}\cr
\omit& \vphantom{A}\cr
.8        & 1\cr
6.2       & .& 1\cr
51.2      & .& 1& 1\cr
3.5       & 1& 1& .& 1\cr
3^22.     & .& .& .& .& 1\cr
31.4      & .& 1& 1& 1& .& 1\cr
41^2.2    & .& .& 1& .& .& .& 1\cr
32^2.1    & .& .& .& .& 1& .& .& 1\cr
31^2.3    & .& .& 1& .& .& 1& 1& .& 1\cr
.53       & 1& .& .& 1& .& .& .& .& .& 1\cr
31^3.2    & .& .& .& .& .& .& 1& .& 1& .& 1\cr
2^3.2     & .& .& .& .& .& .& .& 1& .& .& .& 1\cr
31^4.1    & .& .& .& .& .& .& .& .& .& .& 1& .& 1\cr
31^2.1^3  & .& .& .& .& .& .& .& 1& .& .& .& .& .& 1\cr
.431      & .& .& .& 1& .& 1& .& .& .& 1& .& .& .& .& 1\cr
21^2.21^2 & .& .& .& .& 1& .& .& 1& .& .& .& 1& .& 1& .& 1\cr
21^4.2    & .& .& .& .& .& .& .& .& 1& .& 1& .& 1& .& .& .& 1\cr
31^5.     & .& .& .& .& .& .& .& .& .& .& .& .& 1& .& .& .& .& 1\cr
31.1^4    & .& .& .& .& .& .& .& .& .& .& .& .& .& 1& .& .& .& .& 1\cr
1^3.2^21  & .& .& .& .& 1& .& .& .& .& .& .& .& .& .& .& 1& .& .& .& 1\cr
21.21^3   & .& .& .& .& .& .& .& .& .& .& .& .& .& 1& .& 1& .& .& 1& .& 1\cr
.3^21^2   & .& .& .& .& .& 1& .& .& 1& .& .& .& .& .& 1& .& .& .& .& .& .& 1\cr
1^2.2^21^2& .& .& .& .& .& .& .& .& .& .& .& 1& .& .& .& 1& .& .& .& 1& 1& .& 1\cr
1^6.2     & .& .& .& .& .& .& .& .& .& .& .& .& 1& .& .& .& 1& 1& .& .& .& .& .& 1\cr
.2^4      & .& .& .& .& .& .& .& .& .& .& .& 1& .& .& .& .& .& .& .& .& .& .& 1& .& 1\cr
.2^21^4   & .& .& .& .& .& .& .& .& 1& .& .& .& .& .& .& .& 1& .& .& .& .& 1& .& .& y_3& 1\cr
.21^6     & .& .& .& .& .& .& .& .& .& .& .& .& .& .& .& .& 1& y_2& .& .& .& .& .& 1& y_3\pl2& 1& 1\cr
\noalign{\hrule}
 & ps& ps& ps& ps& ps& ps& ps& ps& ps& .1^2\!& ps& ps& ps& ps& .1^2\!& ps& ps& 1^6\!.& .1^2\!& .1^2\!& .1^2\!& .1^2\!& .1^2\!& A_5& .2^3\!& .1^2\!& .1^6\cr
   }}$$}
\smallskip
\end{table}

Here, the blocks are labelled by the 6-cuspidal unipotent characters of the
centraliser $\tw2D_2(q).(q^3+1)^2$, respectively $\tw2D_3(q).(q^3+1)^2$,
of a $\Phi_6$-torus of rank~2, in accordance with~\cite{BMM}.

\begin{proof}
First consider $\tw2D_8(q)$. All columns $\Psi_i$ in the principal block apart
from $\Psi_{23}$ and $\Psi_{25}$ are obtained by Harish-Chandra induction,
and we find two further projectives with unipotent parts
$$\begin{aligned}
  \tPsi_{23} =\,& [.32^2]+y_3[1.21^4]+y_4[1.1^6],\\
  \tPsi_{25} =\,& [1^7.]+y_1[1.21^4]+y_2[1.1^6].
\end{aligned}$$
Since $[1.21^4]$ and $[1^7.]$ lie in families which are not comparable, the
character $\tPsi_{25}$ must involve $y_1$ copies of $\Psi_{26}$ by $(T_\ell)$.
Then (HCr) shows that $y_1 =0$. 
\par
The columns in the block labelled $\binom{0\ 1\ 2}{1}$ except for the 20th and
25th are again obtained by Harish-Chandra induction, as well as
$$\begin{aligned}
  \Psi_{20}=\,& [21^5.]+[1^6.1]+y_2[.1^7],\\
  \Psi_{25}=\,& [.2^31]+y_3[.2^21^3]+y_4[.1^7].
\end{aligned}$$
\par
We now turn to $\tw2D_9(q)$. For the principal block with (HCi) we find all
columns $\Psi_i$ but $\Psi_1$, $\Psi_{11}$, $\Psi_{13}$, $\Psi_{22}$
and~$\Psi_{26}$, as well as the projectives
$$\begin{aligned}
  \tPsi_{22}=&\, [.3^22]+y_3[1^2.21^4]+y_4[1^2.1^6],\\
  \tPsi_{26}=&\, [1^8.]+y_2[1^2.1^6].
\end{aligned}$$
Furthermore, we get $\Psi_1+\Psi_{11}$, $\Psi_1+\Psi_5$ and
$\Psi_{11}+\Psi_{13}$,
which yield $\Psi_1$, $\Psi_{11}$ and $\Psi_{13}$ by triangularity.
\par
The block for $\tw2D_9(q)$ labelled $\binom{0\ 1\ 2}{2}$ is now obtained as
usually. An application of (HCr) shows that all columns correspond to PIMs.
This completes the proof of Propositions~\ref{prop:2D82D9,d=6}
and~\ref{prop:2D7,d=6}.
\end{proof}

%%%%%%%%%%%%%%%%%%%%%%%%%%%%
\section{Odd-dimensional orthogonal groups}

\begin{prop}   \label{prop:trees Bn d=6}
 Let $q$ be a prime power and $\ell$ a prime with $d_\ell(q)=6$.
 The Brauer trees of the unipotent $\ell$-blocks of $B_n(q)$ and of $C_n(q)$,
 $3\le n\le 6$, with cyclic defect are as given in Table~\ref{tab:Bn,d=6,def1}.
\end{prop}

Here $B_2$ denotes the Harish-Chandra series of the ordinary cuspidal unipotent
character of $B_2(q)$, and $B_3^s$ the cuspidal $\ell$-modular constituent of
the unipotent character $[B_2\co.1]$ of $B_3(q)$.

\begin{table}[ht]
\caption{Brauer trees for $B_n(q)$ and $C_n(q)$ ($3\le n\le6$), $7\le\ell| (q^2-q+1)$} \label{tab:Bn,d=6,def1}
$$\vbox{\offinterlineskip\halign{$#$
        \vrule height10pt depth 2pt width 0pt&& \hfil$#$\hfil\cr
B_3(q):& 3.& \vr& 2.1& \vr& 1.1^2& \vr& .1^3& \vr& \bigcirc& \vr& B_2\co.1& \vr& B_2\co1.\cr
\cr
B_4(q):& 4.& \vr& 2.2& \vr& 1.21& \vr& .21^2& \vr& \bigcirc& \vr& B_2\co.1^2& \vr& B_2\co1^2.\cr
\cr
 & 31.& \vr& 21.1& \vr& 1^2.1^2& \vr& .1^4& \vr& \bigcirc& \vr& B_2\co.2& \vr& B_2\co2.\cr
\cr
B_5(q):& 5.& \vr& 2.3& \vr& 1.31& \vr& .31^2& \vr& \bigcirc& \vr& B_2\co1.1^2& \vr& B_2\co1^2.1\cr
\cr
 & 41.& \vr& 21.2& \vr& 1^2.21& \vr& .21^3& \vr& \bigcirc& \vr& B_2\co.21& \vr& B_2\co21.\cr
\cr
 & 31^2.& \vr& 21^2.1& \vr& 1^3.1^2& \vr& .1^5& \vr& \bigcirc& \vr& B_2\co1.2& \vr& B_2\co2.1\cr
\cr
 & 32.& \vr& 22.1& \vr& 1^2.1^3& \vr& 1.1^4& \vr& \bigcirc& \vr& B_2\co.3& \vr& B_2\co3.\cr
\cr
 & 4.1& \vr& 3.2& \vr& 1.22& \vr& .221& \vr& \bigcirc& \vr& B_2\co.1^3& \vr& B_2\co1^3.\cr
\cr
B_6(q):& 5.1& \vr& 3.3& \vr& 1.32& \vr& .321& \vr& \bigcirc& \vr& B_2\co1.1^3& \vr& B_2\co1^3.1\cr
\cr
 & 42.& \vr& 2^2.2& \vr& 1^2.21^2& \vr& 1.21^3& \vr& \bigcirc& \vr& B_2\co.31& \vr& B_2\co31.\cr
\cr
 & 33.& \vr& 2^2.1^2& \vr& 21.1^3& \vr& 2.1^4& \vr& \bigcirc& \vr& B_2\co.4& \vr& B_2\co4.\cr
\cr
 & 41.1& \vr& 31.2& \vr& 1^2.2^2& \vr& .2^21^2& \vr& \bigcirc& \vr& B_2\co.21^2& \vr& B_2\co21^2.\cr
\cr
 & 4.1^2& \vr& 3.21& \vr& 2.2^2& \vr& .2^3& \vr& \bigcirc& \vr& B_2\co.1^4& \vr& B_2\co1^4.\cr
\cr
 & 321.& \vr& 2^21.1& \vr& 1^3.1^3& \vr& 1.1^5& \vr& \bigcirc& \vr& B_2\co1.3& \vr& B_2\co3.1\cr
  &&  ps&& ps&&  ps&& .1^3&& B_3^s&& B_2\cr
\cr\cr
B_5(q):& .5& \vr& 1.4& \vr& 1^2.3& \vr& 1^3.2& \vr& 1^4.1& \vr& 1^5.& \vr& \bigcirc\cr
  &&  ps&& ps&&  ps&& ps&& ps&& 1^5.\cr
  }}$$
\end{table}

We will encounter the following Hecke algebras:

\begin{lem}   \label{lem:paramBn,d=6}
 Let $q$ be a prime power and $\ell|(q^2-q+1)$. The Hecke algebras of various
 $\ell$-modular cuspidal pairs $(L,\la)$ of Levi subgroups $L$ in $B_n(q)$
 and their respective numbers of irreducible characters are as given in
 Table~\ref{tab:hecke Bn,d=6}.
\end{lem}

\begin{table}[ht]
\caption{Hecke algebras and $|\Irr\cH|$ in $B_n(q)$ for $d_\ell(q)=6$}   \label{tab:hecke Bn,d=6}
$\begin{array}{l|c|ccccc}
 (L,\la)& \qquad\qquad\cH\qquad\qquad\qquad & n=5& 6& 7& 8\cr
\hline
   (B_2,B_2)& \cH(B_{n-2};q^3;q)& 5+0& 6+2+2& 12+4& 13+15\cr
  (B_3,\vhi_{.1^3})& \cH(B_{n-3};q;q)& 5& 6+3& 12+6& 10+20\cr
 (B_3,B_3^s)& \cH(B_{n-3};q;q)& 5& 6+3& 12+6& 10+20\cr
   (A_5,\vhi_{1^6})& \cH(B_{n-5};q^3;q)& -& 1& 1+2& 5\cr
  (B_5,\vhi_{1^5.})& \cH(B_{n-5};q^3;q)& 1& 1& 1+2& 5\cr
\end{array}$
\end{table}

\begin{proof}
In all cases, reduction stability for the cuspidal Brauer characters follows
with Example~\ref{exmp:hecke}(a) as Sylow $\ell$-subgroups of $L$ are cyclic.
The cuspidal $\ell$-modular character $B_2$ is the $\ell$-modular reduction of
the cuspidal unipotent character of $B_2(q)$, thus its Hecke algebra is the
reduction of the one in characteristic~0. The Brauer tree in
Table~\ref{tab:Bn,d=6,def1} shows that
the cuspidal $\ell$-modular Brauer character $B_3^s$ of $B_3(q)$ is a
constituent of the $\ell$-modular reduction of an ordinary cuspidal character
$\la$ in the Lusztig series of an $\ell$-element with centraliser a torus of
order $q^3+1$. The minimal Levi overgroups are of types $B_4$ and $B_3A_1$,
hence the parameters of the Hecke algebra for $\la$ are as given.
\par
Again, by the shape of the Brauer tree the cuspidal $\ell$-modular
Brauer character $\vhi_{1^5.}$ of $B_5(q)$ is liftable to an ordinary
cuspidal character lying in the Lusztig series of an $\ell$-element with
centraliser $B_2(q)(q^3+1)$, and with centraliser $B_3(q)(q^3+1)$ in $B_6(q)$.
Thus its Hecke algebra in $B_6(q)$ is $\cH(B_1;q^3)$. Similarly, the cuspidal
$\ell$-modular Steinberg character $\vhi_{1^6}$ of $A_5(q)$ is liftable to an
ordinary cuspidal character in the Lusztig series of an $\ell$-element with
centraliser $(q^6-1)/(q-1)$, and with centraliser $\GU_2(q^3)$ in
$B_6(q)$. Thus again its Hecke algebra in $B_6(q)$ is $\cH(B_1;q^3)$.
\end{proof}

\begin{prop}   \label{prop:B6,d=6}
 Assume $(T_\ell)$. The decomposition matrix for the principal $\ell$-block
 of $B_6(q)$ for primes $\ell$ with $(q^2-q+1)_\ell>7$ is as given in
 Table~\ref{tab:B6,d=6}.
\end{prop}

\begin{table}[ht]
{\small\caption{$B_6(q)$, $(q^2-q+1)_\ell>7$}   \label{tab:B6,d=6}
$$\vbox{\offinterlineskip\halign{$#$\hfil\ \vrule height10pt depth3pt&
      \hfil\ $#$\ \hfil\vrule&& \hfil\ $#$\hfil\cr
       6.& 1& 1\cr
       .6& \hlf q& .& 1\cr
      51.& \hlf q& 1& .& 1\cr
      2.4& \hlf q^3& 1& 1& .& 1\cr
 B_2\co2^2.& \hlf q^3& .& .& .& .& 1\cr
      .51& \hlf q^4& .& 1& .& .& .& 1\cr
    41^2.& \hlf q^4& .& .& 1& .& .& .& 1\cr
     1.41& \hlf q^5& .& 1& .& 1& .& 1& .& 1\cr
 B_2\co21.1\!& \hlf q^5& .& .& .& .& 1& .& .& .& 1\cr
     21.3& q^5& 1& .& 1& 1& .& .& .& .& .& 1\cr
   1^2.31& \frth q^7& .& .& .& 1& .& .& .& 1& .& 1& 1\cr
   21^2.2& \frth q^7& .& .& 1& .& .& .& 1& .& .& 1& .& 1\cr
B_2\co1^2.2& \frth q^7& .& .& .& .& .& .& .& .& 1& .& .& .& 1\cr
      B_6& \frth q^7& .& .& .& .& .& .& .& .& .& .& .& .& .& 1\cr
    31^3.& \hlf q^9& .& .& .& .& .& .& 1& .& .& .& .& .& .& .& 1\cr
    .41^2& \hlf q^9& .& .& .& .& .& 1& .& 1& .& .& .& .& .& .& .& 1\cr
B_2\co2.1^2& \hlf q^9& .& .& .& .& .& .& .& .& 1& .& .& .& .& .& .& .& 1\cr
   21^3.1& \hlf q^{11}& .& .& .& .& .& .& 1& .& .& .& .& 1& .& .& 1& .& .& 1\cr
 B_2\co1.21\!& \hlf q^{11}& .& .& .& .& 1& .& .& .& 1& .& .& .& 1& .& .& .& 1& .& 1\cr
   1^3.21& q^{11}& .& .& .& .& .& .& .& .& .& 1& 1& 1& .& .& .& .& .& .& .& 1\cr
  1^4.1^2& \hlf q^{15}& .& .& .& .& .& .& .& .& .& .& .& 1& .& .& .& .& .& 1& .& 1& 1\cr
 B_2\co.2^2& \hlf q^{15}& .& .& .& .& 1& .& .& .& .& .& .& .& .& .& .& .& .& .& 1& .& .& 1\cr
    21^4.& \hlf q^{16}& .& .& .& .& .& .& .& .& .& .& .& .& .& .& 1& .& .& 1& .& .& .& .& 1\cr
    .31^3& \hlf q^{16}& .& .& .& .& .& .& .& 1& .& .& 1& .& .& .& .& 1& .& .& .& .& .& .& .& 1\cr
    .21^4& \hlf q^{25}& .& .& .& .& .& .& .& .& .& .& 1& .& .& .& .& .& .& .& .& 1& .& .& .& 1& 1\cr
     1^6.& \hlf q^{25}& .& .& .& .& .& .& .& .& .& .& .& .& .& 2\!& .& .& .& 1& .& .& 1& .& 1& .& .& 1\cr
     .1^6& q^{36}& .& .& .& .& .& .& .& .& .& .& .& .& .& .& .& .& .& .& .& 1& 1& 2& .& .& 1& .& 1\cr
\noalign{\hrule}
  & & ps& ps& ps& ps& B_2& ps& ps& ps& B_2& ps& ps& ps& B_3^s\!& c& ps& .1^3\!& B_3^s\!& ps& B_3^s\!& ps& A_5& c& 1^5\!.& .1^3\!& .1^3\!& c& c\cr
  }}$$}
\end{table}

\begin{proof}
The Hecke algebras for the cuspidal $\ell$-modular Brauer character
$1^5.$ of $B_5(q)$ and the cuspidal $\ell$-modular Steinberg character of
$A_5(q)$ were determined in Lemma~\ref{lem:paramBn,d=6}. As neither is
semisimple modulo~$\ell$, there is just one PIM of $B_6(q)$ in either series.
\par
All columns except for those numbered 14, 22, 26 and~27 are obtained by
Harish-Chandra induction from a Levi subgroup of type $B_5$. According to the
description of the Hecke algebras in Lemma~\ref{lem:paramBn,d=6}, with this we
have accounted for all PIMs from proper Harish-Chandra series, so the remaining
four PIMs must be cuspidal. By triangularity, their non-zero entries lie below
the diagonal, and we denote them by $a_1,\ldots,a_{13}$ for $\Psi_{14}$,
by $a_{14},\ldots,a_{18}$ for $\Psi_{22}$ and by $a_{19}$ for $\Psi_{26}$.
Using the order on families we actually have $a_i = 0$ for $i \in \{1,2,3,14,15\}$.
%It will be shown in the proofs for $B_7(q)$ and $B_8(q)$ in
%Proposition~\ref{prop:B78,d=6} that $a_i=0$ unless possibly for
%$i\in\{7,12,13,18\}$.
\par
To get relations we use the combination of (DL) and (Red). Let
$w$ be a Coxeter element. The coefficients of $R_w$ on $\Psi_{26}$ and
$\Psi_{27}$ are equal to $-a_{17}$ and $2+a_{16}-a_{18}+a_{17}a_{19}$ respectively.
Therefore by (DL) we must have $a_{17} = 0$ and $a_{18} \leq 2+a_{16}$. On the
other hand, if $(q^2-q+1)_\ell > 7$ and $\ell > 7$ then $(q^2-q+1)_\ell > 12$.
By Example \ref{exmp:reg}(a) there exists a regular $\ell$-element of $G^*$.
Therefore one can invoke (Red) to get $a_{18} \geq2 + a_{16}$ hence
$a_{18}=2+a_{16}$. We go on to the next Deligne--Lusztig character $R_w$ for which
$\Psi_{26}$ and $\Psi_{27}$ can potentially occur. It corresponds to
$w =  s_1s_2s_3s_2s_1s_2s_3s_4s_5s_6$.
By (DL), the coefficients on these two PIMs give $2-a_4-a_6+a_7+a_9-a_{12} \geq 0$ and
$$\begin{aligned}
-(1+a_{19})&\underbrace{(2-a_4-a_6+a_7+a_9-a_{12})}_{\geq0 \text{ by (DL)}}\\ & -\underbrace{(-2+2a_4+2a_5+2a_6-2a_7-2a_8-a_9+a_{10}-a_{11}+a_{12}+a_{13})}_{\geq 0 \text{ by (Red)}}\geq0. 
\end{aligned}$$
Therefore we must have $a_{12} = 2-a_4-a_6+a_7+a_9$ and
$a_{13}=2-2a_4-2a_5-2a_6+2a_7+2a_8+a_9-a_{10}+a_{11}-a_{12}$, hence
$a_{13} =-a_4-2a_5-a_6+ a_7+2a_8-a_{10}+a_{11}$. Finally, the coefficient of $\Psi_{27}$
in $R_w$ for $w= s_1s_2s_3s_4s_3s_2s_1s_2s_3s_4s_5s_6$ is $48-24a_{19}$.
By (DL) this forces $a_{19} \leq 2$. 
\par
It will be shown in the proofs for $B_7(q)$ and $B_8(q)$ in
Proposition~\ref{prop:B78,d=6} that $a_i=0$ for $i \in \{4,\ldots,11,16,19\}$.
All columns are indecomposable by (HCr), respectively by the decomposition
numbers for the principal series Hecke algebras.
\end{proof}

As already for $\tw2D_7(q)$ in order to obtain further information on the
decomposition matrix for $B_6(q)$ we will consider blocks of larger groups.

\begin{prop}   \label{prop:B78,d=6}
 Assume $(T_\ell)$. The decomposition matrices for the unipotent
 $\ell$-blocks of $B_7(q)$ and $B_8(q)$ of non-cyclic defect for primes
 $\ell$ with $(q^2-q+1)_\ell>7$ are as given in Tables~\ref{tab:B7,d=6},
 \ref{tab:B8,d=6,bl12}, \ref{tab:B8,d=6,bl34} and~\ref{tab:B8,d=6,bl56}, except
 that for the block labelled $\binom{0\ 1\ 2}{1\ 2}$ the 9th column, labelled
 $B_6$, might not be indecomposable. \par
 Here, $a$ is as in Proposition~\ref{prop:B6,d=6} and we have
 $b_3=2-b_1+b_2$.
\end{prop}

Here $B_6^t$ denotes the Harish-Chandra series of the cuspidal $\ell$-modular
constituent of the unipotent character $B_2\co.2^2$ of $B_6(q)$, $B_2^a$ the
series of the $\ell$-modular cuspidal unipotent character
$B_2\sqtens\vhi_{1^6}$ of $B_2A_5$.

\begin{table}[htbp]
\caption{$B_7(q)$, $(q^2-q+1)_\ell>7$, blocks $\binom{1}{}$ and $\binom{0\ 1}{1}$}   \label{tab:B7,d=6}
{\small
$$\vbox{\offinterlineskip\halign{$#$\hfil\ \vrule height9pt depth0pt&&
      \hfil\ $#$\hfil\cr
       7.& 1\cr
      52.& 1& 1\cr
      1.6& .& .& 1\cr
      2.5& 1& .& 1& 1\cr
  B_2\co32.& .& .& .& .& 1\cr
     1.51& .& .& 1& 1& .& 1\cr
     421.& .& 1& .& .& .& .& 1\cr
 B_2\co31.1& .& .& .& .& 1& .& .& 1\cr
    2^2.3& 1& 1& .& 1& .& .& .& .& 1\cr
    .51^2& .& .& .& .& .& 1& .& .& .& 1\cr
B_2\co1^2.3& .& .& .& .& .& .& .& 1& .& .& 1\cr
   2^21.2& .& 1& .& .& .& .& 1& .& 1& .& .& 1\cr
B_2\co3.1^2& .& .& .& .& .& .& .& 1& .& .& .& .& 1\cr
   321^2.& .& .& .& .& .& .& 1& .& .& .& .& .& .& 1\cr
   1.41^2& .& .& .& 1& .& 1& .& .& .& 1& .& .& .& .& 1\cr
 1^2.31^2& .& .& .& 1& .& .& .& .& 1& .& .& .& .& .& 1& 1\cr
 2^21^2.1& .& .& .& .& .& .& 1& .& .& .& .& 1& .& 1& .& .& 1\cr
 B_2\co1.31& .& .& .& .& 1& .& .& 1& .& .& 1& .& 1& .& .& .& .& 1\cr
  B_6\co1^2& .& .& .& .& .& .& .& .& .& .& .& .& .& .& .& .& .& .& 1\cr
 1^3.21^2& .& .& .& .& .& .& .& .& 1& .& .& 1& .& .& .& 1& .& .& .& 1\cr
  B_2\co.32& .& .& .& .& 1& .& .& .& .& .& .& .& .& .& .& .& .& 1& .& .& 1\cr
   1.31^3& .& .& .& .& .& .& .& .& .& 1& .& .& .& .& 1& 1& .& .& .& .& .& 1\cr
  2^21^3.& .& .& .& .& .& .& .& .& .& .& .& .& .& 1& .& .& 1& .& .& .& .& .& 1\cr
  1^4.1^3& .& .& .& .& .& .& .& .& .& .& .& 1& .& .& .& .& 1& .& .& 1& .& .& .& 1\cr
   1.21^4& .& .& .& .& .& .& .& .& .& .& .& .& .& .& .& 1& .& .& .& 1& .& 1& .& .& 1\cr
     1^7.& .& .& .& .& .& .& .& .& .& .& .& .& .& .& .& .& 1& .& 2& .& .& .& 1& 1& .& 1\cr
    1.1^6& .& .& .& .& .& .& .& .& .& .& .& .& .& .& .& .& .& .& .& 1& 2& .& .& 1& 1& .& 1\cr
\noalign{\hrule}
 & ps& ps& ps& ps& B_2\!& ps& ps& B_2\!& ps& .1^3\!& B_3^s\!& ps& B_3^s\!& ps& ps& ps& ps& B_3^s\!& B_6\!& ps& B_6^t\!& .1^3\!& 1^5\!.\!& A_5\!& .1^3\!& 1^6\!.\!& .1^6\!\cr
\omit& \vphantom{A}\cr
\omit& \vphantom{A}\cr
       .7& 1\cr
      6.1& .& 1\cr
     51.1& .& 1& 1\cr
      3.4& 1& 1& .& 1\cr
     31.3& .& 1& 1& 1& 1\cr
B_2\co2^21.& .& .& .& .& .& 1\cr
      .52& 1& .& .& .& .& .& 1\cr
   41^2.1& .& .& 1& .& .& .& .& 1\cr
     1.42& 1& .& .& 1& .& .& 1& .& 1\cr
   31^2.2& .& .& 1& .& 1& .& .& 1& .& 1\cr
B_2\co21^2.1\!& .& .& .& .& .& 1& .& .& .& .& 1\cr
    B_6\co2& .& .& .& .& .& .& .& .& .& .& .& 1\cr
   1^2.32& .& .& .& 1& 1& .& .& .& 1& .& .& .& 1\cr
B_2\co1^3.2& .& .& .& .& .& .& .& .& .& .& 1& .& .& 1\cr
     .421& .& .& .& .& .& .& 1& .& 1& .& .& .& .& .& 1\cr
   31^3.1& .& .& .& .& .& .& .& 1& .& 1& .& .& .& .& .& 1\cr
  1^3.2^2& .& .& .& .& 1& .& .& .& .& 1& .& .& 1& .& .& .& 1\cr
    31^4.& .& .& .& .& .& .& .& .& .& .& .& .& .& .& .& 1& .& 1\cr
B_2\co2.1^3& .& .& .& .& .& .& .& .& .& .& 1& .& .& .& .& .& .& .& 1\cr
   .321^2& .& .& .& .& .& .& .& .& 1& .& .& .& 1& .& 1& .& .& .& .& 1\cr
   21^4.1& .& .& .& .& .& .& .& .& .& 1& .& .& .& .& .& 1& .& 1& .& .& 1\cr
B_2\co1.21^2\!& .& .& .& .& .& 1& .& .& .& .& 1& .& .& 1& .& .& .& .& 1& .& .& 1\cr
  1^5.1^2& .& .& .& .& .& .& .& .& .& 1& .& .& .& .& .& .& 1& .& .& .& 1& .& 1\cr
B_2\co.2^21& .& .& .& .& .& 1& .& .& .& .& .& .& .& .& .& .& .& .& .& .& .& 1& .& 1\cr
  .2^21^3& .& .& .& .& .& .& .& .& .& .& .& .& 1& .& .& .& 1& .& .& 1& .& .& .& .& 1\cr
    1^6.1& .& .& .& .& .& .& .& .& .& .& .& 2& .& .& .& .& .& 1& .& .& 1& .& 1& .& .& 1\cr
     .1^7& .& .& .& .& .& .& .& .& .& .& .& .& .& .& .& .& 1& .& .& .& .& .& 1& 2& 1& .& 1\cr
\noalign{\hrule}
 & ps& ps& ps& ps& ps& B_2\!& ps& ps& ps& ps& B_2\!& B_6\!& ps& B_3^s\!& .1^3\!& ps& ps& 1^5\!.\!& B_3^s\!& .1^3\!& ps& B_3^s\!& A_5\!& B_6^t\!& .1^3\!& 1^6\!.\!& .1^6\!\cr
   }}$$}
\end{table}

\begin{table}[htbp]
{\small\caption{$B_8(q)$, $(q^2-q+1)_\ell>7$, blocks $\binom{2}{}$ and $\binom{0\ 1\ 2}{1\ 2}$}   \label{tab:B8,d=6,bl12}
$$\vbox{\offinterlineskip\halign{$#$\hfil\ \vrule height9pt depth0pt&&
      \hfil\ $#$\hfil\cr
       8.& 1\cr
      53.& 1& 1\cr
      2.6& 1& .& 1\cr
  B_2\co42.& .& .& .& 1\cr
     1.61& .& .& 1& .& 1\cr
 B_2\co41.1& .& .& .& 1& .& 1\cr
     431.& .& 1& .& .& .& .& 1\cr
     2.51& 1& .& 1& .& 1& .& .& 1\cr
    .61^2& .& .& .& .& 1& .& .& .& 1\cr
B_2\co4.1^2& .& .& .& .& .& 1& .& .& .& 1\cr
   2^2.31& 1& 1& .& .& .& .& .& 1& .& .& 1\cr
B_2\co1^2.4& .& .& .& .& .& 1& .& .& .& .& .& 1\cr
  3^21^2.& .& .& .& .& .& .& 1& .& .& .& .& .& 1\cr
   2.41^2& .& .& .& .& 1& .& .& 1& 1& .& .& .& .& 1\cr
  21.31^2& .& .& .& .& .& .& .& 1& .& .& 1& .& .& 1& 1\cr
 B_2\co1.41& .& .& .& 1& .& 1& .& .& .& 1& .& 1& .& .& .& 1\cr
  2^21.21& .& 1& .& .& .& .& 1& .& .& .& 1& .& .& .& .& .& 1\cr
2^21^2.1^2& .& .& .& .& .& .& 1& .& .& .& .& .& 1& .& .& .& 1& 1\cr
21^2.21^2& .& .& .& .& .& .& .& .& .& .& 1& .& .& .& 1& .& 1& .& 1\cr
  B_2\co.42& .& .& .& 1& .& .& .& .& .& .& .& .& .& .& .& 1& .& .& .& 1\cr
   B_6\co.2& .& .& .& .& .& .& .& .& .& .& .& .& .& .& .& .& .& .& .& .& 1\cr
   2.31^3& .& .& .& .& .& .& .& .& 1& .& .& .& .& 1& 1& .& .& .& .& .& .& 1\cr
 21^3.1^3& .& .& .& .& .& .& .& .& .& .& .& .& .& .& .& .& 1& 1& 1& .& .& .& 1\cr
  2^21^4.& .& .& .& .& .& .& .& .& .& .& .& .& 1& .& .& .& .& 1& .& .& .& .& .& 1\cr
   2.21^4& .& .& .& .& .& .& .& .& .& .& .& .& .& .& 1& .& .& .& 1& .& .& 1& .& .& 1\cr
    21^6.& .& .& .& .& .& .& .& .& .& .& .& .& .& .& .& .& .& 1& .& .& 2& .& 1& 1& .& 1\cr
    2.1^6& .& .& .& .& .& .& .& .& .& .& .& .& .& .& .& .& .& .& 1& 2& .& .& 1& .& 1& .& 1\cr
\noalign{\hrule}
 & ps& ps& ps& B_2\!& ps& B_2\!& ps& ps& .1^3\!& B_3^s\!& ps& B_3^s\!& ps& ps& ps& B_3^s\!& ps& ps& ps& B_6^t\!& B_6\!& \!.1^3\!& A_5\!& 1^5\!.& \!.1^3\!& \!1^6\!.& \!.1^6\!\cr
\omit& \vphantom{A}\cr
\omit& \vphantom{A}\cr
      .71& 1\cr
    6.1^2& .& 1\cr
   51.1^2& .& 1& 1\cr
      .62& 1& .& .& 1\cr
     3.41& 1& 1& .& .& 1\cr
    31.31& .& 1& 1& .& 1& 1\cr
     2.42& 1& .& .& 1& 1& .& 1\cr
B_2\co2^21^2.\!& .& .& .& .& .& .& .& 1\cr
 B_6\co1^2.& .& .& .& .& .& .& .& .& 1\cr
 41^2.1^2& .& .& 1& .& .& .& .& .& .& 1\cr
  31^2.21& .& .& 1& .& .& 1& .& .& .& 1& 1\cr
B_2\co21^3.1\!& .& .& .& .& .& .& .& 1& .& .& .& 1\cr
    21.32& .& .& .& .& 1& 1& 1& .& .& .& .& .& 1\cr
 21^2.2^2& .& .& .& .& .& 1& .& .& .& .& 1& .& 1& 1\cr
B_2\co1^4.2& .& .& .& .& .& .& .& .& .& .& .& 1& .& .& 1\cr
 31^3.1^2& .& .& .& .& .& .& .& .& .& 1& 1& .& .& .& .& 1\cr
    .42^2& .& .& .& 1& .& .& 1& .& .& .& .& .& .& .& .& .& 1\cr
 21^4.1^2& .& .& .& .& .& .& .& .& .& .& 1& .& .& 1& .& 1& .& 1\cr
   .32^21& .& .& .& .& .& .& 1& .& .& .& .& .& 1& .& .& .& 1& .& 1\cr
    31^5.& .& .& .& .& .& .& .& .& .& .& .& .& .& .& .& 1& .& .& .& 1\cr
B_2\co2.1^4& .& .& .& .& .& .& .& .& .& .& .& 1& .& .& .& .& .& .& .& .& 1\cr
   21^5.1& .& .& .& .& .& .& .& .& 2& .& .& .& .& .& .& 1& .& 1& .& 1& .& 1\cr
B_2\co1.21^3\!& .& .& .& .& .& .& .& 1& .& .& .& 1& .& .& 1& .& .& .& .& .& 1& .& 1\cr
  1^6.1^2& .& .& .& .& .& .& .& .& 2& .& .& .& .& 1& .& .& .& 1& .& .& .& 1& .& 1\cr
  .2^31^2& .& .& .& .& .& .& .& .& .& .& .& .& 1& 1& .& .& .& .& 1& .& .& .& .& .& 1\cr
B_2\co.2^21^2\!& .& .& .& .& .& .& .& 1& .& .& .& .& .& .& .& .& .& .& .& .& .& .& 1& .& .& 1\cr
     .1^8& .& .& .& .& .& .& .& .& .& .& .& .& .& 1& .& .& .& .& .& .& .& .& .& 1& 1& 2& 1\cr
\noalign{\hrule}
 & ps& ps& ps& ps& ps& ps& ps& B_2\!& B_6\!& ps& ps& B_2\!& ps& ps& B_3^s\!& ps& .1^3\!& ps& .1^3\!& 1^5\!.& B_3^s\!& 1^6\!.& B_3^s\!& A_5\!& \!.1^3\!& B_6^t\!& \!.1^6\!\cr
  }}$$}
\end{table}

\begin{table}[htbp]
{\small\caption{$B_8(q)$, $(q^2-q+1)_\ell>7$, blocks $\binom{0\ 2}{1}$ and $\binom{1\ 2}{0}$}   \label{tab:B8,d=6,bl34}
$$\vbox{\offinterlineskip\halign{$#$\hfil\ \vrule height9pt depth0pt&&
      \hfil\ $#$\hfil\cr
7.1& 1\cr
1.7& .& 1\cr
3.5& 1& 1& 1\cr
52.1& 1& .& .& 1\cr
B_2\co321.& .& .& .& .& 1\cr
421.1& .& .& .& 1& .& 1\cr
1.52& .& 1& 1& .& .& .& 1\cr
B_2\co31^2.1& .& .& .& .& 1& .& .& 1\cr
32.3& 1& .& 1& 1& .& .& .& .& 1\cr
321.2& .& .& .& 1& .& 1& .& .& 1& 1\cr
.521& .& .& .& .& .& .& 1& .& .& .& 1\cr
321^2.1& .& .& .& .& .& 1& .& .& .& 1& .& 1\cr
1.421& .& .& 1& .& .& .& 1& .& .& .& 1& .& 1\cr
B_6\co1.1& .& .& .& .& .& .& .& .& .& .& .& .& .& 1\cr
B_2\co1^3.3& .& .& .& .& .& .& .& 1& .& .& .& .& .& .& 1\cr
1^2.321& .& .& 1& .& .& .& .& .& 1& .& .& .& 1& .& .& 1\cr
B_2\co3.1^3& .& .& .& .& .& .& .& 1& .& .& .& .& .& .& .& .& 1\cr
321^3.& .& .& .& .& .& .& .& .& .& .& .& 1& .& .& .& .& .& 1\cr
2^21^3.1& .& .& .& .& .& .& .& .& .& 1& .& 1& .& .& .& .& .& 1& 1\cr
1.321^2& .& .& .& .& .& .& .& .& .& .& 1& .& 1& .& .& 1& .& .& .& 1\cr
B_2\co1.31^2& .& .& .& .& 1& .& .& 1& .& .& .& .& .& .& 1& .& 1& .& .& .& 1\cr
1^3.2^21& .& .& .& .& .& .& .& .& 1& 1& .& .& .& .& .& 1& .& .& .& .& .& 1\cr
B_2\co.321& .& .& .& .& 1& .& .& .& .& .& .& .& .& .& .& .& .& .& .& .& 1& .& 1\cr
1^5.1^3& .& .& .& .& .& .& .& .& .& 1& .& .& .& .& .& .& .& .& 1& .& .& 1& .& 1\cr
1.2^21^3& .& .& .& .& .& .& .& .& .& .& .& .& .& .& .& 1& .& .& .& 1& .& 1& .& .& 1\cr
1^7.1& .& .& .& .& .& .& .& .& .& .& .& .& .& 2& .& .& .& 1& 1& .& .& .& .& 1& .& 1\cr
1.1^7& .& .& .& .& .& .& .& .& .& .& .& .& .& .& .& .& .& .& .& .& .& 1& 2& 1& 1& .& 1\cr
\noalign{\hrule}
 & ps& ps& ps& ps& B_2\!& ps& ps& B_2\!& ps& ps& .1^3\!& ps& ps& B_6\!& B_3^s\!& ps& B_3^s\!& 1^5\!.& ps& .1^3\!& B_3^s\!& ps& B_6^t\!& A_5\!& .1^3\!& 1^6\!.& .1^6\!\cr
\omit& \vphantom{A}\cr
\omit& \vphantom{A}\cr
71.& 1\cr
62.& 1& 1\cr
1^2.6& .& .& 1\cr
B_2\co3^2.& .& .& .& 1\cr
21.5& 1& .& 1& .& 1\cr
1^2.51& .& .& 1& .& 1& 1\cr
42^2.& .& 1& .& .& .& .& 1\cr
B_2\co31.2& .& .& .& 1& .& .& .& 1\cr
2^2.4& 1& 1& .& .& 1& .& .& .& 1\cr
B_2\co21.3& .& .& .& .& .& .& .& 1& .& 1\cr
B_2\co3.21& .& .& .& .& .& .& .& 1& .& .& 1\cr
32^21.& .& .& .& .& .& .& 1& .& .& .& .& 1\cr
1^2.41^2& .& .& .& .& 1& 1& .& .& 1& .& .& .& 1\cr
B_2\co2.31& .& .& .& 1& .& .& .& 1& .& 1& 1& .& .& 1\cr
2^3.2& .& 1& .& .& .& .& 1& .& 1& .& .& .& .& .& 1\cr
2^31.1& .& .& .& .& .& .& 1& .& .& .& .& 1& .& .& 1& 1\cr
.51^3& .& .& .& .& .& 1& .& .& .& .& .& .& .& .& .& .& 1\cr
1.41^3& .& .& .& .& .& 1& .& .& .& .& .& .& 1& .& .& .& 1& 1\cr
2^31^2.& .& .& .& .& .& .& .& .& .& .& .& 1& .& .& .& 1& .& .& 1\cr
1^2.31^3& .& .& .& .& .& .& .& .& 1& .& .& .& 1& .& .& .& .& 1& .& 1\cr
B_6\co.1^2& .& .& .& .& .& .& .& .& .& .& .& .& .& .& .& .& .& .& .& .& 1\cr
B_2\co.3^2& .& .& .& 1& .& .& .& .& .& .& .& .& .& 1& .& .& .& .& .& .& .& 1\cr
1^3.21^3& .& .& .& .& .& .& .& .& 1& .& .& .& .& .& 1& .& .& .& .& 1& .& .& 1\cr
1^2.21^4& .& .& .& .& .& .& .& .& .& .& .& .& .& .& .& .& .& 1& .& 1& .& .& 1& 1\cr
1^4.1^4& .& .& .& .& .& .& .& .& .& .& .& .& .& .& 1& 1& .& .& .& .& .& .& 1& .& 1\cr
1^2.1^6& .& .& .& .& .& .& .& .& .& .& .& .& .& .& .& .& .& .& .& .& .& 2& 1& 1& 1& 1\cr
1^8.& .& .& .& .& .& .& .& .& .& .& .& .& .& .& .& 1& .& .& 1& .& 2& .& .& .& 1& .& 1\cr
\noalign{\hrule}
 & ps& ps& ps& B_2\!& ps& ps& ps& B_2\!& ps& B_3^s\!& B_3^s\!& ps& ps& B_3^s\!& ps& ps& .1^3\!& .1^3\!& 1^5\!.& ps& B_6\!& B_6^t\!& ps& .1^3\!& A_5\!& .1^6\!& 1^6\!.\cr
  }}$$}
\end{table}

\begin{table}[htbp]
{\small\caption{$B_8(q)$, $(q^2-q+1)_\ell>7$, blocks $\binom{0\ 1}{2}$ and $\binom{0\ 1\ 2}{}$}   \label{tab:B8,d=6,bl56}
$$\vbox{\offinterlineskip\halign{$#$\hfil\ \vrule height9pt depth0pt&&
      \hfil\ $#$\hfil\cr
        .8& 1\cr
       6.2& .& 1\cr
      51.2& .& 1& 1\cr
       4.4& 1& 1& .& 1\cr
      41.3& .& 1& 1& 1& 1\cr
       .53& 1& .& .& .& .& 1\cr
    41^2.2& .& .& 1& .& 1& .& 1\cr
  B_6\co2.& .& .& .& .& .& .& .& 1\cr
B_2\co2^3.& .& .& .& .& .& .& .& .& 1\cr
      1.43& 1& .& .& 1& .& 1& .& .& .& 1\cr
    41^3.1& .& .& .& .& .& .& 1& .& .& .& 1\cr
      .431& .& .& .& .& .& 1& .& .& .& 1& .& 1\cr
    31^3.2& .& .& .& .& 1& .& 1& .& .& .& 1& .& 1\cr
\!B_2\co21^2.1^2\!& .& .& .& .& .& .& .& .& 1& .& .& .& .& 1\cr
   1^2.3^2& .& .& .& 1& 1& .& .& .& .& 1& .& .& .& .& 1\cr
B_2\co1^3.21& .& .& .& .& .& .& .& .& .& .& .& .& .& 1& .& 1\cr
     41^4.& .& .& .& .& .& .& .& .& .& .& 1& .& .& .& .& .& 1\cr
B_2\co21.1^3& .& .& .& .& .& .& .& .& .& .& .& .& .& 1& .& .& .& 1\cr
    21^4.2& .& .& .& .& .& .& .& .& .& .& 1& .& 1& .& .& .& 1& .& 1\cr
   .3^21^2& .& .& .& .& .& .& .& .& .& 1& .& 1& .& .& 1& .& .& .& .& 1\cr
\!B_2\co1^2.21^2\!& .& .& .& .& .& .& .& .& 1& .& .& .& .& 1& .& 1& .& 1& .& .& 1\cr
   1^4.2^2& .& .& .& .& 1& .& .& .& .& .& .& .& 1& .& 1& .& .& .& .& .& .& 1\cr
    1^5.21& .& .& .& .& .& .& .& .& .& .& .& .& 1& .& .& .& .& .& 1& .& .& 1& 1\cr
     1^6.2& .& .& .& .& .& .& .& 2& .& .& .& .& .& .& .& .& 1& .& 1& .& .& .& 1& 1\cr
B_2\co.2^3& .& .& .& .& .& .& .& .& 1& .& .& .& .& .& .& .& .& .& .& .& 1& .& .& .& 1\cr
   .2^21^4& .& .& .& .& .& .& .& .& .& .& .& .& .& .& 1& .& .& .& .& 1& .& 1& .& .& .& 1\cr
     .21^6& .& .& .& .& .& .& .& .& .& .& .& .& .& .& .& .& .& .& .& .& .& 1& 1& .& 2& 1& 1\cr
\noalign{\hrule}
 & ps& ps& ps& ps& ps& ps& ps& B_6\!& \!B_2\!& ps& ps& .1^3\!& ps& B_2\!& ps& \!B_3^s\!& 1^5\!.& \!B_3^s\!& ps& .1^3\!& \!B_3^s\!& ps& A_5& 1^6\!.& B_6^t\!& .1^3\!& .1^6\!\cr
\omit& \vphantom{A}\cr
\omit& \vphantom{A}\cr
 B_2\co6.& 1\cr
B_2\co51.& 1& 1\cr
     4^2.& .& .& 1\cr
B_2\co41^2.& .& 1& .& 1\cr
     43.1& .& .& 1& .& 1\cr
   42.1^2& .& .& .& .& 1& 1\cr
 B_2\co.6& 1& .& .& .& .& .& 1\cr
B_2\co31^3.& .& .& .& 1& .& .& .& 1\cr
    3^2.2& .& .& .& .& 1& .& .& .& 1\cr
    32.21& .& .& 1& .& 1& 1& .& .& 1& 1\cr
   41.1^3& .& .& .& .& .& 1& .& .& .& .& 1\cr
  31.21^2& .& .& .& .& .& 1& .& .& .& 1& 1& 1\cr
B_2\co.51& 1& 1& .& .& .& .& 1& .& .& .& .& .& 1\cr
B_2\co21^4.& .& .& .& .& .& .& .& 1& .& .& .& .& .& 1\cr
  2^2.2^2& .& .& 1& .& .& .& .& .& .& 1& .& .& .& .& 1\cr
  21.2^21& .& .& .& .& .& .& .& .& 1& 1& .& 1& .& .& 1& 1\cr
    4.1^4& .& .& .& .& .& .& .& .& .& .& 1& .& .& .& .& .& 1\cr
   3.21^3& .& .& .& .& .& .& .& .& .& .& 1& 1& .& .& .& .& 1& 1\cr
 2.2^21^2& .& .& .& .& .& .& .& .& .& .& .& 1& .& .& .& 1& .& 1& 1\cr
 B_2\co1^6.& .& .& .& .& .& .& .& .& .& .& .& .& .& 1& .& .& .& .& .& 1\cr
B_2\co.41^2& .& 1& .& 1& .& .& .& .& .& .& .& .& 1& .& .& .& .& .& .& .& 1\cr
  1^2.2^3& .& .& .& .& .& .& .& .& 1& .& .& .& .& .& .& 1& .& .& .& .& .& 1\cr
   1.2^31& .& .& .& .& .& .& .& .& .& .& .& .& .& .& 1& 1& .& .& 1& .& .& 1& 1\cr
     .2^4& .& .& .& .& .& .& .& .& .& .& .& .& .& .& 1& .& .& .& .& .& .& .& 1& 1\cr
B_2\co.31^3& .& .& .& 1& .& .& .& 1& .& .& .& .& .& .& .& .& .& .& .& .& 1& .& .& b_1& 1\cr
B_2\co.21^4& .& .& .& .& .& .& .& 1& .& .& .& .& .& 1& .& .& .& .& .& .& .& .& .& b_2& 1& 1\cr
 B_2\co.1^6& .& .& .& .& .& .& .& .& .& .& .& .& .& 1& .& .& .& .& .& 1& .& .& .& b_3& .& 1& 1\cr
\noalign{\hrule}
 & B_2\!& B_2\!& ps& B_2\!& ps& ps& B_3^s\!& B_2\!& ps& ps& ps& ps& B_3^s\!& B_2\!& ps& ps& .1^3\!& .1^3\!& .1^3\!& B_2^a\!& B_3^s\!& .1^3\!& .1^3\!& c& \!B_3^s\!& \!B_3^s\!& c\cr
  }}$$}
\end{table}

\begin{proof}
Let us first consider the blocks of $B_7(q)$. In the principal block, (HCi)
yields all columns except for the 19th, and a projective character $\Psi$ with
unipotent part $a_4[2^21^2.1]+a_5[B_2\co1.31]+[B_6\co1^2]$
plus further constituents with larger $a$-value. By our assumption of
triangularity, $\Psi-a_4\Psi_{17}-a_5\Psi_{18}$
must then also be a projective character. By (HCr) this is only the case when
$a_4=a_5=0$. 
\par
Next, consider the blocks of $B_8(q)$. Here, the principal block shows that we
must have $a_6=a_8=0$, the block with label $\binom{0\ 1\ 2}{1\ 2}$ gives
$a_{16}=0$, and the block with label $\binom{1\ 2}{0}$ forces that
$a_9=a_{10}=a_{19}=0$. Thus we have obtained all information that had been
missing in the proof of Proposition~\ref{prop:B6,d=6}. The correctness of the
six printed decomposition matrices is now verified as in our previous proofs.
The relation between the $b_i$ in the block labelled $\binom{0\ 1\ 2}{}$ is
obtained from (DL) using the Coxeter element.
\par
Finally, to conclude that $a_{11}=0$ (resp. $a_7 =0$) we have to go up to the
unipotent block of $B_9(q)$ (resp.~$B_{10}(q)$) labelled by
$\binom{1\ 2\ 3}{0\ 1}$ (resp.~by $\binom{2\ 3}{0}$) and invoke ($T_l$)
(we do not print the corresponding decomposition matrices).
\end{proof}

%%%%%%%%%%%%%%%%%%%%%%%%%%%%
\section{Symplectic groups}

Finally, we consider the symplectic groups. The Brauer trees here are the same
as for the odd-dimensional orthogonal groups and had already been given in
Proposition~\ref{prop:trees Bn d=6}. The arguments used to determine the
decomposition matrices for blocks of defect 2 involve only Harish-Chandra
induction/restriction, unipotent Deligne--Lusztig characters (which depend
only on $(W,F)$) and the existence of $\ell$-regular elements (which are
similar for $G$ and $G^*$ when $\ell$ is odd, see Remark \ref{rem:dual}).
Consequently, under our assumptions on $\ell$, the unipotent part of the
decomposition matrices of the unipotent $\ell$-blocks of $C_n(q)$
and $B_n(q)$ are identical.

\begin{cor}   \label{cor:C6,d=6}
 Assume $(T_\ell)$. The decomposition matrices for the unipotent $\ell$-blocks
 of $C_n(q)$, $n=6,7,8$, of non-cyclic defect for primes $\ell$ with
 $(q^2-q+1)_\ell>7$ are the same as for $B_n(q)$ and hence as given in
 Tables~\ref{tab:B6,d=6}--\ref{tab:B8,d=6,bl56}, except possibly for the values
 of the yet unknown entries $b_1$ and $b_2$ (which need not be the same as for
 type $B_n$).
\end{cor}

%%%%%%%%%%%%%%%%%%%%%%%%%%%%
\section{Unipotent decomposition matrix of $F_4(q)$}

\begin{thm}   \label{thm:F4,d=6}
 Let $(q,6)=1$. The decomposition matrix for the principal $\ell$-block of
 $F_4(q)$ for primes $\ell\ge7$ with $(q^2-q+1)_\ell>7$ is approximated
 by Table~\ref{tab:F4,d=6}.

 Here, the unknown entries satisfy $y_6 = 1+y_1+y_2$, $z_3\leq 1$ and
 $x_1 + 3(z_1+z_2)\leq5$ (in particular all the $z_i$ are equal to either
 $0$ or $1$).
\end{thm}

\begin{table}[ht]
\caption{$F_4(q)$, $(q^2-q+1)_\ell > 7$, $(q,6)=1$}   \label{tab:F4,d=6}
$$\vbox{\offinterlineskip\halign{$#$\hfil\ \vrule height11pt depth4pt&
      \hfil\ $#$\ \hfil\vrule&& \hfil\ $#$\hfil\cr
  \phi_{1,0}&                                    1& 1\cr
 \phi_{2,4}'&                \hlf q\Ph4\Ph8\Ph{12}& .& 1\cr
\phi_{2,4}''&                \hlf q\Ph4\Ph8\Ph{12}& .& .& 1\cr
    B_2\co2.&               \hlf q\Ph1^2\Ph3^2\Ph8& .& .& .& 1\cr
 \phi_{8,3}'&                 q^3\Ph4^2\Ph8\Ph{12}& 1& 1& .& .& 1\cr
\phi_{8,3}''&                 q^3\Ph4^2\Ph8\Ph{12}& 1& .& 1& .& .& 1\cr
\phi_{12,4}&\frac{1}{24}q^4\Ph2^4\Ph3^2\Ph8\Ph{12}& 1& .& .& .& 1& 1& 1\cr
 \phi_{9,6}'&\frac{1}{8}q^4\Ph3^2\Ph4^2\Ph8\Ph{12}& .& 1& .& .& 1& .& .& 1\cr
\phi_{9,6}''&\frac{1}{8}q^4\Ph3^2\Ph4^2\Ph8\Ph{12}& .& .& 1& .& .& 1& .& .& 1\cr
     F_4[1]& \frac{1}{8}q^4\Ph1^4\Ph3^2\Ph8\Ph{12}& .& .& .& .& .& .& .& .& .& 1\cr
    B_2\co.2& \frth q^4\Ph1^2\Ph3^2\Ph4\Ph8\Ph{12}& .& .& .& 1& .& .& .& .& .& .& 1\cr
  B_2\co1^2.& \frth q^4\Ph1^2\Ph3^2\Ph4\Ph8\Ph{12}& .& .& .& 1& .& .& .& .& .& .& .& 1\cr
     F_4[-1]&   \frth q^4\Ph1^4\Ph3^2\Ph4^2\Ph{12}& .& .& .& .& .& .& .& .& .& .& .& .& 1\cr
  F_4[\ze_3]&      \thrd q^4\Ph1^4\Ph2^4\Ph4^2\Ph8& .& .& .& .& .& .& .& .& .& .& .& .& .& 1\cr
F_4[\ze_3^2]&      \thrd q^4\Ph1^4\Ph2^4\Ph4^2\Ph8& .& .& .& .& .& .& .& .& .& .& .& .& .& .& 1\cr
 \phi_{8,9}'&                 q^9\Ph4^2\Ph8\Ph{12}& .& .& .& .& 1& .& 1& 1& .& .& .& .& .& y_1& y_1& 1\cr
\phi_{8,9}''&                 q^9\Ph4^2\Ph8\Ph{12}& .& .& .& .& .& 1& 1& .& 1& .& .& .& .& y_2& y_2& .& 1\cr
\phi_{2,16}'&           \hlf q^{13}\Ph4\Ph8\Ph{12}& .& .& .& .& .& .& .& 1& .& .& .& .& .& y_1& y_1& 1& .& 1\cr
\phi_{2,16}''&          \hlf q^{13}\Ph4\Ph8\Ph{12}& .& .& .& .& .& .& .& .& 1& .& .& .& .& y_2& y_2& .& 1& .& 1\cr
  B_2\co.1^2&          \hlf q^{13}\Ph1^2\Ph3^2\Ph8& .& .& .& 1& .& .& .& .& .& .& 1& 1& .& 1&    1& .& .& .& .& 1\cr
 \phi_{1,24}&                                      q^{24}& .& .& .& .& .& .& 1& .& .& x_1& .& .& 2& y_6& y_6& 1& 1& z_1& z_2& z_3& 1\cr
\noalign{\hrule}
 \omit& & ps& ps& ps& B_2& ps& ps& ps& ps& ps& c& B_3& C_3& c& c& c& B_3& C_3& c& c& c& c\cr
   }}$$
\end{table}

\begin{proof}
K\"ohler \cite[Table~A.162]{Koe06} has shown uni-triangularity and determined
the $\ell$-modular Harish-Chandra series of the unipotent Brauer characters
under the stated assumptions on $\ell$.
He has also proved that the ordinary cuspidal characters $F_4[1]$ and $F_4[-1]$
can only occur in the reduction of the Steinberg character $\phi_{1,24}$.
We denote these multiplicities by $x_1$ and $x_2$. Similarly, he showed that
the cuspidal irreducible Brauer characters belonging to $\phi_{2,16}'$,
$\phi_{2,16}''$ and $B_2\co.1^2$ can occur only in the reduction of the
Steinberg character, and we denote by $z_1,z_2,z_3$ their respective
multiplicities. The projective characters corresponding to the Galois conjugate
cuspidal characters $F_4[\ze_3]$ and $F_4[\ze_3^2]$ can have potentially
several unipotent constituents. We will write the unipotent part of these
projective characters as
\begin{align*}
&& \Psi_{14} & = F_4[\ze_3] + y_1  \phi_{8,9}'+y_2\phi_{8,9}''+y_3\phi_{2,16}'
	+y_4\phi_{2,16}''+y_5B_2\co.1^2+y_6 \phi_{1,24} &&\\
\text{and} && \Psi_{15} & = F_4[\ze_3^2] + y_1  \phi_{8,9}'+y_2\phi_{8,9}''+y_3\phi_{2,16}'+y_4\phi_{2,16}''+y_5B_2\co.1^2+y_6 \phi_{1,24}. &&
\end{align*}
We now use (DL) to compute some of these coefficients. We start with the
Deligne--Lusztig character $R_w$ where $w$ is a Coxeter element.
The coefficients of $\Psi_{18}$, $\Psi_{19}$ and $\Psi_{20}$ in $R_w$ are
$2y_1-2y_3$, $2y_2-2y_4$ and $2-2y_5$ respectively and are all non-negative
by (DL). The coefficient
on $\Psi_{21}$ is
$$X = 2+2y_1+2y_2-2y_6 -z_1(2y_1-2y_3)-z_2(2y_2-2y_4)-z_3(2-2y_5) \geq 0.$$
Now if $(q^2-q+1)_\ell > 7$, we can use (Red) for the maximal torus of order
$(q^2-q+1)^2$ to obtain $-3-2y_1-2y_2+y_3+y_4+2y_5+y_6 \geq 0$. Adding twice
this non-negative number to $X$ we get
$$-(1+z_1)(2y_1-2y_3)-(1+z_2)(2y_2-2y_4)-(2+z_3)(2-2y_5) \geq 0$$
which forces $2y_1-2y_3=2y_2-2y_4=2-2y_5=0$ and $X = 0$. This gives
$$y_1=y_3,\quad y_2=y_4,\quad y_5 =1\quad\text{and}\quad y_6 = 1+y_1+y_2.$$
\par
We continue with the Deligne--Lusztig character associated with
$w =s_1s_2s_3s_4s_2s_3$. The coefficient of $\Psi_{21}$ in $R_w$ is $2-x_2$,
hence $x_2 \leq 2$ by (DL). On the other hand, (Red) applied to the case of a
torus gives $x_2 \geq 2$ hence $x_2 =2$.
\par
We finish with the Deligne--Lusztig character $R_w$ with
$w = s_1s_2s_3s_4s_1s_2s_3s_4$. The coefficient of the PIM $\Psi_{21}$ is
$25-x_1-9z_1-9z_2-14z_3$. Therefore by (DL) we must have $z_3 \leq 1$ and
$z_1+z_2 \leq 2$. If all of the $z_i$ are zero then we can only deduce that
$x_1 \leq 25$, which is not satisfactory. However, if one considers the
generalised $q^4$-eigenspace of $F$ on $R_w$, then the coefficient of
$\Psi_{21}$ is $5-x_1-3z_1-3z_2$, which gives $x_1 \leq 5$ if $z_1=z_2 =0$
or $x_1 \leq 2$ if $z_1$ or $z_2$ is non-zero.
\end{proof}

Again, we collect data on the $\Phi_6$-modular Harish-Chandra series in the
blocks $b$ considered in this section in Tables~\ref{tab:6-blocks},
\ref{tab:6-blocks,twisted} and~\ref{tab:6-blocks,Bn}. None of the decomposition
matrices determined here agree after any permutations of rows and columns,
so none of the blocks can be Morita equivalent.

\begin{table}[ht]
\caption{Modular Harish-Chandra-series in $\Phi_6$-blocks}  \label{tab:6-blocks}
$$\begin{array}{rc|cc|ccccccc}
 G& b& W_G(b)& |\IBr b|& ps& D_4& .1^4& A_5& .1^6& D_6^s& E_6\\
\hline
    D_6&    & G(6,2,2)& 18& 10& 2& 2& 2& 1& 1\\
    D_8&   2& G(6,2,2)& 18& 10& 2& 2& 2& 1& 1\\
    D_7&    & G(6,1,2)& 27& 14& 4& 4& 1& 2& 2\\
    D_8& 1,3& G(6,1,2)& 27& 14& 4& 4& 1& 2& 2\\
    E_6&    &      G_5& 21& 11& 2& 2& 1& & & 5\\
    E_8&   2&      G_5& 21& 11& 2& 2& 1& & & 5\\
\hline\hline
  &    &    &         & ps& B_2& B_3& C_3& c\\
\hline
    F_4&    &      G_5& 21&  8&  1& 2& 2& 8\\
\end{array}$$
\end{table}

\begin{table}[ht]
\caption{Modular Harish-Chandra-series in $\Phi_6$-blocks, twisted groups}  \label{tab:6-blocks,twisted}
$$\begin{array}{rc|cc|ccccccc}
 G& b& W_G(b)& |\IBr b|& ps& .1^2& A_5& 1^6.& .2^3& .1^6\\
\hline
\tw2D_7&    & G(6,1,2)& 27& 15& 8& 1& 1& 1& 1&\\
\tw2D_8& 1,2& G(6,1,2)& 27& 15& 8& 1& 1& 1& 1&\\
\tw2D_9& 1,2& G(6,1,2)& 27& 15& 8& 1& 1& 1& 1&\\
\end{array}$$
\end{table}

\begin{table}[ht]
\caption{Modular Harish-Chandra-series in $\Phi_6$-blocks, type $B_n$}  \label{tab:6-blocks,Bn}
$$\begin{array}{rc|cc|cccccccccccc}
 G& b& W_G(b)& |\IBr b|& ps& B_2& .1^3& B_3^s& A_5& B_6^t& 1^5\!.& B_6& 1^6\!.& .1^6& B_2^a& c\\
\hline
 B_6,B_7& & G(6,1,2)&  27& 13&  2&   3&   3&   1&   1&   1&   1&   1&   1\\
 B_8& 1$-$5& G(6,1,2)& 27& 13&  2&   3&   3&   1&   1&   1&   1&   1&   1\\
 B_8&    6& G(6,1,2)&  27&  9&  5&   5&   5&   0&   0&   0&   0&   0&   0& 1& 2\cr
\end{array}$$
\end{table}

%%%%%%%%%%%%%%%%%%%%%%%%%%%%%%%%%%%%%%%%%%%%%%%%%%%%%%%%%%%%%%%%%%%%%%%%%
\chapter{Decomposition matrices at $d_\ell(q)=5,7,8,10,12,14$}   \label{chap:d>=7}

In this final section we collect the  decomposition matrices for primes
$\ell$ with $d_\ell(q)\notin\{1,2,3,4,6\}$ for the classical groups considered
in this work. (The Brauer trees for exceptional groups can be found in
\cite{Cra12}.) In all cases, the corresponding cyclotomic polynomial divides
the order polynomial of the groups in question at most once, so all such blocks
are of cyclic defect. We thus give their Brauer trees; they were all first
determined by Fong and Srinivasan \cite{FS90}.

\begin{prop}   \label{prop:d=5}
 The Brauer trees of the unipotent $\ell$-blocks, for $11\le\ell|\Phi_5(q)$,
 of $B_5(q)$, $D_n(q)$ with $5\le n\le 7$ and of $\tw2D_n(q)$ with
 $6\le n\le7$, are as given in Table~\ref{tab:d=5}.
\end{prop}

\begin{table}[htb]
\caption{Brauer trees for $11\le\ell|\Phi_5(q)$} \label{tab:d=5}
{\small
$$\vbox{\offinterlineskip\halign{$#$
        \vrule height10pt depth 2pt width 0pt&& \hfil$#$\hfil\cr
D_5(q):& .5& \vr&  .41& \vr&  .31^2& \vr&  .21^3& \vr&  .1^5& \vr& \bigcirc\cr
\cr
D_7(q):& 1.6& \vr& 1.42& \vr& 1.321& \vr& 1.2^21^2& \vr& 1.1^6& \vr& \bigcirc\cr
  & & ps& & ps& & ps& & ps& & 1^4\cr
\cr
D_6(q):& .6& \vr& .42& \vr& .321& \vr& .2^21^2& \vr& .1^6& \vr& \bigcirc& \vr& 1.1^5& \vr& 1.21^3& \vr& 1.31^2& \vr& 1.41& \vr& 1.5\cr
\cr
D_7(q):& .7& \vr& .43& \vr& .3^21& \vr& .2^21^3& \vr& .21^5& \vr& \bigcirc& \vr& 1^5.2& \vr& 2.21^3& \vr& 2.31^2& \vr& 2.41& \vr& 2.5\cr
\cr
 & .61& \vr& .52& \vr& .32^2& \vr& .2^31& \vr& .1^7& \vr& \bigcirc& \vr& 1^2.1^5& \vr& 1^2.21^3& \vr& 1^2.31^2& \vr& 1^2.41& \vr& 1^2.5\cr
  & & ps& & ps& & ps& & ps& & 1^4& & 1^4& & ps& & ps& & ps& & ps\cr
\cr
B_5(q):& 5.& \vr& 41.& \vr& 31^2.& \vr& 21^3.& \vr& 1^5.& \vr& \bigcirc& \vr& .1^5& \vr& .21^3& \vr& .31^2& \vr& .41& \vr& .5\cr
\cr
\tw2D_6(q):& 5.& \vr& 41.& \vr& 31^2.& \vr& 21^3.& \vr& 1^5.& \vr& \bigcirc& \vr& .1^5& \vr& .21^3& \vr& .31^2& \vr& .41& \vr& .5\cr
\cr
\tw2D_7(q):& 6.& \vr& 42.& \vr& 321.& \vr& 2^21^2.& \vr& 1^6& \vr& \bigcirc& \vr& 1.1^5& \vr& 1.21^3& \vr& 1.31^2& \vr& 1.41& \vr& 1.5\cr
\cr
 & 5.1& \vr& 41.1& \vr& 31^2.1& \vr& 21^3.1& \vr& 1^5.1& \vr& \bigcirc& \vr& .1^6& \vr& .2^21^2& \vr& .321& \vr& .42& \vr& .6\cr
  & & ps& & ps& & ps& & ps& & 1^4& & 1^4& & ps& & ps& & ps& & ps\cr
  }}$$}
\end{table}

Here, ``$1^4$'' denotes the cuspidal Steinberg PIM of $A_4(q)$, and similarly in
the subsequent Brauer trees, the Harish-Chandra series are labelled by names
of unipotent characters in which the corresponding cuspidal Brauer character
first appears.

\newpage
\begin{table}[hbt]
\caption{Brauer trees for $17\le\ell|\Phi_{8}(q)$} \label{tab:d=8}
{\small
$$\vbox{\offinterlineskip\halign{$#$
        \vrule height9pt depth 2pt width 0pt&& \hfil$#$\hfil\cr
B_4(q):& 4.& \vr& 3.1& \vr& 2.1^2& \vr& 1.1^3& \vr& .1^4& \vr& \bigcirc& \vr& B_2\co.1^2& \vr& B_2\co1.1& \vr& B_2\co2.\cr
\cr
B_5(q):& 5.& \vr& 3.2& \vr& 2.21& \vr& 1.21^2& \vr& .21^3& \vr& \bigcirc& \vr& B_2\co.1^3& \vr& B_2\co1^2.1& \vr& B_2\co21.\cr
\cr
 & 41.& \vr& 31.1& \vr& 21.1^2& \vr& 1^2.1^3& \vr& .1^5& \vr& \bigcirc& \vr& B_2\co.21& \vr& B_2\co1.2& \vr& B_2\co3.\cr
 & & ps& & ps& & ps& & ps& & .1^4& & \!B_2\co.1^2\!& & B_2& & B_2\cr
\cr
D_5(q):& .5& \vr& 1.4& \vr& 1^2.3& \vr& 1^3.2& \vr& 1.1^4& \vr& .1^5& \vr& \bigcirc& \vr& D_4\co1^2& \vr& D_4\co2\cr
\cr
D_6(q):&  .6& \vr& 2.4& \vr& 21.3& \vr& 21^2.2& \vr& 21^3.1& \vr& 21^4.& \vr& \bigcirc& \vr& D_4\co.1^2& \vr& D_4\co1^2.\cr
\cr
 & .51& \vr& 1.41& \vr& 1^2.31& \vr& 1^3.21& \vr& 1^4.1^2& \vr& 1^6.& \vr& \bigcirc& \vr& D_4\co.2& \vr& D_4\co2.\cr
\cr
D_7(q):& .7& \vr& 3.4& \vr& 31.3& \vr& 31^2.2& \vr& 31^3.1& \vr& 31^4.& \vr& \bigcirc& \vr& D_4\co1.1^2& \vr& D_4\co1^2.1\cr
\cr
 & 1.6& \vr& 2.5& \vr& 2^2.3& \vr& 2^21.2& \vr& 2^21^2.1& \vr& 2^21^3.& \vr& \bigcirc& \vr& D_4\co.1^3& \vr& D_4\co1^3.\cr
\cr
 & .61& \vr& 2.41& \vr& 21.31& \vr& 21^2.21& \vr& 21^3.1^2& \vr& 21^5.& \vr& \bigcirc& \vr& D_4\co.21& \vr& D_4\co21.\cr
\cr
 & .52& \vr&  1.42& \vr& 1^2.32& \vr& 1^3.2^2& \vr& 1^5.1^2& \vr& 1^6.1& \vr& \bigcirc& \vr& D_4\co.3& \vr& D_4\co3.\cr
\cr
 & .51^2& \vr& 1.41^2& \vr& 1^2.31^2& \vr& 1^3.21^2& \vr& 1^4.1^3& \vr& 1^7.& \vr& \bigcirc& \vr& D_4\co1.2& \vr& D_4\co2.1\cr
 & & ps& & ps& & ps& & ps& & ps& & .1^5& & \!D_4\co.1^2\!& & D_4\cr
\cr
\tw2D_5(q):& 4.& \vr& 2.2& \vr& 1.21& \vr& .21^2& \vr& \bigcirc& \vr& .1^4& \vr& 1^2.1^2& \vr& 21.1& \vr& 31.\cr
\cr
\tw2D_6(q):& 5.& \vr& 2.3& \vr& 1.31& \vr& .31^2& \vr& \bigcirc& \vr& 1.1^4& \vr& 1^2.1^3& \vr& 2^2.1& \vr& 32.\cr
\cr
 & 4.1& \vr& 3.2& \vr& 1.2^2& \vr& .2^21& \vr& \bigcirc& \vr& .1^5& \vr& 1^3.1^2& \vr& 21^2.1& \vr& 31^2.\cr
\cr
\tw2D_7(q):& 6.& \vr& 2.4& \vr& 1.41& \vr& .41^2& \vr& \bigcirc& \vr& 2.1^4& \vr& 21.1^3& \vr& 2^2.1^2& \vr& 3^2.\cr
\cr
 & 51.& \vr& 21.3& \vr& 1^3.31& \vr& .31^3& \vr& \bigcirc& \vr& 1.21^3& \vr& 1^2.21^2& \vr& 2^2.2& \vr& 42.\cr
\cr
 & 5.1& \vr& 3.3& \vr& 1.32& \vr& .321& \vr& \bigcirc& \vr& 1.1^5& \vr& 1^3.1^3& \vr& 2^21.1& \vr& 321.\cr
\cr
 & 41^2.& \vr& 21^2.2& \vr& 1^3.21& \vr& .21^4& \vr& \bigcirc& \vr& .2^21^2& \vr& 1^2.2^2& \vr& 31.2& \vr& 41.1\cr
\cr
 & 4.1^2& \vr& 3.21& \vr& 2.2^2& \vr& .2^3& \vr& \bigcirc& \vr& .1^6& \vr& 1^4.1^2& \vr&21^3.1& \vr& 31^3.\cr
 & & ps& & ps& & ps& & .1^3& & .1^3& & ps& & ps& & ps\cr
\cr
\tw2D_4(q):& 3.& \vr& 2.1& \vr& 1.1^2& \vr& .1^3& \vr& \bigcirc\cr
\cr
\tw2D_6(q):& 41.& \vr& 21.2& \vr& 1^2.21& \vr& .21^3& \vr& \bigcirc\cr
 & & ps& & ps& & ps& & .1^3\cr
  }}$$}
\end{table}

\begin{prop}   \label{prop:d=8}
 The Brauer trees of the unipotent $\ell$-blocks, for $17\le\ell|\Phi_8(q)$,
 of $B_4(q)$, $B_5(q)$, $D_n(q)$ with $5\le n\le7$ and of $\tw2D_n(q)$ with
 $4\le n\le 7$, are as given in Table~\ref{tab:d=8}.
\end{prop}

\begin{prop}   \label{prop:d=7}
 The Brauer tree of the principal $\ell$-block of $D_7(q)$, for
 $29\le\ell|\Phi_7(q)$, is as given by
$$\vbox{\offinterlineskip\halign{$#$
        \vrule height10pt depth 2pt width 0pt&& \hfil$#$\hfil\cr
 & .7& \vr& .61& \vr& .51^2& \vr& .41^3& \vr& .31^4& \vr& .21^5& \vr& .1^7& \vr& \bigcirc\cr
 & & ps& & ps& & ps& & ps& & ps& & ps& & c\cr
  }}$$
\end{prop}

\begin{prop}   \label{prop:d=14}
 The Brauer tree of the principal $\ell$-block of $\tw2D_7(q)$, for
 $29\le\ell|\Phi_{14}(q)$, is as given by
 $$\vbox{\offinterlineskip\halign{$#$
        \vrule height10pt depth 2pt width 0pt&& \hfil$#$\hfil\cr
& 6.& \vr& 5.1& \vr& 4.1^2& \vr& 3.1^3& \vr& 2.1^4& \vr& 1.1^5& \vr& .1^6& \vr& \bigcirc\cr
  &&  ps&& ps&&  ps&& ps&& ps&& ps&& c\cr
  }}$$
\end{prop}

\begin{prop}   \label{prop:d=10}
 The Brauer trees of the unipotent $\ell$-blocks, for $11\le\ell|\Phi_{10}(q)$,
 of $B_5(q)$, $D_n(q)$ with $6\le n\le7$ and of $\tw2D_n(q)$ with
 $5\le n\le 7$, are as given in Table~\ref{tab:d=10}.
\end{prop}

\begin{table}[htb]
\caption{Brauer trees for $11\le\ell|\Phi_{10}(q)$} \label{tab:d=10}
{\tiny
$$\vbox{\offinterlineskip\halign{$#$
        \vrule height10pt depth 2pt width 0pt&& \hfil$#$\hfil\cr
B_5(q):& 5.& \vr& 4.1& \vr& 3.1^2& \vr& 2.1^3& \vr& 1.1^4& \vr& .1^5& \vr& \bigcirc& \vr& B_2\co.1^3& \vr& B_2\co1.1^2& \vr& B_2\co2.1& \vr& B_2\co3.\cr
 & & ps& & ps& & ps& & ps& & ps& & c& & c& & B_2& & B_2& & B_2\cr
\cr
D_6(q):& .6& \vr& 1.5& \vr& 1^2.4& \vr& 1^3.3& \vr& 1^4.2& \vr& 1^5.1& \vr& 1^6.& \vr& \bigcirc& \vr& D_4\co.1^2& \vr& D_4\co1.1& \vr& D_4\co2.\cr
\cr
D_7(q):& .7& \vr& 2.5& \vr& 21.4& \vr& 21^2.3& \vr& 21^3.2& \vr& 21^4.1& \vr& 21^5.& \vr& \bigcirc& \vr& D_4\co.1^3& \vr& D_4\co1^2.1& \vr& D_4\co21.\cr
\cr
 & .61& \vr& 1.51& \vr& 1^2.41& \vr& 1^3.31& \vr& 1^4.21& \vr& 1^5.1^2& \vr& 1^7.& \vr& \bigcirc& \vr& D_4\co.21& \vr& D_4\co1.2& \vr& D_4\co3.\cr
 & & ps& & ps& & ps& & ps& & ps& & ps& & 1^6.& & \!\!D_4\co.1^2\!\!& & D_4& & D_4\cr
\cr
 \tw2D_6(q):&  5.& \vr& 3.2& \vr& 2.21& \vr& 1.21^2& \vr& .21^3& \vr& \bigcirc& \vr& .1^5& \vr& 1^2.1^3& \vr& 21.1^2& \vr& 31.1& \vr& 41.\cr
\cr
\tw2D_7(q):& 6.& \vr& 3.3& \vr& 2.31& \vr& 1.31^2& \vr& .31^3& \vr& \bigcirc& \vr& 1.1^5& \vr& 1^2.1^4& \vr& 2^2.1^2& \vr& 32.1& \vr& 4.2\cr
\cr
 & 5.1& \vr& 4.2& \vr& 2.2^2& \vr& 1.2^21& \vr& .2^21^2& \vr& \bigcirc& \vr& .1^6& \vr& 1^3.1^3& \vr& 21^2.1^2& \vr& 31^2.1& \vr& 41^2.\cr
 & & ps& & ps& & ps& & ps& & .1^4& & .1^4& & ps& & ps& & ps& & ps\cr
\cr
\tw2D_5(q):& 4.& \vr& 3.1& \vr& 2.1^2& \vr& 1.1^3& \vr& .1^4& \vr& \bigcirc\cr
\cr
\tw2D_7(q):& 51.& \vr& 31.2& \vr& 21.21& \vr& 1^2.21^2& \vr& .21^4& \vr& \bigcirc\cr
 & & ps& & ps& & ps& & ps& & .1^4\cr
  }}$$}
\end{table}

\begin{prop}   \label{prop:d=12}
 The Brauer trees of the unipotent $\ell$-blocks, for $13\le\ell|\Phi_{12}(q)$,
 of $B_6(q)$, $D_7(q)$,$\tw2D_6(q)$ and $\tw2D_7(q)$ are as given in
 Table~\ref{tab:d=12}.
\end{prop}

\begin{table}[htb]
\caption{Brauer trees for $13\le\ell|\Phi_{12}(q)$} \label{tab:d=12}
$\vbox{\offinterlineskip\halign{$#$
        \vrule height10pt depth 2pt width 0pt&& \hfil$#$\hfil\cr
%%B_6(q):& 6.& \vr& 5.1& \vr& 4.1^2& \vr& 3.1^3& \vr& 2.1^4& \vr& 1.1^5& \vr& .1^6& \vr& \bigcirc& \vr& B_2\co.1^4& \vr& B_2\co1.1^3& \vr& B_2\co2.1^2& \vr& B_2\co3.1& \vr& B_2\co4.\cr
\cr
D_7(q):& .7& \vr& 1.6& \vr& 1^2.5& \vr& 1^3.4& \vr& 1^4.3& \vr& 1^5.2& \vr& 1^6.1& \vr& 1^7.\cr
 & & & & & & & & & & & & & & & \vert\cr
 & & & & & & & D_4\co3.& \vr& D_4\co2.1& \vr& D_4\co1.1^2& \vr& D_4\co.1^3& \vr& \bigcirc\cr
\cr\cr
\tw2D_6(q):& 5.& \vr& 4.1& \vr& 3.1^2& \vr& 2.1^3& \vr& 1.1^4& \vr& .1^5& \vr& \bigcirc\cr
  &&  ps&& ps&&  ps&& ps&& ps&& c\cr
\cr
\tw2D_7(q):&6.& \vr& 4.2& \vr& 3.21& \vr& 2.21^2& \vr& 1.21^3& \vr& .21^4& \vr& \bigcirc\cr
 & & & & & & & & & & & & & \vert\cr
 & & & 51.& \vr& 41.1& \vr& 31.1^2& \vr& 21.1^3& \vr& 1^2.1^4& \vr& .1^6\cr
\cr\cr  }}$
\end{table}

%%%%%%%%%%%%%%%%%%%%%%%%%%%%%%%%%%%%%%%%%%%%%%%%%%%%%%%%%%%%%%%%%%%%%%%%%
\chapter{On a conjecture of Craven}   \label{chap:craven}

Let $G=\bG^F$ for $\bG$ connected reductive defined over $\FF_q$. By \cite{Lu84}
for any unipotent character $\rho$ of $G$ there exists a \emph{degree polynomial}
$R_\rho\in\QQ[x]$ such that, in particular, $\rho(1) = R_\rho(q)$. The degree
(resp.~the valuation at $x$) of $R_\rho$ is denoted $A_\rho$ (resp.~$a_\rho$).
By the main theorem of \cite{BDT19}, under suitable conditions the decomposition
matrix of the unipotent blocks of $G$ is unitriangular with respect to the
ordering on families. This induces a bijection $\rho\mapsto\vhi_\rho$ between
the unipotent characters and the irreducible Brauer characters lying in
unipotent blocks. Since both the $a$-function and $A$-function are decreasing
with respect to the ordering on families, we have, for unipotent characters
$\rho \neq \rho'$
$$\langle \rho,\Psi_{\vhi_{\rho'}} \rangle \neq 0 \Longrightarrow
   a_\rho > a_{\rho'} \text{ and } A_\rho > A_{\rho'}. $$
In \cite{Cra12}, Craven predicts the existence of perverse equivalences between
unipotent $\ell$-blocks of $G$ and their Brauer correspondents, for primes
$\ell$ such that Sylow $\ell$-subgroups of $G$ are abelian. A consequence of
the existence of such equivalences is that the decomposition matrix would be
unitriangular with respect to the perversity function (see
\cite[Prop.~8.1]{CR17}).

Let us recall how the perversity function of Craven's conjectural perverse
equivalence is defined. For a root of unity $\zeta\in\mu(\CC)$ we denote by
$\Arg(\zeta)\in]0,2\pi]$ the argument of $\zeta$. We write $m(\zeta,R)$ for
the multiplicity of $\zeta$ as a root of a polynomial~$R$. For any positive
integer $d$ and unipotent character $\rho$ of $G$ we then define
$$\pi_d(\rho):= \frac{A_\rho+a_\rho}{d} + \frac{m(1,R_\rho)}{2} 
  + \sum_{\begin{subarray}{c} \zeta \in \mu(\CC) \smallsetminus \{1\} \\
  \Arg(\zeta)<2\pi/d \end{subarray} } m(\zeta,R_\rho).$$
It is shown in \cite{Cra12} that $\pi_d(\rho)-\pi_d(\rho')$ is an integer
whenever $\rho$ and $\rho'$ lie in the same $d$-Harish-Chandra series of $G$.
In particular $\pi_d(\rho)$ is an integer for every unipotent character $\rho$
in the principal $d$-Harish-Chandra series, as $\pi_d(1_G)=0$.

\begin{conj}[Craven]   \label{conj:craven}
 Let $\ell$ be a prime such that Sylow $\ell$-subgroups of $G$ are abelian, and
 let $d=d_\ell(q)$ be the order of $q$ in $\FF_\ell^\times$. Then for any two
 unipotent characters $\rho\ne\rho'$ of $G$ we have 
 $$\langle \rho,\Psi_{\vhi_{\rho'}} \rangle \neq 0 \Longrightarrow
   \pi_d(\rho) > \pi_d(\rho').$$
\end{conj}

When $d=2$, every non-real root $\zeta$ of $R_\rho$ satisfies
$\Arg(\zeta,R_\rho) < 2\pi/d$ or $\Arg(\overline{\zeta},R_\rho)< 2\pi/d$.
Since $R_\rho$ has real coefficients, this shows that 
$$\pi_2(\rho) = A_\rho - \frac{m(-1,R_\rho)}{2}.$$
In addition, for $\ell$ a good prime for $\bG$, $m(-1,R_\rho)$ is constant on
the unipotent characters in a fixed unipotent $\ell$-block since they form a
single 2-Harish-Chandra series, therefore in this case
Conjecture~\ref{conj:craven} follows from the result of \cite{BDT19}. 

On the other hand, when $d>2$ the order coming from $\pi_d$ might be quite
different from the one given by the families. As an example, let us consider the
principal $\Phi_6$-block of $B_6(q)$ whose decomposition matrix is given in
Table~\ref{tab:B6,d=6}. The unipotent characters
$$21^3.1,B_2\co1.21\!,1^3.21,1^4.1^2,B_2\co.2^2,21^4.,.31^3,.21^4,
 1^6., \text{ and }.1^6$$
all lie in families which are smaller than the family containing the cuspidal
unipotent character $B_6$. However the $\pi_6$-function on these characters is
given by
$$\begin{array}{c|ccccccccccc}
  \rho & B_6 & 21^3.1 & B_2\co1.21\! & 1^3.21 & 1^4.1^2 & B_2\co.2^2 & 21^4. & .31^3. & .21^4 & 1^6. & .1^6\\
\hline
  \pi_6(\rho) & 11 & 10 & 10 & 10 & 11 &  11 & 11 & 10 & 11 & 12 & 12
\end{array}$$
which predicts that the character of the PIM $P_{B_6}$ can only have $1^6.$ and
$.1^6$ as unipotent constituents.

We have checked that the non-zero entries in the decomposition matrices computed
in the previous chapters are compatible with Conjecture~\ref{conj:craven}. 
Furthermore, the conjecture predicts the following about yet unknown
decomposition numbers in our tables:

\begin{prop}   \label{prop:craven}
 Assume that Craven's Conjecture~\ref{conj:craven} holds. Then the following
 decomposition numbers in our tables vanish:
 \begin{itemize}
  \item for $E_6(q)$ with $d=3$: $b_1=0$,
  \item for $E_6(q)$ and $E_8(q)$ with $d=6$: $a_1=a_2=a_4=b_1=b_4=0$,
  \item for $\tw2D_7(q)$ with $d=6$, $\ell>7$: $y_2=y_3=0$,
  \item for $F_4(q)$ with $d=6$: $x_1=y_1=y_2=z_1=z_2=z_3=0$. 
 \end{itemize}
\end{prop}

Moreover, several of our proofs would become easier given the conjecture.
Observe that Sylow $\ell$-subgroups of $G$ are abelian in all cases of the
proposition. 
\medskip

In \cite{DM16} decomposition matrices for $d_\ell(q)=4$ were computed up
to a few unknown entries. Craven's conjecture also predicts that some 
of them vanish.

\begin{prop}   \label{prop:craven4}
 Assume that Craven's Conjecture~\ref{conj:craven} holds. Then the following
 decomposition numbers at $d_\ell(q)=4$ computed in \cite{DM16} vanish:
 \begin{itemize}
  \item for  $D_7(q)$: $b_1=b_3=b_4=b_5=b_7=c=d=0$,
  \item for $\tw2D_5(q)$, $\tw2D_7(q)$: $a=0$,
  \item for $\tw2E_6(q)$: $c_1=c_4=c_6=d_2=0$,
  \item for $C_4(q)$: $a=0$.
 \end{itemize}
\end{prop}

\begin{rem} \label{rem:craven4}
Note that we did not use the assumption $(T_\ell)$ when computing the matrices
in \cite{DM16}, but rather proved a weaker property for the matrices we
considered. Using $(T_\ell)$ for larger groups and (HCi), (HCr), we can actually
deduce that all the entries listed in Proposition~\ref{prop:craven4} vanish
except possibly for $c_4$ and $c_6$ in $\tw2E_6(q)$.
\end{rem}

%%%%%%%%%%%%%%%%%%%%%%%%%%%%%%%%%%%%%%%%%%%%%%%%%%%%%%%%%%%%%%%%%%%%%%%%%

\listoftables

\backmatter
\bibliographystyle{amsalpha}

\begin{thebibliography}{131}

\bibitem{BoMi11}
{\sc C. Bonnaf\'e, J. Michel}, Computational proof of the Mackey formula for
  $q > 2$. \emph{J. Algebra \bf327} (2011), 506--526.
 
\bibitem{BR03}
{\sc C. Bonnaf\'e, R. Rouquier}, Cat\'egories d\'eriv\'ees et vari\'et\'es de
  Deligne--Lusztig. \emph{Publ. Math. Inst. Hautes \'Etudes Sci. \bf97} (2003),
  1--59.

\bibitem{Bki456}
{\sc N. Bourbaki}, \emph{\'El\'ements de Math\'ematique. Fasc.~XXXIV. Groupes
  et Alg\`ebres de Lie. Chapitres IV, V et VI}. Hermann, Paris 1968.

\bibitem{BMM}
{\sc M. Brou\'e, G. Malle, J. Michel}, Generic blocks of finite reductive
  groups. \emph{Ast\'erisque \bf212} (1993), 7--92.

\bibitem{BDT19}
{\sc O. Brunat, O. Dudas, J. Taylor}, Unitriangular shape of decomposition
  matrices of unipotent blocks. Preprint, 2019.

\bibitem{CE94}
{\sc M. Cabanes, M. Enguehard}, On unipotent blocks and their ordinary
  characters. \emph{Invent. Math. \bf117} (1994), 149--164.

\bibitem{CE}
{\sc M.~Cabanes, M.~Enguehard}, \emph{Representation Theory of Finite
  Reductive Groups}. New Mathematical Monographs, 1, Cambridge University
  Press, Cambridge, 2004.

\bibitem{Ca}
{\sc R. Carter}, \emph{Finite Groups of Lie type: Conjugacy Classes and
  Complex Characters}. Wiley, Chichester, 1985.

\bibitem{CR17}
{\sc J. Chuang and R. Rouquier}, Perverse equivalences. Preprint, 2017.

\bibitem{Cra12}
{\sc D. Craven}, Perverse equivalences and Brou\'e's conjecture II: The cyclic
  case. Preprint, 2012. 

\bibitem{DLM14}
{\sc F. Digne, G. Lehrer, J. Michel}, On character sheaves and characters of
  reductive groups at unipotent classes. \emph{Pure Appl. Math. Q.} {\bf10}
  (2014), 459--512.

\bibitem{DM91}
{\sc F.~Digne, J.~Michel}, \emph{Representations of Finite Groups of Lie Type}.
  London Math. Soc. Student Texts, 21. Cambridge University Press, 1991.

\bibitem{DJ92}
{\sc R. Dipper, G. James}, Representations of Hecke algebras of type $B_n$.
  \emph{J. Algebra \bf146} (1992), 454--481.

\bibitem{DGHM}
{\sc R. Dipper, M. Geck, G. Hiss, G. Malle}, Representations of Hecke algebras
  and finite groups of Lie type. \emph{Algorithmic Algebra and Number Theory
 (Heidelberg, 1997)}, 331--378, Springer, Berlin, 1999. 

\bibitem{Du13}
{\sc O. Dudas}, A note on decomposition numbers for groups of Lie type of
  small rank. \emph{J. Algebra \bf388} (2013), 364--373.

\bibitem{DM14}
{\sc O. Dudas, G. Malle}, Projective modules in the intersection cohomology of
  Deligne--Lusztig varieties. \emph{C. R. Acad. Sci. Paris S\'er. I Math.
  \bf352} (2014), 467--471.

\bibitem{DM15}
{\sc O. Dudas, G. Malle}, Decomposition matrices for low rank unitary groups.
  \emph{Proc. London Math. Soc. \bf110} (2015), 1517--1557.

\bibitem{DM16}
{\sc O. Dudas, G. Malle}, Decomposition matrices for exceptional groups at
  $d=4$. \emph{J. Pure Appl. Algebra \bf220} (2016), 1096--1121.

\bibitem{FS}
{\sc P. Fong, B. Srinivasan}, The blocks of finite classical groups.
  \emph{J. reine angew. Math. \bf396} (1989), 122--191.

\bibitem{FS90}
{\sc P. Fong, B. Srinivasan}, Brauer trees in classical groups. \emph{J. Algebra
  \bf131} (1990), 179--225.

\bibitem{GeThesis}
{\sc M. Geck}, \emph{Verallgemeinerte Gelfand--Graev Charaktere und
  Zerlegungszahlen endlicher Gruppen vom Lie-Typ}. Dissertation, RWTH Aachen,
  1990.

\bibitem{Ge93}
{\sc M. Geck}, Basic sets of Brauer characters of finite groups of Lie type II.
  \emph{J. London Math. Soc. \bf47} (1993), 255--268.

\bibitem{Ge93b}
{\sc M. Geck}, The decomposition numbers of the Hecke algebra of type
  $E^\ast_6$. \emph{Math. Comp. \bf61} (1993), 889--899.

\bibitem{GH91}
{\sc M. Geck, G. Hiss}, Basic sets of Brauer characters of finite groups of
  Lie type. \emph{J. reine angew. Math. \bf418} (1991), 173--188.

\bibitem{GH97}
{\sc M. Geck, G. Hiss}, Modular representations of finite groups of Lie type
  in non-defining characteristic. \emph{Finite Reductive Groups (Luminy, 1994)},
  195--249, Progr. Math., 141, Birkh\"auser Boston, Boston, MA, 1997.

\bibitem{Chv}
{\sc M.~Geck, G.~Hiss, F.~L\"ubeck, G.~Malle, and G.~Pfeiffer},
  {\sf CHEVIE} -- A system for computing and processing generic character
  tables for finite groups of Lie type, Weyl groups and Hecke algebras.
  \emph{Appl. Algebra Engrg. Comm. Comput. \bf7} (1996), 175--210.

\bibitem{GHM}
{\sc M. Geck, G. Hiss, G. Malle}, Cuspidal unipotent Brauer characters.
  \emph{J. Algebra \bf168} (1994), 182--220.

\bibitem{GHM2}
{\sc M. Geck, G. Hiss, G. Malle}, Towards a classification of the irreducible
  representations in non-describing characteristic of a finite group of Lie
  type. \emph{Math. Z. \bf221} (1996), 353--386.

\bibitem{GJ11}
{\sc M. Geck, N. Jacon}, \emph{Representations of Hecke algebras at Roots of
  Unity}. Algebra and Applications, 15. Springer-Verlag London, London, 2011.

\bibitem{GM09}
{\sc M. Geck, J. M\"uller}, James' conjecture for Hecke algebras of exceptional
  type. I. \emph{J. Algebra \bf321} (2009), 3274--3298.

\bibitem{GP92}
{\sc M. Geck, G. Pfeiffer}, Unipotent characters of the Chevalley groups
  $D_4(q)$, $q$ odd. \emph{Manuscripta Math. \bf76} (1992), 281--304.

\bibitem{Gr74}
{\sc J.~A. Green}, Walking around the Brauer tree. \emph{J. Austral. Math.
  Soc. \bf 17} (1974), 197--213.

\bibitem{GrHi}
{\sc J. Gruber, G. Hiss}, Decomposition numbers of finite classical groups
  for linear primes. \emph{J. reine angew. Math. \bf485} (1997), 55--91.

\bibitem{HN14}
{\sc F. Himstedt, F. Noeske}, Decomposition numbers of $\SO_7(q)$ and
  $\Sp_6(q)$. \emph{J. Algebra \bf413} (2014), 15--40.

\bibitem{Ho80}
{\sc R. Howlett}, Normalizers of parabolic subgroups of reflection groups.
  \emph{J. London Math. Soc. \bf21} (1980), 62--80.

\bibitem{Ja05}
{\sc N. Jacon}, An algorithm for the computation of the decomposition matrices
  for Ariki-Koike algebras. \emph{J. Algebra \bf292} (2005), 100--109.

\bibitem{Ja90}
{\sc G. James}, The decomposition matrices of $\GL_n(q)$ for $n\le 10$.
  \emph{Proc. London Math. Soc. (3) \bf60} (1990), 225--265.

\bibitem{Koe06}
{\sc C. K\"ohler}, \emph{Unipotente Charaktere und Zerlegungszahlen der
  endlichen Chevalleygruppen vom Typ~$F_4$}. Dissertation, RWTH Aachen, 2006.

\bibitem{Lu78}
{\sc G. Lusztig}, \emph{Representations of Finite Chevalley Groups}. CBMS
  Regional Conference Series in Mathematics, 39. American Mathematical
  Society, Providence, R.I., 1978.

\bibitem{Lu84}
{\sc G. Lusztig}, \emph{Characters of Reductive Groups over a Finite Field}.
  Annals of Mathematics Studies, 107. Princeton University Press, Princeton,
  NJ, 1984.

\bibitem{Lu92}
{\sc G. Lusztig}, A unipotent support for irreducible representations.
  \emph{Adv. Math. \bf94} (1992), 139--179.

\bibitem{MT}
{\sc G. Malle, D. Testerman}, \emph{Linear Algebraic Groups and Finite Groups of
  Lie Type}. Cambridge Studies in Advanced Mathematics, 133, Cambridge
  University Press, 2011.

\bibitem{Mi15}
{\sc J. Michel}, The development version of the CHEVIE package of GAP3.
  \emph{J. Algebra \bf435} (2015), 308--336.

\bibitem{N19}
{\sc E. Norton}, On Harish-Chandra series of finite unitary groups and the
  $\mathfrak{sl}_\infty$-crystal on level $2$ Fock spaces. Preprint,
  arXiv:1908.09694, 2019.

\bibitem{Pa18}
{\sc A. Paolini},On the decomposition numbers of $\SO_8^+(2^f)$. \emph{J. Pure
  Appl. Algebra \bf222} (2018), 3982--4003.

\bibitem{SpB}
{\sc N. Spaltenstein}, \emph{Classes Unipotentes et Sous-groupes de Borel}.
  Lecture Notes in Mathematics, 946, Springer Verlag, 1982.

\bibitem{Sp74}
{\sc T. A. Springer}, Regular elements of finite reflection groups.
  \emph{Invent. Math. \bf25} (1974), 159--198.

\bibitem{Tay14}
{\sc J. Taylor}, Generalised Gelfand--Graev representations in small
  characteristics. \emph{Nagoya Math. J. \bf224} (2016), 93--167.
\end{thebibliography}
%%%%%%%%%%%%%%%%%%%%%%%%%%%%%%%%%%%%%%%%%%%%%%%%%%%%%%%%%%%%%%%%%%%%%%%%%

\end{document}